\documentclass[11p,english,pdflatex,memoir,report]{smfbook} 
\usepackage[latin1]{inputenc}
\usepackage[french,english]{babel} 
\usepackage{smfthm} 
\pdfoutput=1
\usepackage{multicol} 
\usepackage{graphicx, color, multicol}

\usepackage{cancel}

\usepackage{ulem}

 \renewcommand{\theequation}{\thesection.\arabic{equation}}
\definecolor{myblue}{rgb}{0,0,0.6}
\definecolor{myred}{rgb}{0.8,0,0}
\numberwithin{equation}{section}

\DeclareMathAlphabet{\mathscr}{OT1}{pzc}{m}{it} 

\usepackage{chngcntr}

\font\tenms=msbm10
\font\sevenms=msbm7
\font\fivems=msbm5
\newfam\bbfam \textfont\bbfam=\tenms \scriptfont\bbfam=\sevenms
\scriptscriptfont\bbfam=\fivems

\def\L{\Lambda}
\newtheorem{Thm}{Theorem}
\newtheorem{Thmapp}{Theorem}[chapter]
\newtheorem{Def}{Definition}[section]
\newtheorem{Prop}[Def]{Proposition}
\newtheorem{Cor}[Def]{Corollary}
\newtheorem{Lem}[Def]{Lemma}

\newtheorem{Rmk}[Def]{Remark}

\newtheorem{Notation}[Def]{Notation}
\counterwithout{figure}{chapter}

\def\bB {\mathbb{B}}

 \def\N {\mathbb{N}}
\def\R {\mathbb{R}}
\def\D {\mathbb{D}}
\def\T {\mathbb{T}}
\def\Z {\mathbb{Z}}

\def\cA {\mathcal{A}}
\def\cB {\mathcal{B}}
\def\cC {\mathcal{C}}

\def\cD {\mathcal{D}}

\def\cH {\mathcal{H}}
\def\cI {\mathcal{I}}
\def\cJ {\mathcal{J}}
\def\cG {\mathcal{G}}
\def\cL {\mathcal{L}}
\def\cF {\mathcal{F}}
\def\cM {\mathcal{M}}
\def\cN {\mathcal{N}}

\def\cP {\mathcal{P}}

\def\cZ {\mathcal{Z}}

\def\cR {\mathcal{R}}
\def\cT {\mathcal{T}}

\newcommand{\bbA}{{\ensuremath{\mathbb A}} }
\newcommand{\bbB}{{\ensuremath{\mathbb B}} }
\newcommand{\bbC}{{\ensuremath{\mathbb C}} }
\newcommand{\bbD}{{\ensuremath{\mathbb D}} }
\newcommand{\bbE}{{\ensuremath{\mathbb E}} }

\newcommand{\bbN}{{\ensuremath{\mathbb N}} }

\newcommand{\bbP}{{\ensuremath{\mathbb P}} }

\newcommand{\bbR}{{\ensuremath{\mathbb R}} }

\newcommand{\bbT}{{\ensuremath{\mathbb T}} }

\newcommand{\bbV}{{\ensuremath{\mathbb V}} }

\newcommand{\bbZ}{{\ensuremath{\mathbb Z}} }

\newcommand{\gga}{\gamma}            % \gg already exists...
\newcommand{\gd}{\delta}
\newcommand{\gep}{\varepsilon}       % \ge already exists...
\newcommand{\gp}{\varphi}
\newcommand{\gr}{\rho}
\newcommand{\gz}{\zeta}

\newcommand{\gP}{\Phi}

\newcommand{\gk}{\kappa}
\newcommand{\go}{\omega}

\newcommand{\gl}{\lambda}
\newcommand{\gL}{\Lambda}
\newcommand{\gs}{\sigma}
\newcommand{\gS}{\Sigma}

\def\b {{\beta}}

\def\eps {{\varepsilon}}
\def\e {{\varepsilon}}

\def\indc {{\bf 1}}

\def\la {\langle}
\def\ra {\rangle}
\def \La {\bigg\langle}
\def \Ra {\bigg\rangle}
\def \lA {\big\langle \! \! \big\langle}
\def \rA {\big\rangle \! \! \big\rangle}
\def \LA {\bigg\langle \! \! \! \! \! \; \bigg\langle}
\def \RA {\bigg\rangle \! \! \! \! \! \; \bigg\rangle}

\def\cum{{f^\eps_{n,[0,t]} (H^{\otimes n})}}
\def\lcum{{f_{n,[0,t]} (H^{\otimes n})}}
\def\cumt{{f^{\eps, m_0}_{n,[0,t]} (H^{\otimes n} )}}
\def\tcumt{{\tilde f^{\eps, m_0}_{n,[0,t]} (H^{\otimes n} )}}
\def\tcum{{\tilde f^{\eps}_{n,[0,t]} (H^{\otimes n} )}}
\def\lcumt{{\tilde f^{m_0}_{n,[0,t]} (H^{\otimes n} )}}

\def\d {{\partial}}

\newcommand{\ba}{\begin{aligned}}
\newcommand{\ea}{\end{aligned}}

\newcommand{\be}{\begin{equation}}
\newcommand{\ee}{\end{equation}}

\newcommand{\mb}{  \Delta  \hskip-.2cm  \Delta }

\makeatletter
\newcommand{\parttext}[1]{\def\@parttext{#1}}
\def\@endpart{\vskip 0pt plus 0.5fil
              \begin{formatparttext}
                \@parttext % on imprime le texte sp\'ecifique à une partie
                \gdef\@parttext{}% on vide le texte sp\'ecifique à une partie
             \end{formatparttext}
              \vskip 0pt plus 0.5fil
              \newpage
              \if@twoside
               \if@openright
                \null
                \thispagestyle{empty}%
                \newpage
               \fi
              \fi
              \if@tempswa
                \twocolumn
              \fi}
\makeatother

\newenvironment{formatparttext}{\begin{center}\itshape\small}{\end{center}}
 \makeindex

\begin{document}

\author[T. Bodineau]{Thierry Bodineau}
\address[T. Bodineau]%
{CMAP, CNRS, Ecole Polytechnique, I.P. Paris
\\
Route de Saclay, 91128 Palaiseau Cedex, FRANCE}
\email{thierry.bodineau@polytechnique.edu}

\author[I. Gallagher]{Isabelle Gallagher}
\address[I. Gallagher]%
{DMA, \'Ecole normale sup\'erieure, CNRS, PSL Research University\\
45 rue d'Ulm, 75005 Paris, FRANCE \\
and Universit\'e de Paris}
\email{gallagher@math.ens.fr}

\author[L. Saint-Raymond]{Laure Saint-Raymond}
\address[L. Saint-Raymond]
{UMPA UMR 5669 du CNRS, ENS de Lyon,Universit\'e  de Lyon\\
46  all\'ee d'Italie, 69007 Lyon,
FRANCE
}
\email{Laure.Saint-Raymond@ens-lyon.fr}

\author[S. Simonella]{Sergio Simonella}
\address[S. Simonella]
{UMPA UMR 5669 du CNRS, ENS de Lyon, Universit\'e de Lyon\\
46  all\'ee d'Italie, 69007 Lyon,
FRANCE
}
\email{sergio.simonella@ens-lyon.fr}

\title{Statistical dynamics of  a hard sphere gas:  \\
fluctuating Boltzmann equation and large deviations}

 \begin{abstract}
 We present a mathematical theory of dynamical fluctuations for the hard sphere gas in the Boltzmann-Grad limit.
 We prove that: (1) fluctuations of the empirical measure from the solution of the Boltzmann equation, scaled with the square root of the average number of particles, converge to a Gaussian process driven by the fluctuating Boltzmann equation, as predicted in~\cite{S81}; (2)  large deviations are exponentially small in the average number of particles and are characterized, under regularity assumptions, by a large deviation functional as previously obtained in~\cite{Rez2} for dynamics with stochastic collisions. The results are valid away from thermal equilibrium, but only for short times. Our strategy is based on uniform a priori bounds on the cumulant generating function, characterizing the fine structure of the small correlations. 
   \end{abstract}

\maketitle
 
%\frontmatter
%%%%%%%%%%%%%%%%%%%%%%%%%%%%%%%%%%%%%%%%%%%%%%%%%%%%%%

%\include{pref}
\tableofcontents
\mainmatter

%%%%%%%%%%%%%%%%%%%%%%%%%%%%%%%%%%%%%%%%%%%%%%%%%%%%%%%

%%%% Introduction & Definition of dynamical cumulants

\newpage
\chapter{Introduction}
\label{chap-intro}

This paper is devoted to a detailed analysis of the dynamical  correlations arising, at low density, in a
deterministic particle system obeying Newton's laws. In this chapter we   start  by defining our model precisely, 
and recalling the fundamental result of Lanford on the short-time derivation  of the Boltzmann equation, as a law of large numbers. 
After that, we   state our main results, Theorem \ref{thmTCL} and Theorem \ref{thmLD} below,
regarding    small fluctuations and large deviations of the empirical measure, respectively. 
Finally, the last part of this introduction describes the essential features of the proofs,   the organization of the paper, and presents some open problems. 

\section{The hard-sphere model with random initial data}
\setcounter{equation}{0}

 We consider a system of~$N\geq 0$ spheres of diameter~$\eps>0$ in the~$d$-dimensional torus~$ \T^{dN}$ with~$d \geq 2$. 
 The positions~$ ({\bf x}^{\e}_1,\dots,{\bf x}^{\e}_N) \in  \T^{dN}$ and velocities~$ ({\bf v}^{\e}_1,\dots,{\bf v}^{\e}_N) \in  \R^{dN}$ of the particles  satisfy Newton's laws 
 \begin{equation}
\label{hard-spheres1}
{d{\bf x}^{\e}_i\over dt} =  {\bf v}^{\e}_i \,,\quad {d{\bf v}^{\e}_i\over dt} =0 \quad \hbox{ as long as \ } |{\bf x}^{\e}_i(t)-{\bf x}^{\e}_j(t)|>\eps 
\quad \hbox{for \ } 1 \leq i \neq j \leq N
\, ,
\end{equation}
with specular reflection at collisions  %for two indexes~$i$ and~$j$ satisfying~$1 \leq i \neq j \leq N$, we have
\begin{equation}
\label{defZ'nij}
\begin{aligned}
\left. \begin{aligned}
 \left({\bf v}^{\e}_i\right)'& := {\bf v}^{\e}_i - \frac1{\eps^2} ({\bf v}^{\e}_i-{\bf v}^{\e}_j)\cdot ({\bf x}^{\e}_i-{\bf x}^{\e}_j) \, ({\bf x}^{\e}_i-{\bf x}^{\e}_j)   \\
\left({\bf v}^{\e}_j\right)'& := {\bf v}^{\e}_j + \frac1{\eps^2} ({\bf v}^{\e}_i-{\bf v}^{\e}_j)\cdot ({\bf x}^{\e}_i-{\bf x}^{\e}_j) \, ({\bf x}^{\e}_i-{\bf x}^{\e}_j)  
\end{aligned}\right\} 
\quad  \hbox{ if } |{\bf x}^{\e}_i(t)-{\bf x}^{\e}_j(t)|=\eps\,.
\end{aligned}
\end{equation}
Observe that these boundary conditions do not cover all possible situations, as for instance triple collisions  are excluded. Nevertheless the hard-sphere flow generated by \eqref{hard-spheres1}-\eqref{defZ'nij}
(free transport of $N$ spheres of diameter $\eps$, plus instantaneous reflection
 $$\big({\bf v}^{\e}_i,{\bf v}^{\e}_j\big) \to \Big(\big({\bf v}^{\e}_i\big)',\big({\bf v}^{\e}_j\big)' \Big)
 $$ at contact) is well defined on a full measure subset of ${\mathcal D}^{\eps}_{N} $ (see~\cite{Ale75}, or~\cite{GSRT} for instance) where~${\mathcal D}^{\eps}_{N} $ is the
canonical
phase space
$$\label{D-def}
{\mathcal D}^{\eps}_{N} := \big\{Z_N \in \D^N \, : \,  \forall i \neq j \, ,\quad |x_i - x_j| > \eps \big\} \, .
$$
We have denoted~$Z_N:= (X_N,V_N) \in ( \T ^d\times\R^d)^{N }$ the positions and velocities in the extended space~$\D^N:=( \T ^d\times\R^d)^{N }$\label{defDN}  with~$X_N := (x_1,\dots,x_N) \in \T^{dN}$ and~$V_N := (v_1,\dots,v_N) \in \R^{d N}$. We set~$Z_N =(z_1,\dots,z_N)$ with $z_i = (x_i,v_i)$.

The probability density  $W^{\eps}_N$\label{defWepsN} of finding $N$ hard spheres of diameter $\eps$ at configuration $Z_N$ at time $t$ is   governed by the Liouville equation\index{Liouville equation} in the~$2dN$-dimensional phase space
\begin{equation}
\label{Liouville}
	\d_t W^{\eps}_N +V_N \cdot \nabla_{X_N} W^{\eps}_N =0  \,\,\,\,\,\,\,\,\, \hbox{on } \,\,\,{\mathcal D}^{\eps}_{N}\, ,
\end{equation}
	with specular reflection on the boundary. If we denote\label{DepspmNij}
$$
	\begin{aligned}\d {\mathcal D}^{\eps \pm}_{N}(i,j):= \Big \{Z_N \in\D^N \, : \,  &|x_i-x_j| = \eps \, , \quad \pm (v_i-v_j) \cdot (x_i- x_j) >0 \\
& \mbox{and}  \quad\forall (k,\ell) \in  [1,N]^2\setminus \{i,j\}  , \,\,\, k \neq \ell \, ,\,\,\,     |x_k-x_\ell| > \eps\Big\} \, ,
\end{aligned}
$$
	then
\begin{equation}
\label{tracecondition}
	\forall Z_N\in  \d {\mathcal D}^{\eps +}_{N}(i,j)\, , i\neq j\, , \quad W^{\eps}_N (t,Z_N ) :=  W^{\eps}_N (t,Z_N^{'i,j}) \, ,
\end{equation}
where $Z^{'i,j}_N$\label{scattZnij} differs from $Z_N$ only by~$\left(v_i,v_j\right) \to \left(v'_i,v'_j\right)$, given by \eqref{defZ'nij}.

The canonical formalism consists in fixing the number~$N$ of particles, and in studying the  probability density~$W^{\eps}_N$ of particles in the state~$Z_N$ at time~$t$, as well as its marginals.
The main drawback of this formalism is that fixing the number of particles creates spurious correlations (see e.g.\,\cite{EC81,PS17}). 
We are  rather going to define   a particular class of distributions on the grand canonical phase space
$$
{\mathcal D}^\eps := \bigcup_{N\geq 0} {\mathcal D}^{\eps}_{N} \, ,
$$
where the  number of particles is not fixed but given by a modified Poisson law
(actually ${\mathcal D}^{\eps}_{N} = \emptyset$ for large~$N$).
For notational convenience, we   work with functions extended to zero over $\D^N \setminus \overline{{\mathcal D}^{\eps}_{N}}$.
 Given a probability distribution $f^0 :  \D\to \R $ satisfying
\begin{equation}
\label{lipschitz}
 |f^0(x,v)|  + | \nabla_x f^0(x,v)|\ \leq  C_0\,\exp \Big(- \frac{\beta_0} 2 |v|^2 \Big) 
\, , \quad C_0 \geq 1  \, , \, \, \beta_0>0 \,, 
 \end{equation}
 the initial probability density is defined on the configurations  $(N,Z_N) \in \D^\N$ as 
\begin{equation}
\label{eq: initial measure}
\frac{1}{N!}W^{\eps 0}_{N}(Z_N) 
:= \frac{1}{\cZ^ \eps} \,\frac{\mu_\eps^N}{N!} \, \prod_{i=1}^N f^0 (z_i) \,
\indc_{{\mathcal D}^{\eps}_{N}}(Z_{N})
 %\quad\mbox{on\ }{\mathcal D}^{\eps}_{N}
\end{equation} 
where $\mu_\e > 0$ and the normalization constant $\cZ^\eps$ is given by 
$$\label{def-Zeps}
\cZ^\eps :=  1 + \sum_{N\geq 1}\frac{\mu_\eps^N}{N!}  
\int_{\D^N} dZ_N \prod_{i=1}^N f^0 (z_i) \, \indc_{{\mathcal D}^{\eps}_{N}}(Z_{N})\; .
$$
Here and below, $\indc_A$ will be the indicator function of the set $A$. We will also use
the symbol $\indc_{``{\rm *}"}$ for the indicator function of the set defined by condition $``{\rm *}"$.

Note that in the chosen  probability measure, particles are ``exchangeable", in the sense that $W^{\eps 0}_{N}$ is invariant by permutation
of the particle labels in its argument. Moreover, the choice~(\ref{eq: initial measure}) for the initial data is the one   guaranteeing the ``maximal factorization'', in the sense that particles would be~i.i.d.\,were it not for the indicator
function (`hard-sphere exclusion').

Our fundamental random variable is the time-zero configuration,
consisting of the initial positions and velocities of all the particles of the gas.
We will denote~$\cN$ the total number of particles (as a random variable)
and~${\mathbf Z}^{\eps0}_{\cN} = \left(  \mathbf z^{\eps0}_i \right)_{i=1,\dots,\cN}$ the  initial particle
configuration. The particle dynamics 
\begin{equation}
\label{def: trajectory}
t \mapsto {\mathbf Z}^\eps_{\cN}(t) = \left( \mathbf z^\eps_i(t)\right)_{i = 1,\dots, \cN}
\end{equation}
is then given by the hard-sphere flow solving~\eqref{hard-spheres1}-\eqref{defZ'nij}
with random initial data ${\mathbf Z}^{\eps0}_{\cN}$ (well defined with probability 1).
The probability of an event~$X$ with respect to the measure  \eqref{eq: initial measure} will be denoted~${\mathbb P}_\eps(X)$\label{def-Peps}, and the corresponding expectation symbol will be denoted~${\mathbb E}_\eps$\label{def-Eeps}. Notice that particles are identified by their label, running from $1$ to $\cN$. We shall mostly deal with expectations of observables of type ${\mathbb E}_\eps\big(\sum_{i=1}^\cN \dots \big)$. Unless differently specified, we always imply that ${\mathbb E}_\eps\big(\sum_i \dots \big) = {\mathbb E}_\eps\big(\sum_{i=1}^\cN \dots \big)$.

The average total number of particles $\cN$ is fixed in such a way that 
\begin{equation}
\label{eq: Boltzmann Grad}
\lim_{\gep \to 0} {\mathbb E}_\eps  \left(  \cN  \right)  \eps^{  d-1 } =  1 \;.
\end{equation}
The limit \eqref{eq: Boltzmann Grad} ensures that the   {\it Boltzmann-Grad scaling} holds, i.e.\,that  the inverse mean free path is of order $1$ \cite{Gr49}.
Thus from now on we will set $$\mu_\eps = \eps^{- (d-1)}\;.$$

Let us define the rescaled initial $n$-particle correlation function\index{Correlation function!rescaled} \label{grandcanonical-initdata}
%The idea is to count how many groups of $j$ particles fall, in average, in a given configuration $Z_j = (z_1,\dots,z_j)$:
%%
%\begin{align*}
%F_j (Z_j):=\Big\langle \sum_{k_i\neq   k_\ell} \delta_{\zeta_{k_1}}(z_1)\dots 
%\delta_{\zeta_{k_j}}(z_j)\Big\rangle
%\end{align*}
%%
%where we are labelling the particles and indicating their (random) configuration by $\zeta_1,\dots,\zeta_N$, and
%the brackets denote the average with respect to the grand canonical state.
%The $j$-particle initial  correlation function  
%can be rewritten
%
\begin{align*}
F_n^{\eps 0} (Z_n)
:= \mu_\eps^{-n} \,
\sum_{p=0}^{\infty} \,\frac{1}{p!}\, \int_{ \D^p} dz_{n+1}\dots dz_{n+p} \,
%\indc_{{\mathcal D}^{\eps}_{n+p}}(Z_{n+p})\,
W_{n+p}^{\eps 0} (Z_{n+p}) 
%\quad\mbox{on\ }{\mathcal D}^{\eps}_{n}\,.
\;.
\end{align*}
%The rescaled $n$-particle correlation function is  equivalent to the marginal of order $n$ in the canonical setting. 
%In particular 
We say that the initial measure admits correlation functions when the series in the right-hand side is convergent, which is the case with our choice~(\ref{eq: initial measure}) of initial data, together with the series in the inverse formula
$$W_n^{\eps 0} (Z_n) = \mu_\eps^{n} \,
\sum_{p=0}^{\infty} \,\frac{(-\mu_\e)^p}{p!}\, \int_{ \D^p} dz_{n+1}\dots dz_{n+p} \,
F_{n+p}^{\eps 0} (Z_{n+p}) \, .
$$
In this case, the set of functions $\left(F_n^{\eps 0}\right)_{n \geq 1}$ describes all the properties of the system.

For any  test function $h_n: \D^n\rightarrow \bbR$, 
the following holds~:
\begin{equation}
\label{eq: marginal time 0}
\begin{aligned}
\bbE_\eps \Big( \sum_{\substack{i_1, \dots, i_n \\ i_j \neq i_k, j \neq k}} h_n \big( {\bf z}_{i_1}^{\e 0}, \dots ,  {\bf z}_{i_n}^{\e 0} \big) \Big)
& = \bbE_\eps \Big( \gd_{\cN \geq n} \frac{\cN!}{(\cN - n)!}  h_n \big( {\bf z}_{1}^{\e 0}, \dots ,  {\bf z}_{n}^{\e 0} \big) \Big)\\
& = \sum_{p = n}^\infty \int_{ \D^p} dZ_p \,\frac{W^{\eps 0}_{p}(Z_p)}{p!}\, \frac{p!}{(p-n)!}\,  h_n \big( Z_n \big)   \\
& =  \mu_\eps^{n} \int_{ \D^n} dZ_n \, F_n^{\eps 0} (Z_n )\, h_n( Z_n)   \;.
\end{aligned}
\end{equation}
%where the expectation is with respect to the measure \eqref{eq: initial measure} and we used the exchangeability 
%of the particles in the first equality.
%
Starting from the initial distribution $W_N^{\eps 0}$,
the density $W^{\eps}_N(t)$ evolves on ${\mathcal D}^{\eps}_{N}$ according to the Liouville equation~\eqref{Liouville} with specular boundary reflection   \eqref{tracecondition}.
At time $t \geq 0$, the (rescaled)~$n$-particle correlation function\index{Correlation function!$n$-particle} is defined 
%on ${\mathcal D}^{\eps}_{n}$ 
as 
\begin{align}
\label{eq: densities at t}
F^{\eps}_n (t,Z_n)
:= \mu_\eps^{-n} \,\sum_{p=0}^{\infty} \,\frac{1}{p!} \,\int_{ \D^p} dz_{n+1}\dots dz_{n+p} \,
%\indc_{{\mathcal D}^{\eps}_{n+p}}(Z_{n+p})\,
W_{n+p}^\eps (t, Z_{n+p})
\end{align}
and, as in \eqref{eq: marginal time 0}, we get 
\begin{equation}\label{eq: marginal time t}\begin{aligned}
\bbE_\eps \Big( \sum_{\substack{i_1, \dots, i_n \\ i_j \neq i_k, j \neq k}} h_n \big( {\mathbf z}^\eps_{i_1}(t), \dots ,  {\mathbf z}^\eps_{i_n}(t) \big) \Big)
=
\mu_\eps^{n} \int_{\D^n} dZ_n\, F^{\eps}_n (t, Z_n) \,h_n \big( Z_n \big)\;,
\end{aligned}
\end{equation}
where we used the notation \eqref{def: trajectory}. Notice that $F_n^\eps (t, Z_n) = 0$
for $Z_n \in \D^n \setminus \overline{{\mathcal D}^{\eps}_{n}}$.  

\section{Lanford's theorem : a law of large numbers}
\label{subsec: Generalized OU}
\setcounter{equation}{0}

In the Boltzmann-Grad limit $\mu_\eps \to \infty$, the average behavior is governed by the Boltzmann equation~:
%\begin{equation}
%%\label{boltzmann-eq}
%\left\{ \begin{aligned}
%& \d_t f +v \cdot \nabla _x f = \! \displaystyle\int_{\R^d}\int_{{\mathbb S}^{d-1}}   \! \Big(  f(t,x,w') f(t,x,v') - f(t,x,w) f(t,x,v)\Big) 
%\big ((v-w)\cdot \omega\big)_+ \, d\omega \,dw \,  ,\\
%&  f(0,x,v) = f^0(x,v)
% \end{aligned}
%  \right.\end{equation}
  \begin{equation}
\label{boltzmann-eq}
\left\{ \begin{aligned}
& \d_t f +v \cdot \nabla _x f = \! \displaystyle\int_{\D}\int_{{\mathbb S}^{d-1}}   \! \Big(  f(t,y,w') f(t,x,v') - f(t,y,w) f(t,x,v)\Big) 
 d\mu_{(x,v)}  (  (y,w),  \omega)\,,\\
&  f(0,x,v) = f^0(x,v)
 \end{aligned}
  \right.\end{equation}
where, for any $(x,v) \in \D$,
\begin{equation} \label{defdmuzi}
  d\mu_{(x,v)}  (  (y,w),  \omega):  = \delta_{y-x} \big ( (   w - v) \cdot \omega\big)_+ d\omega\,  
dy\, d w  
\end{equation}
and   where the precollisional velocities $(v',w')$ are defined by the scattering law
\begin{equation}\label{scattlaw}
 v' := v- \big( (v-w) \cdot \omega\big)\,  \omega \,  ,\qquad
 w' :=w+\big((v-w) \cdot \omega\big)\,  \omega  \, .
\end{equation}
More precisely, the convergence is described by Lanford's theorem~\cite{La75} (in the canonical setting --- for the grand-canonical setting see~\cite{Ki75}, where the case of smooth compactly supported potentials is also addressed),  which we state here in the case of the initial measure \eqref{eq: initial measure}.
\begin{Thm}[Lanford~\cite{La75}]
\label{thm: Lanford}
Consider a system of hard spheres initially distributed according to  the grand canonical measure {\rm(\ref{eq: initial measure})} with~$f^0 $ satisfying the estimate~{\rm(\ref{lipschitz})}. 
Then,  in the Boltzmann-Grad limit $\mu_\eps \to \infty$, the rescaled one-particle density~$F^\eps_1(t)$ converges uniformly on compact sets to  the solution~$f(t)$ of the Boltzmann equation {\rm(\ref{boltzmann-eq})}  on a time interval~$[0,T_0]$ (which depends only on~$f^0$ through $C_0,\beta_0$).
Furthermore for each~$n$, the rescaled $n$-particle correlation function $F^\eps_n(t)$ converges almost everywhere
in $\D^n$
to~$f^{\otimes n}(t)$ on the same time interval. 
\end{Thm}
We refer to \cite{IP89,S2,CIP94,CGP97} for detailed proofs. The topic continues to be studied and developed, see \cite{MT12, GSRT, denlinger, PS17,GG18,GG21, PS21} for more recent contributions.

%To simplify the previous theorem, the initial particle correlations are chosen to be only due to the non-overlapping condition on the hard spheres, which vanishes asymptotically as $\eps $ goes to~$ 0$. Thus, at time zero,  the particles  are chosen initially ``almost independently" with the same distribution $f^0$. 
 Let us define the empirical measure
\begin{equation}
\label{eq:empmeasnontested} 
\pi^\eps_t :=  \frac{1}{\mu_\eps} \sum_{i=1}^{\cN} \delta_{{\bf z}^{\eps}_i(t)}\,,
\end{equation}
where~$\delta_{{\bf z}^{\eps}_i(t)}$ denotes the Dirac mass at point~${{\bf z}^{\eps}_i(t)}$.
Tested on a (one-particle) function $h : \D \to \R$, it reads\label{empirical} 
\begin{equation}
\label{eq:empmeas} 
\pi^\eps_t(h) =  \frac{1}{\mu_\eps} \sum_{i=1}^{\cN} h\left( {\bf z}^{\eps}_i(t)\right)\,.
\end{equation}
By definition, $F_1^{\eps}$ describes the average behavior of (exchangeable) particles~:
\begin{equation}
\label{Eepsh1} 
\bbE_\eps\big(\pi^\eps_t(h)\big) =\int_{\D} F^{\eps}_1(t,z)\, h(z)\, dz\,.
\end{equation}
The propagation of chaos derived in Theorem~\ref{thm: Lanford} implies in particular  that  the empirical measure concentrates on the solution of Boltzmann equation: let us prove the following law of large numbers, which is an easy corollary to  Theorem~\ref{thm: Lanford}.
\begin{Cor}\label{LLNcorollary}
Under the assumptions of Theorem~{\rm\ref{thm: Lanford}},  for all~$\delta > 0$ 
and smooth $h : \D \to \R$,
$$\bbP_\eps \left( \Big|\pi^\eps_t(h) -  \int_{\D}  f (t,z)h(z) dz\Big| > \delta \right) \xrightarrow[\mu_\eps \to \infty]{} 0\;.
$$
\end{Cor}
\begin{proof}
Computing the variance for any test function $h$, we get that
\begin{equation}\begin{aligned}
\label{Eepsh2} 
& \bbE_\eps \Big( \big(\pi^\eps_t(h)- \int F^\eps_1(t,z)\, h(z)\, dz \big)^2 \Big)\\
%& \qquad 
%=  \bbE_\eps \left( \left( \frac{1}{\mu_\eps} \sum_{i=1}^{\cN} \Psi^\eps\left( z_i(t)\right) \right)^2 \right) 
%-   \left(\int F_1(t,z)\, \Psi^\eps(z)\, dz \right)^2 \\
&   
=   \bbE_\eps \Big( \frac{1}{\mu_\eps^2}  \sum_{i=1}^{\cN} h^2 \big( {\bf z}^{\eps}_i(t)\big) 
+ \frac{1}{\mu_\eps^2}  \sum_{i \not = j}  h \big( {\bf z}^{\eps}_i(t)\big) h \big( {\bf z}^{\eps}_j(t)\big)
\Big) -   \Big(\int F^\eps_1(t,z)\, h (z)\, dz \Big)^2 \\
&  
=  \frac{1}{\mu_\eps}  \int F^\eps_1(t,z)\, h^2(z)\, d z  
+ \int F^\eps_2(t,Z_2)\, h(z_1) h(z_2)\, d Z_2 -   \Big(\int F^\eps_1(t,z)\, h(z)\, dz \Big)^2
\xrightarrow[\mu_\eps \to \infty]{} 0\, ,
\end{aligned}
\end{equation}where the convergence to 0 follows from the fact that $F^\eps_2$ converges to $f^{\otimes 2}$ and $F^\eps_1$ to $f$
almost everywhere.\end{proof}

\begin{Rmk} \label{rem:Tetoile}
The   restriction to the time interval~$[0,T_0]$ in the statement of Theorem~{\rm{\ref{thm: Lanford}}}  originates from a Cauchy-Kovalevskaya argument 
 in a scale of  Banach spaces.
 A (non optimal) estimate of~$T_0$ in terms of~$C_0$ and~$\beta_0 $ is provided in Theorem~{\rm\ref{cumulant-thm1*}} of the present paper, of the form~$T_0\sim \,C_0^{-1}\beta_0^{(d+1)/2}$  (notice  that in this estimate the inverse temperature is given by~$\beta_0$, while the physical density is~$C_0/\beta_0^{\frac{d}{2}}$).    Remark that the  Cauchy-Kovalevskaya argument provides  the
 same dependence  in terms of~$C_0 $ and~$\beta_0$ for the wellposedness time of the Boltzmann equation: see Appendix~{\rm\ref{boltz-Cauchy}}.  \end{Rmk}

\section{The fluctuating Boltzmann equation}
\setcounter{equation}{0}

Describing the fluctuations around the Boltzmann equation  is a way to 
capture part of the  information which has been lost in the limit $\mu_\e \to \infty$.

As in the classical central limit theorem, we expect these fluctuations to be of order $1/\sqrt{\mu_\eps}$, which is the typical size 
%(in some suitable norms to be understood) 
of the remaining correlations.
We therefore define  the fluctuation field $\gz^\gep$ as follows: for any test function $h: \D\to\R$ (recall~(\ref{Eepsh1}))
\begin{equation}
\label{eq: fluctuation field}
\gz^\gep_t \big(  h  \big) :=  { \sqrt{\mu_\eps }} 
\left( \pi^\eps_t(h) -  \int \, F^\eps_1(t,z) \,  h \big(  z \big)\, dz   \right) \, .
\end{equation}

Initially the empirical measure starts close to the density profile $f^0$ and~$\gz^\gep_0$ converges in law towards a Gaussian white noise~$\gz_0$ with covariance 
\begin{equation}
\label{eq: fluctuation field at time zero}
\bbE \big( \gz_0(h_1)\, \gz_0(h_2) \big)
= \int h_1(z) \,h_2(z)\, f^0(z) \,dz\, .
\end{equation}
%and
%the fluctuation field measures how the initial randomness of the particle system evolves along the Newtonian dynamics.
%
This follows from a computation similar to 
 \eqref{Eepsh2} because, with our choice of initial data given in~\eqref{eq: initial measure}, $\mu_\eps\left( F_2^\eps\left(0\right) - \left(F_1^\eps\right)^{\otimes 2}\left(0\right)\right)$ vanishes as $\mu_\eps \to \infty$ (the Gaussian character requires an estimate of higher order cumulants, which is made precise in Proposition \ref{prop:ID} below). Note that, for more general initial states, a smoothly correlated part may appear in the covariance \cite{S83,PS17}.

\medskip
In this paper we   prove that in the limit  $\mu_\eps \to \infty$, starting from ``almost independent" hard spheres, $\zeta^\eps_t$ converges to a  Gaussian process, solving formally
\begin{equation}
\label{eq: OU}
d \gz_t   = \cL_t \,\gz_t\, dt + d\eta_t\,,
\end{equation}
where $\cL_t $ is the   {\it linearized Boltzmann operator} around the solution $f(t)$ of the Boltzmann equation~\eqref{boltzmann-eq} 
\begin{equation}
\label{eq: LBO}
\begin{aligned}
&\cL_t \,h(z) := - v \cdot \nabla_x h(z)+\int_{\D}\int_{{\mathbb S}^{d-1}}  \, d\mu_{z}  ( z_1,  \omega) \\
&   \qquad  \times \big( f (t,x_1,v_1') h(x,v') + f (t,x,v') h(x_1,v_1') 
   - f (t,z)  h(z_1) -  f (t,z_1)  h(z) \big) \, .
\end{aligned}
\end{equation}
The noise $d \eta_t(z)$ is Gaussian,  with zero mean and covariance 
\begin{equation}
\label{eq: final result covariance'}
\begin{aligned}
& \bbE \left( \int  dt_1\, dz_1   h_1 (z_1)  \eta_{t_1} (z_1)   \int dt_2 \, dz_2 \, h_2 (z_2) \eta_{t_2} (z_2) \right)
 \\ &\qquad \qquad= \frac{1}{2} \int  dt \, d\mu(z_1, z_2, \omega) 
 f (t, z_1)\, f (t, z_2) \Delta h_1 \, \Delta h_2
\end{aligned}
\end{equation}
denoting
\begin{equation}
\label{eq: measure mu}
d\mu (z_1, z_2, \omega): = \delta_{x_1 - x_2}  \,  \big( (v_1 - v_2) \cdot \omega \big)_+ d \omega\, d v_1\, d v_2 dx_1 
\end{equation}
and defining  for any $h$
\begin{equation}
\label{Delta-def}
  \Delta h (z_1, z_2, \omega) := h(z_1') +  h(z_2') -  h(z_1) -  h(z_2)\,, 
  \end{equation}
where~$z_i':=(x_i,v_i')$ with notation~(\ref{scattlaw}) for the velocities obtained after scattering.
We postpone the precise definition of a weak solution to~\eqref{eq: OU} to Section \ref{subsec: Functional setting}.

Our   result is   the following.  
\begin{Thm} 
\label{thmTCL}
Consider a system of hard spheres initially distributed according to  the grand canonical measure {\rm(\ref{eq: initial measure})} where $f^0$ is a function satisfying~{\rm(\ref{lipschitz})}. 
Then, there exists $T >0$ (depending   on $f^0$  as~$T \sim C_0^{-1}\beta_0^{\frac{d+1}2}$) such that, 
in the Boltzmann-Grad limit~$\mu_\eps \to \infty$, the fluctuation field $\left(\zeta^\eps_t\right)_{t \geq 0}$ converges in law to a  Gaussian process, uniquely determined by its covariance, which solves~\eqref{eq: OU} in a weak sense on the time interval $[0,T ]$. \end{Thm}

The convergence towards the limiting process \eqref{eq: OU} was conjectured  by Spohn in \cite{S83} and
the non-equilibrium covariance of the process at two different times was computed in \cite{S81}, see also \cite{S2}.
The noise emerges after averaging the deterministic microscopic dynamics.
It is white in time and space, but correlated in  
velocities so that momentum and energy are conserved.

At equilibrium the convergence of a discrete-velocity version of the same process was derived rigorously
by  Rezakhanlou in \cite{Rez}, starting from a dynamics with stochastic collisions (see also \cite{vK74,KL76,Tanaka82,Uchiyama83,Uchiyama88,meleard} for fluctuations and space-homogeneous models).

The physical aspects of the fluctuations for the rarefied gas have been thoroughly investigated in \cite{EC81,S81,S83}.  We also refer to \cite{BGSRSshort}, where we gave an outline of our results and strategy. Here we would like to recall only a few important features.

1) {\it The noise in \eqref{eq: OU} originates from dynamical correlations.}

\noindent
It is a very general fact that, when the macroscopic equation is dissipative, the dynamical equation for the fluctuations contains a term of noise. In the case under study, dynamical correlations correspond for example to two given particles having interacted directly or indirectly  backward in time on $[0,t]$ --- a precise, albeit technical definition will be given later on in terms of a suitable class of pseudo-dynamics (Definition \ref{def: external recollision} below). These correlations have a negligible contribution to the limit $\pi^\e_t \to f(t)$ (see  Corollary~\ref{LLNcorollary}).
The proof of Theorem \ref{thmTCL} provides a further insight on the relation between collisions and noise. Following \cite{S81}, we   represent the dynamics in terms of a special class of trajectories, for which one can classify precisely the dynamical correlations responsible for the term~$d\eta_t$; see Section~\ref{section1.5} for further explanations. For the moment we just remind the reader that there is no a priori contradiction between the dynamics being deterministic, and  the appearance of noise from collisions in the singular limit. Indeed when~$\eps$ goes to zero, the deflection angles are no longer deterministic (as in the probabilistic interpretation of the Boltzmann equation). The randomness, which is entirely coded on the initial data of the hard sphere system, is transferred to the dynamics in the limit.

2) {\it Equilibrium fluctuations can be deduced by the fluctuation-dissipation theorem.}

\noindent
As a particular case, we obtain the result at thermal equilibrium $f^0= M$, where $M$ is a Maxwellian. The stochastic process~\eqref{eq: OU} boils down to a generalized Ornstein-Uhlenbeck process. The
noise term compensates the dissipation induced by  the linearized Boltzmann operator, and the covariance of the noise~\eqref{eq: final result covariance'} can be predicted heuristically by using the invariant measure.
More precisely at equilibrium, one has the equation $d \gz_t   = \cL_{\rm eq} \,\gz_t\, dt + d\eta_t$ where $\cL_{\rm eq}$ is the linearized Boltzmann operator around $M$. To determine the structure of the Gaussian noise, one can formally express the time-independent quantity $\bbE \big( \gz_t(h_1)\, \gz_t(h_2) \big) = \int h_1 \,h_2\, M \,dz\,$ in terms of the initial fluctuations $\gz_0$, and of~$d\eta$. Using that $\cL_{\rm eq}$ is contracting, the limit $t \to \infty$ cancels the dependence on $\gz_0$ and provides formula~\eqref{eq: final result covariance'}, with $f=M$, for the covariance of the noise; see~\cite{S2} for details, and also Remark~\ref{rmk equiliibrium} page~\pageref{rmk equiliibrium}.

3) {\it Away from equilibrium, the fluctuating equations keep the same structure.}

\noindent
The most direct way   to guess \eqref{eq: OU}-\eqref{eq: final result covariance'} is starting from the equilibrium prediction (previous point) and {\it assuming} that $M= M(v)$ can be substituted with $f=f(t,x,v)$.
This heuristics is known as ``extended local equilibrium'' assumption, in the context of fluctuating hydrodynamics; we refer again to~\cite{S2} for details. The hypothesis is based on the remark that the noise 
in the fluctuating equation \eqref{eq: OU} should be white in space and time ($\delta-$correlated in $t$ and $x$) and therefore it should be determined completely by the local properties of the gas. If locally the system is at equilibrium, then the non equilibrium equation \eqref{eq: OU} should be simply the one obtained from the equilibrium equation by adjusting the local parameters.
This procedure turns out to give the right result also for our gas at low density, even if $f=f(t,x,v)$ is not locally Maxwellian. The reason is that a form of local equilibrium is still true, in terms of ideal gases; namely, around a little cube of volume $\mu_\eps^{-1}$ centered in $x$ at time~$t$, the hard sphere distribution converges, as $\mu_\eps\to \infty$, to a uniform Poisson measure with constant density $\int f(t,x,v) dv$ and independent velocities distributed according to $f(t,x,v)/\int f(t,x,v) dv$ (see Corollary~4.7 in~\cite{S2}).

4) {\it Away from equilibrium, fluctuations exhibit long range correlations.}

The covariance of the fluctuation field %(at equal times and) 
at different points $x_1, x_2$ is not zero when~$|x_1-x_2|$ is of order one (and decays slowly with $|x_1-x_2|$). 
At variance with \eqref{eq: fluctuation field at time zero} which is $\delta-$correlated, at positive times a smooth dynamical contribution  to the covariance emerges, which is non zero on macroscopic distances. This feature is typical of non equilibrium fluctuations as discussed in \cite{EC81}. 
%and it can be observed experimentally by means of light scattering on stationary states. 
In the hard sphere gas at low density, this dynamical contribution originates again from dynamical correlations. The proof of Theorem \ref{thmTCL} will provide an explicit formula describing this effect, showing that  the long range contribution to the covariance formula can be expressed in terms of dynamics involving correlations (see \cite{S81}, and Proposition \ref{Prop: covariance OU} page~\pageref{Prop: covariance OU}).

\begin{Rmk}
Note that a fluctuation theorem in the spirit of Theorem {\rm\ref{thmTCL}} was proved first in the context of a mean-field limit of Hamiltonian particle systems, interacting by means of smooth, weak and long-range forces \cite{BH77} (see also \cite{HL73,GV79} for early results on quantum mechanical models). However, this situation is deeply different from ours. The macroscopic limit is governed by the Vlasov equation, which is a reversible equation with no entropy production. Correspondingly, there is no dynamical noise in the fluctuating equation: the fluctuations evolve deterministically
according to the linearized Vlasov equation.
\end{Rmk}

\section{Large deviations}
\setcounter{equation}{0}

While typical fluctuations are of order $O(\mu_\eps^{-1/2})$, they may sometimes happen to be large, leading to a dynamics which is different from the Boltzmann equation.
A classical problem is  to evaluate the probability of such an atypical event, namely that the empirical measure 
remains close to a probability density $\gp \neq f$   during a time interval $[0,t]$.
The following explicit formula for the large deviation functional on $[0,t]$ was obtained by Rezakhanlou~\cite{Rez2} in the case of a one-dimensional stochastic dynamics mimicking the hard-sphere dynamics, and then 
conjectured for the deterministic hard-sphere dynamics in~\cite{RezLNM,bouchet}:
\begin{align}
\label{eq: fonc GD sup p}
\widehat \cF(t,\gp) 
& :=  \widehat \cF (0,\gp_0) + \sup_p \left\{ \int_0^{t} ds 
\, \left[\int_{\bbT^d} dx  \int_{\bbR^d} dv \; p(s,x,v)  \, D_s\gp(s,x,v) - \cH \big( \gp(s) ,p(s) \big) \right]
\right\},
\end{align}
where the supremum is taken over bounded measurable functions $p$,
%(cf. Equation (1.14) in \cite{Rez2}), 
and the Hamiltonian is given by
\begin{equation}
\label{eq: Hamiltonian}
\cH(\gp,p) := \frac{1}{2} \int  d\mu (z_1, z_2, \omega) 
\gp (z_1) \gp (z_2) \big( \exp \big( \Delta p (z_1, z_2)\big) -1 \big) \, ,
\end{equation}
with $d\mu$ and $\Delta  p $ defined in (\ref{eq: measure mu})-(\ref{Delta-def}). 
We have denoted~$D_t$   the transport operator
\begin{equation}
\label{eq: derivative transport}
D_t \gp (t,z) := \partial_t \gp(t,z) + v \cdot \nabla_x \gp(t,z)\;,
\end{equation}
  and finally
\begin{equation}
\widehat \cF (0,\gp_0) := \int_{\D} dz \; \left( \gp_0 \log \left( \frac{\gp_0}{f^0} \right)
-  \gp_0 +f^0 \right)
\end{equation}
with~$\gp_0 = \gp|_{t=0}$,  is the large deviation  rate for the empirical measure at time zero.

The functional~$ \widehat\cF (0)$ can be obtained by a standard procedure, modifying the measure \eqref{eq: initial measure}
in such a way to make the (atypical) profile~$\varphi_0$ typical\footnote{{In \cite{Scola}, at equilibrium, a derivation of large deviations by means of cluster expansion methods is discussed for a larger range of densities.}}. Similarly, to obtain the collisional term $\cH$ in $ \widehat \cF(t,\varphi)$, one would like to understand
the mechanism leading to an atypical path $\varphi = \varphi(s)$ at positive times. 
A serious difficulty then arises, due to the deterministic dynamics. 
%Note that collisions on scale $\e$ translate into an intricate geometry of the initial data.
Ideally, one should conceive a way of tilting the initial measure in order to observe a given
trajectory.  Whether such an efficient bias exists, we do not know. 
We shall proceed in a different way and 
 deduce the large deviations from the cumulant generating function 
\begin{equation}
\label{eq:Letehfirst}
\Lambda ^\eps_t (e^h) := \frac1{\mu_\eps} \log \bbE_\eps \Big( \exp \big( \mu_\eps \, \pi^\eps_t(h) \big) \Big)
\end{equation}
in the spirit of the G\"{a}rtner-Ellis Theorem
which is classical in the large deviation theory \cite{dembozeitouni}.
In this approach, the main difficulty is the explicit characterization of the cumulant generating function
which requires to control the dynamics at {\it all} scales in $\e$.
For our purpose, we will actually need to sample the empirical measure on the whole interval $[0,t]$ and not only at time $t$, which will be implemented by a more general functional (see Eq.\,\eqref{Ieps-def dynamics} below).

%Define the class of admissible conjugate functions  {\color{blue} \tt on a un peu chang\'e les notations dans le chapitre 7, il faudra retirer cette def}
%\begin{equation}
%\label{eq: set hat Br}
%\widehat \bB_r := \left\{ p \in C^1\left([0,T]\times \D\right)  \, \Big| \quad  \|p\|_{W^{1,\infty} ([0,T]\times \D)} < r \right\}
%\end{equation}
%for some $r>0$ and a time $T$ which will be tuned later.
We will be able to evaluate the asymptotic probability of observing any trajectory $\varphi$ satisfying
$D_t\varphi = \frac{\partial\cH}{\partial p}$, namely the biased Boltzmann equation
\begin{equation}
\label{biased-Boltz}
\begin{aligned}
 D_t\varphi = \int_{\D}\int_{{\mathbb S}^{d-1}}& \Big( \varphi  (t,y,w') \varphi (t,x,v') e^{-\Delta p (t,x,v,y,w,\omega)}    \\
  &- \varphi  (t,y,w) \varphi (t,x,v) e^{\Delta p (t,x,v,y,w,\omega)}   \Big)  \, d\mu_{(x,v)}  (  (y,w),  \omega) 
   \end{aligned}
\end{equation}
for some Lipschitz~$p$, and with initial data
\begin{equation}
\label{biased-data}
 \varphi(0,x,v)  = f^0 (x,v) \,e^{p(0,x,v)} \,.
\end{equation}
It is known indeed (see \cite{Rez2}) that 
\eqref{biased-Boltz} allows to code a large class of macroscopic profiles which can be attained in a  large deviation regime. The perturbed equation~\eqref{biased-Boltz} describes a collision process with biased transition rate.

  It can be proved easily (see Chapter \ref{LDP-chap} and  Appendix \ref{CauchyKol}) that \eqref{biased-Boltz}, in mild form, has a unique solution in the class of continuous functions with Gaussian decay in $v$. Such solutions will be called strong solutions.

Consider~$\cM (\bbD)$   the set of positive measures on $\D$\label{def-calM} with finite mass (metrized with the topology of weak convergence).
Define the  set of trajectories in $[0,t]$ taking values in $\cM (\bbD)$  as the Skorokhod space~$D([0,t], \cM(\bbD))$\label{def-Skorokhod}    and denote by $d_{[0,t]}$ the corresponding distance (see \cite{Billingsley} page 121).
The large deviation theorem states as follows -- a more complete version is proved in Chapter~\ref{LDP-chap}  (see Theorems~\ref{thm: F = ha F} and \ref{thmLDbis}).

\begin{Thm} 
\label{thmLD}
Consider a system of hard spheres initially distributed according to  the grand canonical measure {\rm(\ref{eq: initial measure})} where $f^0$ satisfies~{\rm(\ref{lipschitz})}. 
For any $r>0$, there exists a time $T >0$ (depending only on $C_0,\beta_0,r$) such that the following holds.
Define
$$
\begin{aligned}
\cR_{r,T} := \Big\{ \varphi :  [0,T] \times \bbD \mapsto \bbR^+ \, : \,  \varphi \hbox{ is the strong solution of   \eqref{biased-Boltz}-\eqref{biased-data} on $[0,T]$ for some  }  p \\   \hbox{ such that } \, \,  \|p\|_{W^{1,\infty}([0,T] \times {\mathbb D})} \leq r  \Big\} \,.
\end{aligned}$$
For any $\varphi \in  \cR_{r,T}$, in the Boltzmann-Grad limit $\mu_\eps \to \infty$, the empirical measure satisfies the 
large deviation estimates
$$
\begin{aligned}
 \lim_{\delta  \to 0} \limsup_{ \mu_\eps \to \infty}  
 {1\over \mu_\eps} \log\bbP_\eps [ d_{[0,T]} (\pi^\eps, \varphi) \leq \delta]   = -\widehat \cF(T,\gp) \, ,\\
  \lim_{\delta  \to 0} \liminf_{ \mu_\eps \to \infty}  {1\over \mu_\eps} \log\bbP_\eps [ d_{[0,T]} (\pi^\eps, \varphi) \leq \delta] = -\widehat \cF(T,\gp) \, .
\end{aligned}
$$
\end{Thm}

A companion program for large deviations (including gradient flows) has been developed for spatially homogeneous models and stochastic particle systems, in the spirit of Kac's approach for the justification of kinetic theory \cite{Leonard95,Heydecker21,BBBO21,BBBC21,BBBC22}.
For (regular) homogeneous observables $\gp$, the functional $\widehat \cF$ coincides with the functional obtained for the Kac model (see also \cite{Rez2} for the additional spatial dependence). 

Thus a feature of Theorem \ref{thmLD} is that the large deviation behaviour of the mechanical dynamics is also ruled by the large deviation functional of the stochastic process.
It is generally accepted that there is good similarity between deterministic systems displaying some chaoticity and random stochastic processes, an idea that has been used several times in mathematical physics. Our context is rather simple, because of the property of molecular chaos which underlies the kinetic theory of gases. Traditionally, the rigorous justification of this theory is based on two approaches, the programs of Grad \cite{Gr58} and Kac \cite{Kac56}, corresponding respectively to the deterministic and the random case which are both effective with some limitations. It is therefore natural to ask to what extent the ``equivalence'' of dynamical system and stochastic process can be pushed. Our result proves such equivalence up to dynamical events of exponentially small probability.

\bigskip

For an extensive formal discussion on large deviations in the Boltzmann gas, as well as for some physical motivations, we refer to \cite{bouchet} (see also \cite{BdSGJ-LL15} for diffusive systems).
As argued in the following section, fluctuations and large deviations are a systematic way to probe the physical system on finer and finer scales, characterizing all the correlations. In particular, they complement the rigorous explanation of the transition to irreversibility, by showing that stochastic reversibility is recovered if one retains all the information discarded in Lanford's analysis. Finally, we mention that the large deviations add a formal geometric structure to the limit, of gradient-flow type as discussed in \cite{bouchet} (Section 5.4), which might motivate further investigations. 

\section{Strategy of the proofs}\label{section1.5}
\setcounter{equation}{0}

In this section we provide an overview of the paper and describe, informally, the core of our argument leading to Theorems \ref{thmTCL} and \ref{thmLD}. 
 
We should start by recalling the basic features of the proof of Theorem \ref{thm: Lanford}. 
For a deterministic dynamics of interacting particles, so far there has been only one way to access the law of large numbers rigorously. The strategy is based on the `hierarchy of moments' corresponding to the family of correlation functions~$\left(F^{\eps}_n \right)_{n \geq 1}$, Eq.\,\eqref{eq: densities at t}.  The main role of $F^{\eps}_n $ is to project the measure on finite groups of particles (groups of cardinality~$n$), out of the total $\cN$. The term `hierarchy' refers to the set of linear BBGKY equations satisfied by this collection of functions (which will be written in Section \ref{sec:SCE}), where the equation for~$F^{\eps}_n$ has a  source term   depending on $F^{\eps}_{n+1}$. This hierarchy is completely equivalent to the Liouville equation~\eqref{Liouville} for the family $\left(W_N^\eps\right)_{N \geq 0}$, as it contains exactly the same amount of information. However as $\cN \sim \mu_\e$ in the Boltzmann-Grad limit \eqref{eq: Boltzmann Grad}, one should make sense of a Liouville density depending on infinitely many variables, and the BBGKY hierarchy  becomes the natural convenient way to grasp the relevant information. Lanford succeeded to show that the explicit solution~$F^\eps_n(t)$ of the BBGKY hierarchy, obtained by iteration of the Duhamel formula, converges to a product~$f^{\otimes n}(t)$ (propagation of chaos), where $f$ is the solution of the Boltzmann equation~(\ref{boltzmann-eq}).  

This result based on the hierarchy of moments has two important limitations. The first one is the   restriction on its time of validity, which comes from too many terms in the iteration: we are indeed unable to take advantage of   cancellations between gain and loss terms. The second one is a drastic loss of information. We shall not give here a precise notion of `information'. We limit ourselves to stressing that $\left(F^{\eps}_n \right)_{n \geq 1}$ is suited to the description of typical events. In the limit, everything is encoded in $f $, no matter how large~$n$. Moreover, the Boltzmann equation produces some entropy along the dynamics: at least formally, $f$ satisfies
$$\d_t \big(-  \int f\log f \,dv \big)+ \nabla_x \cdot \big( -\int f\log f \,v \,dv\big)   \geq 0\,,
$$
which is in contrast with the time-reversible hard-sphere dynamics.   Our main purpose  here is to overcome this second limitation (for short times) and to perform the Boltzmann-Grad limit in such a way as to keep most of the information lost in Theorem \ref{thm: Lanford}. In particular, the limiting functional~\eqref{eq: fonc GD sup p} coincides with the large deviations functional of a genuine reversible Markov process, in agreement with the microscopic reversibility~\cite{bouchet}. We face a significant difficulty: on the one hand, we know that  {\it averaging} is important in order to go from Newton's equations to Boltzmann's equation; on the other hand, we  want to keep track of some of the  microscopic structure. 

To this end, we need to go beyond  the BBGKY hierarchy and turn to  a more powerful representation of the dynamics.
We shall replace the family $\left(F^{\eps}_n \right)_{n \geq 1}$ (or $\left(W_N^\eps\right)_{N \geq 0}$) with a third, equivalent, family of functions $\left(f^{\eps}_n \right)_{n \geq 1}$, called (rescaled)  {\it cumulants}\footnote{Cumulant type expansions within the framework of kinetic theory appear in \cite{BGSR2,PS17,LMN16,GG18,GG22}.}. Their role is to    grasp information on the dynamics on finer and finer scales. Loosely speaking, $f^{\eps}_n (t)$ will collect events where~$n$ particles are ``completely connected'' by a chain of interactions. We shall say that the $n$ particles form a  {\it cluster}. Since a collision between two given particles is typically of order~$t/\mu_\eps$, a ``complete connection'' would account for events of probability  of order~$(t/\mu_\eps)^{n-1}$. We therefore end up with a hierarchy of rare events, which we need to control at all orders  to obtain Theorem~\ref{thmLD}. At variance with $\left(F^{\eps}_n \right)_{n \geq 1}$, even  {\it after} the limit $\mu_\eps \to \infty$ is taken, the rescaled cumulant $f^{\eps}_n$ cannot be trivially obtained from the cumulant $f^{\eps}_{n-1}$. Each step entails extra information, and events of increasing complexity, and decreasing probability.

The cumulants, which are a standard probabilistic tool, will be  investigated here in the dynamical, non-equilibrium context.
Their precise definition and basic properties are discussed in Chapter \ref{cluster-chap}. 

The introduction of cumulants will not entitle us to avoid the BBGKY hierarchy entirely. Unfortunately, the equations for $\left(f^{\eps}_n \right)_{n \geq 1}$ are difficult to handle. But the moment-to-cumulant relation~$ \left(F^{\eps}_n\right)_{n \geq 1} \to  \left(f^{\eps}_n \right)_{n \geq 1}$ is a bijection and, in order to construct $ f^{\eps}_n(t)$, we can still resort to the same solution representation of  \cite{La75} for the correlation functions $\left(F^{\eps}_n (t)\right)_{n \geq 1}$.
 This formula is an expansion over  {\it collision trees}, meaning that it has a geometrical representation as a sum over binary tree graphs, with vertices accounting for collisions. 
The formula will be presented in Chapter \ref{tree-chap} (and generalized from the finite-dimensional case to the case of functionals over trajectories, which is needed to deal with space-time processes). For the moment, let us give an idea of the structure of this tree expansion. The Duhamel iterated solution for $F^{\eps}_n(t)$ has a peculiar characteristic flow: $n$ hard spheres (of diameter~$\eps$) at time $t$ flow backwards, and collide (among themselves or) with a certain number of external particles, which are added at random times and at random collision configurations. The following picture (Figure \ref{fig:exampleintro1}) is an example of such flow (say, $n=3$).
 \begin{figure}[h] 
\centering
\includegraphics[width=3in]{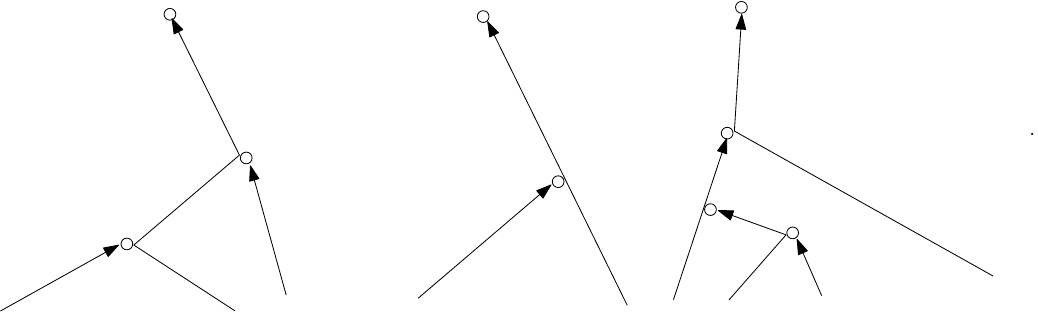} 
\caption{}
\label{fig:exampleintro1}
\end{figure}

\noindent
The net effect resembles a binary tree graph. The real graph is  just a way to record which pairs of particles collided, and in which order.

It is important to notice that different subtrees are unlikely to interact: since the hard spheres are small and the trajectories involve finitely many particles, two subtrees will encounter each other with small probability. This is a rather pragmatic point of view on the propagation of chaos, and the reason why $F^{\eps}_n (t)$ is close to a tensor product (if it is so at time zero) in the classical Lanford argument. Observe that, in this simple argument, we are giving a notion of dynamical  {\it correlation} which is purely geometrical. Actually we will use this idea over and over. Two particles %, say 1 and 2, 
are correlated if their generated subtrees are  {\it connected}, as represented for instance in the following picture (Figure \ref{fig:exampleintro2}).
 \begin{figure}[h] 
\centering
\includegraphics[width=1.5in]{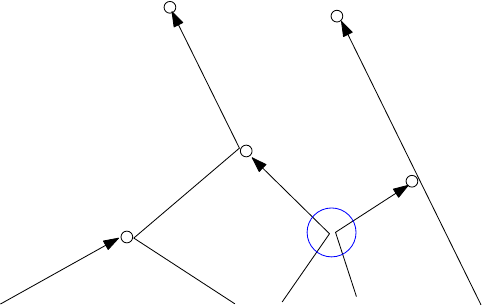} 
\caption{}
\label{fig:exampleintro2}
\end{figure}

\noindent 
The event in Figure \ref{fig:exampleintro2} has `size' $t / \mu_\eps$ (the volume of a tube of diameter $\eps$ and length $t$). In Chapter \ref{cumulant-chap}, we will give a precise definition  of correlation (connection) based on geometrical constraints. %(``recollisions'' and ``overlaps'', corresponding to trees that meet and interact, or cross without interacting, respectively).
It will be the elementary brick to characterize $f^{\eps}_n (t)$ explicitly in terms of the initial data. The formula for~$f^{\eps}_n(t)$ 
(Section \ref{sec:dyncum}) will be supported on characteristic flows with $n$ particles connected, through their generated subtrees (hence of expected size $(t / \mu_\eps)^{n-1}$). In other words, while $F^\eps_n$ projects the measure on  {\it arbitrary} groups of particles of size $n$, the improvement of $f^\eps_n$ consists in restricting to  {\it completely connected} clusters of the same size.

With this naive picture in mind, let us briefly comment again on information, and irreversibility. One nice feature of the geometric analysis of dynamical correlations is that it reflects the transition from a time-reversible to a time-irreversible model. In \cite{BGSRS18}
%,  collision trees are not symmetric in time (they go backwards), and there exist 
we identified, and quantified, the microscopic singular sets where $F^\eps_n$ does  {\it not} converge. These sets are not invariant by time-reversal (they have a direction always pointing to the past, and not to the future). Looking at $F^\eps_n(t)$, we lose track of what happens in these small sets. This implies, in particular, that Theorem \ref{thm: Lanford} cannot be used to come back from time $t>0$ to the initial state at time zero. 
%Events as  ``two collision trees meet each other'' do not contribute to the limit. 
%On the other hand, ``collision trees encountering each other'' are exactly the events used to build up 
The cumulants describe what happens on all the small singular sets,
%. It is therefore not surprising that optimal bounds on the cumulants 
therefore providing the information missing {\color{black} to  recover the reversibility.

At the end of Chapter \ref{cumulant-chap}, we   give a uniform estimate  on these cumulants (Theorem \ref{cumulant-thm1}), which is  the main advance of this paper. 
This $L^1$-bound is  sharp in $\eps$ and $n$ ($n$-factorial bound), roughly stating that the unscaled cumulant decays as $(t / \mu_\eps)^{n-1} n^{n-2}.$ This estimate is intuitively simple. We have given a geometric notion of correlation as a {\it link} between two collision trees.
Based on this notion, we can draw a random graph telling us which particles are correlated and which particles are not (each collision tree being one vertex of the graph). Since the cumulant describes $n$ completely correlated particles, there will be at least $n-1$ edges, each one of small `volume' $t / \mu_\eps$. Of course there may be more than $n-1$ connections (if the random graph has cycles), but these are hopefully unlikely as they produce extra smallness in $\eps$. If we ignore all of them, we are left with minimally connected graphs, whose total number is $n^{n-2}$ by Cayley's formula. Thanks to the good dependence in $n$ of these uniform bounds, we can actually  {\it sum up} all the family of cumulants into an analytic series, referred to as `cumulant generating function' (coinciding with formula \eqref{eq:Letehfirst}).

The second central result of this paper,  stated in Chapter \ref{HJ-chapter} (Theorem \ref{cumulant-thm2}), is the characterization of the rescaled cumulants in the Boltzmann-Grad limit, with minimally connected graphs.
Using this minimality property, we derive a Hamilton-Jacobi equation for the limiting cumulant generating function,
% Wellposedness and uniqueness for this equation can be achieved by abstract methods, based on analyticity.
%All the information of the microscopic mechanical model is actually encoded in this Hamilton-Jacobi equation, 
which is our ultimate point of arrival (allowing us, in particular, to characterize the covariance of the fluctuation field  and the large deviation functional).

 The rest of the paper is devoted to the proofs of our main results.
 
 Chapter \ref{TCL-chap} proves Theorem \ref{thmTCL}. Here, the uniform bounds of Theorem \ref{cumulant-thm1} are considerably better than what is required, and the proof amounts  to looking at a characteristic function living on larger scales. 
Indeed a simple expansion shows that the characteristic function of the fluctuation field is determined, at leading order, by $ f_1^\e$, $(\mu_\e^{1-\frac{n}{2}} f_n^\e)_{n \geq 2}$ so that only the first two cumulants contribute to the limit. This proves the Gaussian character of the process (implying in particular the Wick Theorem for the moments of the limiting field).
 The more technical part of the proof concerns the tightness of the process for which we adapt a Garsia-Rodemich-Rumsey's inequality on the modulus of continuity, to the case of a discontinuous process.
 
In Chapter \ref{LDP-chap} we prove Theorem \ref{thmLD}, and actually even a slightly more general statement. Our purpose is to show that the cumulant generating function  obtained in Chapter~\ref{HJ-chapter} is dual, through the Legendre transform, to a large deviation rate function. Restricting to the class $\cR_{r,T}$ of observables,  this rate functional can be identified with the one predicted in the literature, based on the analogy with stochastic dynamics.

Finally, Chapters \ref{estimate-chap} and \ref{convergence-chap} are devoted to the proof of Theorems \ref{cumulant-thm1} and  \ref{cumulant-thm2}, respectively. We encounter here a combinatorial issue. The number of terms in the formula for $f^{\eps}_n (t)$ grows, at first sight, badly with~$n$, and cancellations need to be exploited to obtain a factorial growth. At this point, cluster expansion methods \cite{Ru69} (summarized in Chapter \ref{cluster-chap}), applied to the collision trees, enter the game. The decay~$(t / \mu_\eps)^{n-1}$ follows instead from a geometric analysis on hard-sphere trajectories with $n-1$ connecting constraints, in the spirit of previous work \cite{BGSR2,BGSRS18,PS17}.

 Many different types of PDEs appear in this text, which are all  solved, locally in time, by an application of an abstract Cauchy-Kovalevskaya theorem in the spirit of Nishida~\cite{nishida2}. The statement of the theorem, as well as various applications, are provided in the Appendix.

\section{Remarks, and open problems}\label{chapter:CoRe}
\setcounter{equation}{0}

%We have developed a mathematical theory of fluctuations for a hard sphere system in the limit of Grad \cite{Gr49}, starting from a nonequilibrium state and thoroughly quantifying the statistical correlations developed over a short kinetic time.

We conclude with a few remarks on our results.

\begin{itemize}
\item To simplify our proof, we assumed that the initial datum is a quasi-product measure, with the minimal amount of correlations (only the mutual exclusion between hard spheres is taken into account). This assumption is useful to isolate the dynamical part of the problem in the clearest way. More general initial states could be dealt with along the same lines (see \cite{S83,PS17}). However the cumulant expansions would contain more terms, describing the deterministic (linearized) transport of initial correlations. 
\item Similarly, fixing only the average number of particles (instead of the exact number of particles) allows to avoid spurious correlations. We therefore work in a grand canonical setting, as is customary in  statistical physics when dealing with fluctuations. Notice that fixing $\cN = N$ produces a long range term of order $1/N$ in the covariance of the fluctuation field. Note also that the cluster expansion method, which is crucial in our analysis, is developed (with few exceptions, see~\cite{PT} for instance) in a grand canonical framework \cite{PU09}.
\item Our results could be established in the whole space $\R^d$, or in a parallelepiped box with periodic or reflecting boundary conditions. Different domains might be also covered, at the expense of complications in the geometrical estimates of dynamical correlations (see~\cite{EGM11,Dolmaire,LeBihan21} for instance).
\item We do not deal with the original BBGKY hierarchy of equations, which was written for smooth potentials, but always restrict to the hard-sphere system. It is plausible that our results could be extended to smooth, compactly supported potentials as considered in \cite{GSRT,PSS17} (see~\cite{Ayi} for a fast decaying case), but the proof would be considerably more involved. 
\item  {At thermal equilibrium,  we expect Theorem \ref{thmTCL} to be true globally in time: see \cite{BGSR2} for a first step in this direction{\footnote{After submission of this work, this program was completed in references \cite{BGSRScov,BGSRStcl,BGSRSsurv}.}}}.
%Our results prove that, far from equilibrium, the structure of the noise is obtained by simply replacing a Maxwellian distribution with the velocity distribution $f(t,x,\cdot)$. This is coherent with the physical picture of a dilute gas at local equilibrium. In fact on the small scale the system is described by a Poisson measure with density modulated by the solution to the Boltzmann equation \cite{S2}.
\end{itemize}

\bigskip
\bigskip
\bigskip

{\bf Acknowledgements.} We are very grateful to H. Spohn and M. Pulvirenti for many enlightening discussions on the subjects treated in this text. We thank also F. Bouchet, F. Rezakhanlou,  G. Basile, D. Benedetto, L. Bertini for sharing their insights on large deviations and A. Debussche, A.  de Bouard, J. Vovelle for their explanations on SPDEs. Finally, we thank the anonymous reviewers for their remarks and suggestions, which have led to a substantial improvement of our manuscript.

 This work was partially supported by the ANR-15-CE40-0020-01 grant LSD. IG and LSR acknowledge the support of a grant from the Simons Foundation MPS No651463-Wave Turbulence.

%
%As a concluding remark, two assumptions are made (which would be probably unnecessary) to simplify our proof. First of all, the initial datum is a quasi-Poisson measure, with the minimal amount of correlations (only the mutual exclusion between hard spheres is taken into account). This assumption is useful to isolate clearly the dynamical part of the problem. For more general initial states, the cumulant expansions would contain more terms, describing the deterministic (linearized) transport of initial correlations. Secondly, we do not deal with the original BBGKY hierarchy of equations, which was written for smooth potentials, but always restrict to the hard-sphere system. Our results could well be extended to smooth, short-range potentials as considered in \cite{GSRT,PSS17}, but the proof would be considerably more involved. 
%

% OTHER POSSIBLE COMMENTS

% perchÃÂ grancanonico

% Poisson locale : da equilibrio a nonequilibrio

% confronto con metodo Rezakhanlou

% "modified Boltzmann equation"

% cita Basile-Benedetto-Bertini ? gradient flows

% link entropia. 

% irreversibilitÂ / bilancio dettagliato.

% tempi lunghi

\part{Dynamical cumulants}

\chapter{Combinatorics on connected clusters}
\label{cluster-chap} 
\setcounter{equation}{0}

 This preliminary chapter consists in presenting a few  notions (well-known in statistical mechanics)
 that will be essential in our analysis: the content of this chapter is classical, but proofs are given for completeness and   to prepare the less familiar reader to some of the combinatorial notions and techniques used in this article. We present in particular cumulants, and their link with exponential moments as well as with cluster expansions. We conclude the chapter with   some combinatorial identities that will be useful throughout this work.

\section{Generating functionals and cumulants} 
\setcounter{equation}{0}

Let $h:\D \to \R$ be a bounded continuous function.
%in the spaces
%\begin{equation}
%\label{eq: space X}
%{\mathcal X}_{\alpha,\beta}
%:= 
%\Big\{ h \in C^0 ({\mathbb D}) 
%\quad  \Big| \qquad   \|h (x,v)  e^{ \alpha  +\frac{\beta } 2 |v|^2} \|_{L^\infty({\mathbb D})}  \leq 1 
%\Big\}\;,\quad \alpha,\beta \in \R\;.
%\end{equation}
%
%For a collection of such test functions $\left( h_i \right)_{i=1}^n$, 
We shall use the functional notation
%\begin{equation}   \label{spaceduality}
%\left\langle F^\eps_{n}(t), h^{\otimes n} \right\rangle = \int_{\D^n} dZ_n\,F^\eps_{n}(t, Z_n)h(z_1)\dots h(z_n)\;,
%\end{equation}
\begin{equation}   \label{spaceduality}
F^\eps_{n,t}\left(h^{\otimes n}\right) = \int_{\D^n} dZ_n\,F^\eps_{n}(t, Z_n)h(z_1)\dots h(z_n)\;,
\end{equation}
(see formula \eqref{FnH-expectation} below for a generalization) and  $$\cP^s_n = \mbox{  set of partitions   of $\{1,\dots,n\}$ into $s$ parts}\;, \label{partition-def}$$
with $$\sigma \in \cP^s_n \Longrightarrow \sigma = \{\sigma_1,\dots,\sigma_s\}\;,\quad |\sigma_i | = \kappa_i\;,
\quad\sum_{i=1}^s \kappa_i =n\;.$$
%where $F^\eps_{n}$ are extended to zero on $\D^n \setminus {\mathcal D}_{n}^\eps$.
%
%The parameters $\alpha,\beta$ will eventually be chosen negative, but large enough (depending on $\alpha,\beta_0$).

The moment generating functional of the empirical measure~(\ref{eq:empmeas}), namely
$\bbE_\eps \Big( \exp \big( \pi^\eps_t(h)  \big )   \Big)$
is related to the rescaled correlation functions~\eqref{eq: densities at t} by the following remark. We recall that
\begin{equation}
\label{eq:empmeas2}
\bbE_\eps \Big(  \exp \big( \pi^\eps_t(h)  \big )  \Big) = \bbE_\eps\left[  \exp \big( \frac{1}{\mu_\eps} \sum_{i=1}^{\cN} h\big( {\bf z}^{\eps}_i(t)\big)   \big )\right]  \, .
\end{equation}
\begin{Prop}
\label{prop: exponential moments}
We have that
\begin{equation}
\label{eqprop: exponential moments}
%\forall t \in [0,T_\alpha]\, , \quad   
\bbE_\eps \Big( \exp \big( \pi^\eps_t(h)  \big )   \Big) =
1+\sum_{n=1} ^\infty {\mu_\eps^n\over n!} \,F^\eps_{n,t}\left(\left(e^{h / \mu_\eps} - 1\right)^{\otimes n}\right)
\end{equation}
if the series is absolutely convergent.
\end{Prop}

\begin{proof}
Starting from \eqref{eq:empmeas2}, one has  
$$
\begin{aligned}
\sum_{k\geq 1}  {1\over k!}\,  \bbE_\eps \Big(   \Big( \pi^\eps_t(h)  \Big)^k \Big) & =
\sum_{k\geq 1} {1\over k!} \, \sum_{n=1}^k\sum_{\sigma \in \cP^n_k } \mu_\eps^{-k}\, 
\bbE_\eps \Big( \sum_{\substack{i_1, \dots, i_n \\ i_j \neq i_\ell, j \neq \ell}} h \left( {\mathbf z}^\eps_{i_1}(t)\right)^{\kappa_{1}}
\dots h \left( {\mathbf z}^\eps_{i_n}(t)\right)^{\kappa_{n}}\Big)\\
& = \sum_{k\geq 1} {1\over k!} \,
\sum_{n=1}^k\sum_{\sigma \in \cP^n_k } \mu_\eps^{-k}\, 
\mu_\eps^{n} \int_{\D^n} dZ_n\, F^{\eps}_n (t, Z_n) h(z_1)^{\kappa_1}\dots h(z_n)^{\kappa_n }
\end{aligned}
 $$
 where in the last equality we used \eqref{eq: marginal time t}. 
On the other hand for fixed $n$
$$
\begin{aligned}
 \sum_{k\geq n} {\mu_\eps^{-k}\over k!} \sum_{\sigma \in \cP^n_k } \prod _{i = 1} ^n h(z_i)^{\kappa_i} 
 & = \sum_{k\geq n} {\mu_\eps^{-k}\over k!\,n!} \sum_{\substack{ \kappa _1\cdots \kappa_n \geq 1\\ \sum\kappa_i=k } }  \binom{k}{\kappa_1}\binom{k-\kappa_1}{\kappa_2}\cdots \binom{k - \kappa_1- \cdots - \kappa_{n-2}}{\kappa_{n-1}}
 \prod _{i = 1}^n  h(z_i)^{\kappa_i} \\
 &= {1\over n!} \prod_{i= 1}^n \sum_{\kappa_i \geq 1} {h(z_i)^{\kappa_i}\over \mu_\eps^{\kappa_i}\kappa_i !}  = {1\over n!} \prod_{i=1}^n  \left( e^{h(z_i)/\mu_\eps } - 1\right) \;.
 \end{aligned}
 $$
Therefore
 $$
 \bbE_\eps \left( \exp \Big( \pi^\eps_t(h)   \Big)   \right) = 1 + \sum_{n \geq 1}
 \mu_\eps^{n} \int_{\D^n} dZ_n\, F^{\eps}_n (t, Z_n) 
 {1\over n!} \prod_{i=1}^n  \left( e^{h(z_i)/\mu_\eps } - 1\right)\;,
 $$
 which proves the proposition.
 \end{proof}

The moment generating functional  is just a compact representation
of the information coded in the family $\left(F^{\eps}_n (t)\right)_{n \geq 1}$. 
After the Boltzmann-Grad limit~$\mu_\eps \to\infty$, the right-hand side of~(\ref{eqprop: exponential moments}) reduces to~$\displaystyle\sum_{n=0} ^\infty {1\over n!}  \Big( \int f (t)h\Big)^n = \exp \Big( \int f (t)h\Big)$, i.e.\,to the solution of the Boltzmann equation.

As discussed in the introduction, our purpose is to keep a much larger amount of information. To this end, we study the cumulant generating functional which is, by Cram\'{e}r's theorem, an obvious candidate to reach atypical profiles \cite{Varadhan}. Namely, we pass to the logarithm and rescale as follows:
\begin{equation}\label{def-Lambdaepst}
\gL^\eps_{t} (e^h) 
:= \frac{1}{\mu_\eps}\log \bbE_\eps \left( \exp \Big(   \mu_\eps\,\pi^\eps_t(h)  \Big)   \right) = \frac{1}{\mu_\eps}\log \bbE_\eps \Big( \exp \big(   \sum_{i=1}^{\cN} h\big( {\bf z}^{\eps}_i(t) \big)   \big)\Big)  \;.
\end{equation}

The first task is to look for a proposition analogous to the previous one. In doing so, the following definition emerges naturally, where we use the notation: 
   \begin{equation}
   \label{shorthand}
  G_{\sigma_j}    := G_{|\sigma_j|} (Z_{\sigma_j}) \, , \quad  G_{\sigma}
:= \prod_{j=1}^{|\sigma| } G_{\sigma_j}
\end{equation}
for~$\sigma = \{\sigma_1,\dots,\sigma_s\} \in \cP^s_n$.
  
 \begin{Def}[Cumulants]\label{def:cumulants} 
 Let~$(G_n)_{n\geq 1}$  be a family of distributions  of~$n$ variables invariant by permutation of the labels of the variables.  The rescaled {\rm cumulants}\index{Cumulant} associated with~$(G_n)_{n\geq 1}$ form the family~$(g_n)_{n\geq 1}$   defined, for all~$n \geq 1$,  by
\begin{equation} \label{eq:definizionediretta}
g_n=\mu_\eps^{n-1} \sum_{s= 1}^n \sum _{\sigma \in \cP^s_n } (-1) ^{s-1}  (s-1) ! \, G_{\sigma}
 \,.
 \end{equation}
 \end{Def}
The scaling factor $\mu_\eps^{n-1}$ (although unnecessary in this chapter) is introduced for later convenience, and will ensure that the cumulants are of order $1$ in $\varepsilon$.
 
We then have
 the following result, which is well-known in the theory of point processes (see~\cite{daley}). 
 %We stress the fact that in that statement, $ f^\eps_{n}$ are rescaled cumulants.
\begin{Prop}
\label{prop: exponential cumulants}
Let $  (  f^\eps_{n} )_{n \geq 1}$ be the family of rescaled cumulants associated with $\left(F^\eps_{n}\right)_{n \geq 1}$. We have  
\begin{equation*}
\gL^\eps_{t} (e^h)  =
\sum_{n=1} ^\infty {1\over n!} f^\eps_{n,t}\left(\left(e^h - 1\right)^{\otimes n}\right)\;,
\end{equation*}
if the series is absolutely convergent.
 \end{Prop}

\begin{proof}
%A straightforward series expansion gives
%\begin{equation}
%\label{log-exp-series}
%\begin{aligned}
%\frac{1}{\mu_\eps} \log \bbE_\eps  \left(  \exp \Big(  \mu_\eps\,\pi^\eps_t(h)  \Big)  \right) 
%&= \frac{1}{\mu_\eps} \sum_{n= 1}^\infty {(-1)^{n-1} \over n}  
%\left(  \bbE_\eps \left( \sum_{k= 1}^\infty 
%{1\over k!}  \Big(  \mu_\eps\,\pi^\eps_t(h)\Big)^k  
%\right) \right)^n \\
%&= \frac{1}{\mu_\eps} \sum_{n= 1}^\infty {(-1)^{n-1} \over n}   \prod_{\ell = 1}^n \sum_{k_\ell} {1\over k_\ell!}  \bbE_\eps \left(  \Big(  \mu_\eps\,\pi^\eps_t(h) \Big)^{k_\ell}  \right).
%\end{aligned}
%\end{equation}
Applying Proposition~\ref{prop: exponential moments} to~$h$ in place of~$h/\mu_\eps$, expanding the logarithm in a series and using Definition \ref{def:cumulants},  we get
 $$
 \begin{aligned}
\frac{1}{\mu_\eps} \log \bbE_\eps \left( \exp \Big(   \mu_\eps\,\pi^\eps_t(h) \Big) \right) 
&= \frac{1}{\mu_\eps} \sum_{n= 1}^\infty {(-1)^{n-1} \over n}   \prod_{\ell = 1}^n\left[ \sum_{p_\ell} {\mu_\eps^{p_\ell} \over p_\ell!}     F^\eps _{p_\ell,t}\left((e^{ h} - 1) ^{\otimes p_\ell} \right)\right]  \\
& =  \frac{1}{\mu_\eps} \sum_{n= 1}^\infty {(-1)^{n-1} \over n}  \sum_{ p_1, \dots, p_n} {\mu_\eps^{p_1 + \dots + p_n} \over p_1! \dots p_n!}   \prod_{\ell = 1}^n    F^\eps _{p_\ell,t}\left((e^{ h} - 1) ^{\otimes p_\ell} \right) \\
& =  \sum_{p = 1}^\infty \frac{\mu_\eps^{p-1}}{p!} \sum_{ n = 1}^p \sum_{\sigma \in \cP^n_p} (-1)^{n-1} (n-1)!    \prod _{\ell = 1} ^n  F^\eps _{p_\ell,t}\left((e^{ h} - 1) ^{\otimes p_\ell} \right)\\
&=  \sum_{p = 1}^\infty {1\over p!} f^\eps_{p,t}\left(\left(e^h - 1\right)^{\otimes p}\right)\;.
\end{aligned}
$$
In the third equality, we used that the number of partitions of $\{1, \dots , p\}$ into $n$ sets with cardinals $p_1, \dots, p_n$ is given by 
\begin{equation}
\label{eq: decomposition en k sets}
\left|\cP_p^n(p_1, \dots, p_n)\right| = \frac{1}{n!} \; 
\binom{p}{p_1} \binom{p-p_1}{p_2}\cdots \binom{p- p_1- \cdots - p_{n-1}}{p_n} 
= \frac{1}{n!} \;  \frac{ p !}{p_1 ! \;  \cdots \;  p_n !} \,,
\end{equation}
where the factor $n!$ arises to take into account the fact that the sets of the partition are not ordered. 
This proves the result.
\end{proof}
 Note that cumulants measure departure from chaos in the sense that    they vanish identically at order~$n\geq 2$   in the case of i.i.d.\,random variables.

\section{Inversion formula for cumulants}\label{sct:inversion}
\setcounter{equation}{0}

In this section we prove that the  cumulants~$(g_n)$ associated with a family $(G_n)$  in the sense of Definition~\ref{def:cumulants},  encode  all the correlations,  meaning that    $G_n$ can be reconstructed from $(g_k)_{k\leq n}$ for all~$n \geq 1$. More precisely,   the following inversion formula holds.
\begin{Prop}
\label{prop: inversion cumulants general}
Let~$(G_n)_{n\geq 1}$ be a family of distributions and~$(g_n)_{n\geq 1}$ its cumulants
in the sense of Definition~{\rm\ref{def:cumulants}}. 
Then the map from~$(G_n)_{n\geq 1}$ to its cumulants~$(g_n)_{n\geq 1}$ is a bijection and, for each~$n\geq 1$, the   distribution~$G_n$ can be recovered from the cumulants~$(g_k)_{ k\leq n} $ by the inversion formula
\begin{align}
\label{eq: cumulants inverse general}
\forall n \geq 1 \, , \qquad 
G_n = \sum_{s =1}^n   \sum_{\gs \in \cP_n^s}  
   \;    \mu_\eps^{-(n-s)} g_{\gs} \, .
\end{align}
Equations \eqref{eq: cumulants inverse general} and \eqref{eq:definizionediretta} are equivalent definitions of~$(g_n)_{n\geq 1}$.
 \end{Prop}
\begin{proof}
Let us check that
\begin{equation*}
G_n  = \mu_\eps^{-(n-1)}g_n  +  \sum_{s =2}^n   \mu_\eps^{-(n-s)}\sum_{\gs \in \cP_n^s} g_{\gs} \,.
\end{equation*}
Replacing the cumulants~$g_{\gs_j}$ by their definition, we get
$$\bbA_n  := \sum_{s =2}^n   \sum_{\gs \in \cP_n^s}  
\mu_\eps^{-(n-s)} g_{\gs}  = 
\sum_{s =2}^n   \sum_{\gs \in \cP_n^s}   \prod_{j =1   }^s\Big( \sum_{k_j =1}^{|\gs_j|} 
\sum_{ \gk_j   \in \cP_{ \gs_j}^{k_j}} (-1)^ {k_j-1} ( k_j-1)! 
\;  G_{\gk_{j}} \Big) \,.
$$

Using the Fubini Theorem, we can index the sum by the partitions 
with $r :=\displaystyle  \sum_{j=1}^s k_j$ sets and obtain
\begin{align*}
\label{eq: bbAn}
\bbA_n  = 
\sum_{r = 2}^n 
\sum_{\gr \in \cP_n^r}  G_{\gr} \Big(
\sum_{s =2}^r    
\sum_{\omega \in \cP_r^s}   (-1)^{r-s }\prod_{i=1}^s(|\omega_i|-1)!
\Big)\, .
 \end{align*}
Note that the partition $\gs$ in the definition of $\bbA_n$ can be  recovered as
$$
\forall i \leq s\, , \qquad \gs_i = \bigcup_{j \in \go_i} \rho_j\,.
$$
Using the combinatorial identity 
$$
\sum_{k = 1}^n   \sum_{\gs \in \cP_n^k} (-1)^k \prod_{i =1}^k (|\gs_i| -1)! = 0
$$
(see Lemma~\ref{lem: combinatorial identities} below %in Section \ref{combinatorics} below 
for a proof), we find that
$$\sum_{s =2}^r    
\sum_{\omega \in \cP_r^s} (-1)^{r-s }\prod_{i=1}^s(|\omega_i|-1)! =- (-1)^{r-1}(r-1)!\,,$$
hence it follows that 
 \begin{align*}
\bbA_n & =- \sum_{r = 2}^n 
\sum_{\gr \in \cP_n^r}     G_{\gr}  (-1)^{r-1}(r-1)! = -\mu_\eps^{-(n-1)}g_n + G_n \, ,
\end{align*}
where the last equality follows from the definition of $g_n$.
%Finally, the fact that~$(G_n)_{n \geq 1} \longmapsto (g_n)_{n \geq 1} $ is a bijection  can be verified by induction on $n$
%(using the inversion formula).
Similarly, \eqref{eq: cumulants inverse general} $\Rightarrow$ \eqref{eq:definizionediretta} can be verified by induction on $n$.
This completes the proof of Proposition~\ref{prop: inversion cumulants general}.
\end{proof}

\section{Clusters and the tree inequality}\label{sct:support}
\setcounter{equation}{0}

We now  prove that the cumulant of order $n$ is supported on  clusters (connected groups) of cardinality~$n$. We shall consider an abstract situation based on a  ``disconnection" condition, the definition of which may change according to the context.
\begin{Def}\label{def:cbr}
A {\rm connection} is a commutative binary relation $\sim$ on a set $V$:
$$x \sim y\;,\quad x,y \in V\;.$$
The (commutative) complementary relation, called {\rm disconnection}, is denoted $\not\sim$, that is $x \not \sim y$ if and only if $x\sim y$ is false.
%\;,\quad x \neq y\;.$$
\end{Def}
%We consider a family of distributions~$(\gP_n)_{n\geq 1}$ depending on a set of~$n$ variables~$\eta_1, \dots, \eta_n$, and the non overlapping condition between two   variables~$\eta_i$ and~$\eta_j$ is denoted by
%$$
%\eta_i \not \sim \eta_j\, .
%$$
Consider the indicator function that   $n$ elements~$\{ \eta_1, \dots, \eta_n \}$  are disconnected
$$
\gP_n \big( \eta_1, \dots, \eta_n \big) := \prod_{1 \leq i \not = j \leq n } \indc_{ \eta_i \not \sim \eta_j } \,.
$$
For $n=1$, we set  $\gP_1 \big( \eta_1) \equiv 1$.

The following proposition   shows that the cumulant of order~$n$ of~$\gP_n$ is supported on clusters of length~$n$, meaning configurations~$(\eta_1, \dots, \eta_n)$ in which all elements
are linked by a  chain  of connected elements.
%We thus have an exact correspondence between the geometric picture of clusters, and the analytical decomposition in cumulants.  
Before stating the proposition let us recall some classical terminology on graphs. This definition, as well as Proposition~\ref{prop: tree inequality} and its proof, are taken from~\cite{gibbspp}.
\begin{Def}\label{def:graphs}
Let~$V$ be a set of vertices\index{Graph!vertex} and~$E\subset\big\{\{v,w\},\;v,w \in V\, ,\; v\neq w\big\}$ a set of edges\index{Graph!edge}.  The pair~$G = (V,E)$ is called a {\rm graph} (undirected, no self-edge, 
no multiple edge)\index{Graph}. 
Given a graph~$G$ we denote by~$E(G)$ the set of all edges in~$G$. The graph is said  {\rm connected}\index{Graph!connected} if for all~$v,w \in V$, $v\neq w$, there exist~$v_0 = v,v_1,v_2,\dots,v_n=w$ such that~$\{v_{i-1},v_i\}\in E$ for all~$i = 1,\dots,n$.

We denote by~$\mathcal C_V$  the set of  connected graphs with~$V$ as vertices, and by~$\mathcal C_n$  the set of  connected graphs with~$n$  vertices when~$V = \{1,\dots,n\}$. A  {\rm minimally connected}\index{Graph!minimally connected}, or {\rm tree graph}, is a connected graph with~$n-1$ edges. We denote by~$\mathcal T_V$  the set of  minimally connected graphs with~$V$ as vertices, and by~$\mathcal T_n$  the set of minimally connected graphs with~$n$  vertices when~$V = \{1,\dots,n\}$. 

Finally, the union of two graphs $G_1=(V_1,E_1)$ and $G_2=(V_2,E_2)$ is $G_1 \cup G_2 = (V_1\cup V_2,E_1\cup E_2)$.
\end{Def}
The following result was originally derived by Penrose~\cite{Pe67}.
 \begin{Prop}
\label{prop: tree inequality}\index{Tree inequality}
The (unrescaled) cumulant of~$\Phi_n$ defined as in Definition~{\rm\ref{def:cumulants}} is equal to 
\begin{equation}
\label{graph-representation}
\gp_n  \big( \eta_1, \dots, \eta_n \big) =
\sum_{G\in \cC_n} \prod_{ \{i,j\}  \in E (G)} (-\indc_{ \eta_i \sim \eta_j}) \,.
\end{equation}

Furthermore, one has the following ``tree inequality"
\begin{equation}
\label{tree-ineq}
|\gp_n  \big( \eta_1, \dots, \eta_n \big)| \leq 
\sum_{T\in \cT_n} \prod_{ \{i,j\}  \in E(T)} \indc_{ \eta_i \sim \eta_j}   \,.
\end{equation}
\end{Prop}

\begin{proof}
The first step is to check  the representation formula (\ref{graph-representation}) for the cumulant~$\gp_n$. The starting point is the definition of $\gP_n$
$$
\gP_n \big( \eta_1, \dots, \eta_n \big)  = \prod_{1 \leq i \not = j \leq n } (1- \indc_{ \eta_i  \sim \eta_j } )
 = \sum_{ G }  \prod_{ \{i,j\}  \in E(G)} (-\indc_{ \eta_i \sim \eta_j}) \,,
$$
where the sum over~$G$ runs over all graphs with~$n$ vertices.
We then decompose these graphs into connected components and obtain that 
$$\gP_n \big( \eta_1, \dots, \eta_n \big)  = \sum_{s= 1}^n  \sum_{ \sigma \in \cP^s_n}  \prod _{k = 1}^s \left( \sum_{ G_k   \in \cC_{\sigma_k}}  \prod_{ \{i,j\}  \in E(G_k)} (-\indc_{ \eta_i \sim \eta_j}) \right) \,.$$
By the uniqueness of the cumulant decomposition as given in Proposition \ref{prop: inversion cumulants general} (without the rescaling), we therefore find~(\ref{graph-representation}).
%deduce that 
%$$\gp_n  \big( \eta_1, \dots, \eta_n \big) =
%\sum_{G\in \cC_n} \prod_{ \{i,j\}  \in E(G)} (-\indc_{ \eta_i \sim \eta_j}) \,.$$

The second step is  to compare connected graphs and trees. This is achieved by defining a tree partition scheme, i.e.\,a map~$\pi : \cC_n \to \cT_n$ such that for any~$T \in \cT_n$, there is a graph $R(T) \in \cC_n$ satisfying
$$ \pi^{-1} (\{T\}) =\big \{ G\in \cC_n \,:\, E(T) \subset E(G) \subset E(R(T))\big\}\,.$$
Penrose's partition scheme \label{Penrose} is obtained in the following way. Given a graph $G$, we define its image $T$ iteratively starting from the root $1$
\begin{itemize}
\item the first generation of $T$ consists of all $i$ such that $\{1, i\} \in G$;  these vertices are accepted and labeled in increasing order $t_{1,1},\dots, t_{1,r_1}$;
\item the $\ell$-th generation consists of all $i$ which are not already in the tree,  and such that~$\{t_{\ell-1,j}, i\} $  belongs to~$ E(G)$ for some $j\in \{1,\dots, r_{\ell-1}\} $;   these vertices are labeled
 in increasing order of~$j = 1, \dots, r_{\ell-1}$, then increasing order of $i$.
 \end{itemize}
The procedure ends with a  unique tree $T\in \cT_n$. In order to characterize $R(T)$, we now investigate which edges of $G$ have been discarded. Denote by $d(i)$ the graph distance of  the vertex $i$ to the root (which is just its generation). Let  $\{i,j\} \in E(G) \setminus E(T)$ and assume without loss of generality that $d(i) \leq d(j)$. By construction $d(j)  \leq  d(i) +1$. Furthermore, if $d(j) = d(i) +1$, the parent~$i'$ of $j$ in the tree is such that $i'<i$. 
Therefore $E(G) \setminus E(T)$ is a subset  of the set~$E'(T)$ consisting of edges within a generation ($d(i) = d(j)$), and of edges towards a younger uncle ($d(j) = d(i) +1$ and~$i'<i$). Conversely, we can check that any graph satisfying  $E(T) \subset G \subset E(T) \cup E'(T)$ belongs to $\pi^{-1} (\{T\})$. We therefore define $R(T)$ as the graph with edges $E(T) \cup E'(T)$.

The last step is to exploit the non trivial cancellations between graphs associated with the same tree. There holds, with the above notation,
$$
\begin{aligned}
\sum_{G\in \cC_n} \prod_{ \{i,j\}  \in E(G)} (-\indc_{ \eta_i \sim \eta_j})&= \sum_{T\in \cT_n} \sum_{ G \in \pi^{-1} (T)} \prod_{ \{i,j\}  \in E(G)} (-\indc_{ \eta_i \sim \eta_j})\\
&=  \sum_{T\in \cT_n}\left(  \prod _{ \{i,j\}  \in E(T)} (-\indc_{ \eta_i \sim \eta_j}) \right) \left( \sum_{ E' \subset E'(T)} \prod _{ \{i,j\}  \in E'} (-\indc_{ \eta_i \sim \eta_j})\right)  \\
& = \sum_{T\in \cT_n}\left(  \prod _{ \{i,j\}  \in E(T)} (-\indc_{ \eta_i \sim \eta_j})\right) \left(  \prod_ { \{i,j\} \in E'(T)} (1 - \indc_{ \eta_i \sim \eta_j})\right)\,.
\end{aligned}
$$
The conclusion follows from the fact that $(1 - \indc_{ \eta_i \sim \eta_j} )\in [0,1]$.
The proposition is proved.
\end{proof}

\section{Number of minimally connected graphs}\label{Cayley+}
\setcounter{equation}{0}
The following classical result will be used in Chapter~\ref{estimate-chap}.
\begin{Lem}
\label{lem: Cayley+} The cardinality of the set of minimally connected graphs on $n$ vertices with degrees (number of edges per vertex) of the vertices $1,\dots,n$ fixed respectively at the values 
$d_1,\dots, d_n$ is 
\begin{equation}
\begin{aligned}
\Big|\Big\{ T \in \cT_n \;\;: \;\; d_1(T)=d_1,\dots,d_n(T) = d_n\Big\}\Big| =  \frac{(n-2)!}{\prod_{i=1}^{n} (d_i-1)!}\,\cdotp
\end{aligned}
\end{equation}
\end{Lem}
Before proving the lemma, let us notice   that it implies Cayley's formula $|\cT_n| = n^{n-2}$. Indeed the graph is minimal, so there are exactly $n-1$ edges hence (each edge has two vertices) the sum of the degrees
has to be equal to $2n-2$. Thus
$$
|\cT_n| = \sum_{ \substack{d_1,\dots,d_n \\ 1 \leq d_i \leq n-1 \\ \sum_{i}d_i = 2(n-1)} }
 %\substack{ \\ÃÂ1 \leq d_i \leq n-1 \\ÃÂ\sum_{i}d_i = 2(n-1)}}
 \frac{(n-2)!}{\prod_{i=1}^{n} (d_i-1)!}
=  
%\sum_{ \substack{d_1,\dots,d_n \\ÃÂ1 \leq d_i \leq n-2 \\ÃÂ\sum_{i}d_i = n-2}}
\sum_{ \substack{d_1,\dots,d_n \\ 0 \leq d_i \leq n-2 \\ \sum_{i}d_i = n-2} }
\frac{(n-2)!}{\prod_{i=1}^{n} d_i!} = 
\left( \sum_{i=1}^n\, 1
%\underbrace{1+\dots + 1}_{n}
\right)^{n-2}\;.
$$
\begin{proof}
The lemma can be proved by induction. For $n=2$ the result is trivial, so we suppose to have proved it for the set $
\cT_n^{d_1,\dots,d_n}:=\{ T \in \cT_n \;\; | \;\; d_1(T)=d_1,\dots,d_n(T) = d_n\}$,
for arbitrary $d_1,\dots,d_n$, and consider the set $\cT_{n+1}^{d_1,\dots,d_{n+1}}$. Since there is always at least one vertex of degree 1, we can assume without loss of generality that $d_{n+1} = 1$. Notice that, if the vertex $n+1$ is linked to the vertex $j$, then necessarily $d_j \geq 2$. We therefore compute the number of minimally connected graphs
on $n$ vertices with degrees $d_1,\dots,d_{j-1},d_j-1,d_{j+1},\dots,d_n$, and sum then over $j$ (all the ways to attach the vertex $n+1$ of degree 1). This leads to
$$
|\cT_{n+1}^{d_1,\dots,d_{n+1}} | = \sum_{j=1}^n \frac{(n-2)!}{(d_j-2)!\prod_{i \neq j} (d_i-1)!}\,,
$$
hence
$$
|\cT_{n+1}^{d_1,\dots,d_{n+1}} | =\frac{(n-2)!}{\prod_{i=1}^{n+1} (d_i-1)!}  \sum_{j=1}^{n+1}(d_j-1) =  \frac{(n-1)!}{\prod_{i=1}^{n} (d_i-1)!}
$$
having used again $\sum_{j=1}^{n+1}d_j = 2(n+1-1)$.
\end{proof}

%%%%%%%%%%%%%%%%%%%%%%%%%%%%%%%%%%%%%%%%%%%%%%%%%%%%%%%%%
 \section{Combinatorial identities}\label{combinatorics}
 \setcounter{equation}{0}

  The following combinatorial identities have been used  in the previous sections.
\begin{Lem}
\label{lem: combinatorial identities}
For $n \geq 2$ there holds
\begin{align}
& \sum_{k = 1}^n   \sum_{\gs \in \cP_n^k} (-1)^k (k-1)! = 0 \,,\label{eq: combinatorial identities 1}\\
& \sum_{k = 1}^n   \sum_{\gs \in \cP_n^k} (-1)^k \prod_{i =1}^k (|\gs_i| -1)! = 0 \, .\label{eq: combinatorial identities 2} 
\end{align}
%and for $m\geq n\geq 1$
%\begin{equation}
%\label{eq: combinatorial identities 4}
%\sum_{k=0}^{n} (-1)^{m+n-k}  (m+n-k-1)! C^k_{n} {m!  \over (m-k)!}= 0 \, .
%\end{equation}
\end{Lem}
\begin{proof}
From the Taylor series of $x \mapsto \log \big( \exp (x) \big)$, we deduce that 
$$
\forall n \geq 2, \qquad
\sum_{k = 1}^n  \sum_{\ell_1 + \dots + \ell_k = n}
\frac{(-1)^k}{k}  \; \frac{1}{ \ell_1 !  \dots  \ell_k !} = 0 \,.
$$
%The number of partitions of $\{1, \dots , n\}$ into $k$ sets with cardinals $\ell_1, \dots, \ell_k$ is given by 
%\begin{equation}
%\label{eq: decomposition en k sets}
%\sharp  \cP_n^k(\ell_1, \dots, \ell_k) = \frac{1}{k!} \; 
%\binom{n}{\ell_1} \binom{n-\ell_1}{\ell_2}\cdots \binom{n- \ell_1- \cdots - \ell_{k-1}}{\ell_k} 
%= \frac{1}{k!} \;  \frac{ n !}{\ell_1 ! \;  \dots \;  \ell_k !} \,,
%\end{equation}
%where the factor $k!$ arises to take into account the fact that the sets of the partition are not ordered. 
Combining \eqref{eq: decomposition en k sets} and the previous identity, we get 
\begin{align*}
0 =
\sum_{k = 1}^n  \sum_{\ell_1 + \dots + \ell_k = n}
\frac{(-1)^k}{k}  \; \frac{1}{ \ell_1 !  \dots  \ell_k !}
&= 
\sum_{k = 1}^n  \frac{(-1)^k}{k}  \sum_{\ell_1 + \dots + \ell_k = n} {k!\over n!}  \sharp  \cP_n^k(\ell_1, \dots, \ell_k)\\
= {1\over n!}
\sum_{k = 1}^n (-1)^k (k-1)!    \sharp  \cP_n^k
\end{align*}
and this completes the first identity \eqref{eq: combinatorial identities 1}.

\bigskip
 From the Taylor series of $x \mapsto  \exp \big( \log(1+x) \big)$, we deduce that 
\begin{equation*}
\label{eq: combinatorial bis}
\forall n \geq 2, \qquad 
\sum_{k = 1}^n \frac{1}{k!} \sum_{\ell_1 + \dots + \ell_k = n}
\frac{(-1)^k}{\ell_1  \dots  \ell_k} = 0 \,.
\end{equation*}
Combining \eqref{eq: decomposition en k sets} and the previous identity, we get 
\begin{align*}
0 =
\sum_{k = 1}^n \frac{1}{k!} \sum_{\ell_1 + \dots + \ell_k = n}
\frac{(-1)^k}{\ell_1  \dots  \ell_k} 
= \frac{1}{n!} 
\sum_{k = 1}^n   \sum_{\gs \in \cP_n^k} (-1)^k \prod_{i =1}^k (|\gs_i| -1)!
\end{align*}
and this completes the second identity \eqref{eq: combinatorial identities 2}.

%\bigskip
% From the Taylor series of $x \mapsto  (1+ (e^x - 1))^{-1} = e^{-x}$, we deduce that 
%\begin{equation}
%\forall n \geq 2, \qquad 
%\sum_{k = 1}^n (-1)^k  \sum_{\ell_1 + \dots + \ell_k = n} {k!\over n!}  \sharp  \cP_n^k(\ell_1, \dots, \ell_k)  = {(-1)^n \over n!}  \,\cdotp
%\end{equation}
%Combining \eqref{eq: decomposition en k sets} and the previous identity, we get 
%\begin{align*}
%(-1)^n 
%=
%\sum_{k = 1}^n   \sum_{\gs \in \cP_n^k} (-1)^k k!
%\end{align*}
%and this completes the proof of the  last identity \eqref{eq: combinatorial identities 3}.
%
%
%\bigskip
% From the Taylor series of $x \mapsto  (1-x)^{-1} $, we deduce that 
%$$
%\begin{aligned}
% (1-x)^{-m} &= (\sum_{k=0}^\infty x^k) ^m = \sum_{m=0}^\infty x^n \sum_{k_1+k_2 +\dots k_m = n} 1\\
% & = \sum_{n=0} ^\infty x^n C^n_{n+m-1}
% \end{aligned}
% $$
% Taking the product with $(1-x)^m$, we get
% $$ \left( \sum_{n=0} ^\infty x^n C^n_{n+m-1}\right) \left( \sum_{s=0}^m C^s_m (-1)^s x^s \right) = 1$$
% from which we deduce that for $n \geq 1$,
% $$\sum_{k=0 } ^n  (-1)^{k} C^{k}_m C^{n-k} _{n-k+m-1} = 0 \,.$$
% Multiplying this identity by $(m-1)! / n!$, we obtain exactly \eqref{eq: combinatorial identities 3}.
  The lemma is proved.
\end{proof}

%
%\bigskip
%Later on, we will also need to estimate the convolution of two sequences with factorial growth.
%
%\begin{Lem}\label{convolution}
%Define 
%$
%S_n : =\displaystyle \sum_{p=0}^n  { p!} { (n-p)!}$.
%Then
%$$S_n = n! \Big( 2 + O( \frac 1 n)\Big).$$
%\end{Lem}
%
%\begin{proof}
%By definition, $S_n$ is the product of 
% $n!$ by the sum of  the inverse combinatorial factors~$1/ \left(\begin{array}{c}n \\ p\end{array}\right)$. 
% 
% This  sum is at least 2 (because of the two extreme terms) and at most (because~$\left(\begin{array}{c}n \\ 2\end{array}\right)\leq \left(\begin{array}{c}n  \\p\end{array}\right)=\left(\begin{array}{c}n  \\n-p\end{array}\right)$ when~$p\geq 2$)
% $$ 1 +1 + \frac 1n+ \frac1 n + \sum_{p=2}^{n-2} \frac 2{n(n-1)} = 2 +\frac 2 n +\frac{2( n-3)}{ n(n-1) } \leq 2 + \frac 4 n$$
%which proves the result.
%\end{proof}

%%%%%%%%%%%%%%%%%%%%%%%%%%%%%%%%%%%%%%%%%%%%%%%%%%%%%%%
\chapter{Tree expansions of the hard-sphere dynamics}
\label{tree-chap} 
\setcounter{equation}{0}

Here and in the next chapter, we explain how the combinatorial  methods presented in the previous chapter can be applied to study the dynamical correlations of hard spheres. The first steps in this direction are to define a suitable family describing the correlations of order $n$, and then to obtain a  graphical representation of this family  which will be helpful to identify the clustering structure.

\section{Space correlation functions}\label{sec:SCE}
\setcounter{equation}{0}

 For the sake of simplicity, we   start by describing correlations in phase space.
Recall that the $n$-particle correlation function $F^\eps_n \equiv F^\eps_n(t, Z_n)$ defined by 
(\ref{eq: densities at t}) counts how many groups of $n$ particles are, in average, in a given configuration~$Z_n$ at time $t$: see Eq.\,\eqref{eq: marginal time t}.
%ecalling that~$({\bf z}_{i_1} (t),\dots, {\bf z}_{i_n}(t))$
%represents the trajectory of~$n$ hard-spheres, labeled~$i_1$ to~$i_n$, among the~$  {\mathcal N} $ particles at time~$t$,     we have that for any test function~$h_n: {\mathbb D}^n\to \R$
%\begin{equation}\begin{aligned}
%\bbE_\eps \Big( \sum_{\substack{i_1, \dots, i_n \\ i_j \neq i_k, j \neq k}} h_n \big( {\mathbf z}_{i_1}(t), \dots ,  {\mathbf z}_{i_n}(t) \big) \Big)
%=
%\mu_\eps^{n} \int_{\D^n} dZ_n\, F^{\eps}_n (t, Z_n) \,h_n \big( Z_n \big)\,,
%\end{aligned}
%\end{equation}
%where we recall that~$\bbE_\eps$ denotes the expectation with respect to the measure~(\ref{eq: initial measure})
%on the initial configurations.
%In particular, for any~$p\in \N$ and any family of continuous test functions $(h_j)_{1\leq j \leq p}$ on~$ {\mathbb D}$, we can compute the moment associated with the random fields $ \pi^\eps_t  (h_j):= \sum_{\ell=1} ^\cN h_j ({\mathbf z}_{\ell} (t) )$~: 
%\begin{equation}\label{eq: moment}\begin{aligned}
%\bbE_\eps \Big( \prod_{ j = 1}^p   \pi^\eps_t (h_j)\Big)
%= \sum_{n= 1}^p \sum_{\sigma \in \cP^n_p } 
%\mu_\eps^{n} \int_{\D^n} dZ_n\, F^{\eps}_n (t, Z_n) \prod_{j = 1} ^n  \prod_{i \in \sigma_j} h _{i} (z_j) \,.
%\end{aligned}
%\end{equation}

Let us now discuss the time evolution of the correlation functions: by integration of the Liouville equation (\ref{Liouville}), we get that the family $(F_n^\eps)_{n\geq 1}$ satisfies the so-called BBGKY hierarchy\index{BBGKY hierarchy} (going back to \cite{Ce72})~:
\begin{equation}
\label{BBGKYGC}
 \partial_t F^\eps_n + V_n \cdot \nabla_{X_n} F^\eps_n = C^\eps_{n,n+1} F^\eps_{n+1} \quad \mbox{in} \quad {\mathcal D}^\eps_{n  }
\end{equation}
 with specular boundary reflection 
\begin{equation}\label{specFn}
 \forall Z_{n} \in \d {\mathcal D}^{\eps +}_{n}(i,j)\, , \quad F^\eps _n (t, Z_n) :=  F^\eps _n (t, Z_{n}^{'i,j}) \, ,
\end{equation}
where $Z^{'i,j}_N$ differs from $Z_N$ only by~(\ref{defZ'nij}).
The collision operator in the right-hand side of (\ref{BBGKYGC}) comes from the boundary terms in Green's formula (using the reflection condition to rewrite the gain part in terms of pre-collisional velocities):
 $$
C^\eps_{n,n+1}  F^\eps_{n+1}:=  \sum_{i=1}^nC_{n,n+1}^{i,\eps} F^\eps_{n+1}
 $$
 with 
\begin{equation}\label{defcollisionintegral}
\begin{aligned}
(C_{n,n+1}^{i,\eps} F^\eps_{n+1}) (Z_n) := \int F^\eps_{n+1} (Z_n^{\langle i \rangle},x_i,v'_i,x_i+\e \omega,w') \big( (
w- v_i
)\cdot \omega \big)_+ \, d \omega dw\\
-   \int  F^\eps_{n+1} (Z_{n},x_i+\e \omega,w ) \big( (
w - v_i
)\cdot \omega \big)_- \, d \omega dw \, ,
\end{aligned}
\end{equation}
where~$(v'_i,w')$ is recovered from~$(v_i,w)$ through the scattering laws~(\ref{defZ'nij}), and with the notation
\begin{equation}
\label{def langle i rangle}
Z_n^{\langle i \rangle} := (z_1,\dots,z_{i-1},z_{i+1},\dots,z_n )\,.
\end{equation}

Note that the collision operator is defined as a trace, and thus  some regularity on~$F^\eps_{n+1}$ is required to make sense of this operator. The classical way of dealing with this issue (see for instance~\cite{GSRT,S})  is to consider the integrated form of the equation, obtained by Duhamel's formula
$$ F^\eps _n(t) =  S^\eps_n(t) F_n^{\eps 0 }+\int_0^t  S^\eps_n(t-t_{1} ) C^\eps_{n,n+1}F^\eps _{n+1} (t_{1}) dt_{1} \,,$$
denoting by~$S^\eps_n$\label{Sn-def}  the group associated with free transport in $\cD^\eps_n$ with specular reflection on the boundary~$ \d \cD^\eps_n$.

	\bigskip	
Iterating Duhamel's formula, we can express the solution as a sum of operators acting on the initial data~:
\begin{align}
\label{eq: duhamel1}
F^\eps _n  (t) =\sum_{m\geq0}    Q^\eps_{n,n+m}(t) F_{n+m}^{\eps 0} \, ,
\end{align} 
where we have defined for $t>0$
\begin{equation}  \begin{aligned}\label{Q-def}
Q^\eps_{n,n+m}(t) F_{n+m}^{\eps 0 }  := \int_0^t \int_0^{t_{1}}\dots  \int_0^{t_{m-1}}  S^\eps_n(t-t_{ 1}) C^\eps_{n,n+1}  S^\eps_{n+1}(t_{1}-t_{2}) C^\eps_{n+1,n+2}   \\
\dots  S^\eps_{n+m}(t_{m})    F_{n+m}^{\eps 0} \: dt_{m} \dots d t_{1} 
\end{aligned}\end{equation}
and~$Q^\eps_{n,n}(t)F^{\eps0}_{n} := S^\eps _n(t)F_{n}^{\e 0}$,  $Q^\eps_{n,n+m}(0)F^{\eps 0} _{n+m} := \gd_{m,0}F^{\eps0}_{n+m}$.

  \section{Geometrical representation with collision trees} \label{sec:grct}
\setcounter{equation}{0}

The usual way to study the Duhamel series \eqref{eq: duhamel1} is to  introduce 
 ``pseudo-dynamics" describing the action of the operator $Q^\eps_{n, n+m}$. In the following,   particles will be denoted by two different types of labels: either integers~$i$ or labels~$i*$ (this difference   will correspond to the fact that particles labeled with an integer~$i$ will be added to the pseudo-dynamics  through the Duhamel formula as time goes backwards, while those labeled by~$i*$  are already present at time~$t$). The configuration of the particle labeled~$i*$ will be denoted indifferently~$z_i^*=(x_i^*,v_i^*)$ or~$z_{i*}=(x_{i*},v_{i*})$.

\begin{Def}[Collision trees]
\label{trees-def}
 Given $n \geq 1\,, m\geq 0$, an (ordered)  {\rm collision tree}\index{Collision!tree}  $a \in \cA_{n,   m} $\label{simple-tree} is   a family $(a_i) _{  1\leq i \leq   m}$ with $a_i \in \{1,\dots, i-1\}\cup \{1*,\dots,n*\}$.
\end{Def}
Note that~~$|\cA_{n,  m}| = n(n+1) \dots (n+m-1)$.

\medskip

Given a collision tree $a \in \cA_{n,   m}$,  we define  pseudo-dynamics starting from a configuration~$Z_n^*= (x_i^*, v_i^*) _{1\leq i\leq n}$
in the $n$-particle phase space  at time $t$  as follows.

\begin{Def}[Pseudo-trajectory] 
\label{def: pseudo-trajectory}
Given~$Z_n^*  \in \cD^\eps_n$,~$m\in \N$ and  $a \in \cA_{n,   m} $, we consider a collection of   times, angles and velocities~$(T_m, \Omega_{  m}, V_{  m}) := (t_i, \omega_i, v_i)_{  1\leq i\leq   m}$ satisfying the constraint
$$
  0 \leq  t_{  m} < \cdots <  t_{  1} \leq t = t_0 \, .
$$ 
We   define recursively    {\rm pseudo-trajectories}\index{Pseudo-trajectory} as follows:
\begin{itemize}
\item in between the  collision times~$t_i$ and~$t_{i+1}$   the particles follow the~$(n+i)$-particle (backward) hard-sphere flow;
\item at time~$t_i^+$,  particle   $i$ is adjoined to particle $a_i$ at position~$x_{a_i} + \eps \omega_i$  and 
with velocity~$v_i$, provided it remains at a distance larger than~$\eps$  from all the other particles. %and  impact parameter~$\omega_i$. 
If~$(v_i - v_{  a_i} (t_i^+)) \cdot \omega_i >0$, velocities at time $t_i^-$ are given by the scattering laws
\begin{equation}
\label{V-def}
\begin{aligned}
v_{  a_i}(t^-_i) &:= v_{  a_i}(t_i^+) - \left((v_{  a_i}(t_i^+)-v_i) \cdot \omega_i\right) \, \omega_i  \, ,\\
v_i(t^-_i) &:= v_i+ \left((v_{  a_i}(t_i^+)-v_i) \cdot \omega_i\right) \,  \omega_i \, .
\end{aligned}
\end{equation}
\end{itemize}
We denote  by~$\Psi^\eps_{n,m}  = \Psi^\eps_{n,m}(t)$ (we shall   sometimes omit to emphasize the number of created particles and   denote  it simply  by~$\Psi^\eps_{n} $) the  so constructed pseudo-trajectory, and
   by $Z_{n,m}(\tau ) =\big (Z_{n}^*(\tau ) ,Z_m(\tau)\big)$ the coordinates of the particles in the pseudo-trajectory  at time~$\tau \leq t_m$.   It depends   on   the parameters~$a, Z_n^*, T_m, \Omega_{  m}, V_{     m}$, and $t  $. 
   We also define~$\cG_{ m}^{\e}(a, Z_n ^*) $  to be the set of parameters~$(T_m, \Omega_{   m}, V_{   m}) $ such that the pseudo-trajectory   exists up to time~$0$, meaning in particular that on adjunction of a new particle, its distance to the others remains larger than~$\e$.  For~$m=0$, there is no adjoined particle and the pseudo-trajectory~$\Psi^\eps_{n,0} (\tau)= Z_{n,0}(\emptyset,Z_n^*,\tau )$ for~$\tau \in 
(0,t)$ is   the $n$-particle (backward) hard-sphere flow.

For a given time $t >0$, the sample path pseudo-trajectory of the $n$ ($*-$labeled) particles  is denoted by~${  Z}_n^* ([0,t]) $\label{def-Zn0t}.
\end{Def}

\begin{Rmk} \label{rm:realvsvirtual}
We stress the difference in notation: ``$z_i(\tau)$''   in the above definition denotes the configuration of particle $i$ in the
pseudo-trajectory while the real, $\cN$-particle hard-sphere flow  is denoted~${\mathbf Z}^\eps_{\cN}(\tau) $ as in  {\rm(\ref{def: trajectory})}: particle~$i$ has configuration~${\mathbf z}^\eps_{i}(\tau)$ in the hard-sphere flow.
\end{Rmk}

With these notations, the representation formula \eqref{eq: duhamel1} for the  $n$-particle   correlation function can be rewritten as 
\begin{equation}\label{BBGKY}
F^\eps _n (t, Z_n^*) =
 \sum_{m \geq 0} \, \sum_{a \in \cA_{n,   m} }\, \int_{\cG_{  m}^{\e}(a, Z_n^* )}    dT_m  d\Omega_{    m}  dV_{   m}    \, \Big( \prod_{i=  1}^{  m}  \big( v_i -v_{  a_i} (t_i)\big) \cdot \omega_i \Big)  \,
F_{n+m}^{\eps 0} \big (\Psi^{\eps0}_{n,m}\big) \,,
\end{equation}
where
$$
 dT_m := dt_1\dots dt_m \,  {\mathbf 1}_{0 \leq t_m \leq \dots \leq t_1 \leq t}\, ,
$$
  we have denoted by~$(F_{n}^{\eps 0})_{n\geq 1}$ the initial rescaled correlation function, and~$\Psi^{\eps0}_{n,m}$ is the configuration at time 0 associated with the pseudo-trajectory~$\Psi^\eps_{n,m}$.
Note that the variables $\omega_i$ are integrated over spheres and the scalar products take positive and negative values (corresponding to the positive and negative parts of the collision operators). Equivalently, we can introduce decorated trees $\left(a,s_1,\dots,s_m\right)$ with signs $s_i = \pm $ specifying the collision hemispheres: denoting by $ \cA_{n,   m}^{\pm}$ the set of all such trees, we can write Eq.\,\eqref{BBGKY} as
\begin{equation}\label{BBGKYwithsigns}
F^\eps _n (t, Z_n^*) =
 \sum_{m \geq 0} \, \sum_{a \in \cA^\pm_{n,   m} }\,\int_{\cG_{  m}^{\e}(a, Z_n^* )}    dT_m  d\Omega_{    m}  dV_{   m}    \, \Big( \prod_{i=  1}^{  m} s_i  \left(\big( v_i -v_{  a_i} (t_i)\big) \cdot \omega_i \right)_+\Big)  \,
F_{n+m}^{\eps 0} \big (\Psi^{\eps0}_{n,m}\big) \,,
\end{equation}
where the pseudo-trajectory is defined as before, with the scattering \eqref{V-def} applied in the case $s_i = +$
and the creation at position $x_i + s_i \e \omega_i$.

 \begin{figure}[h] 
\centering
\includegraphics[width=6in]{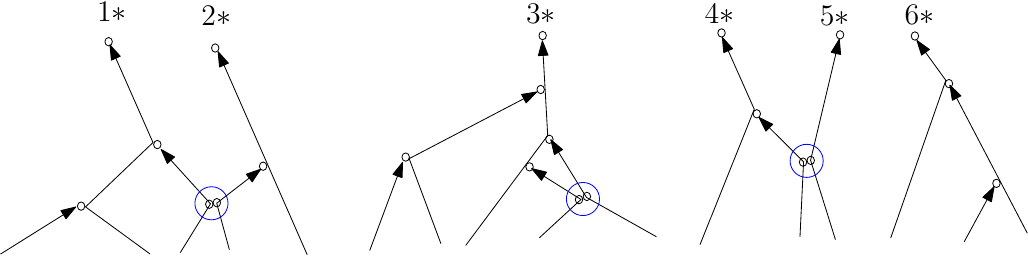} 
\caption{An example of pseudo-trajectory with $n=6$, $m = 10$. In this symbolic picture, time is thought of as flowing upwards (at the top we have a configuration $Z_6^*$, at the bottom $\Psi^{\eps0}_{6,10}$). The little circles represent hard spheres of diameter $\eps$. Notice that several collisions are possible between the adjunction times $T_m$.
These collisions are highlighted by blue circles. For simplicity, the hard spheres have been drawn only at their first time of existence (going backwards), and at collisions between adjunction times. }
\label{fig:ex-pt}
\end{figure}

\section{Averaging over trajectories} \label{subsec:GRCT}
\setcounter{equation}{0}

{\color{black}
To describe dynamical correlations more precisely, we  are going to follow the particle trajectories.
As noted in Remark \ref{rm:realvsvirtual},
pseudo-trajectories provide a geometric representation of the iterated Duhamel series~\eqref{eq: duhamel1}, but they are not physical trajectories of the particle system.
Nevertheless,  the probability on the trajectories of $n$ particles can be derived from the Duhamel series,
as we are  going to explain now.

For a given time $t >0$, the sample path of $n$ particles  labeled~$i_1$ to~$i_n$,  among the~$  {\mathcal N} $  hard spheres, is denoted~$({\bf z}^\eps_{i_1} ([0,t] ),\dots, {\bf z}^\eps_{i_n}([0,t] ))$\label{def-Zn0teps}. In the case when~$i_j=j$ for all~$1 \leq j \leq n$ we denote that sample path by~${\mathbf Z}^\eps_n ([0,t]) $.
As ${\mathbf Z}^\eps_n$  has jumps in   velocity, it is   convenient to work 
 in the space $D_n([0,t])$\label{def-DnOt} of functions that are right-continuous with left limits in $ \D^n  $. 
 This space is endowed with the Skorokhod topology. In the case when~$n=1$ we denote it simply by~$D([0,t])$.

Let $H_n$ be a bounded measurable function on $D_n([0,t])$  (the assumption on boundedness will be relaxed later). We define 
\begin{equation}
\label{FnH}\begin{aligned}
 F^\eps _{n, [0,t]}  ( H_n) :=
& \int  dZ_n^* \sum_{m \geq 0}  \sum_{a \in \cA^\pm_{n,   m} }  \int_{\cG_{ m}^{\e}(a, Z_n^* )}  dT_m  d\Omega_{  m}  dV_{    m}  \\
&  \qquad \times H_n \big(Z_n^*([0,t]) \big)
\Big( \prod_{i=  1}^{  m}  s_i\left(\big( v_i -v_{  a_i} (t_i)\big) \cdot \omega_i \right)_+\Big)  
F_{n+m}^{\eps0} \big (\Psi^{\eps0}_{n,m}\big) \, .
\end{aligned}
\end{equation}
This  formula generalizes the representation introduced in Section \ref{sec:grct} in the sense that,
in the case when~$H_n (Z_n^*([0,t] )) = h_n(Z_n^*(t))$, we obtain 
$$ F^\eps _{n, [0,t]}  ( H_n) = \int  F^\eps _n (t,Z_n^*) h_n(Z_n^*) dZ_n^*\,.$$

More generally,  in analogy with \eqref{eq: marginal time t}, Eq.\,\eqref{FnH}
gives the average (under the initial probability measure) of the function $H_n $ as stated in the next proposition.
\begin{Prop}
\label{prop: identification proba Duhamel}
Let $H_n$ be a bounded measurable function on $D_n([0,t])$.
%  \xout{and with compact support}.  
Then 
\begin{equation}\label{FnH-expectation}\begin{aligned}
\bbE_\eps \Big( \sum_{\substack{i_1, \dots, i_n \\ i_j \neq i_k, j \neq k}} H_n \big( {\mathbf z}^\eps_{i_1}([0,t] ), \dots ,  {\mathbf z}^\eps_{i_n}([0,t] ) \big) \Big)
=
\mu_\eps^{n} F^{\eps}_{n, [0,t]} (H_n)  \,.
\end{aligned}
\end{equation}
\end{Prop}

\begin{proof}
To establish (\ref{FnH-expectation}), we  first look at the case of a discrete sampling of  trajectories
$$
H_n ({\mathbf Z} ^\eps_n([0,t]) )  = \prod_{i= 1}^p h_n^{(i)} ({\mathbf Z} ^\eps_n(\theta_i) )
$$
for some decreasing sequence of times $\Theta = (\theta_i)_{1\leq i \leq p}$ in $[0,t]$,  and some family of bounded continuous functions~$\left(h_n^{(i)}\right)_{1\leq i \leq p}$ with $h_n^{(i)} : \bbD^n \to \bbR$.

% \begin{figure}[h] %  figure placement: here, top, bottom, or page
%\centering
%\includegraphics[width=5in]{sampling.pdf} 
%\caption{\small 
% A discrete sampling of correlations between trajectories. We represent the collision tree coding the structure
% of the pseudo-trajectory. Each line corresponds to one particle. The circled dots correspond to times at which a weight appears in the expansion \eqref{eq: duhamel tilde} {\color{blue} renumeroter $\theta, h$}.}
% \label{fig:sampling}
%\end{figure}

\medskip
\noindent
\underline{First step.} $ $ To take into account the discrete sampling $H_n$, we  proceed recursively and define for any $\tau \in [0,t]$ 
$$H_{n, \tau}  ({\mathbf Z} ^\eps_n([0,t]) )  :=\left(  \prod_{\theta_i \leq \tau }  h_n^{(i)} ({\mathbf Z} ^\eps_n(\theta_i) )\right) \left(  \prod_{\theta_j >\tau}  h_n^{(j)}({\mathbf Z}^\eps _n(\tau) )\right)  \,.$$
In particular, for $\tau \leq \theta_p \leq \dots \leq \theta_1$,
the function $H_{n,\tau}$ depends only on the density at time $\tau$ so that
$$ 
\bbE_\eps \Big( \sum_{\substack{i_1, \dots, i_n \\ i_j \neq i_k, j \neq k}} H_{n,\tau}  \big( {\mathbf z}^\eps_{i_1}([0,t] ), \dots ,  {\mathbf z}^\eps_{i_n}([0,t] ) \big) \Big)
=\mu_\eps^n \int F_n^\eps (\tau, Z_n^*)  \prod_{j= 1} ^p h_n^{(j)}(Z^*_n) dZ^*_n\,.$$
We then define the  biased distribution
$$
\tilde F_n ^\eps( \tau, Z_n^*)
:=F_n^\eps (\tau, Z_n^*)  \prod_{j= 1} ^p h_n^{(j)} (Z^*_n)\, \,  \hbox{ for }\, \,  \tau \in [0,\theta_p]
$$ 
and then extend this   biased correlation function $\tilde F_n^\eps (\tau, Z_n^*) $ on $[0,t]$ so that
$$ 
\bbE_\eps \Big( \sum_{\substack{i_1, \dots, i_n \\ i_j \neq i_k, j \neq k}} H_{n,\tau}  \big( {\mathbf z}^\eps_{i_1}([0,t] ), \dots ,  {\mathbf z}^\eps_{i_n}([0,t] ) \big) \Big)
=\mu_\eps^n \int \tilde F_n^\eps (\tau, Z_n^*)  dZ^*_n\,.$$

In order to characterize  $\tilde  F_n^\eps  (\tau)$, we
 have to iterate the Duhamel formula~(\ref{eq: duhamel1}) in time slices $[\theta_{i+1}, \theta_i]$
 as in the proof of Proposition~2.4 of~\cite{BGSR3} (see also \cite{BLLS80,BGSR2}). 
More precisely we start by 
 writing the Duhamel formula~(\ref{eq: duhamel1}) on~$[\theta_1,t]$, and bias the data at time~$\theta_1^-$ by~$h_n^{(1)}$. This gives, with the notation introduced in Definition~\ref{def: pseudo-trajectory} for the pseudo-trajectories~$Z_{n,m}(\tau)$,
  $$
 \begin{aligned}
  \tilde F^\eps_n(t,Z_n^*) &=  \sum_{ k_{1} \geq0}Q ^\eps _{n,n+ k_{1}} ( t - \theta_{1} ) \tilde F^\eps_{n+ k_{1}}\big(\theta_1^+,Z_{n,k_{1}}(\theta_1)\big) \\
 &=  \sum_{ k_{1} \geq0}Q ^\eps _{n,n+ k_{1}} (t - \theta_{1})h_n^{(1)}(Z_n^* (\theta_{1})) \tilde F^\eps_{n+ k_{1}}\big(\theta_1^-,Z_{n,k_{1}}(\theta_1)\big) \, .
 \end{aligned}$$ 
 Similarly
  $$
  \tilde F^\eps_{n+ k_{1}}\big(\theta_1^-,Z_{n,k_{1}}\big) =  \sum_{ k_{2} \geq0}Q ^\eps _{n+ k_{1},n+ k_{1}+k_2} (  \theta_{1}-  \theta_{2})h_n^{(2)}(Z_n^* (\theta_{2})) \tilde F^\eps_{n+ k_{1}+k_2}\big(\theta_2^-,Z_{n,k_{1}+k_2}(\theta_2)\big) \, .
 $$ 
   We  obtain by iteration   that
 \begin{equation}
 \label{eq: tilde f}
 \begin{aligned}
\tilde F_n ^\eps( t ) 
 &= \sum_{k_1 + \dots+ k_{p+1} \geq0} 
 Q ^\eps _{n,n+ k_1} (t - \theta_{1} ) 
h_n^{(1)}(Z_n^* (\theta_{1})) Q ^\eps _{n+ k_1,n+ k_1+k_2} (\theta_{1} - \theta_{2} )   \\
&\qquad \dots  h_n^{(p)}(Z_n^*(\theta_{p}))  Q ^\eps _{n+ k_1+\dots + k_{p} ,n+ k_1+ \dots + k_{p+1}} (\theta_p) F_{n+ k_1+ \dots + k_{p+1}}^{\eps 0}  
 \;,
\end{aligned}
 \end{equation}
  which leads to \eqref{FnH-expectation} for discrete samplings. 
  
  \medskip
\noindent
\underline{Second step.} $ $  More generally any function $H_n$ on $( \D^n )^p$ can be approximated in terms of products of functions on $\D^n$, thus \eqref{eq: tilde f} leads to 
\begin{align*}
\bbE_\eps \Big( \sum_{\substack{i_1, \dots, i_n \\ i_j \neq i_k, j \neq k}} H_{n}  \big( {\mathbf z}^\eps_{i_1}([0,t] ), \dots ,  {\mathbf z}^\eps_{i_n}([0,t] ) \big) \Big)
=\mu_\eps^n \sum_{k_1 + \dots+ k_{p+1} \geq0} 
 Q ^\eps _{n,n+ k_1} (t - \theta_{1} ) 
 Q ^\eps _{n+ k_1,n+ k_1+k_2} (\theta_{1} - \theta_{2} )   \\
  \dots   Q ^\eps _{n+ k_1+\dots + k_{p} ,n+ k_1+ \dots + k_{p+1}} (\theta_p) H_n ( Z^*_n(\theta_1), \dots, Z^*_n(\theta_p))  F_{n+ k_1+ \dots + k_{p+1}}^{\eps 0} 
\end{align*}
where the Duhamel series is weighted by the $n$-particle pseudo-trajectories at times $\theta_1, \dots , \theta_p $.

\medskip
\noindent
\underline{Third step.} $ $ 
For any $0 \leq \theta_p < \dots <\theta_1 <t $, we denote by~$\pi_{\theta_1, \dots ,\theta_p}$  
the projection from $D_n([0,t])$ to~$( \D^n )^p$
\begin{equation}
\label{eq: projection}
\pi_{\theta_1, \dots ,\theta_p} (Z_n([0,t])) = ( Z_n(\theta_1), \dots, Z_n(\theta_p)) \,.
\end{equation}
The $\sigma$-field of Borel sets for the Skorokhod topology can be generated by the sets of the form~$\pi_{\theta_1, \dots ,\theta_p}^{-1} A$ with $A$ a subset of $( \D^n )^p$ (see Theorem 12.5 in \cite{Billingsley}, page 134).
This completes the proof of Proposition~\ref{prop: identification proba Duhamel}.
  \end{proof}
}
\bigskip

To simplify notation, we are going to denote by $\Psi^\eps_n$ the pseudo-trajectory during the whole time interval~$[0,t]$, which  is  encoded by its starting points $ Z_n^* $ and the evolution parameters~$(a, T_{m}, \Omega_{m}, V_{m})$. 
  Similarly we use the compressed notation~$ \indc_{\cG^{\e}} $    for the constraint that the   parameters~$( T_{m}, \Omega_{m}, V_{m})$ should be in~$\cG_{ m}^{\e}(a, Z_n ^*) $ as in Definition~\ref{def: pseudo-trajectory}.
The parameters~$(a, T_{m}, \Omega_{m}, V_{m})$  are distributed according to the measure
\begin{equation}
\label{eq: measure nu}
d\mu ( \Psi^\eps_n) := \sum_{m} \sum_{a \in \cA^\pm_{n,m} } dT_{m}  d\Omega_{m}  dV_{m} \indc_{\cG^{\e}} (\Psi^\eps_n) \prod_{k=1}^{m}   \Big (s_k\left(\big( v_k -v_{a_k} (t_k)\big) \cdot \omega_k\right)_+\Big ) \, .
\end{equation}
%The product of the cross--section factors of the pseudo-trajectory $\Psi^\eps_n$ will be denoted by 
%\begin{equation}
%\label{eq: cross section}
%\cC \big( \Psi^\eps_n \big) := \prod_{k=1}^{m}    \big( v_k -v_{a_k} (t_k)\big) \cdot \omega_k  \; .
%\end{equation}
The weight coming from the function~$H_n$ 
will be denoted by
\begin{equation}
\label{eq: poids1}
\cH \big( \Psi^\eps_n \big) :=   H_n\big( Z_n^*([0,t])\big)  \,.
\end{equation}
%Finally $\indc_{\cG^{\e}} ( \Psi^\eps_n \big)$ will be the indicator function $\indc_{\cG_{ m}^{\e}(a, Z^*_n)}$.

Formula \eqref{FnH} can  be rewritten
\begin{equation}
\label{eq: notation compressee}
 F_{n, [0,t]} ^\eps(H_n )
 = \int dZ_n^* \int d\mu ( \Psi^\eps_n)  \; \cH \big( \Psi^{\eps}_n \big)
\; F^{\eps 0} \big (  \Psi_n^{\eps0}\big) \,  ,
\end{equation}
and $F^{\eps 0} \big (\Psi_n^{\eps0}\big)$ stands for the initial data evaluated on the configuration at time 0 of 
the pseudo-trajectory (containing $n+m$ particles).

The series expansion \eqref{eq: notation compressee} is absolutely convergent, uniformly in $\e$, for times smaller than some~$T_0>0$: this determines the time restriction in Theorem \ref{thm: Lanford} (see Remark \ref{rem:Tetoile}). 
%More precisely, $T_0$
%is defined by the following condition:
%\begin{equation}
%\forall t \in[0,T_0]\;,\qquad \sup_{n\geq 1} \left[\sup_{\D^n}\int d|\mu| ( \Psi^\eps_n)\,
%C_0^n\,  e^{-\frac{\beta_0}{4} |V^*_n|^2}\right]^{\frac{1}{n}}< 1\;,
%\end{equation}
%and can be bounded by~$c \sqrt{\beta_0}/C_0$, where~$c$ is a universal constant depending only on the dimension (see Chapter~\ref{estimate-chap} for details).}

%%%%%%%%%%%%%%%%%%%%%%%%%%%%%%%%%%%%%%%%%%%%%%%%%%%%%%%

\chapter{Cumulants for the hard-sphere dynamics}
\label{cumulant-chap} 
\setcounter{equation}{0}

To understand the structure of dynamical correlations, we are going to describe how the collision trees introduced in the previous chapter (which are the elementary dynamical objects) can be grouped into clusters. We shall identify three different types of correlations \textcolor{black}{(treated in Section \ref{sec:ER}, \ref{sec:O}, \ref{sec:IC} respectively).} Our starting point will be Formula~\eqref{eq: notation compressee}. We will also need the notation $\Psi^\eps_{n} = \Psi^\eps_{\{1,\dots,n\}}$, where a  pseudo-trajectory is  labeled by the ensemble of its roots.

\smallskip

Notice that the two collision trees in~$\Psi^\eps_{\{1,2\}}$ do   {\it not} scatter if and only if $\Psi^\eps_{\{1\}}$ and $\Psi^\eps_{\{2\}}$ keep a mutual distance larger than~$  \eps$. We shall then write the non-scattering condition as  the complement of an   {\it overlapping} condition, meaning that~$\Psi^\eps_{\{1\}}$ and $\Psi^\eps_{\{2\}}$ reach a mutual distance smaller than $\eps$ (without scattering with each other). \textcolor{black}{The scattering, disconnection and overlap situations are represented in Figure~\ref{rec-nonrec-ov} (recall also Figure \ref{fig:ex-pt}), together with some nomenclature which is made precise below.}

 \begin{figure}[h] 
\centering
\includegraphics[width=6in]{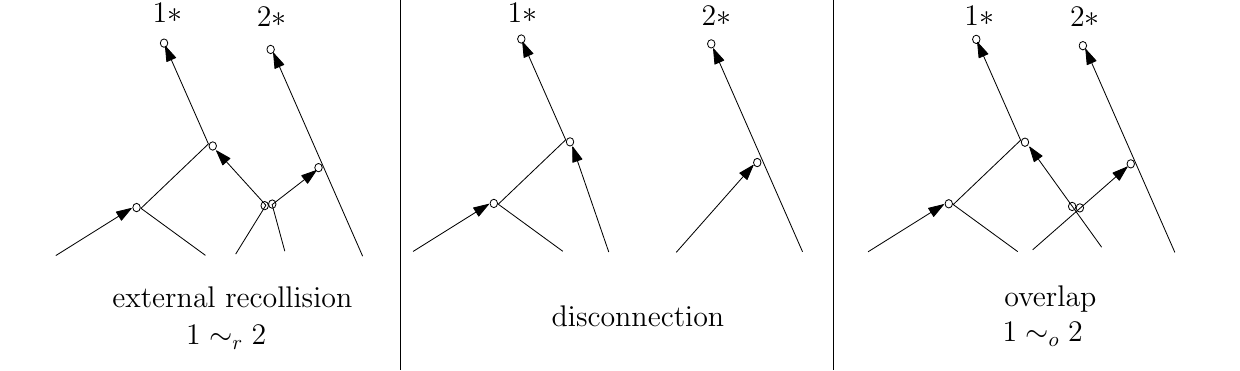} 
\caption{}
%{\color{red} \tt rajouter (independent trees) sous disconnection}
\label{rec-nonrec-ov}\end{figure}

\section{External recollisions}\label{sec:ER}
\setcounter{equation}{0}

A pseudo-trajectory $\Psi^\eps_n$ is made of $n$ collision trees starting  from the roots $Z_n^*$.
These elementary collision trees will be called {\it subtrees}, and will be indexed by the label of their root.
 The parameters $(a, T_{m}, \Omega_{m}, V_{m})$ associated with each collision tree
are independent, and can be separated into $n$ subsets.

The corresponding pseudo-trajectories $\Psi^\eps_{\{1\}}, \dots \Psi^\eps_{\{n\}}$ evolve independently until two particles belonging to different trees collide,
 in which case the corresponding two trees get correlated.  The next definition introduces the notion of  recollision and  distinguishes whether the recolliding particles are in the same tree or not.

\begin{Def}[External/internal recollisions]\index{Recollision!external}\index{Recollision!internal}
\label{def: external recollision}
A  {\it recollision} occurs when 
two pre-existing particles in a  pseudo-trajectory scatter.
A recollision between two particles  will be called an  {\it external recollision} if the two particles involved are in different subtrees (see Figure~{\rm\ref{rec-nonrec-ov}}).  A recollision  between two particles  will be called an  {\it internal recollision} if the two particles involved are  in the same subtree. 
\end{Def} 

Let us now decompose the integral~(\ref{eq: notation compressee}) depending on whether subtrees are correlated or not. 
Recall Definitions \ref{def:cbr} and \ref{def:graphs}.

\begin{Notation} \label{notationmb}
%Given  a partition~$\gl \in \cP^\ell_n$, 
We denote by 
$$\{ j \} \sim_r \{  j'\} $$ the condition: ``there exists an external recollision between particles in the subtrees indexed by~$j$ and~$j'$".
Given $\gl \subset \{1, \dots , n\}$, we denote by~$\mb_{\gl} $ the indicator function that 
any two elements of $\lambda$ are connected by a chain of external recollisions. In other words
%(using Definition \ref{def:graphs})
%for all~$i ,j$ in~$ \gl$, there exists a chain $i_1= i, i_2,\dots, i_{k_{ij}} = j$ of~$k_{ij}$  elements of $\gl$ such that there  is at least one  external recollision between particles in the subtrees indexed by $ {i_p}$ and ${i_{p+1}}$.
\begin{equation}\label{eq:defDM}
\mb_{\gl} = 1 \quad\iff\quad  \exists\, G \in \mathcal C_{\lambda}  \, , \quad 
\prod_{\{j,j'\}\in E(G)}\indc_{\{ j \} \sim_r \{ j' \} }= 1\;.
\end{equation}
%\prod _{i\neq j   }   \prod _{p = 1}^{k_{ij}-1} \Big(1- \indc_{\{{\gl_{i_p}} \not \sim {\gl_{i_{p+1}}} \}} \Big)  = 1\,.$$
Notice that $\mb_{\gl}$ depends only on $\Psi^\eps_{\gl}$. We set $\mb_{\gl}=1$ when $|\gl|=1$. We extend $\mb_{\gl}$ to zero outside~$\cG^{\e}(Z_{\gl}^*)$.
We therefore have the partition of unity
\begin{equation}
\indc_{\cG^{\e}}  \big (  \Psi^\eps_n\big) =  \sum_{\ell =1}^n \sum_{\gl \in \cP_n^\ell}  \left( \prod_{i=1}^\ell \mb_{\gl_i}\,\indc_{\cG^{\e}} \big (  \Psi^\eps_{\gl_i} \big) \right)
\gP_\ell \left( \gl_1, \dots, \gl_\ell \right)%\indc_{\cG^{\e}} \big (  \Psi^\eps_n\big)  \,,
\end{equation}
where $\gP_1 =1$, and $\gP_\ell$ for $\ell>1$ is the indicator function that the subtrees indexed by
 ${\gl_1}, \dots, {\gl_\ell}$ keep mutual distance   larger than~$\eps$. $\gP_\ell$ is defined on $\cup_i \cG^{\e}(Z_{\gl_i}^*)$.
\end{Notation}

Using the notation \eqref{eq: notation compressee}, we can partition the pseudo-trajectories in terms of the external recollisions
\begin{align*}
F_{n, [0,t] }^\eps(H^{\otimes n} )&=  \int dZ_n^*  \sum_{\ell =1}^n \sum_{\gl \in \cP_n^\ell}
 \int d\mu ( \Psi^\eps_n)   \cH \big( \Psi^\eps_n \big)
\;  \Big( \prod_{i=1}^\ell \mb_{\gl_i} \Big)
\; \gP_\ell \big( \gl_1, \dots, \gl_\ell \big)
 F^{\eps 0} \big (  \Psi^{\eps0}_n\big)  \,  .
\end{align*}

There is no external recollision between the subtrees indexed by $\gl_1, \dots, \gl_\ell$, so   the pseudo-trajectories are defined independently;  in particular,  assuming from now on that
$$
H_n = H^{\otimes n}
$$ 
with $H$ a measurable function on the space of trajectories $D([0,t])$,
the cross-sections,  the weights and the 
 constraint imposed by~$\cG^{\e}$ factorize
$$
\gP_\ell  \big( \gl_1, \dots, \gl_\ell \big)   \cH \big( \Psi^\eps_n \big)
d\mu \big (  \Psi^\eps_n\big) 
= \gP_\ell  \big( \gl_1, \dots, \gl_\ell \big)
\Big ( \prod_{i=1}^\ell \cH \big( \Psi^\eps_{\gl_i} \big)d \mu\big( \Psi^\eps_{\gl_i} \big)  \Big)  
$$
and we get 
\begin{equation}
\begin{aligned}
\label{eq: decomposition clusters 1}
F_{n, [0,t] }^\eps(H^{\otimes n} )=  \int dZ_n^* \sum_{\ell =1}^n \sum_{\gl \in \cP_n^\ell}
 \int  \;  \Big( \prod_{i=1}^\ell d\mu\big( \Psi^\eps_{\gl_i} \big)   \cH \big( \Psi^\eps_{\gl_i} \big)  \mb_{\gl_i} \Big)
\;  \gP_\ell \big( \gl_1, \dots, \gl_\ell \big) 
 F^{\eps 0} \big (  \Psi^{\eps0}_n\big)\,    .
 \end{aligned}
\end{equation}
The function $ \gP_\ell $ forbids any overlap between different subtrees $\lambda_i$ in \eqref{eq: decomposition clusters 1}. In particular, notice that $ \gP_\ell $ is equal to zero if $|x^*_i-x^*_j| < \e$ for some $i \neq j$ (compatibly with the definition of $F_{n, [0,t] }^\eps$)\,.

Although the subtrees  $\Psi^\eps_{\gl_1}, \dots, \Psi^\eps_{\gl_\ell}$ in the above formula have no external recollisions, they are  not yet fully  independent as their parameters are  constrained precisely by the fact that 
no external recollision should occur.
Thus we are going to decompose further the collision integral.

\section{Overlaps} \label{sec:O}
\setcounter{equation}{0}

In order to identify all possible correlations, we   now introduce a cumulant expansion of the constraint~$\Phi_\ell$ encoding the fact that no external recollision should occur between the different~$\lambda_i$.

\begin{Def}[Overlap]\index{Overlap}
\label{def:overlap}
An  {\it overlap} occurs between two subtrees if two pseudo-particles, one in each subtree, find themselves at a distance   less than~$\eps$ one from the other for some $\tau \in [0,t]$ (see Figure~{\rm\ref{rec-nonrec-ov}}).
\end{Def} 
\begin{Notation} \label{notationov}
%Given  a partition~$\gl \in \cP^\ell_n$, 
We denote by 
$${\gl_i} \sim_o {\gl_j}$$ the relation: ``there exists an overlap between two subtrees belonging to $\gl_i$ and $\gl_j$ respectively", and we denote ${\gl_i} \not\sim_o {\gl_j}$ the complementary relation. % (note that $\not\sim_o = \not\sim_r$).
Therefore
\begin{equation}
\gP_\ell \big( \gl_1, \dots, \gl_\ell \big) = \prod_{1 \leq i \not = j \leq \ell } \indc_{ {\gl_i} \not \sim_o {\gl_j} } \, .
\end{equation}
\end{Notation}

The inversion formula \eqref{eq: cumulants inverse general} (for unrescaled cumulants)
implies that 
$$\label{def-phirho}
\gP_\ell \big( \gl_1, \dots, \gl_\ell \big)
=  \sum_{r =1}^\ell   \sum_{\gr \in \cP_\ell^r}  
\;  \gp_{ \gr} \, , 
$$
denoting $$ \varphi_\rho := \prod_{j=1}^r \varphi_{\rho_j} \,.$$
The cumulants associated with   
the partition $\{\gl_1, \dots,\gl_\ell\}$ are defined for any subset~$ \gr _j$ of~$ \{1, \dots, \ell\}$
as
\begin{equation}\label{defphitilderho}
\gp_{ \gr_j} =
\sum_{u =1}^{| \gr_j|} \sum_{\go \in \cP_{\gr_j}^{u}} (-1)^{u-1} (u-1)! \,
 \gP_{\omega}   \, ,
\end{equation}
where $\go$ is a partition in~$u$ subparts of $ \gr_j$, and recalling the notation \label{notationPhiomega}
$$
 \gP_{\omega}   = \prod_{i=1}^u\gP_{\omega_i} \, , \quad \gP_{\omega_i}  =  \gP_{|\omega_i|} ( \lambda_k ; k \in \omega_i)\,.
$$  
Note that as stated in Proposition~\ref{prop: tree inequality},   the function
$\gp_{\gr}$ is supported on clusters formed by overlapping collision  trees, i.e.
\begin{equation}
\label{eq: phi rho tilde graphs}
\gp_{ \gr_j} =
\sum_{G\in \cC_{ \rho_j}} \prod_{\{ i_1,i_2\}  \in E (G)} (-\indc_{\gl_{i_1} \sim_o \gl_{i_2}}) \,.
\end{equation}

 For the time being let us return to~\eqref{eq: decomposition clusters 1}, which  can thus be further decomposed as 
\begin{equation}
\label{eq: decomposition clusters 2}
\begin{aligned}
 F_{n, [0,t] }^\eps(H^{\otimes n} ) \! =     \int    dZ_n^*\sum_{\ell =1}^n \sum_{\gl \in \cP_n^\ell}
\sum_{r =1}^\ell   \sum_{\gr \in \cP_\ell^r} 
 \int \Big( \prod_{i=1}^\ell d\mu\big( \Psi^\eps_{\gl_i} \big)   \cH \big( \Psi^\eps_{\gl_i} \big)  \mb_{\gl_i} \Big)
\;  \gp_{ \gr}   F^{\eps 0} \big (  \Psi^{\eps0}_n\big)\;.
\end{aligned}\end{equation}
By abuse of notation, the partition $\gr$  can be also interpreted as a partition of~$\{1,\dots, n\} $ 
\begin{equation}\label{relativecoarseness}
\forall j \leq |\rho|\, , \qquad 
\gr_j = \bigcup_{i \in \gr_j} \gl_i \, ,
\end{equation}
coarser than the partition $\gl$. 
 The relative coarseness~{\rm(\ref{relativecoarseness})} will be denoted by $$\gl \hookrightarrow \gr\, .$$

\section{Initial clusters} \label{sec:IC}
\setcounter{equation}{0}
In \eqref{eq: decomposition clusters 2}, the pseudo-trajectory is evaluated at time 0 on the initial distribution $F^{\eps 0} \big (  \Psi^{\eps0}_n\big)$.
Thus the pseudo-trajectories $\{ \Psi^\eps_{\gr_j} \}_{j \leq r}$ remain correlated by the initial data, so we are finally going to decompose the initial measure in terms of cumulants\index{Initial clusters}.

Given $\gr =\{ \gr_1, \dots, \gr_r \}$ a partition of $\{1,\dots, n\}$ into~$r$ subsets, we define the cumulants of the initial data associated with $\gr$ as follows.
For any subset $\tilde \gs$ of $\{1, \dots, r\}$, we set
\begin{equation}
\label{eq : cumulants time 0}
f^{\eps 0}_{\tilde \gs} := 
\sum_{u =1}^{|\tilde \gs|} \sum_{\go \in \cP_{\tilde \gs}^u} (-1)^{u-1} (u-1)!
\;  F^{\eps 0} _{\omega} \,  ,
\end{equation}
where $\go$ is a partition of $\tilde \gs$, and denoting  as previously
$$
 F^{\eps 0} _{\omega}  = \prod_{i=1}^u  F^{\eps 0} _{\omega_i} \, , \quad F^{\eps 0} _{\omega_i}  =F^{\eps 0}  (\Psi^{\eps0}_{\rho_j} ; j \in \omega_i)\,.
$$
 We recall that~$\Psi^{\eps0}_{\gr_j}$ represents
  the pseudo-trajectories rooted in~$Z_{\gr_j}^*$ computed at time~0. They  involve~$m_j$ new particles, so there are~$ |\rho_j| + m_j$ particles at play at time~0, with of course~$\sum_{j=1}^r ( |\rho_j| + m_j)= n+\sum_{j=1}^r m_j = n+m$.
 We  stress  that the cumulant decomposition depends
on~$\gr$ (in the same way as~(\ref{defphitilderho}) was depending on $\gl$).

Given~$\gr =\{ \gr_1, \dots, \gr_r \}$, the initial data can thus be decomposed as\label{initialcluster}
\begin{align*}
F^{\eps 0 }\big (  \Psi^{\eps0}_n\big)  =  \sum_{s =1}^r    \sum_{\gs \in \cP_r^s}  
  f^{\eps 0}_{\gs}\,,  \quad \hbox{with}\quad  f^{\eps 0}_{\gs} =\prod_{i=1} ^s f^{\eps 0}_{\sigma_i} \,  .
\end{align*}
By abuse of notation as above in~(\ref{relativecoarseness}), the partition $\gs$  can be also interpreted as a partition of~$\{1,\dots, n\}$ 
$$
\forall i \leq |\sigma| \,, \qquad 
\gs_i = \bigcup_{j \in \gs_i} \gr_j  \,  ,
$$
coarser than the partition $\gr$. Hence there holds~$\gr \hookrightarrow \gs$.

\bigskip

We finally get
$$    
F_{n, [0,t] }^\eps(H^{\otimes n} )=  \int dZ_n^* \sum_{\ell =1}^n  \sum_{\gl \in \cP_n^\ell}
\sum_{r =1}^\ell   \sum_{\gr \in \cP_\ell^r} \sum_{s =1}^r    \sum_{\gs \in \cP_r^s}   \int   \Big( \prod_{i=1}^\ell d\mu\big( \Psi^\eps_{\gl_i} \big)   \cH \big( \Psi^\eps_{\gl_i} \big)  \mb_{\gl_i} \Big)\; 
 \gp_{ \gr}   \; 
  f^{\eps 0}_{\gs}   \,  .
$$
The $n$ subtrees generated by   $Z_n^*$ have been decomposed into 
nested partitions~$\gl  \hookrightarrow \gr \hookrightarrow \gs$ (see~Figure~\ref{fig: clusters}).
\begin{figure}[h] %  figure placement: here, top, bottom, or page
\centering
\includegraphics[width=6in]{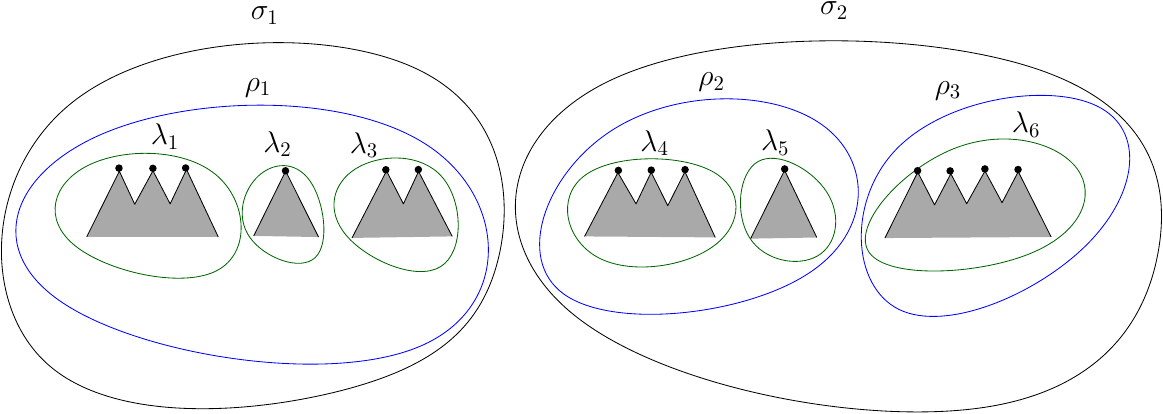} 
\caption{\small 
The figure illustrates the nested decomposition $\lambda \hookrightarrow \gr \hookrightarrow \gs$ in \eqref{eq: decomposition clusters 3}.
The configuration $Z_n^*$ at time $t$ is represented by $n = 14$ black dots. 
Collision trees, depicted by grey triangles,  are created from each dots and all the trees with labels in a subset $\lambda_i$ interact via external recollisions, forming connected clusters (grey mountains).
These trees are then regrouped in coarser partitions $\gr$ and $\gs$ in order to evaluate the corresponding cumulants.  
Green clusters $\lambda$ are called forests, blue clusters $\gr$ are called jungles, and black clusters $\gs$ are called initial clusters.}
\label{fig: clusters}
\end{figure}

Thus we can write
\begin{equation}
\label{eq: decomposition clusters 3}
F_{n, [0,t] }^\eps(H^{\otimes n} )  
=  \int dZ_n^*\!    \sum_{\gl , \gr, \gs \atop \gl  \hookrightarrow \gr \hookrightarrow \gs}
 \int   \Big( \prod_{i=1}^\ell d\mu\big( \Psi^\eps_{\gl_i} \big)   \cH \big( \Psi^\eps_{\gl_i} \big)  \mb_{\gl_i} \Big)\;  \gp_{ \gr} \,      f^{\eps 0}_{\gs}    \,    .
\end{equation}
The order of the sums can be exchanged, starting from the coarser partition $\gs$: we obtain
\begin{equation}
\label{eq: decomposition clusters 4}
 F_{n, [0,t] }^\eps(H^{\otimes n} ) 
 =  \int dZ_n^* \!\sum_{s =1}^n    \sum_{\gs \in \cP_n^s} 
  \prod_{j =1}^s   \! \!  \sum_{\gl, \gr \atop \gl \hookrightarrow \gr \hookrightarrow \gs_j}\!
\int  \Big( \prod_{i=1}^\ell d\mu\big( \Psi^\eps_{\gl_i} \big)   \cH \big( \Psi^\eps_{\gl_i} \big)  \mb_{\gl_i} \Big)
\;  \gp_{ \gr}  
f^{\eps 0}_{\gs_j}   \end{equation}
where the generic  variables  $\gl, \gr$ denote now nested partitions of the subset $\gs_j$.

\section{Dynamical cumulants} \label{sec:dyncum}
\setcounter{equation}{0}
%\subsection{Scaled cumulants}
Using the inversion formula \eqref{eq: cumulants inverse general}, 
the cumulant of order $n$\index{dynamical cumulant} is defined as the term in 
\eqref{eq: decomposition clusters 4} such that~$\gs$ has only 1 element, i.e. $\gs = \{1,\dots, n\} $. We therefore define the (scaled) cumulant, recalling notation~(\ref{eq : cumulants time 0}),
\begin{equation}
\begin{aligned}
\label{eq: decomposition cumulant}
\cum =   \int  dZ_n^* \mu_\eps^{n-1} \sum_{\ell =1}^n \sum_{\gl \in \cP_n^\ell}
\sum_{r =1}^\ell   \sum_{\gr \in \cP_\ell^r}  
\int  \Big( \prod_{i=1}^\ell d\mu\big( \Psi^\eps_{\gl_i} \big)   \cH \big( \Psi^\eps_{\gl_i} \big)  \mb_{\gl_i} \Big)\; \gp_{ \gr} \;  f^{\eps 0}_{\{1,\dots,r\}} (\Psi^{\eps 0}_{\rho_1}, \dots , \Psi^{\eps 0}_{\rho_r})  \,    .
\end{aligned}
\end{equation}

In the simple case $n=2$, the above formula reads
\begin{equation*}
\begin{aligned} 
& f_{2, [0,t] }^\eps(H^{\otimes 2} )  =   \int  dZ_2^*\, \mu_\eps \Big\{ 
\int  d\mu ( \Psi^\eps_{\{1,2\}})   \,\indc_{ \{ 1\} \sim_r \{ 2\} }\,
\cH \big( \Psi^\eps_{\{1,2\}} \big) F^{\eps 0} (\Psi^{\eps 0}_{\{1,2\}}) 
\\
&\quad\;\;  -
\int  \prod_{i=1}^{2} \Big[d \mu ( \Psi^\eps_{\{i\}})  \; 
\cH \big( \Psi^\eps_{\{i\}} \big) \Big] \indc_{ \{1\} \sim_o \{2\}}
F^{\eps 0} \left(\Psi^{\eps 0}_{\{1\}},\Psi^{\eps 0}_{\{2\}}\right) 
\\
&\quad\;\;  +
\int  \prod_{i=1}^{2} \Big[d \mu ( \Psi^\eps_{\{i\}})  \;
\cH \big( \Psi^\eps_{\{i\}} \big) \Big] \left( 
F^{\eps 0} \left(\Psi^{\eps 0}_{\{1\}},\Psi^{\eps 0}_{\{2\}}\right) 
- F^{\eps 0} \left(\Psi^{\eps 0}_{\{1\}}\right) F^{\eps 0} \left(\Psi^{\eps 0}_{\{2\}}\right) 
\right)
 \Big\} ,
\end{aligned}
\end{equation*}
where we used \eqref{eq:defDM}, \eqref{eq: phi rho tilde graphs} and \eqref{eq : cumulants time 0}. The three lines on the right hand side represent the three possible correlation mechanisms between particles $1^*$ and $2^*$ (i.e.\,between the subtrees $1$ and $2$): respectively the  recollision, the  overlap and the correlation of initial data. 
%Notice that for $|x^*_1-x^*_2| < \e$ we have that $\indc_{ \{1\} \sim_r \{2\} } = 0$ and $\indc_{ \{1\} \sim_o \{2\}} = 1$, and therefore only the very last term inside the curly brackets survives, which is completely factorized
%(and equal to $ - F^\eps _1 (t, z^*_1) F^\eps _1 (t, z^*_2)$ in the case $H = 1$).

%\xout{Notice that to simplify we write~$\cum$ for~$f^\eps_{n,[0,t]} (h^{\otimes n})$.}
More generally, looking at Eq.\,\eqref{eq: decomposition cumulant}, we are going to check that  $\cum $ is a cluster of order $n$, and identify a minimal structure in the spirit as the Penrose partition scheme recalled in Chapter~\ref{cluster-chap}.

\begin{itemize}
\item
We start with $n$ trees which are grouped into  $\ell$ \underline{forests}\index{forest} in the partition $\gl$.
In each forest~$\gl_i$ we shall identify   $|\lambda_i| - 1$ ``clustering recollisions". These recollisions give rise to~$ \sum_{i=1}^\ell (|\lambda_i| - 1) = n-\ell$ constraints. %For example in Figure~\ref{fig: clusters} there are 6 forests.

\item
The $\ell$ forests are then grouped into $r$ \underbar{jungles}\index{jungle} $\gr$ and in each jungle $\gr_i$,
we shall identify~$|\gr_i| - 1$ ``clustering overlaps".  These give rise to~$\sum_{i=1}^r (|\gr_i| - 1) = \ell-r$ constraints. \item
The $r$ elements of $\gr$ are then coupled by the initial cluster, and this gives rise to~$r-1$ 
%independent 
constraints.
%For example in Figure~\ref{fig: clusters} there are 3 jungles.
\end{itemize}

By construction $n- 1 = \sum_{i=1}^r  (|\gr_i| - 1) + \sum_{i=1}^\ell (|\lambda_i| - 1) + r-1$.
The dynamical decomposition~\eqref{eq: decomposition cumulant}
implies therefore that the cumulant of order $n$  is associated with pseudo-trajectories with $n-1$ clustering constraints,  and we expect that each of these~$n-1$ clustering constraints will provide a small factor of order   $1/\mu_\eps$.
To quantify rigorously this smallness, we need to identify   $n-1$ ``independent'' degrees of freedom. For clustering overlaps this will be an easy task. Clustering recollisions will require more attention, as they introduce a strong dependence between different trees.

Let us now analyze Eq.\,\eqref{eq: decomposition cumulant} in more detail.
The decomposition can be interpreted in terms of a graph in which the edges represent all possible correlations (between points in a tree, between trees in a forest and between forests in a jungle). In these correlations, 
some play a special role as they specify minimally connected subgraphs in jungles or forests: this is made precise in      the two following   important notions.

Let us start with the easier case of overlaps in a jungle. The following definition assigns a minimally connected graph (cf.\,Definition \ref{def:graphs}) on the 
set of forests grouped into a given jungle. 
\begin{Def}[Clustering overlaps]\index{clustering overlap}
\label{def:ClustOv}

Given a jungle $\rho_i = \{\lambda_{j_1},\dots,\lambda_{j_{|\gr_i|}}\}$ and a pseudo-trajectory $\Psi^\eps_{\rho_i}$, 
we call ``clustering overlaps'' the set of 
$|\gr_i|-1$ overlaps 
%such that
%- each overlap is an  external overlap between forests 
%between forests $v_j$ and $v'_j$, with $v_j,v'_j \in \rho_i\,, v_j\neq v'_j$;
%there are exactly  external overlaps, and 
\begin{equation} \label{ovrel}
(\lambda_{j_1} \sim_o \lambda_{j'_1}),\dots,(\lambda_{j_{|\gr_i|-1}}\sim_o\lambda_{j'_{|\gr_i|-1}})
\end{equation} 
such that
$$\Big\{\{\lambda_{j_1} , \lambda_{j'_1}\},\dots,\{\lambda_{j_{|\gr_i|-1}},\lambda_{j'_{|\gr_i|-1}}\}\Big\} = E(T_{\gr_i})$$ 
where $T_{\gr_i}$ is the minimally connected graph on $\rho_i$ constructed via the Penrose algorithm. 
%$k=1,\cdots,|\gr_i|-1$, we call $k$-th overlap time the time
Given a pseudo-trajectory $\Psi^\eps_{\rho_i}$ with clustering overlaps, we define $|\gr_i|-1$ overlap times as follows: the $k$-th overlap time is
\begin{equation} \label{eq:deftauOV}
\tau_{{\rm{ov}},k} := \sup\Big\{ \tau \geq 0\;:\, 
\min_{\substack{  \mbox{\tiny $q$  {\rm  in}  $\Psi^\e_{\gl_{j_k}}$} \\ \mbox{\tiny $q'$ {\rm  in} $\Psi^\e_{\gl_{j_ k'}}$}}}
 |x_{q'}(\tau) - x_q(\tau)| < \e \Big\}\;.
\end{equation}
% Finally, if $q,q'$ realize the minimum in \eqref{eq:deftauOV}, we define the corresponding overlap vector as
%\begin{equation}\label{eq:omov}
%\omega_{{\rm{ov}},k} := \frac{x_{q'}(\tau_{{\rm{ov}},k}) - x_q(\tau_{{\rm{ov}},k})}{\e}\;.
%\end{equation}
\end{Def}
\begin{Rmk}Contrary to the case of   clustering recollisions defined below (Definition~{\rm\ref{def:ClustRec}}), there is no privileged way of extracting this minimally connected graph, so we choose   the Penrose algorithm (see the proof of Proposition~{\rm\ref{graph-representation}}) for simplicity.
Remark that the times~$\tau_{{\rm{ov}},k} $ are not ordered.
\end{Rmk}

\smallskip

Each one of the $|\gr_i|-1$ overlaps is a strong geometrical constraint which will be used in Part III to gain a small factor $t/ \mu_\e$. More precisely, in Chapter \ref{estimate-chap} we   assign to each forest $\lambda_{j_k}$ a root~$z^*_{\lambda_{j_k}}$ (chosen among the roots of $\Psi^\eps_{\lambda_{j_k}}$).  
Then, it will be possible to ``move rigidly" the whole pseudo-trajectory~$\Psi^\eps_{\lambda_{j_k}}$, acting just on $x^*_{\lambda_{j_k}}$. It follows that one easily translates the condition of ``clustering overlap'' into~$|\gr_i|-1$ independent
constraints on the relative positions of the roots. In fact remember that
%t, in formula \eqref{eq: decomposition cumulant}, 
the pseudo-trajectories~$\Psi^\eps_{\lambda_{j_k}}, \Psi^\eps_{\lambda_{j'_k}}$ do not interact with each other by construction. Therefore~$\lambda_{j_k} \sim_o \lambda_{j'_k}$ means that the two pseudo-trajectories meet at some time $\tau_{{\rm{ov}},k}>0$ and, immediately after (going backwards), they cross each other freely.  This  corresponds to a small measure set in the variable $x^*_{\lambda_{j_k'}}-x^*_{\lambda_{j_k}}$. 
%Thanks to the non-interaction between different forests, one has 
%have always one vertex of degree $1$, it is be enough to ``prune'' vertices of degree $1$ in an arbitrary order.
%In other words, every forest as a whole is an independent degree of freedom.

Contrary to overlaps,  recollisions are unfortunately not independent from one another. For this reason, the study of recollisions 
of trees in a forest needs more care. 
In this case we need to fix the order of the recollision times. Then
we can identify an ordered sequence of relative positions (between trees) which do not affect the previous recollisions.
One by one and following the ordering, such degrees of freedom are shown to belong to a small measure set.
The precise identification of degrees of freedom will be explained in Section \ref{subsec:estDC} and is based on
the following notion.

%We recall that a collision tree is identified by the label of its root. The following definition assigns a minimally connected graph on the 
%set of collision trees grouped into a given forest. 
%

\begin{Def}[Clustering recollisions]\index{clustering recollision}
\label{def:ClustRec}

Given a forest $\lambda_i = \{i_1,\dots,i_{|\lambda_i|}\}$ and a pseudo-trajectory $\Psi^\eps_{\lambda_i}$, we call ``clustering recollisions'' the set of recollisions 
%between trees $$(j_1,j'_1),\cdots,(j_{|\lambda_i|-1},j'_{|\lambda_i|-1})$$
identified by the following iterative procedure.

- The first clustering recollision is the first external recollision in $\Psi^\eps_{\lambda_i}$ (going backward in time);
we    rename the recolliding trees $j_1,j'_1$ and the recollision time $\tau_{\rm{rec},1}$.% and choose $j_1 < j'_1$. 

- The $k$-th clustering recollision is the first external recollision in $\Psi^\eps_{\lambda_i}$ (going backward in time) 
such that, calling $j_k,j'_k$ the recolliding trees, $\{\{j_1,j'_1\},\dots,\{j_k,j'_k\}\} = E\left(G^{(k)}\right)$ where $G^{(k)}$ is a graph  with no cycles (and no multiple edges). We denote the recollision time $\tau_{\rm{rec},k}$.

In particular, 
\begin{equation}
\label{eq:defCR} 
\tau_{\rm{rec},1} \geq \dots \geq \tau_{\rm{rec},|\lambda_i|-1} \quad \mbox{and} \quad 
 \Big\{\{j_1,j'_1\},\dots,\{j_{|\lambda_i|-1},j'_{|\lambda_i|-1}\}\Big\} = E(T_{\gl_i})  
\end{equation}
where $T_{\gl_i}$ is a minimally connected graph on $\lambda_i$.

If $q,q'$ are the particles realizing the $k$-th recollision, we define the corresponding recollision vector by
\begin{equation}\label{eq:omrec}
  \omega_{\rm{rec},k} := \frac{x_{q'}  (\tau_{\rm{rec},k}) -   x_{q} (\tau_{\rm{rec},k}) }{\eps}\;.
 \end{equation}
\end{Def}

The important difference between Definition \ref{def:ClustRec} and Definition \ref{def:ClustOv} is that we have given an order to the recollision times in Eq.\,\eqref{eq:defCR} (which does not exist in Eq.\,\eqref{eq:deftauOV}).

From now on, in order to distinguish, at the level of graphs, between clustering recollisions and clustering overlaps, we shall decorate edges
as follows.
\begin{Def}[Edge sign]\index{edge sign}
\label{def: edge sign}
An edge has sign $ +$ if it represents a clustering recollision.
An edge has sign $ - $ if it represents a clustering overlap.
\end{Def}

Collecting together clustering recollisions and clustering overlaps, we obtain $r$ minimally connected 
clusters, one for each jungle. In particular, we can construct a graph $G_{\gl , \gr}$ made of $r$ minimally connected 
components. To each $e \in E(G_{\lambda, \rho})$, we associate a sign (+ for a recollision and $-$ for an overlap), and a clustering time $\tau^{clust}_e$.
%One way to do this is the following. We first set $G'_{\gr_i} = \cup_{\lambda_i \in \gr_i}T_{\gl_i}$ \textcolor{black}{
%(a graph with all positive edges)}. Secondly, we construct trees $\left(T_{\gr_i}\right)_{i=1}^r$ as in Definition \ref{def:ClustOv}.
%%In this step we have some freedom % for instance according to the Penrose partition scheme at page \pageref{Penrose} 
%(notice that $T_{\gl_i}$ is uniquely determined, while~$T_{\gr_i}$ is not). 

%Next, for each one of the overlap relations $\lambda_{i} \sim_o \lambda_{i'}$ in \eqref{ovrel}, we specify the two subtrees $j \in \lambda_{i} , j' \in  \lambda_{i'}$ realizing the first overlap between
%the forests going backward in time, and we call $G''_{\gr_i}$ the graph on $\gr_i$ with \textcolor{black}{
%(negative)} edges $\cup \{j,j'\}$. 
%Then we obtain that $T'_{\gr_i} := G'_{\gr_i} \cup G''_{\gr_i}$ is a minimally connected graph on the jungle.
%Finally, we set
%\begin{equation} \label{eq:dcg}
%G_{\gl , \gr} := \bigcup_{i=1}^{r}T'_{\gr_i}
%\end{equation}
%where the union runs over disjoint trees. 
%For each $\{j,j'\} \in E(G_{\gl,\gr})$, we have identified
%a clustering (recollision or overlap) time, which we will denote $\tau^{\rm clust} _{\{j,j'\}}$.
%

Our main results describing the structure of dynamical correlations will be  proved in the third part of this paper.  
The major breakthrough in this work    is to remark that one can obtain uniform bounds for the cumulant of order $n$  for all $n$ with a controlled growth.  We recall indeed that we expect each  clustering to produce a small factor~$t/\mu_\eps$, so that the (scaled) cumulant~$f^\eps_n(t) $ of order~$n$ defined in~(\ref{eq: decomposition cumulant}) should be bounded in~$\eps$.  Moreover the number of minimally connected graphs with~$n$ vertices is~$n^{n-2}$ so we expect~$f^\eps_n(t)$ to grow as~$(Ct)^{n-1} n!$. This is made precise in the following theorem, which provides in particular 
  sharp controls on the cumulant generating function $\gL^\eps_{[0,t]}$ from which the large deviation estimates are derived in Chapter  \ref{LDP-chap}.
The following theorem will be proved in Section~\ref{subsec:proofThm} as Theorem~\ref{cumulant-thm1*}.
 
\begin{Thm}
\label{cumulant-thm1}
Consider the system of hard spheres under the  initial measure~{\rm(\ref{eq: initial measure})},  with~$f^0 $ satisfying~{\rm(\ref{lipschitz})}.
Let  $H : D([0,\infty[) \mapsto \bbR$  be a continuous function such that 
\begin{equation} \label{eq:boundHn}
|H^{\otimes n}(Z_n ([0,t]) )| \leq \exp \Big(\alpha\, n + \frac{\beta_0} 4 \sup_{s\in [0,t] }   |V_n(s)|^2 \Big) 
\end{equation}
for some~$\alpha \in \R$.
Define  the scaled cumulant  $ \cum$  by {\rm(\ref{eq: decomposition cumulant})}, with the   notation~{\rm(\ref{eq: measure nu})}. 
  Then  there exists a  positive constant $C$   such that  the following uniform a priori bound holds for any~$t \leq T_0$:
 \begin{equation}  \label{eq:numerozero}
 | \cum|  \leq (Ce^{\alpha})^n\big(t+\eps \big ) ^{n-1}  n!\;.
 \end{equation}
In particular  there is a constant~$c<1$ depending only on the dimension such that  setting $H = e^{h}-1$,
%Let  $h : D([0,\infty]) \to \bbR$  be a continuous function such that 
%$$|(e^h) ^{\otimes n}(Z_n ([0,t]) )| \leq \exp \Big(\alpha n + \frac{\beta_0} 4 \sup_{s\in [0,t] }   |V_n(s)|^2 \Big) 
%$$
%with $ \alpha, \beta_0, T_\alpha$ defined by {\rm(\ref{defTstar})}.
%
the  series defining the cumulant generating function is absolutely convergent on a time $[0,T_\alpha]$ with~$T_\alpha = c\, e^{-\alpha} \beta_0^{(d+1)/2}/C_0$:
\begin{equation}
\label{Ieps-def dynamics}
\begin{aligned}
\forall t \leq T_\alpha \, , \quad  \gL^\eps_{[0,t]} (e^h) 
:=\frac{1}{\mu_\eps} \log \bbE_\eps \left( \exp \Big(  \sum_{i =1}^\cN h \big( {\bf z}^\eps_i ([0,t] \big)  \Big)   \right)
=  \sum_{n = 1}^\infty {1\over n!}  f^\eps_{n, [0,t]} \big( ( e^h - 1)^{\otimes n} \big) \, .
\end{aligned}
\end{equation}
   \end{Thm}

Note that \eqref{Ieps-def dynamics} follows easily from the uniform bounds \eqref{eq:numerozero} on the rescaled cumulants, recalling Proposition~\ref{prop: exponential cumulants}.

In the next chapter, we shall prove the existence of the limiting cumulant generating function (Theorem~\ref{cumulant-thm2}) and the form of the limit will be characterized explicitly (Theorem~\ref{Thm: exponential cumulants dynamics}). As is known from the general theory \cite{DV,dembozeitouni,RezLect} such a result implies upper and lower large deviation bounds, which will be obtained later on in Chapter~\ref{LDP-chap} (see Sections~\ref{sec:  Upper bound LD} and~\ref{sec:  Lower bound LD}).

%%%% Analysis of the limiting cumulants

%\input{BGSRS_part2.tex}

%%% TCL + grandes deviations

\chapter{Characterization of the limiting cumulants}\label{HJ-chapter}
\label{Limiting cumulants}
\setcounter{equation}{0}

Thanks to the uniform bounds obtained in Theorem \ref{cumulant-thm1} we expect that, for all $n$,  there is a limit $\lcum$ for~$\cum$ as $\mu_\eps \to \infty$.
 Our goal in this chapter is first to obtain  a description 
 of~$\lcum$  in terms of
 a series expansion  similar to {\rm(\ref{eq: decomposition cumulant})}, with a precise definition of the limiting pseudo-trajectories (see Theorem~\ref{cumulant-thm2} in Section~\ref{limiting pseudo-traj}
 below): the main feature of those pseudo-trajectories is that they correspond to minimally connected collision graphs.

  In Section~\ref{Limiting cumulant generating function} we derive a series expansion for the limiting
 cumulant generating function (Theorem~\ref{Thm: exponential cumulants dynamics}) which 
 is shown to satisfy a Hamilton-Jacobi equation in Section~\ref{limitequations} (Theorem~\ref{HJ-prop}); the fact that the limiting  graphs  have no cycles  is crucial for the derivation of this equation.

This Hamilton-Jacobi equation encodes all the dynamical correlations. In particular, 
the convergence of the typical density to the Boltzmann equation is recovered from the Hamilton-Jacobi equation in Section  \ref{sec: Boltzmann equation} and the limit covariance in Section  \ref{sec: Equation for the limit covariance}.

\section{Limiting pseudo-trajectories and graphical representation of   limiting cumulants}\label{limiting pseudo-traj}
\setcounter{equation}{0}

In this section we       characterize the limiting cumulants~$\lcum$  by their integral representation. 
This means that we have to specify both the  limiting pseudo-trajectories and the limiting measure.
 
 We first describe the formal limit of \eqref{eq: decomposition cumulant}. To   this end, we
  start by giving a definition of minimal pseudo-trajectories associated with cumulants for fixed~$\eps$. Recall that the cumulant $\cum$ of order $n$ corresponds to graphs of size $n$ which are completely connected, either by recollisions, or by overlaps, or by initial correlations.  
  It will be proved in Chapter~\ref{convergence-chap}  that
  
  \begin{itemize}
\item clusterings coming from the defect of factorization of the initial data are smaller by a factor $O(\eps)$ and thus will not contribute to the limit,
  
  \item cycles are created by    additional  (non clustering)  recollisions or overlaps and have a vanishing contribution in the limit.    
\end{itemize}
Thus 
only  pseudo-trajectories corresponding to  minimally connected graphs  will be considered in this section.

\begin{Def}[Minimal cumulant pseudo-trajectories]\label{Psineps}
Let $m\geq 0$. 
The  cumulant pseudo-trajectory~$\Psi^\e_{n,m}$ associated with the minimally connected graph~$T \in \cT_n^\pm$ decorated with edge signs~$\left(s^{\rm clust}_e\right)_{e \in E(T)}$, and the decorated collision tree $a \in \cA_{n,m} ^\pm $ is obtained by fixing~$Z_n^*$ and a collection of~$m$     creation times~$T_m$ in decreasing order,   and parameters~$( \Omega_m, V_m)$.
 The  cumulant pseudo-trajectory is constructed backward according to the following rules. At each step the set of particles follows the backward free transport until two of them approach at a distance~$\eps$ or we reach a time $t_k$. 
 
 At a time~$t_k$, a new particle, labeled~$k$, is adjoined at   position~$x_{a_k}(t_k)+s_k\e\omega_k$ and with velocity~$v_k$. 
 \begin{itemize}
\item 
  If~$s_k >0 $ 
 then the velocities~$v_k$ and~$v_{a_k}$ are changed to~$v_k(t_k^-)$ and~$v_{a_k}(t_k^-)$ according to the laws~\eqref{V-def},
 
 \item then all particles are transported (backwards) in $\cD^\eps_{n+k} $.  
 \end{itemize}
 
 When two  particles, say $\{q_e,q_e'\}$,  touch, we look at the roots $j$ and $j'$ of their respective subtrees. 
 \begin{itemize}
 \item If $e=\{j,j'\} $ is not an edge of $T$ or if  this edge has already appeared before in the (backward) process, then the pseudo-trajectory is not admissible.
 \item  Else   we have a   clustering recollision if $s^{\rm clust}_e = +$  or a clustering  overlap if $s^{\rm clust}_e = -$. 
  We say that $\{q_e,q_e'\}$ is a representative of the edge $e$, and we denote this by $\{q_e,q_e'\}  \approx e $.  The clustering time is denoted~$\tau_e^{\rm clust}$, and the clustering angle can be  defined 
  %(after neglecting a set of parameters whose contribution will be shown to vanish, as~$\e \to 0$) 
  by
  $$ 
 \omega^{\rm clust}_{e } := \frac{ x_{q_e}(\tau^{\rm clust}_{e }) - x_{q_e'} (\tau^{\rm clust}_{e})} {\eps}  \in {\mathbb S}^{d-1} \,.
$$
  \end{itemize}
  The pseudo-trajectory is admissible if at time 0 all edges of $T$ have appeared in the construction. We  will order the clustering times, and the edges of $T$ accordingly, and we will denote by~$(  \Theta^{\rm clust}_{n-1} , \Omega^{\rm clust}_{n-1}) $ the collection of clustering times and angles.
\end{Def}
Theorem \ref{cumulant-thm1} will be proved in Section \ref{subsec:proofThm} by establishing, in particular, the uniform convergence of the series expansion \eqref{eq: decomposition cumulant} (on the number of created particles $m$, see \eqref{eq: measure nu}). We  thus focus here on a fixed $m$ and a fixed tree $a \in \cA^\pm_{n,m}$.

The clustering constraints provide $n-1$ conditions on the roots $(z_i^*)_{1\leq i \leq n}$ of the trees, so only one root will be  free.  We set this root to be~$z_{n}^*$.
Given   $(x_i^*, v_i^*)$ and $v_j^*$ as well as collision parameters~$(a, T_m,\Omega_m,V_m)$, since the trajectories are piecewise affine one can perform the local change of variables
\begin{equation}
\label{infinitesimalij}
x_j^* \in\bbT^d \mapsto(\tau^{\rm clust}_{e } ,  \omega^{\rm clust}_{e }) \in (0,t) \times  {\mathbb S}^{d-1}
\end{equation}
with Jacobian
    $ \mu_\eps^{-1}\big( (v_{q_e}(\tau^{\rm clust+}_{e })-v_{q'_e}(\tau^{\rm clust+}_{e}))\cdot \omega^{\rm clust}_{e}\big)_+ \,.
  $ 
  This provides  the   identification of measures
\begin{equation}
\label{identificationmeasures}
 \mu_\eps dx_i^* dv_i^* dx_j^* dv_j^*  \\
= dx_i^* dv_i^*dv_j^* 
d\tau ^{\rm clust}_{e} d\omega^{\rm clust}_{e }
\big( (v_{q_e}(\tau^{\rm clust}_{e})-v_{q_e'}(\tau^{\rm clust}_{e }))\cdot\omega^{\rm clust}_{e}\big)_+ \,.
\end{equation}
We shall explain in Section~\ref{subsec:estDC}
 how to identify a good sequence of roots to perform this change of variables iteratively (see Figure~\ref{fig: roots Chapter 5}).
 
\begin{figure}[h] 
\centerline{\includegraphics[width=4in]{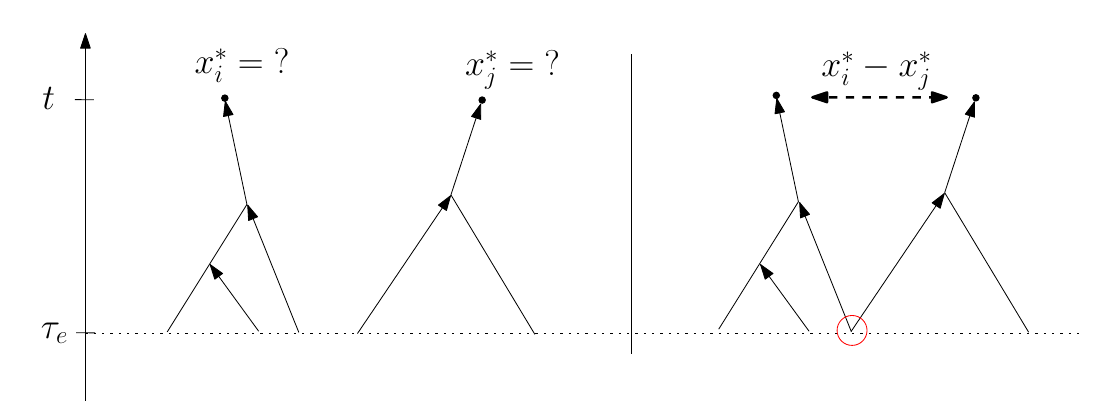}}
\caption{On the left figure, two trees (with roots $x_i^*, x_j^*$) are built independently in the time interval $[\tau_e, t]$ and their roots are not fixed a priori. On the right figure, the clustering condition at time $\tau_e$ imposes a constraint on the relative position  $x_i^*- x_j^*$ of the roots : the trees are shifted rigidly to satisfy the clustering. This procedure is applied iteratively to determine all relative positions at time $t$. Only one root, say $x_n^*$, has to be prescribed.}
\label{fig: roots Chapter 5}
\end{figure}

 For each  tree $a\in \cA^\pm_{n,m} $,  and each minimally connected graph $T\in \cT_n^\pm$, the cumulant pseudo-trajectories are then reparametrized by the  root $x_n^*$, the  velocities $V_n^*$ at time $t$, the sequence $(q_e, q'_e)_{e \in E(T)}$ of clustering particles, the clustering parameters $(\Theta_{n-1}^{\rm clust}, \Omega_{n-1}^{\rm clust})$ and  the collision parameters $(T_m, \Omega_m, V_m)$.

Now let us introduce the limiting cumulant pseudo-trajectories and measure.

\begin{Def}[Limiting cumulant pseudo-trajectories]\label{barPsi}
Let $m\geq 0$. The limiting cumulant pseudo-trajectories~$\Psi_{n,m}$ associated with the ordered trees~$T \in \cT_n^\pm$ and $a \in \cA_{n,m} ^\pm $ are obtained by fixing~$x_n^*$ and $V_n^*$,
 \begin{itemize}
 \item %a partition in forests $\gl \in \cP_n^\ell$, 
 for each~$e  \in E(T)$, a representative~$\{q_e,q'_e\}\approx e$
  \item a collection of~$m$  ordered   creation times $T_m$,   and parameters~$( \Omega_m, V_m)$
\item    a collection of  clustering times and angles~$(  \Theta^{\rm clust}_{n-1} , \Omega^{\rm clust}_{n-1}) $.
\end{itemize}
 At each creation time~$t_k$, a new particle, labeled~$k$, is adjoined at 
 position~$x_{a_k}(t_k) $ and with velocity~$v_k$:
 \begin{itemize}
  \item   if~$s_k = +$,  then the velocities~$v_k$ and~$v_{a_k}$ are changed to~$v_k(t_k^-)$ and~$v_{a_k}(t_k^-)$ according to the laws~\eqref{V-def}, 
\item then all particles follow the backward free flow until the next creation or clustering  time.
\end{itemize}
At each  clustering time $\tau^{\rm clust} _{e}$ the particles $q_e$ and~$q_e'$ are  at the same position:
   \begin{itemize}
  \item  if~$s_{e} = +$, then the velocities~$v_{q_e}$ and~$v_{q'_e}$ are changed  according to the scattering rule, with scattering vector $ \omega^{\rm clust}_{e }$,
 % if~$ (v_{k_i} -v_{k_j})\cdot\omega^{\rm clust} _{\{i,j\}} >0 $ then the velocities~$v_{k_i}$ and~$v_{k_j}$ are changed  according to \eqref{V-def},
\item then  all particles follow the backward free flow until the next creation or clustering time. 
\end{itemize}
 \end{Def}

 Note that, in Definition \ref{Psineps}, positions $X_{n}^*$ at time $t$ were fixed and  clustering conditions were considered as admissibility constraints, while here the positions $X_{n}^*$ at time $t$ are 
not prescribed: they are determined according to an algorithm devised in Section~\ref{subsec:estDC}.

 We can therefore define the limiting   measure, with the notation introduced above: 
\begin{equation}
\label{defdmusing}
 \begin{aligned}
d\mu_{{\rm sing}, T,a}\left( \Psi_{n,m}\right)  &:=  dT_{m}  d\Omega_{m}  dV_{m} dx_n^*  dV_{n}^*  d \Theta^{\rm clust}_{n-1} d\Omega^{\rm clust}_{n-1} \prod_{i=1} ^m  s_i  \big( (v_{i}-v_{a_i}(t_i)\cdot\omega_i\big)_+\\
&\qquad \times\prod_{e\in E(T)} \sum_{\{q_e,q_e'\}  \approx e}  s^{\rm {clust}}_{e  }  \big( (v_{q_e}(\tau^{\rm clust}_{e })-v_{q_e'}(\tau^{\rm clust}_{e }))\cdot\omega^{\rm clust}_{e }\big)_+\,.
\end{aligned}\end{equation}
 We stress the fact that this measure is supported on singular pseudo-trajectories, in the sense that the pseudo-particles interact one with the other at  distance~0.

Equipped with these notations, we can now state the result  that will be proved in 
 Chapter~\ref{convergence-chap}.

 \begin{Thm}
\label{cumulant-thm2}
With the previous notation and the assumptions of Theorem~{\rm\ref{cumulant-thm1}}, for all $t \leq T_0$, the cumulant $\cum$ converges   when~$\mu_\eps \to \infty$ to~$\lcum $  given by
  \begin{equation}
\label{eq: decomposition lcumulant}
\forall t \leq T_0 \, , \quad  \lcum =    \sum_{T \in \cT^\pm_n} \sum_{m=0}^\infty \sum_{a \in \cA^\pm_{n,m} } \int d \mu_{{{\rm sing}, T,a }} ( \Psi_{n,m})  \,   \cH \big( \Psi_{n,m}  \big)      \left(f^0\right)^{\otimes m+n} ( \Psi^0_{n,m} )\,  .
\end{equation}
In particular by Theorem~{\rm\ref{cumulant-thm1}} there exists a constant~$C>0$ and a time~$T_\alpha<1/C$ depending only on~$\alpha, C_0, \beta_0$ such that
 $$\forall t \leq T_\alpha \, , \quad | \lcum|  \leq C^n t^{n-1}   n!\, ,$$
and  the limiting cumulant generating function~{\rm(\ref{Ieps-def dynamics})} has the form  
\begin{equation}
\label{I-def dynamics}
\forall t \leq T_\alpha \, , \quad \gL_{[0,t]}( e^h)
=  \sum_{n = 1}^\infty {1\over n!}  f_{n, [0,t]} \big( ( e^h - 1)^{\otimes n} \big)
=  \lim_{\mu_\eps \to \infty}  \gL^\eps_{[0,t]} (e^h) \, .
\end{equation}
\end{Thm}
  \medskip
 
 Recall that the convergence time $T_0$, in Theorem \ref{thm: Lanford}, 
of the particle system to the solution $f$ of the Boltzmann equation depends only on~$f^0$ 
through $C_0,\beta_0$: as noted in Remark~\ref{rem:Tetoile}, there holds~$T_0\sim C_0^{-1}\beta_0^{(d+1)/2}$.
The parameter $\alpha $ quantifies the size of the deviations from $f$ which can be observed. The time $T_\alpha $ is then adjusted accordingly~: $T_\alpha \sim T_0 e^{-\alpha}$.}

\section{Limiting cumulant generating function}\label{Limiting cumulant generating function}
\setcounter{equation}{0}
 The following result provides a graphical expansion of~$ \gL_{[0,t]}(e^h)$.  \begin{Thm}
\label{Thm: exponential cumulants dynamics}
 Under the assumptions of Theorem~{\rm\ref{cumulant-thm1}},  the limiting cumulant generating function~$ \gL_{[0,t]} $ satisfies for all~$t \leq T_\alpha$
\begin{equation}\label{decgLeh}
 \gL_{[0,t]}(e^h) +1=\sum_{K = 1}^\infty {1\over K!} \sum_{  \tilde T \in \cT_K^\pm} \int d\mu_{{\rm sing},   \tilde T } (\Psi_{K,0}) (e^h)^{ \otimes  K}  (\Psi_{K,0})  f^{0 \otimes K} (\Psi_{K,0}^0)   \, ,
\end{equation}
where
\begin{equation}\label{defdmusingtilde}
 \begin{aligned}
d\mu_{{\rm sing},   \tilde T }  &:=  dx_K^* dV_K  \prod_{e= \{q,q'\} \in E(  \tilde T)}  s_e  \big( (v_{q}(\tau_e)-v_{q'}(\tau_e ))\cdot\omega_{e }\big)_+d\tau_e d\omega_e \, .
\end{aligned}\end{equation}
Furthermore the series is absolutely convergent for~$t \in [0,T_\alpha]$~:
 \begin{equation}
 \label{Lambda-absconv}
  \int d |\mu_{{\rm sing},  \tilde T } (\Psi_{K,0})| \,  (e^h)^{ \otimes  K}  (\Psi_{K,0})   \, f^{0 \otimes K} (\Psi_{K,0}^0)  \leq \big(Ct \big ) ^{K-1}   \, .
  \end{equation}
\end{Thm}
 Compared to Theorem~\ref{cumulant-thm2}, all dynamical connections are dealt with in a symmetric way, resorting to one connected graph~$ \tilde T \in  \cT_K^\pm$, rather than a graph~$T \in \cT^\pm_n$ encoding recollisions and overlaps and a tree~$a \in \cA^\pm_{n,m}$ encoding collisions.

\begin{proof}
By definition and thanks to Theorem~\ref{cumulant-thm2}, $$
 \gL_{[0,t]} \left(e^h   \right)  = \sum_{n = 1}^\infty  {1\over n!}   \sum_{T \in \cT^\pm_{n} }\sum_{m=0}^\infty \sum_{a \in \cA^\pm_{n,m} } \int d \mu_{{\rm sing}, T,a } ( \Psi_{n,m})  (e^h- 1)^{\otimes n}  \left(f^0\right)^{\otimes (m+n)} \, .
 $$
Note that  the trajectories of particles  $i\in \{1, \dots,  m\}$ can be extended on the whole interval $[0,t]$ just by transporting $i$ without collision on $[t_i, t]$~: this  is actually  the only way to have a set of~$m+n$ pseudo-trajectories which is minimally connected (any additional collision would add a non clustering constraint, or require   adding new particles). It can therefore be identified to some $\Psi_{m+n,0}$ (see Figure~\ref{fig:Thm6}).

 Let us now fix $K=n+m$ and symmetrize over all arguments~:
 $$
 \begin{aligned}
 \gL_{[0,t]} \left(e^h  \right)  &= \sum_{K = 1}^\infty {1\over K!} \sum_{n=1} ^K  {K! \over n! (K-n) ! }  (K-n) !  \sum_{T \in \cT^\pm_{n} }  \sum_{a \in \cA^\pm_{n,K-n} } \int d \mu_{{\rm sing}, T,a }( \Psi_{n, K-n })  (e^h- 1)^{\otimes n } \left(f^0\right)^{\otimes K} \\
 & = \sum_{K = 1}^\infty {1\over K!} \sum_{n=1} ^K  \sum_{\eta \atop |\eta| =n}   \sum _{(\eta^c)^\prec}  \sum_{T \in \cT^\pm_{\eta } }  \sum_{a \in \cA^\pm_{\eta ,(\eta^c)^\prec } } \int d \mu_{{\rm sing}, T,a } ( \Psi_{\eta , (\eta^c)^\prec })  (e^h- 1)^{\otimes \eta}   \left(f^0\right)^{\otimes K} 
 \end{aligned}
 $$
 where $\eta$ stands for a subset of $\{1^*,\dots, n^*, 1,\dots  , K-n\} $ with cardinal $n$;  $\eta^c$ denotes its complement and $(\eta^c)^\prec$ indicates that we have chosen an order on the set $\eta^c$. We denote by $\cA^\pm_{\eta ,(\eta^c)^\prec }$
the set of signed trees with roots $\eta$ and added particles with prescribed order in $(\eta^c)^\prec$. 
 
Note that the combinatorics of collisions $a$ and recollisions or overlaps~$T$ (together with the choice of the representatives $\{ q_e, q'_e \}_{e\in E(T)}$) can be described by a single minimally connected graph $\tilde T \in \cT^\pm_{K}$.
In order to apply Fubini's theorem, we then need to understand  the mapping 
$$( a, T, \{ q_e, q'_e \}_{e\in E(T)} ) \mapsto (\tilde T ,\eta)\,.$$ 

It is easy to see that  this mapping is injective but not surjective. Given a pseudo-trajectory $\Psi_{K,0}$ compatible with $\tilde T$ and a set $\eta$ of cardinality $n$, we reconstruct $( a, T, \{q_e, q'_e \}_{e\in E(T)} ) $ as follows.
We  color in red the~$n$ particles  belonging  to  $\eta$ at time $t$,  and in blue the~$K-n$  other particles. Then we follow the dynamics backward. At each clustering, we apply the following rule
\begin{itemize}
\item if the clustering involves one red particle and one blue particle, then it corresponds to a collision in the Duhamel pseudo-trajectory. The corresponding edge of $\tilde T$ will be described by $a$. We then change the color of the blue particle to red. 
\item if the clustering involves two red particles, then it corresponds to a recollision in the Duhamel pseudo-trajectory. The corresponding edge of $\tilde T$ is therefore an edge $e \in E(T)$ and the two colliding particles determine the representative $\{ q_e, q'_e \}$.
\item if the clustering involves two blue particles, then the pseudo-trajectory is not admissible for~$(\tilde T, \eta)$, as it is not associated to any $( a, T, \{ q_e, q'_e \}_{e\in E(T)} ) $.
\end{itemize}

\begin{figure}[h] 
\centerline{\includegraphics[width=3.5in]{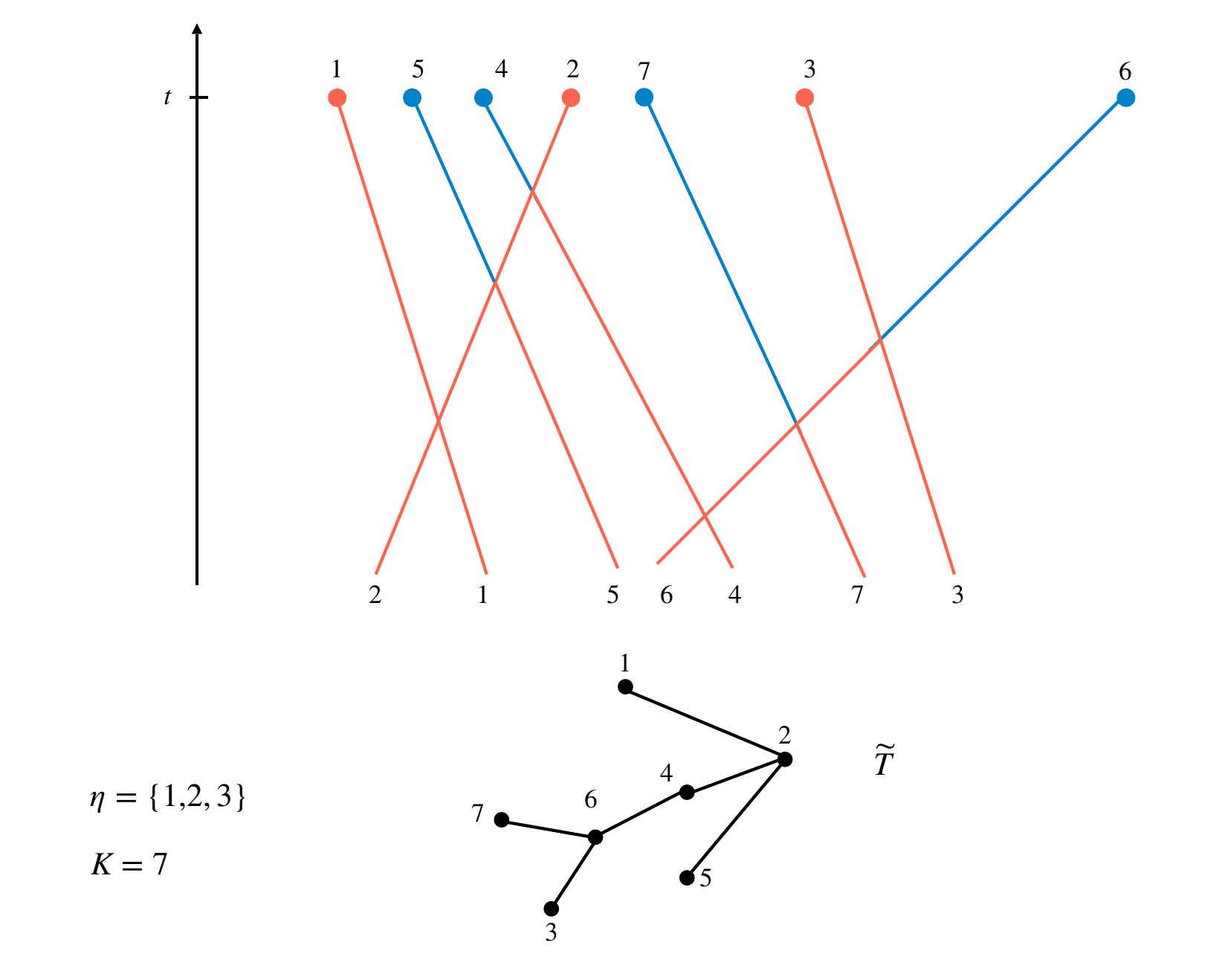}}
\caption{A couple~$(\eta,\tilde T)$ and an associate pseudo-trajectory~$\Psi_{K,0}$.}
\label{fig:Thm6}
\end{figure}

However the  contribution   of the non admissible pseudo-trajectories $\Psi_{K,0}$  to
$$\sum_{\tilde T \in \cT_\eta^\pm} \int d\mu_{{\rm sing}, \tilde T } (\Psi_{\eta,0}) (e^h)^{ \otimes  \eta}   \left(f^0\right)^{\otimes K} $$
 is  exactly zero. Indeed the blue parts of the trajectories are not weighted, so that the overlap and the recollision terms  associated with the first clustering between two blue particles (i.e.\,the $\pm$ signs of the corresponding edge)   exactly compensate.
 
We therefore conclude that 
$$\begin{aligned}
 \gL_{[0,t]} \left(e^h \right)
&= \sum_{K\geq 1 }{1\over K!}  \sum_{\tilde T \in \cT^\pm_{K} } \int d  \mu_{{\rm sing}, \tilde T  } ( \Psi_{K,0}) \left(f^0\right)^{\otimes K}  \sum_{n = 1} ^K \sum _{\eta \in \cP^n_K}  (e^h-1)^{\otimes \eta}  \\
&= \sum_{K\geq 1}{1\over K!}  \sum_{\tilde T \in \cT^\pm_{K} }\int d  \mu_{{\rm sing}, \tilde T  } ( \Psi_{K,0}) (e^h)  ^{\otimes K}  \left(f^0\right)^{\otimes K} - 1 
\end{aligned}
$$
which is exactly (\ref{decgLeh}). 
Note that the compensation mechanism described above does not work for~$n=0$ and $K=1$, which is the reason for the $-1$ in the final formula.

The bound (\ref{Lambda-absconv}) comes from the definition of $\mu_{{\rm sing}, \tilde T  }$ together with the estimates used in the proof of Theorem 4 to control the collision cross-sections.
\end{proof}

\bigskip
\section{Hamilton-Jacobi  equations}\label{limitequations}
\setcounter{equation}{0}

We   consider test functions on the trajectories which write as 
\begin{equation}
\label{defHgammaphi}
h   (z([0,t])  ) = g \big (t,z(t)\big)  -  \int_0^t D_s  g \big ( s,z(s) \big) ds 
\end{equation}
recalling the notation~$D_s  g:= \partial_s g + v\cdot \nabla_x g$. The effect of this specific choice will be to integrate the  transport term in the Hamilton-Jacobi equation. 
We choose complex-valued functions here as we shall be  using properties of analytic functionals of~$g$ later; all the results obtained so far can easily be adapted to this more general setting.
To stress the dependence on $g$, we introduce a specific notation for the corresponding exponential moment~\eqref{I-def dynamics} 
\begin{align}
\label{I-def dynamics J}
\cI (t,g) 
%& :=\frac{1}{\mu_\eps} \log 
%\bbE_\eps \left( \exp \Big(  \sum_{i =1}^\cN H \big( {\bf z}_i ([0,t] \big)  \Big)   \right)\\
 :=\gL_{[0,t]} ( e^{ g(t) -\int_0^t D_s g} ) \, .
\end{align}
Note that $g$ is defined here by its final value $g(t)$ and its transport $Dg = (D_s g)_{0 \leq s \leq t}$, and these two functions will be considered as two independent variables. 

The following statement specifies the functional framework in which~$\cI$ is well defined as a convergent series, and identifies the equation it satisfies.
We recall that for any $\alpha \geq 0$, there exists  $T_\alpha$ (depending only on $\alpha$, $C_0$ and $\beta_0$) such that the cumulant generating function $\Lambda^\eps_{[0,t]}(e^h)$ is uniformly convergent on $[0,T_\alpha]$ provided that $e^h - 1$ satisfies (\ref{eq:boundHn}).
We then define
\begin{equation}
\label{eq: space B}
\begin{aligned}
{\bB }_{\alpha} :=    \Big\{g \in C^1([0,T_\alpha] \times \D;{\mathbb C}) \
&\,   : \qquad 
 | g(t,z)| \leq  ( 1-{t\over 2T_\alpha} )( \alpha +\frac{\beta_0}8 |v|^2) \, ,\\
 &  \sup_{s \in [0,T_\alpha]}   | D_s g (s,z) |  \leq {1\over 2T_\alpha} (\alpha  +\frac{\beta_0}8 |v|^2)
\Big\}\, .
\end{aligned}
\end{equation}

Let us translate Theorems~\ref{cumulant-thm1} and~\ref{Thm: exponential cumulants dynamics} in terms of the functional~$\cI $. 
For~$t $ in~$ [0,T_\alpha]$, let $h$ be defined as in~\eqref{defHgammaphi} with 
$g$ in~${\bB }_{\alpha}$. One has
\begin{equation} 
\begin{aligned}
\label{eq: borne uniforme}
\Big| 
\Big(e^{h (z_i([0,t]) )}-1\Big)^{\otimes n} \Big|  &\leq 
e^{\sum_{i = 1}^n \big| h (z_i([0,t]))  \big|} 
\leq
e^{ \alpha_0n +\frac{\beta_0}8 (1-\frac t{2T_\alpha})   |V_n(t)|^2  + \frac{\beta_0}8 {\frac1{ 2T_\alpha}}\int_0^t |V_n(s)|^2\, ds}\\
&\leq e^{ \alpha n +\frac{\beta_0}8  \sup_{s \in [0,t]}|V_n(s)|^2} \, , 
\end{aligned}
\end{equation}
which is the assumption on~$H = e^{h}-1$ of Theorem~\ref{cumulant-thm1}.
In particular, the   series  
\begin{equation} 
\label{eq:Itglimserie}
\cI (t,g) := - 1
+ \sum_{K = 1}^\infty {1\over K!} \sum_{  T \in \cT_K^\pm} \int d\mu_{{\rm sing},   T } (\Psi_{K,0}) (e^{g(t) - \int_0^t D_sg(s)ds})^{ \otimes  K}  (\Psi_{K,0})  f^{0 \otimes K} (\Psi_{K,0}^0) 
\end{equation}
    is absolutely convergent for~$t \in [0,T_\alpha]$ and~$g \in \bB_\alpha$.  
Note that~(\ref{eq:Itglimserie}) shows that~$\cI$ is    analytic with respect to~$g(t)$:  in particular one can differentiate~$ \cI (t,g)$ with respect to the final condition~$g(t)$, in a direction~$\Upsilon $ and by   term-wise derivation of the series~\eqref{eq:Itglimserie} we find:
\begin{equation}
\label{J-derivative 2}
\begin{aligned}
\int_\bbD dz     {\d \cI (t,g) \over \d g(t) } (z)   \Upsilon (z) 
&=
  \sum_{K}{1\over K!}  \sum_{\tilde T \in \cT^\pm_{K} }\sum_{i=1}^K  \int d  \mu_{{{\rm sing}, \tilde T  }} (\Psi_{K,0})  \Upsilon (z_i(t) )  \\
  & \qquad\qquad\qquad \times
  \left(e^{g (t) -  \int_0^{t}  D_sg  ds} \right)^{\otimes K} (\Psi_{K,0}) \left(f^0\right)^{\otimes K} (\Psi^0_{K,0})
 \, .
 \end{aligned}
 \end{equation}

%\begin{Rmk}
% Additional estimates on $\cI$ will be derived in Proposition  
% \ref{prop: analyticity J part 1}.  
%The fact that the right-hand side of~{\rm(\ref{HJ-eq})} is well defined for~$(t,g) \in [0,T_\alpha] \times \bB_\alpha$, and that~{\rm(\ref{HJ-eq})}   has a unique solution  in the space of differentiable functions in~$t$ with values in~$\bB_\alpha$, is proved in Chapter~{\rm\ref{LDP-chap}} using analytic techniques as in Appendix~{\rm\ref{CauchyKol}}.
%\end{Rmk}

We first state a regularity result on ${\d \cI (t,g) \over \d g(t) }$ needed to define the singularity in the Hamilton-Jacobi equation derived in Theorem \ref{HJ-prop}.
Additional results on $\cI$ in an appropriate functional setting will be derived later in Proposition  \ref{prop: analyticity J} in order to obtain the uniqueness of the Hamilton-Jacobi equation. 
\begin{Prop}
\label{prop: analyticity J part 1} 
For $t \leq T_\alpha$ and~$g \in \bB_\alpha$, the functional 
derivative $(x,v) \mapsto \displaystyle{\d \cI (t,g) \over \d g(t)} (x,v)$
is a continuous function in~$x \in \T^d$ with values in the space~$ {\mathcal M}_v( \bbR^d)$ of weighted measures  in~$v \in \R^d$: there is a constant~$C$ such that for any $g \in\bB_\alpha$,
$$\label{improved-est}
\forall t \leq T_\alpha \, , 
\forall x \in \T^d,
\qquad 
 \Big \| {\d  \cI ( t, g )\over \d g(t)} (x,v) \;  \exp ( \frac{\beta_0}  8 |v|^2) (1+|v|)  \Big  \|_{  {\mathcal M}_v( \bbR^d)}   
 \leq  C \,.
$$
 \end{Prop}

\begin{proof}
Given $K$, we consider the associated integral in the series expansion \eqref{J-derivative 2}. 
The integrand is uniformly bounded by the assumption \eqref{lipschitz} on $f^0$ and  inequality  \eqref{eq: borne uniforme}   
\begin{equation}
\label{eq: upper bound integrand}
\Gamma_K  (\Psi_{K,0})  := 
\left(e^{g (t) -  \int_0^{t}  D_sg  ds} \right)^{\otimes K} (\Psi_{K,0}) 
\left(f^0\right)^{\otimes K} (\Psi^0_{K,0})
 \leq e^{ \alpha K - \frac{3\beta_0}8  |V_K(0)|^2} .
\end{equation}
The   measure~$\mu_{{\rm sing},\tilde T}$ is invariant under global translations in $x$.
Thanks to the upper bound \eqref{eq: upper bound integrand},  each integral in \eqref{J-derivative 2} is uniformly bounded in terms of 
$\| \exp ( - \frac{ \beta_0}8 |v|^2 ) (1+|v|)^{-1}   \Upsilon \|_{L^1_x ( L^\infty_v)}$
\begin{equation}
\label{eq: borne L1 L infty}
\begin{aligned}
&\Big|  \int d  \mu_{{{\rm sing}, \tilde T  }} (\Psi_{K,0})  \Gamma_K  (\Psi_{K,0}) 
 \Upsilon (z_i(t) ) \Big|\\
&\leq  \Big|  \int d  \mu_{{{\rm sing}, \tilde T  }} (\Psi_{K,0})  e^{ \alpha K - \frac{\beta_0}8 |V_K(0)|^2} \Big| \; \| \exp (   -\frac{\beta_0}8 |v|^2 ) (1+|v|)^{-1}  \Upsilon \|_{L^1_x ( L^\infty_v)}\, .
\end{aligned}
\end{equation}

Furthermore, using the continuity of $g$ and $f^0$, we deduce that  $\Gamma_K  (\Psi_{K,0})$ is a continuous function of the root $z_i(t)$, 
as changing the position of the root boils down to translating rigidly the whole pseudo-trajectory. Therefore, by density approximation, one can  extend the convergence  and the bound~(\ref{eq: borne L1 L infty})  to any $\Upsilon$ such that $ \Upsilon \exp (-\frac{\beta_0}  8 |v|^2) (1+|v|)^{-1}   \in {\cM}_x\left(L^\infty_v \right)$ where ${\cM}_x$ is the space of measures. 
Proposition~\ref{prop: analyticity J part 1}  is proved by summing the expansion \eqref{J-derivative 2}.
\end{proof}

The next theorem  is the key to derive the large deviation functional in Chapter~\ref{LDP-chap}.  As a byproduct, it will also allow us  to prove  that  the limit first cumulant~$f_1$ solves  the Boltzmann  equation, and to derive  the equation on the limit covariance. 
\begin{Thm}[Hamilton-Jacobi equation for the limit cumulant generating function]
\label{HJ-prop}
For any~$\alpha>0$, the functional~$ \cI(t,g) $  is well defined on $[0,T_\alpha] \times \bB_\alpha$, and  the series defining~$ \cI(t,g) $ is a solution of the mild form of the Hamilton-Jacobi equation on $[0,T_\alpha] \times \bB_\alpha$~:
\begin{equation}
\label{HJ-eq}
\displaystyle
\begin{cases}
\cI(t,g) &= 
\displaystyle{ \cI (0, g) +
\frac12 \int_0^t d \tau
  \int  {\d \cI \over \d g( \tau )}  ( \tau ,g) (z_1)  {\d  \cI  \over \d g(\tau) } (\tau,g) (z_2)\Big(e ^{\Delta g(\tau)} - 1 \Big)   d\mu (z_1, z_2, \omega) } \, ,\\
 \cI (0, g) &   = \displaystyle {  \int dz \,  f^0(z) (e^{g(0,z)}-1) }\, ,
 \end{cases}
\end{equation}
where we used the notation \eqref{eq: measure mu}-\eqref{Delta-def}
$$
d\mu(z_1, z_2, \omega):=  
\delta_{x_1 - x_2}  ((v_1-v_2)\cdot \omega)_+ d\omega dv_1 dv_2 dx_1 \, ,
$$
and 
$$ \Delta g(z_1, z_2, \omega) := g(z'_1) + g(z'_2) - g(z_1) - g(z_2)\,.$$
\end{Thm}
We will see in Chapter \ref{LDP-chap} that this Hamilton-Jacobi  equation provides a complete characterization of~$\cI$ which will be crucial to identify the large deviation functional by means of Legendre transform.

\begin{proof}
At time 0, the exponential moment \eqref{eq:Itglimserie} reduces to the exponential moment of independent particles thus only the term $K =1$ remains
\begin{equation} 
%\label{eq:Itglimserie}
\cI (0,g) = - 1+ \int d z e^{g(0,z)}   f^0(z)  =  \int dz \,  f^0(z) (e^{g(0,z)} - 1) \, .
\end{equation}

To recover the mild form of the Hamilton-Jacobi equation \eqref{HJ-eq}, 
we are going to reparametrize each term of the series \eqref{eq:Itglimserie}
of $\cI (t,g)$ by singling out the last clustering collision.
Given a tree $T$ in $\cT_K^\pm$ with~$K \geq 2$, let $\tau_e := \tau_e^{\rm clust} \in [0,t]$ be the last clustering time which occurs at the edge $e$ and is associated with the scattering vector $\omega_e := \omega^{\rm clust}_{e }$ and the sign $s_e := s^{\rm {clust}}_{e  } \in \{-1,1\}$.
By removing the edge $e$, the tree $T$ is split into two trees $T_1 \in \cT_{K_1}^\pm$
and $T_2 \in \cT_{K_2}^\pm$ with sizes $K_1 + K_2 = K$ and clustering times belonging to~$[0,\tau_e]$. These trees generate two pseudo-trajectories $\Psi_{K_1,0}, \Psi_{K_2,0}$ on $[0,\tau]$ which are then constrained to cluster at time $\tau_e$. 
The whole pseudo-trajectory $\Psi_{K,0}$ on $[0,t]$ (generated by~$T$) is then recovered  
by merging the pseudo-trajectories $\Psi_{K_1,0}, \Psi_{K_2,0}$ at time $\tau_e$ and extending them on~$[0,t]$ with a scattering, or not, according to the sign $s_e$.
This procedure is abbreviated by 
\begin{equation}
\label{eq: merging pseudo-travectories}
\Psi_{K,0} = \Psi_{K_1,0} \wedge \Psi_{K_2,0}\, .
\end{equation}
This leads to 
\begin{align} 
& \sum_{  T \in \cT_K^\pm} \int d\mu_{{\rm sing},   T } (\Psi_{K,0}) (e^{g(t) - \int_0^t D_sg(s)ds})^{ \otimes  K}  (\Psi_{K,0})  f^{0 \otimes K} (\Psi_{K,0}^0) 
\nonumber \\
& \quad = \frac{1}{2}
\sum_{K_1, K_2 \atop K_1 + K_2  = K} \frac{K!}{K_1 ! \, K_2 !}
\sum_{ T_1 \in \cT_{K_1}^\pm \atop T_2 \in \cT_{K_2}^\pm} 
\int_0^t d \tau_e \int d\mu_{{\rm sing},   T_1 }^{[0, \tau_e]} (\Psi_{K_1,0})
 d\mu_{{\rm sing},   T_2 }^{[0, \tau_e]} (\Psi_{K_2,0})    f^{0 \otimes K_1} (\Psi_{K_1,0}^0) f^{0 \otimes K_2} (\Psi_{K_2,0}^0)
\nonumber  \\
& \qquad \qquad  \qquad \times 
\sum_{ i \in T_1 \atop j \in T_2} 
\sum_{s_e = \pm 1} \int d \omega_e \; s_e \;
\delta_{x_i (\tau_e) - x_j (\tau_e) }  ((v_i (\tau_e^-) -v_j (\tau_e^-))\cdot \omega_e)_+ \;
 \ (e^{g(t) - \int_0^t D_sg(s)ds})^{ \otimes  K},
\label{eq: K1, K2}
\end{align}
where the edge $e = (i,j)$.
By construction the parameters associated with the pseudo-trajectories~$\Psi_{K_1,0}$ and~$\Psi_{K_2,0}$ are independent and the corresponding measures on $[0, \tau_e]$ factorize.
We used the notation~$\mu_{{\rm sing},   T_1 }^{[0, \tau_e]}$ to stress the fact that the clustering times of the measure are restricted to $[0, \tau_e]$.
The last line of the identity \eqref{eq: K1, K2} encodes the clustering constraint at $\tau_e$.

To recover the factorization of the Hamilton-Jacobi equation \eqref{HJ-eq}, we first note that all the particles evolve in straight line in $[\tau_e, t]$, so that for any $k \leq K$
\begin{equation*}
g \big (t,z_k(t)\big)  -  \int_0^t D_s  g \big ( s ,z_k (s) \big) d s 
= g \big ( \tau_e ,z_k( \tau_e^+ )\big)  -  \int_0^{\tau_e} D_s  g \big ( s ,z_k (s) \big) d s .
\end{equation*}
If $s_e =1$, a scattering occurs between the particles $(i,j)$ forming the edge $e$ so that their velocities jump at time $\tau_e$; if~$s_e=-1$ on the other hand,   the trajectories are unchanged. With the notation~\eqref{Delta-def}, the discontinuity at the collision can thus be rewritten as
\begin{align}
(e^{g(t) - \int_0^t D_s g(s)ds})^{ \otimes  K}
& = (e^{g(\tau_e) - \int_0^{\tau_e} D_s g(s)ds})^{ \otimes  K_1} \; (e^{g(\tau_e) - \int_0^{\tau_e} D_s g(s)ds})^{ \otimes  K_2} \nonumber
\\
& \qquad \times \left( 1 + 1_{s_e = 1} \Big[ 
\exp \Big(  \Delta g(\tau_e) \; ( z_i ( \tau_e^- ), z_j ( \tau_e^- ), \omega_e )  \Big) - 1 \Big]
 \right) \, .
 \label{eq: factorisation bis}
\end{align}
It follows that  except for the interaction at time $\tau_e$ between particles $i,j$, the test functions factorize. We can rewrite \eqref{eq: K1, K2} as 
\begin{align} 
& \sum_{  T \in \cT_K^\pm} \int d\mu_{{\rm sing},   T } (\Psi_{K,0}) (e^{g(t) - \int_0^t D_sg(s)ds})^{ \otimes  K}  (\Psi_{K,0})  f^{0 \otimes K} (\Psi_{K,0}^0) 
\nonumber \\
& \quad = \frac{K! }{2}
\sum_{K_1 = 0}^K  \sum_{ T_1 \in \cT_{K_1}^\pm \atop T_2 \in \cT_{K_2}^\pm} 
 \sum_{ i \in T_1 \atop j \in T_2}  \int_0^t d \tau_e  \prod_{\ell = 1,2} \;  
\left[
 \frac{1}{K_\ell ! } \int d\mu_{{\rm sing},   T_\ell }^{[0, \tau_e]} (\Psi_{K_\ell,0})
 f^{0 \otimes K_\ell} (\Psi_{K_\ell,0}^0)   (e^{g(\tau_e) - \int_0^{\tau_e} D_s g(s)ds})^{ \otimes  K_\ell} \right]
\nonumber  \\
& \qquad \quad  \qquad \times 
 \int d \omega_e \; 
\delta_{x_i (\tau_e) - x_j (\tau_e) }  ((v_i (\tau_e^-) -v_j (\tau_e^-))\cdot \omega_e)_+ \;
 \ \Big[ 
\exp \Big(  \Delta g(\tau_e) \; ( z_i ( \tau_e^- ), z_j ( \tau_e^- ), \omega_e )  \Big) - 1 \Big],
\label{eq: K1, K2 bis}
\end{align}
where only the contribution $s_e =1$ remains. Indeed the constant $1$ in the last line of \eqref{eq: factorisation bis} cancels out after summing over $s_e = \pm 1$.

Summing \eqref{eq: K1, K2 bis} over all $K \geq 1$ in order to rebuild $\cI(t,g)$, the product of the functional derivatives~$\displaystyle {\d \cI (\tau_e,g) \over \d g(\tau_e)}$ defined in \eqref{J-derivative 2} can be identified
$$
\cI(t,g) = \cI (0, g) +
\frac12 \int_0^t d \tau_e
  \int  {\d \cI \over \d g( \tau_e )}  ( \tau_e ,g) (z_1)  {\d  \cI  \over \d g(\tau_e) } (\tau_e,g) (z_2)\Big(e ^{\Delta g(\tau_e)} - 1 \Big)   d\mu (z_1, z_2, \omega_e).
$$
Theorem \ref{HJ-prop} is proved.
\end{proof}

\section{The Boltzmann equation for the limit first cumulant}
\label{sec: Boltzmann equation}

The Hamilton-Jacobi equation~\eqref{HJ-eq} encodes all the limiting correlations of the microscopic dynamics. 
As a first consequence, we are going to recover the convergence of the density to the solution of the Boltzmann equation already stated in Theorem  \ref{thm: Lanford}. 
 
%The following notation for the duality on~$\bbD$   will be useful :
%\begin{equation}
%\label{eq: dualite en z}
%\la \gp, \psi \ra := \int_\bbD dz \;  \gp(z) \; \psi (z)\, .
%\end{equation}
Let us denote t the backward transport operator  by $S_t \phi(x,v):=\phi(x-tv,v)$, for any test function $\phi$.

\begin{Prop} 
\label{prop: Boltzmann mild}
In the Boltzmann-Grad limit, 
the rescaled one-particle density  converges  in the time interval $[0,T_0]$ 
 in the sense of measures 
\begin{equation}
\label{eq: limit weak density}
\lim_{\mu_\eps \to \infty} F^\eps_1(t ) = f_1(t )
  = \frac{\partial \cI(t,0)}{\partial g(t)}\, \cdotp\end{equation}
 The limit $f_1$ is a mild solution of the  Boltzmann equation in a weak form
\begin{equation}
\label{eq: mild Boltzmann}
\begin{aligned}
&\int_\bbD f_1(t,z)  \psi (z) \, dz
= \int_\bbD S_t f^0 (z)   \psi   (z) \, dz\\
&\qquad+ \int_0^tds \int  S_{t-s}\big( f_1(s,z'_1)   f_1(s,z'_2)-  f_1(s,z_1)   f_1(s,z_2)\big)      
\psi (z_1) \,   d\mu (z_1, z_2, \omega)\,  ,
\end{aligned}
\end{equation}
for any continuous  bounded test function $\psi$.
%This equation has a unique solution on $[0,T_0]$ which coincides with the one of \eqref{boltzmann-eq} so that  $f_1 = f$.}
\end{Prop}

\begin{proof}
We will consider only functional derivatives of $\cI$ at $g=0$, thus  $\alpha$ can be chosen arbitrarily small so that all the equations obtained from Theorem \ref{HJ-prop} are valid up to the time   $T_0 = T_{\alpha \big | \alpha =0}$. 

By definition~(\ref{I-def dynamics J}) of~$\cI$ 
\begin{equation}
\label{eq: defIChap5}
\cI(t,g) = \sum_{n = 1}^\infty \frac1{n!} f_{n,[0,t]} \Big(\big(e^h - 1 \big)^{\otimes n} \Big)
\quad \text{with} \quad h\big(z([0,t])\big)= g(t,z(t)) -\int_0^t D_s g(s,z(s)) \, ds\, ,
\end{equation}
is a uniformly convergent series for $t \leq T_\alpha$ and in particular it is analytic with respect to  $g(t)$ for $g$ in~$\bB_\alpha$.
Given   a test function~$\psi$ defined on~$\mathbb D$ (and acting at time~$t$), 
 the derivative at $g =0$ is given by 
\begin{equation}
\label{first derivative f1}
 \Big \langle \frac{\partial \cI (t,0)}{\partial g(t)}  , \psi \Big \rangle 
=f_{1,[0,t]} (\psi)  = \int_\bbD  f_1 (t) \psi (z) \, dz
 \, ,
 \end{equation}
where~$\la \, \cdot \,  , \, \cdot \, \ra$ denotes the duality bracket. Theorem \ref{cumulant-thm2} implies that  $f^\eps_{1, [0,t]}$ converges to $f_{1,[0,t]}$.
As  $F^\eps_1(t)  = f^\eps_{1}(t)$, this  leads to \eqref{eq: limit weak density}.

The Hamilton-Jacobi equation~(\ref{HJ-eq}) will  enable  us to obtain rather easily   that the equation satisfied by~$f_1$ is the Boltzmann equation. 
Let us start by 
computing the derivative  with respect to~$g(t)$ of~$  \cI(0,g)$.  
First, we note that for all~$s \in [0,t]$, then $g(s)$ is a function of~$g(t)$ and  $Dg$ through the Duhamel formula
$$
g(t,x+tv,v) = g(s,x+sv,v) + \int_s^t D_\sigma g(\sigma, x+\sigma v, v) \, d\sigma \, , 
$$
which may be recast as follows:
\begin{equation}
\label{g(s)g(t)Dg}
\forall s \in [0,t] \, , \quad 
g(s) = S_{s-t}g (t) -  \int_s^t S_{s-\sigma} D_\sigma g(\sigma) \, d\sigma \, .
\end{equation}
This formula will be key to track the impact of the variations of $g(s)$ in the functional derivatives under a perturbation of  $g$ at time $t$ (or of $Dg$ later on).
Recalling that
$$ 
\cI (0, g) =   \int f^0(z) \big( e^{g(0,z)}-1\big)  \, dz
\, , 
$$
we therefore find that taking the derivative with respect to  $g(t)$ in the direction $\psi$ is given by 
\begin{equation}
\label{derivative at time 0}
\langle \frac{\partial \cI(0,g)}{\partial g(t)}, \psi \rangle   =     \int f^0(z) \big(S_{-t} \psi \big)(z) e^{g(0,z)} \, dz \, ,
\end{equation}
hence in particular at $g=0$
\begin{equation}
\label{derivative at 0}
\langle \frac{\partial \cI(0,0)}{\partial g(t)} , \psi \rangle  =   \int  \big(S_tf^0 \big)(z) \psi (z) \, dz \, .
\end{equation}
Next differentiating~(\ref{HJ-eq})  with respect to~$g(t)$ in the direction~$\psi$, we find
\begin{equation}
\label{eqHJ'}
\begin{aligned}
\langle \frac{\partial \cI(t,g)}{\partial g(t)}  , \psi \rangle
& =  \int  \big(S_tf^0 \big)(z) \psi (z) \, dz \\
&   + \int_0^t ds\int   {\d \cI ( s ,g)\over \d g( s )}   (z_1)\Big \langle  {\d^2  \cI (s,g) \over \d g(s) \d g(t) } , \psi \Big\rangle  (z_2)\Big(e ^{\Delta g(s)} - 1 \Big)   d\mu (z_1, z_2, \omega)  \\
&   +\frac12 \int_0^tds \int   {\d \cI  ( s ,g) \over \d g( s )}  (z_1)   
{\d   \cI  (s,g) \over \d g(s)}   (z_2)
\;  \Delta S_{s-t}\psi \, e^{\Delta g(s)}    d\mu (z_1, z_2, \omega) .
\end{aligned}
\end{equation}
Note that Proposition \ref{prop: analyticity J part 1}  allows to handle  
the singularity of the measure $d\mu$.

Evaluating the result at~$g=0$ produces, thanks to~(\ref{first derivative f1}), (\ref{g(s)g(t)Dg}) and~(\ref{derivative at 0}),
\begin{align*}
\langle \frac{\partial \cI(t,0)}{\partial g(t)}  , \psi \rangle
 =  \int  \big(S_tf^0 \big)(z) \psi (z) \, dz   +\frac12 \int_0^t ds\int  f_1(s,z_1)   f_1(s,z_2) \Delta S_{s-t}\psi \,    d\mu (z_1, z_2, \omega) \, .
\end{align*}
Finally thanks to~(\ref{first derivative f1})  again, we recover that for any smooth function $\psi$
$$
\begin{aligned}
\la f_1(t)  ,  \psi    \ra  &= \int  \big(S_tf^0 \big)(z) \psi (z) \, dz +\int_0^t ds\int \big( f_1(s,z'_1)   f_1(s,z'_2)-  f_1(s,z_1)   f_1(s,z_2)\big)   S_{s-t}\psi     \,    d\mu (z_1, z_2, \omega) \\
&= \int  \big(S_tf^0 \big)(z) \psi (z) \, dz +\int_0^tds \int  S_{t-s}\big( f_1(s,z'_1)   f_1(s,z'_2)-  f_1(s,z_1)   f_1(s,z_2)\big)    \psi  (z_1)    \,   d\mu (z_1, z_2, \omega) .
\end{aligned}
$$
The  proposition is proved. \end{proof}

\section{Equation for the limit covariance}
\label{sec: Equation for the limit covariance}

The fluctuation field covariance is defined for any test functions  
$\psi, \gp$ on~$\mathbb D$ by 
\begin{equation}
\label{eq: definition fluctuation field cov}
\forall s \leq t \, , \qquad  \cC_\eps (t,s, \psi, \gp  ) 
:=  \bbE_\eps \left( \zeta^\eps_t (\psi) \zeta^\eps_s(\gp) \right).
\end{equation}
The Hamilton-Jacobi equation~(\ref{HJ-eq}) enables us   to deduce 
dynamical equations characterizing the limit covariance.
For this, we shall need the following notations :
\begin{Def}
\label{Def: linearized et covariance}
The (adjoint) linearized operator is defined as
 \begin{equation}
\label{eq: adjoint Ls}
\begin{aligned}
  \cL_t^*  \gp(z)&:=v\cdot \nabla_x\gp(z) + {\mathbf L}_t^*\gp(z) \, ,\quad \mbox{with} \\
 {\mathbf L}_t^*\gp(z)&:=\int d\mu_{z}  (  z_2,  \omega)  f (t,z_2) \,\Delta\gp (z,z_2,\omega)  \, ,
  \end{aligned}
  \end{equation}
   with notation~{\rm(\ref{defdmuzi})} for the measure~$d\mu_{z}  (  z_2,  \omega) $. We also set  \begin{equation}
\label{eq: noise covariance}
{\bf  Cov}_t ( \gp, \psi) 
:= \frac{1}{2}  \int d\mu (z_1, z_2, \omega) \,
 f (t, z_1)\, f (t, z_2) \; \Delta \psi \Delta \gp \, . 
 \end{equation}
\end{Def}

\begin{Prop} 
\label{prop: covariance mild}
The covariance of the particle system converges to a quadratic form $\cC$  in the time interval $[0,T_0]$ 
in a weak sense, i.e. for any bounded continuous functions $\gp, \psi$ 
\begin{equation}
\label{eq: limit weak covariance}
\forall s\leq t\leq T_0, \qquad 
\lim_{\mu_\eps \to \infty} \cC_\eps (t, s, \psi, \gp)  = \cC (t,s, \psi, \gp)   \, .
\end{equation}
The limit $\cC$  is a solution of the system of equations for $t\leq T_0$
\begin{equation} \label{eq: systeme covariance}
\left\{ \begin{aligned} 
 \mathcal C(t,t,\psi, \gp) & = \mathcal C(0,0,S_{-t}\psi,S_{-t}\gp ) + \int_0^t ds\, {\bf  Cov}_s( S_{s-t}\psi, S_{s-t}\gp)   \\
    & + \int_0^t ds\, \mathcal C(s,s,S_{s-t}\psi, {\mathbf L}_s^*S_{s-t}\gp)  
  + \int_0^t ds\, \mathcal C(s,s, {\mathbf L}_s^*S_{s-t}\psi, S_{s-t}\gp)   \, ,  
  \\
  \int_0^t   \mathcal C(t,\sigma ,\psi,\phi_\sigma) \,d \sigma
& =  \int_0^t d \sigma   \, 
\left( 
\mathcal C(\sigma ,\sigma , S_{\sigma-t}\psi, \phi_\sigma) 
+  \int_ \sigma^t ds   \;   
\mathcal C \Big( s, \sigma,   {\mathbf L}_s^* S_{s-t}\psi ,  \phi_\sigma \Big)  \right) \, , 
\end{aligned}
\right.
\end{equation}
where $\psi$, $\gp$ and $(\phi_\sigma)_{\sigma \leq T_0}$ are test functions on~$\mathbb D$. 
 \end{Prop}
It is shown in the appendix that~(\ref{eq: systeme covariance}) provides a complete characterization of~$\mathcal C(t,s,\psi, \gp)$, at least for a short time: see Proposition~\ref{prop: well posedness covariance mild}.
 
\begin{proof}
The proof of the proposition is split into 2 steps.

\medskip

\noindent
{\bf Step 1.} Convergence  of  the covariance \eqref{eq: limit weak covariance}.

Recall first that the covariance, for fixed $\eps$,  is determined by the first two cumulants
$$
\begin{aligned}
\forall s \leq t \, , \quad  \cC_\eps (t,s, \psi, \gp)  & = \bbE_\eps \left(\frac1{\mu_\eps}  \sum_{i} \gp ( {\bf z}^\eps_{i} (s))  \psi ( {\bf z}^\eps_{i} (t))\right)+ \bbE_\eps \left(\frac1{\mu_\eps}  \sum_{(i_1,i_2) } \gp ( {\bf z}^\eps_{i_1} (s))  \psi ( {\bf z}^\eps_{i_2} (t))\right) \\
 & \quad  -\mu_\eps\bbE_\eps \left(\frac1{\mu_\eps}  \sum_{i } \gp ( {\bf z}^\eps_{i} (s))\right) \times \bbE_\eps\left(\frac1{\mu_\eps}  \sum_{i }  \psi ( {\bf z}^\eps_{i} (t))\right)
 \\
 & =  f^\eps_{1, [0,t]} ( \varphi(s)\psi(t) ) + f^\eps_{2, [0,t]} ( \varphi(s) \otimes \psi(t) ) 
 \end{aligned}$$
where with slight abuse, we   denote   by $f^\eps_{2,[0,t]}=f^\eps_{2, [0,t]}   \left( \psi  \otimes \gp \right)$ the bilinear form obtained by polarization
$$
f^\eps_{2, [0,t]}   \left( \psi  \otimes \gp \right) 
: = 
\frac12 \Big( f^\eps_{2, [0,t]}   \left( (\psi  + \gp)^{\otimes 2} \right) 
- f^\eps_{2, [0,t]}   \left( \psi^{\otimes 2}  \right) - f^\eps_{2, [0,t]}   \left(  \gp^{\otimes 2} \right) \Big) \;.
$$
By the convergence of the cumulants proved in Theorem \ref{cumulant-thm2}, the limit
covariance is 
\begin{equation}
\label{eq: limiting covariance}
\forall s \leq t \, , \quad \cC(t,s, \psi, \gp)
 :=   f_{1,[0,t]} \big(  \psi(t) \gp \big( s) \big)    +  f_{2,[0,t]} \big(   \psi \big(t) \otimes \gp(s)  \big) \, .
 \end{equation}

\bigskip

\noindent
{\bf Step 2.} Derivation of the system of equations \eqref{eq: systeme covariance}.
 
 We start by establishing the equation for the covariance at a single time $t$.
As in \eqref{first derivative f1}, differentiating twice~$\cI$ with respect to~$g(t)$ in the direction~$\psi$ provides
$$
\Big \langle \frac{\partial^2 \cI}{\partial ^2g(t)}, \psi \otimes \psi \Big \rangle _{\big|g=0}
=f_{1,t} (\psi^2) + f_{2,t} (\psi\otimes\psi) =\mathcal C(t,t,\psi,\psi) \, .
$$  
The corresponding formula for~$\mathcal C(t,t,\varphi,\psi)$ follows by polarization.
Thanks to~(\ref{derivative at time 0}), there holds
$$
\Big \langle 
 \frac{\partial^2 \cI(0,g)}{\partial^2 g(t) },  \psi \otimes  \psi   \Big \rangle _{\big|g=0}
= \int f^0(z)\big(S_{-t}\psi\big)^2(z)\, dz = \mathcal C(0,0,S_{-t}\psi,S_{-t}\psi) \,  .
$$
By using  the identity  \eqref{g(s)g(t)Dg}, the functional can be also differentiated at different times 
\begin{equation}
\label{eq: changement de temps s t}
\Big \langle  {\d^2  \cI  (s,g) \over \d g(s) \d g(t) }, \psi \Big\rangle  (z_1 )
= \Big \langle  {\d^2  \cI  (s,g) \over \d g(s) \d g(s) }, S_{s-t} \psi \Big\rangle  (z_1 ).
\end{equation}
Thus   differentiating~(\ref{eqHJ'}) one more time and computing the result at~$g=0$ provides 
\begin{equation}
\label{Cttpsipsi}
\begin{aligned}
\mathcal C(t,t,\psi,\psi) & = \mathcal C(0,0,S_{-t}\psi,S_{-t}\psi) \\
&\quad 
+ 2 \int_0^t ds \int   
\Big \langle  {\d^2  \cI  (s,0) \over \d g(s) \d g(s) }   , S_{s-t} \psi  \Big \rangle  (z_1 )
\; {\d \cI  ( s ,0)\over \d g( s )}  (z_2)
\; \Delta S_{s-t}\psi  \,   d\mu (z_1, z_2, \omega)  \\
&  \quad 
+\frac12 \int_0^tds \int   {\d \cI ( s ,0) \over \d g( s )}  (z_1)   
{\d   \cI   (s,0)  \over \d g(s)  }  (z_2)
\big( \Delta S_{s-t}\psi \big)^2 d\mu (z_1, z_2, \omega)\\
& = \mathcal C(0,0,S_{-t}\psi,S_{-t}\psi) \\
&\quad 
+ 2 \int_0^t ds \int   
\Big \langle  {\d^2  \cI  (s,0) \over \d g(s) \d g(s) }  (z_1 ) , S_{s-t} \psi  \Big\rangle
\; f  (s, z_2)
\; \Delta S_{s-t}\psi  \,   d\mu (z_1, z_2, \omega)  \\
&  \quad 
+
\frac12 \int_0^tds \int f  (s, z_1)   f (s, z_2)\big( \Delta S_{s-t}\psi \big)^2 d\mu (z_1, z_2, \omega) \, ,
\end{aligned}
\end{equation}
where    $\displaystyle {\d \cI ( s ,0) \over \d g( s )}$ has been replaced by $f(s)$ thanks to Proposition \ref{prop: Boltzmann mild}.

With these notations, \eqref{Cttpsipsi} can be rewritten as 
\begin{equation}
\label{mild form covariance}
\mathcal C(t,t,\psi,\psi) = \mathcal C(0,0,S_{-t}\psi,S_{-t}\psi) + 2\int_0^t ds\, \mathcal C(s,s,S_{s-t}\psi, {\mathbf L}_s^*S_{s-t}\psi) + \int_0^t ds\, {\bf  Cov}_s( S_{s-t}\psi, S_{s-t}\psi) \, .
\end{equation}
Thus the first equation of the system \eqref{eq: systeme covariance} is recovered 
by polarisation.

Now let us turn to the equation on the covariance at two different times.
Given~$\phi$ a test function defined on~$[0,t]\times \mathbb D$, the integrated covariance can be recovered by differentiating with respect to  $Dg$  in the direction~ $\phi_\sigma = \phi (\sigma)$, a given smooth function. Setting  
$$
\Phi(t,z([0,t])) :=\int_0^t \phi(\sigma ,z( \sigma )) \, d \sigma
= \int_0^t \phi_\sigma  \, d \sigma \, ,
$$ 
one has$$
\la \frac{\partial^2 \cI(t,0)}{\partial g(t)\partial Dg}  , \psi \otimes \Phi   \ra
= -f_{1,[0,t]} (\psi \, \Phi) - f_{2,[0,t]} (\psi\otimes\Phi) 
= -\int_0^t\mathcal C(t,s,\psi,\phi_s) \,ds,
$$
where the minus sign comes from the fact that the test function is $g(t) -\displaystyle \int_0^t ds D_s g$.

We are now going to 
 derive the second equation on the covariance at different times, differentiating~(\ref{eqHJ'})
 again. We
recall from \eqref{g(s)g(t)Dg} that the variations of $g(s)$ in the directions $\psi$ and $\phi$ are given by 
\begin{equation}
\label{g(s)g(t)Dg bis}
\forall s \in [0,t] \, , \quad 
\delta g(s) = S_{s-t} \psi  -  \int_s^t S_{s-\sigma} \phi_\sigma \, d\sigma .
%\quad \mbox{with} \quad S_t \phi(x,v):=\phi(x-tv,v)\, .
\end{equation}

We start by observing that taking a second derivative in~\eqref{derivative at time 0}  
leads to
$$
 \Big \la \frac{\partial^2 \cI(0,0)}{\partial g(t)\partial Dg}  ,  \psi \otimes  \Phi  \Big \ra 
 = - \int dz f^0(z) S_{-t} \psi(z) \int_0^t  S_{-\sigma} \phi (\sigma,z) d \sigma
 = -\int_0^t\mathcal C \big( 0,0,  S_{-t}  \psi, S_{-\sigma} \phi_\sigma \big) \,d \sigma .
$$
Taking the derivative at intermediate times $s \in [0,t]$ on $\cI(s,g)$ with respect to  $Dg$ is more delicate as there is a contribution of the variation of $\delta g(s)$ by \eqref{g(s)g(t)Dg bis} and another contribution accounting for the variations on $[0,s]$: recalling~(\ref{eq: defIChap5}),
\begin{equation}
\label{eq: Phi derivee}
\la \frac{\partial \cI(s,0)}{\partial Dg},   \Phi   \ra
= - \int_s^t  \Big \la \frac{\partial \cI(s,0)}{\partial g(s)}  ,    S_{s-\sigma} \phi_\sigma  \Big \ra \, d\sigma  
-   \int_0^s \Big \la \frac{\partial \cI(s,0)}{\partial D_\sigma g}  ,   \phi_\sigma     \Big \ra \, d\sigma  .
\end{equation}

Differentiating~(\ref{eqHJ'}) one more time and using \eqref{eq: changement de temps s t}, there holds
$$
\begin{aligned}
 \int_0^t & \mathcal C(t,\sigma ,\psi,\phi_\sigma) \,d \sigma
=  \int_0^t\mathcal C(0,0,S_{-t}\psi,S_{-\sigma} \phi_\sigma ) \,d \sigma\\
&\quad   
+ \int_0^t ds \int   
 \Big \langle  {\d^2  \cI  (s,0) \over \d g(s) \d g(s) }, S_{s-t} \psi \Big\rangle  (z_1 )
\,  {\d \cI  ( s ,0 ) \over \d g( s )}  (z_2 )
\Big(\Delta \int_s^t  S_{s-\sigma} \phi_\sigma \, d\sigma \Big)  
\,   d\mu (z_1, z_2, \omega)  \\
&\quad 
-  \int_0^t ds  \int  \la
  {\d^2  \cI  (s,0)\over \d g(s) \d Dg}  , \Phi  \ra (z_1)  \,
 {\d \cI  ( s ,0) \over \d g( s )}   (z_2 ) \Delta  S_{s-t}\psi\   \,   d\mu (z_1, z_2, \omega)  \\
&  \quad 
+\frac12 \int_0^t ds  \int   {\d \cI  ( s ,0)\over \d g( s )}  (z_1)   
{\d   \cI (s,0)  \over \d g(s)}    (z_2)
\big( \Delta S_{s-t}\psi \big) \Big(\Delta \int_s^t  S_{s-\sigma} \phi_\sigma \, d\sigma \Big)  d\mu (z_1, z_2, \omega) \, .
\end{aligned}
$$
Using that ${\d   \cI (s,0)  \over \d g(s)} = f(s)$ by Proposition \ref{prop: Boltzmann mild}, the adjoint linearized operator \eqref{eq: adjoint Ls} and the 
covariance~\eqref{eq: noise covariance}, we get 
%  \cL_t^*  \gp(z)&:=v\cdot \nabla_x\gp(z) + {\mathbf L}_t^*\gp(z) \, ,\quad \mbox{with} \\
% {\mathbf L}_t^*\gp(z)
$$
\begin{aligned}
\int_0^t  \mathcal C(t,\sigma ,\psi,\phi_\sigma) \,d \sigma
&
=   \int_0^t\mathcal C(0,0,S_{-t}\psi,S_{-\sigma} \phi_\sigma ) \,d \sigma    
+ \int_0^t ds  \int_s^t d\sigma      \;  
 \Big \langle  {\d^2  \cI  (s,0) \over \d g(s) \d g(s) }, S_{s-t} \psi  \otimes  {\mathbf L}_s^* S_{s-\sigma} \phi_\sigma 
\Big\rangle  \\
& 
-  \int_0^t ds \;  \la  {\d^2  \cI  (s,0) \over \d g(s) \d Dg}  , 
    {\mathbf L}_s^* S_{s-t}\psi \otimes \Phi \ra 
+  \int_0^t ds \int_s^t d\sigma \; 
{\bf  Cov}_s \Big(  S_{s-t}\psi  ,   S_{s-\sigma} \phi_\sigma  \Big)  .
\end{aligned}
$$
 From identity \eqref{eq: Phi derivee}, we finally obtain 
$$
\begin{aligned}
\int_0^t  \mathcal C(t,\sigma ,\psi,\phi_\sigma)\,d \sigma
&=   \int_0^t\mathcal C(0,0,S_{-t}\psi,S_{-\sigma} \phi_\sigma ) \,d \sigma    
+ \int_0^t ds  \int_s^t d\sigma      \;   
\Big \langle  {\d^2  \cI  (s,0) \over \d g(s) \d g(s) }, S_{s-t} \psi  \otimes  {\mathbf L}_s^* S_{s-\sigma} \phi_\sigma \Big\rangle  \\
& 
+  \int_0^t ds  \int_s^t  d \sigma   
\Big \la  {\d^2  \cI  (s,0)\over \d g(s) \d g(s) } , 
    {\mathbf L}_s^* S_{s-t}\psi \otimes S_{s-\sigma} \phi_\sigma \Big\ra \\
& +   \int_0^t ds  \int_0^s d \sigma   \Big \la
  {\d^2  \cI  (s,0)\over \d g(s) \d D g}  , 
    {\mathbf L}_s^* S_{s-t}\psi \otimes  \phi_\sigma \Big\ra \\
& +  \int_0^t ds \int_s^t d\sigma \; 
{\bf  Cov}_s \Big(  S_{s-t}\psi  ,   S_{s-\sigma} \phi_\sigma  \Big)  .
\end{aligned}
$$
Noticing that 
$$
\mathcal C ( s, \sigma, \psi, \phi )
= \Big \la  {\d^2  \cI  (s,0)\over \d g(s) \d D_\sigma g } , \psi \otimes  \phi \Big\ra , 
$$
this can be rewritten in terms on the covariance $\mathcal C$.
$$
\begin{aligned}
\int_0^t   \mathcal C(t,\sigma ,\psi,\phi_\sigma) \,d \sigma
&=   \int_0^t d \sigma  \, \mathcal C(0,0,S_{-t}\psi,S_{-\sigma} \phi_\sigma )    
+ \int_0^t ds  \int_s^t d\sigma \;  
\mathcal C \Big( s,s, S_{s-t} \psi    ,  {\mathbf L}_s^* S_{s-\sigma} \phi_\sigma 
\Big)  \\
& 
+  \int_0^t ds  \int_s^t  d \sigma    \; 
\mathcal C \Big( s,s, {\mathbf L}_s^* S_{s-t}\psi , S_{s-\sigma} \phi_\sigma \Big) 
 +   \int_0^t ds  \int_0^s d \sigma  \;   \mathcal C \Big( s, \sigma   , 
    {\mathbf L}_s^* S_{s-t}\psi ,  \phi_\sigma \Big) \\
& +  \int_0^t ds \int_s^t d\sigma \; 
{\bf  Cov}_s \Big(  S_{s-t}\psi  ,   S_{s-\sigma} \phi_\sigma  \Big)  .
\end{aligned}
$$
Finally swapping the integrals in $s,\sigma$ by Fubini's Theorem, we get    
$$
\begin{aligned}
\int_0^t   \mathcal C(t,\sigma ,\psi,\phi_\sigma) \,d \sigma
& =   \int_0^t d \sigma   \, \mathcal C(0,0,S_{-t}\psi,S_{-\sigma} \phi_\sigma )  
+ \int_0^t d \sigma \int_0^\sigma ds \;  
\mathcal C \Big( s,s, S_{s-t} \psi   ,  {\mathbf L}_s^* S_{s-\sigma} \phi_\sigma 
\Big)  \\
& 
+   \int_0^t d \sigma \int_0^\sigma ds    \; 
\mathcal C \Big( s,s, {\mathbf L}_s^* S_{s-t}\psi , S_{s-\sigma} \phi_\sigma \Big) 
+  \int_0^t d \sigma \int_\sigma^t ds   \;   \mathcal C \Big( s, \sigma   , 
    {\mathbf L}_s^* S_{s-t}\psi ,  \phi_\sigma \Big) \\
& + \int_0^t d \sigma \int_0^\sigma ds \; 
{\bf  Cov}_s \Big(  S_{s  - t}\psi  ,   S_{s - \sigma} \phi_\sigma  \Big)  .
\end{aligned}
$$
Noticing that  \eqref{mild form covariance}  implies 
\begin{align*}
\mathcal C(\sigma ,\sigma , S_{\sigma-t}\psi, \phi_\sigma ) 
&= \mathcal C(0,0,S_{- t }\psi,S_{- \sigma}\phi_\sigma) 
+  \int_0^\sigma ds\, \mathcal C(s,s,S_{s-t}\psi, {\mathbf L}_s^*S_{s-\sigma}\phi_\sigma) \\
& +  \int_0^\sigma ds\, \mathcal C(s,s, {\mathbf L}_s^* S_{s-t}\psi, S_{s-\sigma}\phi_\sigma) 
+ \int_0^\sigma ds\, {\bf  Cov}_s( S_{s-t}\psi, S_{s-\tau}\phi_\sigma) \, ,
\end{align*}
the formula for the covariance simplifies
\begin{align*}
\int_0^t   \mathcal C(t,\sigma ,\psi,\phi_\sigma)\,d \sigma
 =  \int_0^t d \sigma   \, 
\left( 
\mathcal C(\sigma ,\sigma , S_{\sigma-t}\psi, \phi_\sigma) 
+  \int_\sigma^t ds   \;   
\mathcal C \Big( s, \sigma,   {\mathbf L}_s^* S_{s-t}\psi ,  \phi_\sigma \Big)  \right)  .
\end{align*}
This completes the derivation of the system of equations \eqref{eq: systeme covariance}.
\end{proof}

\part{Fluctuations around the Boltzmann dynamics }

\chapter{Fluctuating Boltzmann equation}
\label{TCL-chap} 
\setcounter{equation}{0}

The goal of this chapter is to prove Theorem  \ref{thmTCL}, describing the  limit of the fluctuation field $( \gz^\gep_t )_{t }$, %introduced in \eqref{eq: fluctuation field}, i.e. defined for any smooth test function $\gp$  as
 of which we recall the definition:
\begin{equation*}
\gz^\gep_t \big(  \gp  \big) := \frac1{ \sqrt{\mu_\eps }} 
\Big( \sum_{i = 1}^\cN  \gp \big(  \mathbf z^\eps_i(t)  \big) -  \mu_\eps \bbE_\eps\big(\pi^\eps_t(\gp)\big)  \Big)
\end{equation*}
on test functions $\gp$.
%By definition, for any fixed $\eps$ and any $h \in C^0(\bbD)$,  $t \mapsto\gz^\gep_t \big(  h  \big)$ belongs to the Skorohod space $D( [0,T^\star])$ 
%of functions of time which are right continuous  with left limits.
Namely we prove that, in the Boltzmann-Grad limit, $\zeta^\eps_t$ converges to a process~$\zeta_t$
which solves, in a weak sense clarified below (see Section  \ref{subsec: Functional setting}), the fluctuating Boltzmann equation 
\begin{equation}
\label{eq: OU bis}
d\hat \gz_t   = \cL_t \,\hat\gz_t\, dt + d\eta_t \,.
\end{equation}
We recall that  $f $ is the solution of the Boltzmann equation on $[0,T_0]$, that the linearized  Boltzmann operator is defined as 
$\cL_t := - v \cdot \nabla_x + {\bf L}_t$ with the collision part 
\begin{equation}
\label{defLtbold}
\begin{aligned}
{\bf L}_t \,\gp(z_1) := \int  d\mu_{z_1}  (  z_2,  \omega) 
 \Big( f (t,z_2') \gp(z_1') + f (t,z_1') \gp(z_2')  - f (t,z_2)  \gp(z_1) -  f (t,z_1)  \gp(z_2) \Big) \, ,  
\end{aligned}
\end{equation}
and that $d \eta_t(x,v)$  is a Gaussian noise  with zero mean and covariance given in~\eqref{eq: noise covariance}, which we recall
\begin{equation}
\label{defCovt}
{\bf  Cov}_t ( \psi,\gp) 
:= \frac{1}{2}  \int d\mu (z_1, z_2, \omega) \,  f (t, z_1)\, f (t, z_2) \; \Delta \psi \Delta \gp \, . 
\end{equation}
where the scattering measures are defined as in \eqref{eq: measure mu} and 
\eqref{defdmuzi}
$$\begin{aligned}
& d\mu_{z_1}  (  z_2,  \omega) = \delta_{ x_1-   x_2}  \big ( (    v_1 - v_2) \cdot \omega\big)_+ d\omega  
d   v_2,\\ &
d\mu (z_1, z_2, \omega) = \delta_{x_1-   x_2} \left( (v_1 - v_2) \cdot \omega \right)_+
d \omega \, d x_1 dv_1 d v_2  \, ,
\end{aligned}$$
and we recall the notation 
\begin{equation}
\label{eq: Delta psi}
\Delta \psi (z_1,z_2,\omega) = \psi ( z_1' ) + \psi   ( z_2' )  -  \psi  ( z_1) -  \psi ( z_2 )\, .
\end{equation}
 The limiting Gaussian process \eqref{eq: OU bis} will be characterized  by its covariance in Section  \ref{subsec: Functional setting}.

In order to obtain the convergence of the fluctuation field, 
we shall proceed in two steps, establishing first  the convergence of the characteristic function in Section \ref{characteristic-sec}, and then some tightness in Section~\ref{tightness-sec}.

\section{Weak solutions for the  limit process}
%\label{OU-sec}
\label{subsec: Functional setting}
\setcounter{equation}{0}
A solution~$\hat \gz_t $ to~(\ref{eq: OU bis}) is a Gaussian process: its law   is therefore  completely characterized by its covariance. 
  In this section we study the equation governing this covariance
\begin{equation}
\label{defhatCov}
 \hat \cC(t,s, \psi,\gp) :=\bbE \big(\hat  \gz_t (\psi) \hat \gz_s (\gp) \big) 
\end{equation}
 and prove that it is  precisely the   equation obtained Proposition~\ref{prop: covariance mild}, namely~(\ref{eq: systeme covariance}). Since there is a unique solution to~(\ref{eq: systeme covariance}) (see Proposition~\ref{prop: well posedness covariance mild}),   the limiting covariance~$  \cC(t,s, \psi,\gp) $ is equal to~$ \hat \cC(t,s, \psi,\gp)   $, at least for short times.

\subsection{Equation for the covariance}

Denote by   $\mathcal{U}(t,s)$\label{defmathcalU} the semigroup associated with $\cL_\tau$ between times~$s<t$, meaning that
$$ \d_t \mathcal{U}(t,s) \varphi -\cL_t \mathcal{U}(t,s) \gp = 0\,, \qquad \mathcal{U}(s,s) \gp = \gp\,,$$
and
$$ \d_s  \mathcal{U}(t,s) \gp +  \mathcal{U}(t,s) \cL_s \gp= 0\,, \qquad \mathcal{U}(t,t) \gp = \gp\,.$$
By definition, $  \mathcal{U}^* (t,s) \gp$ satisfies 
\begin{equation}
\label{eq: linearized backward}
\d_s   \mathcal{U}^* (t,s) \gp+  \cL_s^*  \mathcal{U}^* (t,s) \gp= 0\,, \qquad   \mathcal{U}^* (t,t) \gp  = \gp\,,
\end{equation}
and
$$
\d_t \mathcal{U}^* (t,s) \varphi -\mathcal{U}^* (t,s)\cL_t^*   \gp = 0\,, \qquad \mathcal{U}^* (s,s) \gp = \gp\,, $$
where we recall that
$ \cL_s^* = v\cdot \nabla_x +{\bf L}_s^*$ with
\begin{equation}
\label{eq: L*}
{\bf L}_s^*\,\psi (z_1) := \int d\mu_{z_1}  (  z_2,  \omega)  f (s,z_2) \,\Delta\psi (z_1,z_2,\omega)  \,  .
\end{equation}
Formally, a solution of the limit process~\eqref{eq: OU bis} satisfies
for any test function $\gp$
%\begin{equation*}
%\hat \gz_t (\gp)  = \gz_0 (\gp)  + \int_0^t ds  \,\hat  \gz_s (  \cL_s^* \gp)  + 
%\int_0^t   \, d\eta_s (  \mathcal{U}^* (t,s) \gp) \, ,
%\end{equation*}
%or alternatively 
\begin{equation*}
\hat \gz_t (\gp)  = \gz_0 (\mathcal{U}^*(t,0)  \gp)  + \int_0^t   \, d\eta_s ( \mathcal{U}^*(t,s)  \gp)  \,.
\end{equation*}
For any $t \geq s$ and test  functions $\gp,\psi$, the covariance is then given by 
$$\begin{aligned}
\bbE \big(\hat  \gz_t (\psi) \hat \gz_s (\gp) \big) 
& = \bbE \Big(  \gz_0  \big( \mathcal{U}^*(t,0) \psi  \big) \;  \gz_0  \big( \mathcal{U}^*(s,0) \gp  \big) \Big)
+ \bbE \left(  \int_0^t \int_0^s d  \eta_\sigma \, d  \eta_{\sigma'} 
\big( \mathcal{U}^*(t,\sigma) \psi \big) \big( \mathcal{U}^*(s,\sigma') \gp \big)   \right)  \\
& + \bbE \left(  \gz_0  \big( \mathcal{U}^*(t,0) \psi  \big) \;  
 \int_0^s  d  \eta_{\sigma'} 
\big( \mathcal{U}^*(s,\sigma') \gp \big) \right) +
\bbE \left(  \gz_0  \big( \mathcal{U}^*(s,0) \gp  \big) \; \int_0^t   d  \eta_\sigma  
\big( \mathcal{U}^*(t,\sigma) \psi \big) \right)  \\
\end{aligned}$$
so that according to~(\ref{defCovt}) and~(\ref{defhatCov})
\begin{equation}
\label{eq: covariance hors equilibre}
  \hat \cC(t,s, \psi,\gp) = \bbE \Big(  \gz_0  \big( \mathcal{U}^*(t,0) \psi  \big) \;  \gz_0  \big( \mathcal{U}^*(s,0) \gp  \big) \Big)
+   \int_0^s  d\sigma \;
{\bf  Cov}_\sigma \left( \mathcal{U}^*(t,\sigma) \psi , \mathcal{U}^*(s,\sigma) \gp  \right) \, .
 \end{equation}
%{\color{blue} Notice that~(\ref{eq: covariance hors equilibre}) can also be written
%\begin{equation}
%\label{eq: covariance hors equilibre bis}
% \hat \cC(s,t,\gp,\psi) = \bbE \Big(  \gz_0  \big( \mathcal{U}^*(t,0) \psi  \big) \;  \gz_0  \big( \mathcal{U}^*(s,0) \gp  \big) \Big)
%+   \int_0^s  d\sigma \;
%{\bf  Cov}_\sigma \left( \mathcal{U}^*(t,\sigma) \psi , \mathcal{U}^*(s,\sigma) \gp  \right) \, .
%\end{aligned}\end{equation}}
\begin{Def}\label{OU-def}
A weak solution to~{\rm(\ref{eq: OU bis})} is  a Gaussian process with covariance satisfying~{\rm(\ref{eq: covariance hors equilibre})}.
\end{Def}
 
Let us take formally the time derivative of \eqref{eq: covariance hors equilibre}  for~$t> s$. This  gives 
$$ \begin{aligned}
\partial_t \hat \cC(t,s, \psi,\gp)   & = \bbE \left(  \gz_0  \big( \mathcal{U}^*(t,0) \cL_t^* \psi  \big)  \gz_0  \big( \mathcal{U}^*(s,0) \gp  \big) \right)
+    \int_0^s    \!  \! d\sigma 
{\bf  Cov}_\sigma \left( \big( \mathcal{U}^*(t,\sigma)  \cL_t^* \psi  \big) , \big( \mathcal{U}^*(s,\sigma) \gp \big)   \right)
 \\
& =\hat  \cC(t,s, \cL_t^*\psi, \gp) \, .  
\end{aligned}
$$
For $s= t$,  the time derivative is
$$
\begin{aligned}
\partial_t \hat\cC(t,t,\psi,\gp)   & = \hat\cC(t,t, \cL_t^*\psi,\gp) +\hat \cC(t,t, \psi, \cL_t^* \gp)
+ {\bf  Cov}_t (    \psi   ,  \gp  ) \,.
\end{aligned}
$$
We recognize here 
the   equation~(\ref{eq: systeme covariance}) satisfied by the limit covariance~$ \cC(s,t,\gp,\psi)$  (see Proposition~\ref{prop: covariance mild}), written in infinitesimal form: 
\begin{equation}
\label{covariance-dyn}
\forall s \leq t, \qquad 
\begin{cases}
\d_t \cC(t,s, \psi,\gp)    =  \cC(t,s, \cL_t^*\psi,\gp),  \\
\d_t   \cC(t,t,\psi,\gp)    =   \cC(t,t, \cL_t^*\psi,\gp) +   \cC(t,t, \psi, \cL_t^* \gp)
+ {\bf  Cov}_t (    \psi   ,  \gp  )\,, \\
 \cC(0,0,\psi,\gp) =\displaystyle \int dz \gp(z) \psi(z) f^0(z) \, .
\end{cases}
\end{equation}
 The link between~\eqref{eq: systeme covariance} and~(\ref{eq: covariance hors equilibre}) is made rigorous in Lemma \ref{Lem: equiv lim cov} below. 
 The set of equations~\eqref{covariance-dyn} is used in the physics literature to describe correlations at equal and unequal times: we refer to \cite{EC81} which includes a comparison of several equivalent formulations of the right-hand side. 

\begin{Rmk}\label{rmk equiliibrium}
The equilibrium case~(when~$f^0 = M$ is a Maxwellian) is much simpler. 
The linear operator $\cL_{\rm eq}:=-v\cdot \nabla_x + {\bf L}_{\rm eq}$, where ${\bf L}_{\rm eq}$ is   the (autonomous) linearized operator  around~$M$, generates  indeed  a semigroup $\mathcal{U}_{\rm eq} $  of self-adjoint contractions   on~$L^2(M dvdx)$.  
By the method of~\cite{holley1978generalized}, one can  construct a martingale solution of
the  generalized Ornstein-Uhlenbeck  equation
\begin{equation}
\label{eq: OU bis equilibrium} 
d \hat \gz_t   = \cL_{\rm eq} \,\hat\gz_t\, dt + d\eta_t \;.
\end{equation}
Moreover, the covariance structure is such that the fluctuations   exactly compensate the dissipation~: using the symmetry of the equilibrium measure 
$M(z_1') M(z_2') = M(z_1) M(z_2)$ and denoting by~$ \mathcal{U}_{\rm eq}^* $ the adjoint of~$ \mathcal{U}_{\rm eq} $  in~$L^2(\D)$,    one gets 
%the function~$\psi_s :=  \mathcal{U}_{\rm eq}^* (t,s)\gp$ satisfies the backward equation 
%$$
% \d_s \psi_s+  v\cdot \nabla_x \psi_s =- \int_{\R^d}\int_{{\mathbb S}^{d-1}}  \,d \omega \, d w   \left( (v - w) \cdot \omega \right)_+  M(w) \Delta  \psi_s \,, \quad \psi_t  = \gp\,,
%$$
%which gives rise to
$$
\begin{aligned}
\int_0^t du  \, {\bf  Cov}  \big(   \mathcal{U}_{\rm eq}^*(t,\sigma) \gp,  \mathcal{U}_{\rm eq}^*(t,\sigma) \gp \big)
 & = - 2 \int_0^t d\sigma  \int   \mathcal{U}_{\rm eq}^* (t,\sigma) \gp   M {\bf L}_{\rm eq}^* \; \mathcal{U}_{\rm eq}^*(t,\sigma) \gp \\
& =  - 2 \int_0^t d\sigma  \int   \mathcal{U}_{\rm eq}^* (t,\sigma) \gp   M (-\d_\sigma -v\cdot \nabla_x)  \; \mathcal{U}_{\rm eq}^*(t,\sigma) \gp\\
& =\int M |  \gp|^2 -  \int M \, | \mathcal{U}_{\rm eq}^* (t,0 ) \gp|^2 \, .
\end{aligned}
$$
 Out of equilibrium the structure of the linearized operator  is lost: it is no longer autonomous, and the  semigroup generated by $\cL_t$ is no longer a contraction.  
\end{Rmk}

\subsection{Functional setting for~(\ref{eq: covariance hors equilibre})}

 Let us define a functional setting for the semi-group~$\mathcal{U}^*(t,s)$, and check that in this setting the right-hand side of~(\ref{eq: covariance hors equilibre}) is well defined.   
 By a Cauchy-Kovalevskaya type argument (see Theorem~\ref{nishidatheorem} and Section~\ref{boltz-Cauchy}) one can prove that there is a time~$T_0 \sim C_0^{-1}\beta_0^{(d+1)/2}$ such that there is  a unique solution $f$ to the Boltzmann equation  on the time interval $[0,T_0]$ which satisfies
\begin{equation}\label{Mmax}
\|f(t)\|_{L^\infty_{-\beta_0/2}} \leq 4 C_0 \, , 
\end{equation}
with
\begin{equation}
\label{eq: espaces L infini beta}
L^\infty_\beta := \left \{ \gp = \gp(x,v) \, : \, 
\|\gp\|_{ L^\infty_\beta} :=\sup_\D \exp \big( -\frac\beta2  |v|^2\big) |\gp (x,v) |  <+\infty \right\}\,.
\end{equation}
For any $\beta>0$, we  introduce the weighted $L^2$ space\label{L2beta}
\begin{equation}
\label{eq: espaces L2beta}
L^2_\beta := \left \{ \gp = \gp(x,v) \, : \, 
\|\gp\|_{L^2_\beta}:=\Big(\int _\D \exp \big(- \frac\beta2  |v|^2\big) \, \gp^2 (x,v) dxdv  \Big)^\frac12<+\infty \right\}\, .
\end{equation}
In particular, $(L^2_\beta)_{\beta>0}$ is an increasing sequence of Hilbert spaces and an application of Theorem \ref{nishidatheorem}   leads to  the following result: we refer to Section \ref{U*-Cauchy} of  the appendix for the proof.
 \begin{Prop}
\label{prop: borne U*}
    There is a time~$T\in (0,T_0]$ with~$T\sim C_0^{-1}\beta_0^{(d+1)/2}$,  such that for any~$\gp$ in~$ L^2_{\beta_0/4}$
    and
   any $s \leq t \leq T $,     $ \mathcal{U}^* (t,s) \gp $ is well defined and belongs to~$ L^2_{3\beta_0/8}$.
\end{Prop}
This proposition implies that   the  covariance is well defined, as stated in the next proposition. 
 \begin{Prop}
\label{prop:uniq-cov}
    There exists a time~$T\in (0,T_0]$  with~$T\sim C_0^{-1}\beta_0^{(d+1)/2}$, such that for any~$\gp$ and~$\psi$ in~$ L^2_{\beta_0/4}$ and  all times  $0 \leq s \leq t \leq T $,   the  covariance $\hat \cC(t,s,\psi,\gp)$ is well defined by~\eqref{eq: covariance hors equilibre}. \end{Prop}
\begin{proof}[Proof of Proposition {\rm\ref{prop:uniq-cov}}]
Denote~$\psi (\sigma) =  \mathcal{U}^*(t,\sigma) \psi$ and~$\gp(\sigma) =  \mathcal{U}^*(s,\sigma) \gp$. Then
by  the definition of the covariance~\eqref{eq: noise covariance} and by~\eqref{Mmax}, for any $\gp$ and~$\psi \in L^2_{\beta_0/4}$ there holds $ \forall s \leq t \leq T$
$$ \begin{aligned}
\int_0^s  d\sigma \,
{\bf  Cov}_\sigma \Big(\big( \mathcal{U}^*(t,\sigma) \psi \big) , \big( \mathcal{U}^*(s,\sigma) \gp \big)  \Big)
&\leq 2\int_0^s \int  d\mu (z_1, z_2, \omega) f(\sigma, z_1) f(\sigma, z_2)\Big( (\Delta \psi(\sigma))^2 +(\Delta \gp (\sigma))^2\Big) \\
 \leq C \int_0^s& \int d\mu (z_1, z_2, \omega) \exp ( -\frac{\beta_0}4 (|v_1|^2 +|v_2|^2))  \Big( \psi^2(\sigma,z_1) + \gp^2(\sigma,z_1)\Big)  \end{aligned}
$$
which is finite since $\psi(\sigma), \gp(\sigma)$ belong to $L^2_{3\beta_0/8}$ by Proposition~\ref{prop: borne U*}.
Therefore, 
\begin{align}
\label{eq: condition covariance}
\forall s \leq t \leq T\, , \qquad  \int_0^s  d\sigma \,
{\bf  Cov}_\sigma \Big(\big( \mathcal{U}^*(t,\sigma) \psi \big) , \big( \mathcal{U}^*(s,\sigma) \gp \big)  \Big)   < + \infty \, .
\end{align}
Similarly, the first term in the right-hand side of \eqref{eq: covariance hors equilibre} is bounded by applying Proposition \ref{prop: borne U*}, and since
 $$
  \Big|
  \hat \cC(0,0,\psi,\gp)  \Big|=\displaystyle  \Big|\int dz \gp(z) \psi(z) f^0(z)   \Big|< \infty$$thanks to~\eqref{lipschitz}.  This concludes the proof of Proposition \ref{prop:uniq-cov}.
\end{proof} 
 
\subsection{Identification with the limit covariance}

We now prove that the covariance~$\hat \cC(t,s,\psi,\gp)$ constructed above satisfies
the same equation~(\ref{eq: systeme covariance})  as the limiting covariance~$  \cC(t,s,\psi,\gp)$.
 \begin{Lem}
\label{Lem: equiv lim cov}
  The covariance $\hat \cC(t,s)$ defined by {\rm(\ref{eq: covariance hors equilibre})} and Proposition~{\rm\ref{prop:uniq-cov}} satisfies~\eqref{eq: systeme covariance} for~$(s,t) \in [0,T]^2$.
  As a consequence, the covariance $\hat \cC$ coincides on $[0,T]^2$ with the limit covariance $\cC$ of the hard sphere system  defined by~\eqref{eq: limit weak covariance}.
\end{Lem}
\begin{proof}[Proof of Lemma {\rm\ref{Lem: equiv lim cov}}] 
By definition (see Section \ref{U*-Cauchy} of the appendix),
\begin{equation}\label{Ustartsigma'1}
\forall s   \leq t\, , \quad \mathcal{U}^*(t,s) \psi =  S_{s - t}\psi  + \int_s^t du\, \mathcal{U}^*(u,s) {\bf L}_u^* S_{u-t}\psi\, .
\end{equation}
Similarly
$$\begin{aligned}& \mathcal{U}^*(t,s) \psi  \otimes \mathcal{U}^*(t,s) \gp
= S_{s  -t}\psi \otimes S_{s-t}\gp  \\
& \quad +  \int_s^tdu\,\mathcal{U}^*(u,s){\bf L}_u^* S_{u-t}\psi \otimes \mathcal{U}^*(u,s) S_{u- t}\gp
+\int_s^t du \,\mathcal{U}^*(u,s) S_{u-t}\psi \otimes \mathcal{U}^*(u,s) {\bf L}_u^*S_{u- t}\gp\,.
  \end{aligned}
$$ We consider first the case $t=s$ in \eqref{eq: covariance hors equilibre}  which we recall
\begin{align}\label{recall covariance hors equilibre}
\hat \cC(t,t,\psi,\gp)   =  \int  \mathcal{U}^*(t,0) \psi \; \mathcal{U}^*(t,0) \gp \; f^0
+ \int_0^t  d\sigma \; {\bf  Cov}_\sigma \left( \mathcal{U}^*(t,\sigma) \psi , \mathcal{U}^*(t,\sigma) \gp  \right) \,,
\end{align}
and we want to prove that it satisfies~(\ref{eq: systeme covariance}), namely (omitting the integration parameters~$dz$ to lighten notation)
$$
\begin{aligned}\hat \cC  (t,t,\psi,\gp)  &=  \int  S_{-t} \psi \; S_{-t} \gp \; f^0 +  \int_0^t d\sigma\,\hat \cC(\sigma,\sigma,{\mathbf L}_\sigma^*S_{\sigma-t}\psi, S_{\sigma-t}\gp) \\
&\qquad +  \int_0^t d\sigma\,\hat \cC(\sigma,\sigma,S_{\sigma-t}\psi, {\mathbf L}_\sigma^*S_{\sigma-t}\gp) + \int_0^t d\sigma\, {\bf  Cov}_\sigma( S_{\sigma-t}\psi, S_{\sigma-t}\gp) \, .
\end{aligned}$$
Noting that $ {\bf  Cov}_u(\psi,\gp)$ is a linear operator on the tensor product $\psi \otimes \gp$, we find from \eqref{recall covariance hors equilibre}  that
$$
\begin{aligned}&\hat \cC(t,t,\psi,\gp)  =  \int  S_{-t} \psi \; S_{-t} \gp \; f^0 +
\int_0^t  d\sigma \int   \mathcal{U}^*(\sigma,0)
{\bf L}_{\sigma}^* S_{   \sigma-t } \psi  \otimes \mathcal{U}^*( \sigma,0)S_{   \sigma-t } \gp  
  f^0 \\
&+\int_0^t  d\sigma \int   \mathcal{U}^*(\sigma,0)
S_{   \sigma-t } \psi  \otimes \mathcal{U}^*( \sigma,0){\bf L}_{\sigma}^* S_{   \sigma-t } \gp  
  f^0 
+  \int_0^t  d\sigma \; {\bf  Cov}_\sigma  
  \left( S_{\sigma-t} \psi , S_{\sigma-t} \gp  \right) \\
  & +  \int_0^t  d\sigma \int_\sigma^t  d\sigma' {\bf  Cov}_\sigma  
  \left( 
    \mathcal{U}^*(\sigma',\sigma)
 {\bf L}_{\sigma'}^*
  S_{\sigma '-t}  \psi  ,  \mathcal{U}^*(\sigma',\sigma)   S_{\sigma '-t} \gp  \right) 
\\
& +  \int_0^t  d\sigma \int_\sigma^t  d\sigma' {\bf  Cov}_\sigma  
  \left( 
    \mathcal{U}^*(\sigma',\sigma)
  S_{\sigma '-t}  \psi  ,  \mathcal{U}^*(\sigma',\sigma)  {\bf L}_{\sigma'}^*  S_{\sigma '-t} \gp  \right) \, .\end{aligned}$$
To conclude we notice that thanks to~(\ref{recall covariance hors equilibre}) again
$$
\begin{aligned} 
& \int_0^t d\sigma\,\hat \cC(\sigma,\sigma,{\mathbf L}_\sigma^*S_{\sigma-t}\psi, S_{\sigma-t}\gp)  +  \int_0^t d\sigma\,\hat \cC(\sigma,\sigma,S_{\sigma-t}\psi, {\mathbf L}_\sigma^*S_{\sigma-t}\gp)\\
& =  \int_0^t  d\sigma \int   \mathcal{U}^*(\sigma,0)
{\bf L}_{\sigma}^* S_{   \sigma-t } \psi  \otimes \mathcal{U}^*( \sigma,0)S_{   \sigma-t } \gp  
  f^0 +\int_0^t  d\sigma \int   \mathcal{U}^*(\sigma,0)
S_{   \sigma-t } \psi  \otimes \mathcal{U}^*( \sigma,0){\bf L}_{\sigma}^* S_{   \sigma-t } \gp  
  f^0 
  \\
  & +  \int_0^t  d\sigma \int_\sigma^t  d\sigma' {\bf  Cov}_\sigma  
  \left( 
    \mathcal{U}^*(\sigma',\sigma)
 {\bf L}_{\sigma'}^*
  S_{\sigma '-t}  \psi  ,  \mathcal{U}^*(\sigma',\sigma)   S_{\sigma '-t} \gp  \right) 
\\
& +  \int_0^t  d\sigma \int_\sigma^t  d\sigma' {\bf  Cov}_\sigma  
  \left( 
    \mathcal{U}^*(\sigma',\sigma)
  S_{\sigma '-t}  \psi  ,  \mathcal{U}^*(\sigma',\sigma)  {\bf L}_{\sigma'}^*  S_{\sigma '-t} \gp  \right) \, ,
\end{aligned}$$
and the result follows.

We  now study   the case of two different times. Consider~$\psi,(\gp_\sigma)_{\sigma \in [0,t]}$ in~$ L^2_{\beta_0/4}$: recalling
\begin{equation}
\label{recall covariance hors equilibre biss}
\begin{aligned}
 \hat \cC(t,\sigma,\psi,\gp_\sigma)  =\int   \mathcal{U}^*(t,0) \psi  \otimes \mathcal{U}^*(\sigma,0) \gp _\sigma f^0
+   \int_0^\sigma  d\sigma' \;
{\bf  Cov}_{\sigma '} \left( \mathcal{U}^*(t,\sigma') \psi , \mathcal{U}^*(\sigma,\sigma') \gp_\sigma  \right) \, ,
\end{aligned}\end{equation}
   we want to prove that it satisfies~(\ref{eq: systeme covariance}) namely
 \begin{align}\label{recall covariance hors equilibre bis}
\int_0^t  \hat \cC(  t,\sigma,\psi  , \varphi_\sigma) \,d \sigma
 =  \int_0^t d \sigma   \, 
\left( 
 \hat \cC(\sigma ,\sigma,S_{\sigma-t}\psi,\gp_\sigma ) 
+  \int_ \sigma^t d\sigma'   \;   
 \hat \cC \Big( \sigma', \sigma,    {\mathbf L}_{\sigma'}^* S_{\sigma'-t}\psi,  \varphi_\sigma\Big)  \right)  .\end{align}
Note that   by the semi-group property in Corollary \ref{cor: mild formula lin properties},
\begin{equation}\label{Ustartsigma'}
\forall s \leq \sigma\leq t\, , \quad \mathcal{U}^*(t,s) \psi = \mathcal{U}^*(\gs,s) S_{\gs - t}\psi  + \int_\gs^t du\, \mathcal{U}^*(u,s) {\bf L}_u^* S_{u-t}\psi\, ,
\end{equation}
so  identity~(\ref{recall covariance hors equilibre biss}) can be written
 $$
\begin{aligned} 
&\int_0^t  \hat \cC(t,\sigma,\psi,\gp_\sigma) \,d \sigma
 =  \int_0^t d \sigma \int  \mathcal{U}^*(\sigma,0) S_{\sigma - t}\psi  \otimes \mathcal{U}^*(\sigma,0) \gp _\sigma f^0
\\
&+\int_0^t d\gs \int_\gs^t d\gs' 
\int  \mathcal{U}^*(\sigma',0){\bf L}_{\gs'}^* S_{\sigma' - t}\psi  \otimes \mathcal{U}^*(\sigma,0) \gp _\sigma f^0\\
&+ \int_0^t d\gs \int_0^\gs d\gs' 
{\bf  Cov}_{\sigma '} \left( \mathcal{U}^*(\sigma,\sigma') S_{\sigma - t}\psi  , \mathcal{U}^*(\sigma,\sigma')  \gp_\sigma  \right)\\
&+ \int_0^t d\gs \int_0^\gs d\gs'  \int_{\gs }^t du
{\bf  Cov}_{\sigma '} \left( \mathcal{U}^*(u,\sigma') {\bf L}_{u}^* S_{u - t}\psi  , \mathcal{U}^*(\sigma,\sigma')  \gp_\sigma  \right)\, .
 \end{aligned}$$
 Now we note that the first term on the right-hand side adds up to the third to produce
 $$
 \begin{aligned} 
  \int_0^t d \sigma \int  \mathcal{U}^*(\sigma,0) S_{\sigma - t}\psi  \otimes \mathcal{U}^*(\sigma,0) \gp _\sigma f^0+\int_0^t d\gs \int_0^\gs d\gs' 
&{\bf  Cov}_{\sigma '} \left( \mathcal{U}^*(\sigma,\sigma') S_{\sigma - t}\psi  , \mathcal{U}^*(\sigma,\sigma')  \gp_\sigma  \right)\\
&\qquad = \int_0^t d\gs\,\hat \cC(\gs,\gs,S_{\gs-t}\psi,\gp_\gs) \, .
 \end{aligned}$$
 Finally exchanging the role of~$u$ and~$\sigma'$ in the last term on the right-hand side, we find that the two remaining terms    add up to
 $$
   \int_0^t d \sigma   \, 
  \int_ \sigma^t d\sigma'   \;   
 \hat \cC \Big( \sigma', \sigma,    {\mathbf L}_{\sigma'}^* S_{\sigma'-t}\psi,  \varphi_\sigma\Big) \, .  $$
 The result follows.
 By Proposition~\ref{prop: well posedness covariance mild} stating the uniqueness of the solution to~(\ref{eq: systeme covariance}), we deduce that~$ \hat \cC (t,s)=    \cC(t,s) $ for~$0 \leq s \leq t \leq T$. Lemma~{Lem: equiv lim cov}
is proved. 
\end{proof}

\section{Convergence of the process}

The limiting covariance has been characterized in the previous section. 
Let $\theta_1, \dots, \theta_\ell$ be a collection of times in $[0,T]$. 
Given a collection of  smooth bounded test functions $\{ \gp_j \}_{j \leq \ell}$, we consider the discrete sampling
$$H \big(z([0,T_0])\big) = \displaystyle \sum _{j= 1}^\ell \gp_j \big(z(\theta_j)\big)\;.$$
Let us define
\begin{equation}
\label{spacetimeduality}
\lA  \gz^\gep, H \rA  
:= 
\frac{1}{ \sqrt{\mu_\eps }} 
 \sum _{j= 1}^\ell
 \left[
\sum_{i = 1}^\cN  \gp_j \big(  {\bf z}^\eps_i( \theta_j ) \big) -  \mu_\eps \int \, F_1^\eps(\theta_j,z) \,  \gp_j \big(  z \big)\, dz \right] .
\end{equation} 
The convergence of the fluctuation field $\zeta^\eps$  is obtained by proving \begin{itemize}
\item 
  the convergence of the characteristic function $\bbE_\eps  \left(  \exp \big( {\bf i} \lA  \gz^\gep, H \rA  \big) \right)$ which implies that the limiting process is a weak solution of \eqref{eq: OU bis} in the sense of Definition \ref{OU-def}
\item  and the tightness of the fluctuation field.
\end{itemize}
 This will complete the proof of  
Theorem \ref{thmTCL}.

\subsection{Convergence of the characteristic function}\label{characteristic-sec}

\setcounter{equation}{0}

We are going to prove the convergence of time marginals of the process $\left(\gz^\gep_t\right)_{t \geq 0}$.
\begin{Prop}
\label{Prop: characteristic function}
The characteristic function $\bbE_\eps  \left(  \exp \big( {\bf i} \lA  \gz^\gep, H \rA  \big) \right)$
converges to the characteristic function of 
  the Gaussian process
with covariance given by \eqref{eq: covariance hors equilibre}.
%, the covariance of which 
%solves the dynamical equations   \eqref{eq: systeme covariance}.
\end{Prop}

%By Lemma \ref{lem-covariance-dyn} the covariance is given by \eqref{eq: covariance field OU} on $[0,T]$.
%The covariance of the density field out of equilibrium was first computed in \cite{S81}.
%A discussion on the result of \cite{S81} is postponed to Section \ref{appendix-spohn} 
%at the end of this chapter.

\begin{proof}
The characteristic function can be rewritten  in terms of the empirical measure
\begin{align} \label{eq:FFcf22}
\bbE_\eps  \Big( \exp \big( {\bf i}  \lA  \gz^\gep, H \rA \Big) 
= \bbE_\eps \Big( \exp \big( {\bf i} \, \sqrt{\mu_\eps} \lA \pi^\eps , H \rA \big)  \Big)
 \exp \left( - {\bf i} \,  \sqrt{\mu_\eps}   \sum_{j= 1}^\ell  \int  F^\eps_1(\theta_j,z)\, \gp_j(z)\,dz  \right).
\end{align}
Thanks to Proposition \ref{prop: exponential cumulants}, we get
$$
%\label{eq: exponential cumulants}
\log \bbE_\eps  \Big( \exp  \left( {\bf i}  \lA  \gz^\gep, H \rA \right) \Big)   
= \mu_\eps \sum_{ n=1}^\infty   \frac{1}{n !}
 f_{n, [0,t] }^\eps \left( \big( e^{  {\bf i} \, H \over \sqrt{\mu_\eps}} - 1\big)^{\otimes n} \right)  
-   {\bf i}  \, \sqrt{\mu_\eps}   \sum_{j= 1}^\ell  \int  F^\eps_1(\theta_j,z)  \gp_j (z)\,dz\, .
$$
As $H$ is bounded,  the series converges uniformly on $[0,T_0]$ for any $\mu_\eps$ large  enough. 
At leading order, only the terms $n=1$ and $n=2$ will be relevant  in the limit since by Theorem \ref{cumulant-thm1*} 
$$
\Big| f_{n, [0,t] }^\eps \left( \big( e^{  {\bf i} \, H \over \sqrt{\mu_\eps}} - 1\big)^{\otimes n} \right)  \Big| 
\leq 
\left( {C \|H\|_{\infty}  \over {\sqrt{\mu_\eps}}}\right)^n n!  \, .
$$
Expanding the exponential  with respect to $\mu_\eps$, we notice that the  term 
of  order $\sqrt{\mu_\eps}$ cancels  so
$$
\log \bbE_\eps  \left( \exp  \left( {\bf i}  \lA  \gz^\gep, H \rA \right) \right)    
 = - \frac{1}{2}     f_{1, [0,t]} ^\eps  \left( H^2 \right)
-  \frac{1}{2}  f_{2, [0,t]}^\eps  \left( H^{\otimes 2} \right)
+ O \left( \frac{ \|H\|_{\infty}^3 }{\sqrt{ \mu_\eps} } \right) \,.$$
As the cumulants $f_{1, [0,t]} ^\eps  \left( H^2 \right), f_{2, [0,t]} ^\eps  \left( H^{\otimes 2} \right)$ converge (see Theorem \ref{cumulant-thm2}), the characteristic function has a limit
\begin{align*}
\lim_{\mu_\eps \to \infty} \bbE_\eps  \left( \exp \left( {\bf i}  \lA  \gz^\gep, H \rA \right) \right)   
 = \exp \left( - \frac{1}{2}  \sum_{i,j \leq \ell}   \cC(\theta_i,\theta_j, \gp_i, \gp_j) \right),
\end{align*}
where the limiting covariance is given by \eqref{eq: limiting covariance}
   and thus
by \eqref{eq: covariance hors equilibre} thanks to Lemma \ref{Lem: equiv lim cov}. Proposition~\ref{Prop: characteristic function} is proved.
\end{proof} 
%\begin{align}
%\label{eq: limiting covariance}
%\cC(s,t,\gp,\psi)
%= f_{1, [0,t]}   \big( \psi (z(s))  \gp (z(t)) \big)
%+ f_{2, [0,t]}   \big( \psi (z(s)) , \gp (z(t)) \big)
%%\quad \text{with} \quad \psi_s = \psi (z(s)),  \gp_t = \gp (z(t)) \,,
%\end{align}
%denoting abusively by $f_{2, [0,t]}   \left( \psi  , \gp \right)$ the bilinear symmetric form obtained by polarization
%$$
%f_{2, [0,t]}   \left( \psi  , \gp \right) 
%: = 
%\frac12 \Big( f_{2, [0,t]}   \left( (\psi  + \gp)^{\otimes 2} \right) 
%- f_{2, [0,t]}   \left( \psi^{\otimes 2}  \right) - f_{2, [0,t]}   \left(  \gp^{\otimes 2} \right) \Big) \,.
%$$

\begin{Rmk}
The moments of the fluctuation field can be obtained by computing derivatives of~\eqref{eq:FFcf22}. As a byproduct of our analysis, one  then verifies  the Wick's pairing rule: for all $n \geq 1$, the moments of order $2n+1$ vanish in the limit $\mu_\eps \to \infty$ and
$$
\lim_{\mu_\eps \to \infty}\left| \bbE_\eps \left(\prod_{j=1}^{2n} \gz^\eps_{\theta_j} (\varphi_j) \right)
- \sum_{\sigma \in \cP_{2n}^n \atop |\sigma_k| = 2} \prod_{\{i,j\} \in \sigma}
\bbE_\eps \left( \gz^\eps_{\theta_i} (\varphi_i)  \gz^\eps_{\theta_j} (\varphi_j) \right) 
\right| = 0\;.
$$
We omit the details of this computation, which is not to be used in this paper.
\end{Rmk}

%\begin{Rmk}
%Notice that the dynamical equations  \eqref{covariance-dyn} are given by the second order expansion of the Hamilton-Jacobi equations derived in Theorem \ref{HJ-prop}.
%\end{Rmk}

\subsection{Tightness and proof of Theorem 2%\ref{thmTCL}
}
\label{tightness-sec}
In this section we  prove a tightness property for the law of the process~$ \left(\gz^\gep_t\right)_{t \in [0,T_0]}$.  This is  made possible by considering test functions in  a space with more  regularity than~$L^2_{\beta_0}$.
 In order to construct a convenient function space let us consider 
a Fourier-Hermite basis of~$\D$:  let~$\{ \tilde e_{j_1} (x) \}_{j_1 \in \bbZ^d}$ be the Fourier basis of $\bbT^d$
and~$\{ e_{j_2} (v) \}_{j_2 \in \bbN^d}$ be the Hermite basis of~$L^2(\bbR^d)$ constituted of the eigenmodes of the harmonic oscillator $-\Delta_v +  |v|^2$. 
This provides a basis~$\big\{ h_j (z) = \tilde e_{j_1} (x)  e_{j_2} (v)\big\}_{j = (j_1,j_2)}$ 
of Lipschitz functions on~$\bbD$, exponentially decaying in $v$,  such that  for all $j = (j_1,j_2)$
\begin{equation}
\label{eq: sobolev norm hypotheses}
\| h_j  \|_\infty  \leq c  \,  , \qquad 
\| \nabla h_j \|_\infty = \| \nabla_v h_j \|_\infty + \| \nabla_x h_j \|_\infty < c (1+|j|)  \,  , \qquad 
\| v \cdot \nabla_x h_j \|_\infty < c (1+ \, |j|) ^\frac32 \,   ,
\end{equation}
with $|j| := |j_1|+|j_2|$ and for some constant $c$ (see~\cite{Gradshteyn}). Then we define for any    real number~$k\in \R$  the
   Sobolev-type space~${\mathcal H}_{ k} (\bbD)$    by the norm
\begin{equation}
\label{eq: sobolev norm}
\| \varphi \|_{k}^2 := \sum_{j = (j_1,j_2)} (1+ |j|^2)^{k}  \left( \int_\bbD dz  \,  \varphi(z) h_j (z) \right)^2.
%\quad \text{with} \quad   |j| := |j_1|+|j_2|\, .
\end{equation}

 Following \cite{Billingsley} (Theorem 13.2 page 139), 
the tightness of the law of the process  in  $D \big([0,T_0],{\mathcal H}_{-k} (\bbD) \big)$ (for some large positive  $k$)
is a consequence of the following proposition.
\begin{Prop}
\label{Prop: tightness process}
There is $k>0$ large enough such that 
\begin{align}
\label{eq: tighness 1}
\forall \delta'>0\, , \qquad 
&  \lim_{\delta \to 0} \lim_{\mu_\eps \to \infty} \bbP_\varepsilon \Big( \sup_{ |s -t| \leq \delta \atop s,t \in [0,T_0 ]}
\big\| \gz^\gep_t  - \gz^\gep_s \big \|_{-k} \geq \delta'
\Big) = 0\, ,\\
& 
\lim_{A \to \infty} \lim_{\mu_\eps \to \infty} \bbP_\varepsilon \Big( \sup_{ t \in [0,T_0 ]}
\big \| \gz^\gep_t \big \|_{-k} \geq A \Big) = 0\, .
\label{eq: tighness 0}
\end{align}
\end{Prop}

The identification of the limit Gaussian law in Proposition \ref{spacetimeduality} together with the above tightness property complete the characterization of the limiting process and therefore the proof of   Theorem  \ref{thmTCL}.

\medskip

The proof of  Proposition \ref{Prop: tightness process} relies on the following  modified version of the Garsia, Rodemich, Rumsey inequality \cite{Varadhan}
which will be used to control the modulus of continuity (its derivation is postponed 
to Section~\ref{App: Proof of Proposition Modified Garsia, Rodemich, Rumsey inequality}).

\begin{Prop}
\label{prop: Modified Garsia, Rodemich, Rumsey inequality}
Given $b \geq 4$, 
choose  two functions~$\Psi (u) = u^b$ and $p(u) = u^{\gamma/b}$ with $\gamma$ belonging to~$ ] 2,3[$.
Let $\gp : [0,T_0] \to \bbR$ be a given function and define 
for $a >0$ 
\begin{equation}
\label{eq: bound GRR alpha 0} 
B_a : = 
\int_0^{T_0} \int_0^{T_0} ds dt \;
\Psi \left( \frac{| \gp_t - \gp_s|}{p(|t-s|)}\right) {\bf 1}_{|t - s| > a }\, .
\end{equation}
%If $\gp$ satisfies
%\begin{equation}
%\label{eq: modulus short} 
%\sup_{0 \leq s,t\leq {T^\star} \atop |t - s| \leq  2 \alpha} \big| \gp_t - \gp_s \big|
%\leq \nu\, ,
%\end{equation}
The modulus of continuity of $\gp$ is controlled by
\begin{equation}
\label{eq: modulus GRR cutoff} 
\sup_{0 \leq s,t \leq {T_0} \atop |t - s| \leq \delta} \big| \gp_t - \gp_s \big|
\leq 2 \sup_{0 \leq s,t\leq {T_0} \atop |t - s| \leq  2 a } \big| \gp_t - \gp_s \big| 
\; + \; C \; B_a^{1/b} \; \delta^{\frac{\gamma- 2}{b}}\,  ,
\end{equation}
for some constant $C$ depending only on $b$ and $\gamma$.
\end{Prop}
In the standard Garsia, Rodemich, Rumsey inequality, \eqref{eq: bound GRR alpha 0} is assumed to hold with $a =0$
leading to a stronger conclusion as $\gp$ is then proved to be  H\"older continuous. 
The cut-off $a >0$ allows us to consider functions $\gp$ which may be discontinuous.
\bigskip

\begin{proof}[Proof of Proposition {\rm\ref{Prop: tightness process}}]

At time 0, all the moments of  $\gz^\gep_0$ are bounded, so   
\eqref{eq: tighness 0} can be deduced from the control of the initial fluctuations and 
the bound \eqref{eq: tighness 1} on the modulus of continuity.
Thus it is enough to prove \eqref{eq: tighness 1}. 
For this, we are going to show that 
\begin{align}
\label{eq: tighness 2}
& \forall \delta'>0\, , \qquad 
\lim_{\delta \to 0} \lim_{\mu_\eps \to \infty} \bbP_\varepsilon \left( 
\sum_j  \frac{1}{ (1+ |j|^2)^{k} } 
 \sup_{ |s -t| \leq \delta \atop s,t \in [0,T_0]}
\big| \gz^\gep_t (h_j)  - \gz^\gep_s (h_j) \big |^2  \geq \delta' \right) = 0 \,  ,
\end{align}
where $\{ h_j (z)\}_{j = (j_1,j_2)}$ is the family of test functions introduced above.

We are going to apply Proposition \ref{prop: Modified Garsia, Rodemich, Rumsey inequality}
to the functions $t \mapsto \gz^\gep_t (h_j)$ with $b=4$ and a time scale cut-off $a$ vanishing as  $\alpha_\eps = \mu_\eps^{-7/3}$.
In order to do so, the short time fluctuations have first to be controlled.
This will be achieved thanks to the following lemma.
\begin{Lem}
\label{Lem: short time}
The time scale cut-off will be denoted by  $\alpha_\eps = \mu_\eps^{-7/3}$. For the basis of functions introduced in  {\rm\eqref{eq: sobolev norm hypotheses}}, there is~$k>0$ large enough so that
\begin{align} 
\label{eq: tighness short}
& \forall \delta'>0\, , \qquad 
 \lim_{\mu_\eps \to \infty} \bbP_\varepsilon \left( 
\sum_j \frac{1}{ (1+|j|^2)^k } 
\sup_{ |s -t| \leq 2 \alpha_\eps \atop s,t \in [0,T_0]}
\big| \gz^\gep_t (h_j)  - \gz^\gep_s (h_j) \big |^2  \geq \delta'  \right) = 0\, .
\end{align}
\end{Lem}

To control the fluctuations on time scales of order $\delta$, it will be enough to rely on averaged estimates
of the following type.
\begin{Lem}
\label{Lem: delta time}
There exists a constant~$C$ such that for any   function $h$ and for any~$\gep>0$ and~$ s,t $ in~$ [0,T_0]$  
\begin{equation}
\label{eq: bound esperance t,s} 
\bbE_\varepsilon \left( \big( \gz^\gep_t (h)  - \gz^\gep_s (h) \big)^4 
\right) \leq  C  \, \|h\|_{\infty}^2 ( \| \nabla h\|^2_{L^\infty}  + \|h\|^2_{\infty} )  \; \Big(  |t - s|^2 + \frac{1}{\mu_\eps} |t - s| \Big) \,  .
\end{equation}
\end{Lem}

We postpone the proofs of the two previous statements  %\ref{Lem: short time}, \ref{Lem: delta time} 
and conclude first the proof of  \eqref{eq: tighness 2}.

Notice that Lemma \ref{Lem: delta time}
 implies that the random variable associated with any function $h_j$ satisfying~\eqref{eq: sobolev norm hypotheses}
\begin{equation}
%\label{eq: bound GRR esperance} 
B_{\alpha_\gep } (h_j) := 
\int_0^{T_0} \int_0^{T_0} ds \, dt 
 \frac{ \big|  \gz^\gep_t (h_j)  - \gz^\gep_s (h_j) \big|^4}{|t-s|^\gamma}
 {\bf 1}_{|t -s |> \alpha_\gep}
\end{equation}
has finite expectation
\begin{equation}
%\label{eq: bound GRR esperance} 
\bbE_\varepsilon \big( B_{\alpha_\gep}  (h_j) \big) 
\leq C(1+|j| )^2  \int_0^{T_0} \int_0^{T_0} ds dt 
\left( |t-s|^{2 - \gamma}   + \frac{1}{\mu_\eps} |t - s|^{1- \gamma} {\bf 1}_{|t -s |> \alpha_\gep} \right) \,  .
\end{equation}
Setting now $\gamma  = 7/3$, we get an upper bound uniform with respect to $\gep$ for $\alpha_\eps = \mu_\eps^{-7/3}$
\begin{equation}
\label{eq: bound  esperance B} 
\bbE_\varepsilon \big( B_{\alpha_\gep}  (h_j) \big) 
\leq C(1+|j| )^2  \left( 1    + \frac{\alpha_\eps^{2-\gamma}}{\mu_\eps} \right)
\leq C' (1 + |j| )^2\;.
\end{equation}
%Consider a  trajectory sample  such that for the basis of functions introduced in  \eqref{eq: sobolev norm hypotheses}, 
%for some~$k>0$ to be fixed below,
%$$
%\nu_j : = \sup_{ |s -t| \leq 2 \alpha_\eps \atop s,t \in [0,T^\star]}
%\big| \gz^\gep_t (h_j)  - \gz^\gep_s (h_j) \big |  
%\quad \text{with} \quad
%\sum_j  \frac{1}{(1+ |j|^2)^k } \nu_j^2 \leq \frac{\delta'}{16} \, \cdotp
%$$
%Thanks to Proposition \ref{prop: Modified Garsia, Rodemich, Rumsey inequality},
% the modulus of continuity $t \mapsto \gz^\gep_t (h_j)$ is bounded from above by 
% \begin{equation*}
%%\label{eq: modulus GRR cutoff} 
%\sup_{0 \leq s,t, \leq {T^\star} \atop |t - s| \leq \delta} \big| \gz^\gep_t (h_j) - \gz^\gep_s (h_j) \big|
%\leq 2 \nu_j + 8 \sqrt{2} \left( B_{\alpha_\gep} (h_j) \right)^{1/4} \delta^{\frac{\gamma}{4} - \frac{1}{2}}\, ,
%\end{equation*}
%so that
% \begin{equation*}
%%\label{eq: modulus GRR cutoff} 
%\sup_{0 \leq s,t, \leq {T^\star} \atop |t - s| \leq \delta} 
%\sum_j \frac{1}{(1+ |j|^2)^k }
%\big| \gz^\gep_t (h_j) - \gz^\gep_s (h_j) \big|^2
%\leq \frac{\delta'}{2} 
%+ 2^8 \, \delta^{\frac{\gamma}{2} - 1} 
%\sum_j \frac{\sqrt{ B_{\alpha_\gep} (h_j)} }{ (1+ |j|^2)^k  }  \, \cdotp
%\end{equation*}

 From Proposition \ref{prop: Modified Garsia, Rodemich, Rumsey inequality},
 a large modulus of continuity of $t \mapsto \gz^\gep_t (h_j)$ induces a deviation of the 
random variable $B_{\alpha_\gep}(h_j)$. This implies that on average
\begin{align}
& 
\bbP_\varepsilon \Big( 
\sum_j \frac{1}{(1+ |j|^2)^k  } \; \sup_{ |s -t| \leq \delta \atop s,t \in [0,T_0]} \;
\big| \gz^\gep_t (h_j)  - \gz^\gep_s (h_j) \big |^2  
\geq \delta'  
\Big) \nonumber \\
&  \qquad  \leq 
\bbP_\varepsilon \Big( 
\sum_j \frac{1}{ (1+ |j|^2)^k  } 
 \sup_{ |s -t| \leq 2 \alpha_\eps \atop s,t \in [0,T_0]}
\big| \gz^\gep_t (h_j)  - \gz^\gep_s (h_j) \big|^2  
\geq \frac{\delta'}{16}  
\Big)
+
\bbP_\varepsilon \Big( \sum_j \frac{\sqrt{ B_{\alpha_\gep} (h_j)} }{ (1+ |j|^2)^k }   
\geq \frac{\delta'  }{C \, \delta^{\frac{\gamma}{2} - 1}}
\Big)\, .
\label{tight-est}
\end{align}
The first term in (\ref{tight-est}) tends to 0 by Lemma \ref{Lem: short time} and the second one can be estimated by the 
Markov inequality and by the upper bound \eqref{eq: bound  esperance B}, along with the Cauchy-Schwarz inequality
\begin{align*}
\bbP_\varepsilon \Big( \sum_j \frac{\sqrt{ B_{\alpha_\gep} (h_j)} }{  (1+ |j|^2)^k }   
\geq \frac{\delta'  }{C \, \delta^{\frac{\gamma}{2} - 1}} \Big)
\leq  C_1 \frac{\delta^{\gamma - 2 }}{\delta'^2  }  \sum_j \frac{1}{  (1+ |j|^2)^k }   
\bbE_\varepsilon \big( B_{\alpha_\gep} (h_j) \big) 
\leq \frac{C_2}{\delta'^2}  \, \delta^{\gamma - 2}\,,
\end{align*}
for some constants $C_1,C_2$ and $k$ large enough.
As $\gamma = 7/3$, the limit \eqref{eq: tighness 2} holds and Proposition~\ref{Prop: tightness process} is proved.
\end{proof}

\subsection{Averaged time continuity}
\label{subsec: Averaged time continuity}

We prove now Lemma \ref{Lem: delta time}.
Denoting 
$$H(z([0,t])) : = h(z(t)) - h(z(s))\, , 
$$
the moments can be recovered by taking derivatives of the exponential moments
\begin{equation}
\label{eq: control 4th moment}
\bbE_\eps  \left( \big( \gz^\gep_t (h)  - \gz^\gep_s (h) \big)^4  \right) = \left( {\d ^4 \over \d \lambda^4}  \bbE_\eps  \left( \exp  \left( {\bf i}\lambda  \lA  \gz^\gep, H \rA \right) \right)    \right)_{| \lambda = 0}.
\end{equation}
We recall from Proposition \ref{prop: exponential cumulants} that
$$
\log \bbE_\eps  \left( \exp  \left( {\bf i}  \lambda \lA  \gz^\gep, H \rA \right) \right) 
= \mu_\eps \sum_{ n=1}^\infty   \frac{1}{n !}
 f_{n, [0,t] }^\eps \left( \big( e^{  {\bf i}\lambda H\over \sqrt{\mu_\eps}} - 1\big)^{\otimes n}
 \right)  
 -  
 \sqrt{\mu_\eps} \,  {\bf i} \, \lambda    F_1^\eps(H) 
= O(\lambda^2).
$$
Thus expanding the exponential moment at the 4th order leads to
\begin{align*}
\bbE_\eps  \left( \exp  \left( {\bf i} \lambda \lA  \gz^\gep, H \rA \right) \right) 
= &1+ \mu_\eps \sum_{ n=1}^\infty   \frac{1}{n !}
 f_{n, [0,t] }^\eps 
 \left( \big( e^{  {\bf i}\lambda H\over \sqrt{\mu_\eps}} - 1 \big)^{\otimes n}
 \right)   
-  \sqrt{\mu_\eps}  {\bf i}\lambda    F_1^\eps(H) \\
& - \frac{\lambda^4}{2} \left(  \frac{1}{2}     f_{1, [0,t]} ^\eps  \left( H^2 \right)
+ \frac{1}{2}  f_{2, [0,t]} ^\eps  \left( (H)^{\otimes 2} \right)\right)^2 + o(\lambda^4) \, .
%= &
% f_{1, [0,t] }^\eps \Big( \mu_\eps \big( e^{  {\bf i}\lambda H\over \sqrt{\mu_\eps}} - 1- {  {\bf i}\lambda \over \sqrt{\mu_\eps}}  H \big)  \Big) 
% +
% \sum _{n=2}^4  \mu_\eps f_{n, [0,t] }^\eps \left( \big( e^{  {\bf i}\lambda H\over \sqrt{\mu_\eps}} - 1 \big)^{\otimes n} \right)   
% \\
%& + \frac{\lambda^4}{8} \left(   f_{1, [0,t]} ^\eps  \left(  H^2 \right)
%+  f_{2, [0,t]} ^\eps  \left( H^{\otimes 2}  \right)\right)^2 + o(\lambda^4) .
\end{align*}
The fourth moment can be recovered by taking the 4th derivative with respect to $\gl$
\begin{equation}
\label{eq: ordre 4}
\begin{aligned}
\bbE \left( \big( \gz^\gep_t (h)  - \gz^\gep_s (h) \big)^4  \right)  
&=   3 \left(   f_{1, [0,t]} ^\eps  \left(  H^2 \right)
+  f_{2, [0,t]} ^\eps  \left( H^{\otimes 2}  \right)  \right)^2  \\
& \qquad + {1\over \mu_\eps } \sum_{n=1}^4 \sum_{\kappa_1 +\dots +\kappa_n = 4} 
C_\kappa \; f_{n, [0,t]} ^\eps (H^{\kappa_1} \otimes \dots\otimes H^{\kappa_n})
\end{aligned}
 \end{equation}
 denoting abusively by $f^\eps_{n,[0,t]} $  the $n$-linear form obtained by polarization.
Point {\it 3}.~of Theorem  \ref{cumulant-thm1*}   applied with $\delta = t-s$ implies 
\begin{equation}  
% \label{timecontinuity}
\left|  f_{1, [0,t]} ^\eps  \left(  H^2 \right) +  f_{2, [0,t]} ^\eps  \left( H^{\otimes 2}  \right)  \right|  
\leq   C \,( \| \nabla h\|_{\infty} + \|h\|_{\infty}) \; \|h\|_{\infty} \;  |t-s|  \,  (t+\eps)   \,.
\end{equation}
Furthermore for any $\kappa_1 +\dots +\kappa_n = 4$,  Point {\it 3}.~of Theorem  \ref{cumulant-thm1*} implies  also
\begin{equation*}
\left|  f_{n, [0,t]} ^\eps (H^{\kappa_1}\otimes\dots\otimes H^{\kappa_n}) \right|
\leq C \, \|h\|_{\infty}^3 \,  (t+\eps)^3 (t-s)( \| \nabla h\|_{\infty} + \|h\|_{\infty})  \;.
\end{equation*} 
Combined with  \eqref{eq: ordre 4}, this leads to  
\begin{equation}
\label{eq: norme 4 t -s}
\bbE \left( \big( \gz^\gep_t (h)  - \gz^\gep_s (h) \big)^4  \right)  
\leq 
C(t+\eps)^2  \|h\|_{\infty}^2  ( \| \nabla h\|^2_{\infty} +\|h\|^2_{\infty}) \; |t-s| \left( |t-s| + \frac{t+\eps }{\mu_\eps}  \right) . 
\end{equation}
 
This concludes  the proof of Lemma \ref{Lem: delta time}. \qed
 \begin{Rmk}
 Notice that  since the assumption \eqref{eq: time-cont + growth exp} is satisfied, the norms  $\| h \exp ( - \beta_0 v^2 / 4 ) \|_{L^\infty}$ and $\| \nabla h \exp ( - \beta_0 v^2/4)  \|_{L^\infty}$ could have been used instead of~$\| h   \|_{L^\infty}$ and $\| \nabla h    \|_{L^\infty}$. \end{Rmk}
 
\subsection{Control of small time fluctuations}
\label{smalltimefluctuations}

We are now going to prove Lemma \ref{Lem: short time} by localizing the estimates into short time intervals. 
For this divide~$[0,T_0]$ into overlapping intervals~$I_i := [i \alpha_\gep, (i+2) \alpha_\gep]$ of size $2 \alpha_\gep$.
Define also the set of trajectories such that 
at least two distinct collisions occur in the particle system during the time interval $I_i$
\begin{equation}
\label{eq: au moins 2 collisions}
\cA_i:= \Big\{ \text{At least two collisions occur in the Newtonian dynamics $\{ {\bf z}^\eps_\ell (t) \}_{\ell \leq \cN}$ during $I_i$} 
\Big\}.
\end{equation}
We are going to show that the probability of $\cA = \cup_i \cA_i$ vanishes in the limit
\begin{equation}
\label{eq: proba cA}
\lim_{\eps \to 0} \bbP_\eps ( \cA) = 0.
\end{equation}

Assuming the validity of \eqref{eq: proba cA} for the moment, let us first conclude the proof of 
Lemma \ref{Lem: short time} by restricting to the event $\cA^c$.
By construction for any trajectory in $\cA^c$, there is at most one collision during each  time interval $I_i$. Then, except for at most 2 particles, the particles move in straight lines as their 
velocities remain unchanged and it is enough to track 
the variations of the test functions with respect to the positions.
Thus, for any $t,s$ in $I_i$ and a smooth function $h_j$, we get
\begin{align*}
& \sqrt{\mu_\eps} \left( \gz^\gep_t \big(  h_j  \big) 
 - \gz^\gep_s \big(  h_j  \big) \right)
=    
\sum_{\ell = 1}^\cN \big( h_j \big( {\bf z}^\eps_\ell(t) \big) - h_j \big( {\bf z}^\eps_\ell(s) \big) \big) -  \mu_\eps \int dz \big(  F_1^\eps (t,z)  -  F_1^\eps(s,z)  \big) h_j (z)
\\
&  \qquad \qquad =    
\sum_{\ell = 1}^\cN  \int_s^t du  \; {\bf v}^\eps_\ell (u) \cdot \nabla h_j \big( {\bf z}^\eps_\ell( u ) \big) - \mu_\eps
\int dz \big(  F_1^\eps(t,z)  -  F_1^\eps(s,z)  \big) h_j (z)
+ O( \| h_j \|_\infty)\, ,
\end{align*}
where the error occurs from the fact that at most two particles may have collided in the time interval~$[s,t]
\subset I_i$.
Using the Duhamel formula, the particle density (at fixed $\eps$) can be also estimated by
the free transport up to small corrections which may occur from the collision operator
$C^\eps_{1,2} F_2^\eps$
\begin{align}
\label{eq: error term h}
\mu_\eps \int dz \big(  F^\eps_1(t,z)  -  F^\eps_1(s,z)  \big) h_j (z)
= \mu_\eps \int_s^t du \int dz  F^\eps_1(u,z) \,  v \cdot \nabla  h_j (z) 
+ \mu_\eps \alpha_\eps O( \| h_j \|_\infty)\,.
\end{align}
Recall that $\mu_\eps \alpha_\eps \to 0$ when $\mu_\eps$ tends to infinity.
Setting  $\bar h_j (z) := v \cdot \nabla  h_j (z)$, the time difference can be rewritten 
 for any trajectory in $\cA^c$ as a time integral
\begin{align}
\label{eq: h et bar h}
\gz^\gep_t \big(  h_j  \big) - \gz^\gep_s\big(  h_j  \big) 
& =\frac{1}{\sqrt{\mu_\eps}}
 \int_s^t du  \left( \mu_\eps\la \pi^\eps_u , \bar h_j \ra - \mu_\eps \int F_1^\eps (u,z) \bar h_j (z) dz\right)
+  \frac{1}{\sqrt{\mu_\eps}}  O( \| h_j \|_\infty)\\
& = \int_s^t du \; \zeta^\eps_u (\bar h_j ) +  \frac{1}{\sqrt{\mu_\eps}}  O( \| h_j \|_\infty)\, . \nonumber
\end{align}

Thus thanks to  \eqref{eq: h et bar h},  we get
\begin{align*} 
%\label{eq: h norm infinie petite}
U:=\bbP_\varepsilon & \left( \cA^c \bigcap \left \{
\sum_j \frac{1}{ (1+|j|^2)^k } 
\sup_{ |s -t| \leq 2 \alpha_\eps \atop s,t \in [0,T_0]}
\big| \gz^\gep_t (h_j)  - \gz^\gep_s (h_j) \big |^2  \geq \delta'  \right\} \right) \\
& \qquad \leq \bbP_\varepsilon  \left( \cA^c \bigcap \left \{
\sum_j \frac{1}{ (1+|j|^2)^k } 
\sup_{ i \leq \frac{T_0}{\alpha_\gep} } \;  \sup_{s,t \in I_i}
\big| \gz^\gep_t (h_j)  - \gz^\gep_s (h_j) \big |^2  \geq \delta'  \right\} \right) \nonumber\\
& \qquad \leq \bbP_\varepsilon  \left( \cA^c \bigcap \left \{
\sum_j \frac{1}{ (1+|j|^2)^k } 
\sup_{ i \leq \frac{T_0}{\alpha_\gep} } \;  \sup_{s,t \in I_i}
\big|  \int_s^t du \; \zeta^\eps_u (\bar h_j ) \big |^2  \geq \frac{\delta'}{2}  \right\} \right), \nonumber
\end{align*}
where the error term in \eqref{eq: h et bar h} was controlled by choosing $k$ large enough and 
$\eps$ small enough so that $\frac{1}{\sqrt{\mu_\eps}} \ll \gd'/2$. At this stage, the constraint $\cA^c$ can be dropped and 
by the  Bienaym\'e-Tchebichev inequality there holds
\begin{align}
\label{eq: very small intervals}
U  & \leq \sum_j \frac{1}{\gd' (1+|j|^2)^k } 
\bbE_\varepsilon \left(  
\sup_{ i \leq \frac{T_0}{\alpha_\gep} } \;  \sup_{s,t \in I_i}
\big| \int_s^t du \; \zeta^\eps_u (\bar h_j )  \big |^2  \right)  \\
& 
\leq \sum_{i=1}^{\frac{T_0}{\alpha_\gep}}  \sum_j \frac{1}{\gd' (1+|j|^2)^k } 
\bbE_\varepsilon \left(   \sup_{s,t \in I_i}
\big| \int_s^t du \; \zeta^\eps_u (\bar h_j )  \big |^2  \right) . \nonumber
\end{align}
Using  the  Cauchy-Schwarz inequality and then the fact that $t,s$ belong to 
$I_i = [ i \alpha_\gep, (i+1) \alpha_\gep]$, we get
\begin{equation}
\label{eq: etape intermediaire}
\begin{aligned}
\bbE_\varepsilon \left(  \sup_{s,t \in I_i}
\big| \int_s^t du  \;  \zeta^\eps_u (\bar h_j )   \big |^2  \right) 
\leq \bbE_\varepsilon \left(  \sup_{s,t \in I_i}
 |t-s| \;  \int_s^t du   \;   | \zeta^\eps_u (\bar h_j )|^2  \right)\\
\leq \alpha_\eps \int_{i \alpha_\gep}^{(i+1) \alpha_\gep} du \; 
\bbE_\varepsilon \Big(  \gz^\eps_u \big(\bar h_j \big)^2 \Big)  
\leq c \, \alpha_\eps^2   (1+|j|)^3.
\end{aligned}
\end{equation}
In the last inequality, an argument similar argument to \eqref{eq: norme 4 t -s} leads to 
the control of the second moment of $\gz^\eps_u \big(\bar h_j \big)$ by $\| \bar h_j  \|_\infty^2 \leq c (1+|j|)^3$ as $\bar h_j = v \cdot \nabla_x  h_j$ (see \eqref{eq: sobolev norm hypotheses}).

Combining \eqref{eq: very small intervals} and \eqref{eq: etape intermediaire}, we deduce that 
for $k$ large enough
\begin{align} 
\label{eq: very small intervals 2}
U \leq \sum_{i=1}^{\frac{T_0}{\alpha_\gep}}  \sum_j \frac{ c \, \alpha_\eps^2   (1+|j|)^3 }{\gd' (1+|j|^2)^k } \leq \frac{C}{\gd'} \alpha_\eps \xrightarrow{\eps \to 0} 0 .
\end{align}
%By construction the test functions $h_j$ satisfy  \eqref{eq: sobolev norm hypotheses}  so that $\bar h_j (z) = v \cdot \nabla_x  h_j (z)$ are bounded in $L^\infty$ by~$(1+|j|^2)$.
%We deduce that for $k$ large enough, the series converges and

\medskip

Thus to complete the proof of Lemma \ref{Lem: short time}, it remains only to show \eqref{eq: proba cA}, i.e. that  the probability concentrates on $\cA^c$.
To the estimate the probability of the set $\cA_i$  introduced in \eqref{eq: au moins 2 collisions},
we distinguish two cases :
\begin{itemize}
\item A particle has at least two collisions during $I_i$. This event will be denoted by 
$\cA^1_i$ if the corresponding particle has label $1$, and can be separated into two subcases: either particle~$1$ encounters two different particles during~$I_i$, or it encounters the same one due to space periodicity.
\item Two collisions occur involving different particles. This event will be denoted by 
$\cA^{1,2}_i$ if the corresponding particles are $1$ and $2$.
\end{itemize}
The occurence of two collisions in a time interval of length $\alpha_\gep$ has a probability which can be estimated by using Proposition \ref{prop: identification proba Duhamel} with $n=1,2$, which allows  to reduce to an estimate on  pseudo-trajectories thanks to the Duhamel formula: {\color{black}noticing that the space-periodic situation leads to an exponentially small contribution, since it forces the velocity of the colliding particles to be of order~$1/\alpha_\eps$, we find}
\begin{equation}
\label{eq: etape intermediaire 2}
\bbP_\varepsilon \left(  \cA_i  \right)
 \leq  \mu_\eps  \bbP_\varepsilon \left( \cA^1_i  \right) 
 + \mu_\eps^2  \bbP_\varepsilon \left(  \cA^{1,2}_i \right) 
 \leq C \big( \mu_\eps  + \mu_\eps^2 \big) \alpha_\eps^2
 \leq C \alpha_\eps  \mu_\eps^{-1/3},
\end{equation}
where we used that $\alpha_\eps = \mu_\eps^{-7/3}$.
Summing over the
$\frac{T_0}{\alpha_\eps}$ time intervals, we deduce that 
$\bbP_\varepsilon \left(  \cA  \right) \leq C T_0  \mu_\eps^{-1/3}$.
Thus the probability of $\cA$ vanishes as $\eps$ tends to 0. This completes 
the proof of \eqref{eq: proba cA} and thus of Lemma \ref{Lem: short time}. \qed

\begin{Rmk}
\label{Rmk: extension fonction hj}
Remark that the proof of Lemma {\rm\ref{Lem: short time}} still holds for sequences of functions $(h_j)_{j \geq 1}$ satisfying 
$$\| h_j \|_\infty \ll \mu_\eps^{1/2} (1+ j^2)\, , \qquad  
\bbE_\varepsilon \Big(  \gz^\eps_u \big(v \cdot \nabla h_j \big)^2 \Big) \leq  c \,  (1+|j|)^3\, .
$$\end{Rmk}

\section{The modified Garsia, Rodemich, Rumsey inequality}
\label{App: Proof of Proposition Modified Garsia, Rodemich, Rumsey inequality}
\setcounter{equation}{0}

Proposition \ref{prop: Modified Garsia, Rodemich, Rumsey inequality} is a slight adaptation of~\cite{Varadhan}.
For simplicity we suppose that $T_0=1$ and set 
\begin{equation}
%\label{eq: bound GRR alpha 0} 
B_a:= 
\int_0^1 \int_0^1 ds dt \;
\Psi \left( \frac{| \gp_t - \gp_s|}{p(|t-s|)}\right) {\bf 1}_{|t - s| > a}\, .
\end{equation}

\noindent
{\bf Step 1:}\\
We are first going to show that there exists $w,w' \in [0, 2 a]$  
such that
\begin{align}
\label{eq: modulus alpha} 
\big| \gp_{1- w'} - \gp_{w} \big|
\leq 8 \int_0^1 
\Psi^{-1} \left( \frac{4B_a}{u^2} \right) dp(u)
& 
\leq   8 (4  B_a)^{1/b} \int_0^1  \frac{d ( u^{\frac{\gamma}{b}})}{u^{2/b}} 
\leq   C \,  B_a^{1/b}.
%&\leq   8 \sqrt{2} B_a^{1/4} \int_0^1  \frac{d ( u^{\frac{\gamma}{4}})}{\sqrt{u}} 
%\leq   c B_a^{1/4}\, .
\end{align}
Define
\begin{equation}
\label{eq: bound GRR alpha} 
B_a (t) = 
 \int_0^1 ds \;
\Psi \left( \frac{\gp_t - \gp_s }{p(|t-s|)}\right) {\bf 1}_{|t - s| > a}
\quad \text{with} \quad
B_a =  \int_0^1 dt B_a (t).
\end{equation}
There is $t_0 \in (0,1)$ such that $B_a (t_0) \leq B_a $. Suppose that $t_0> 2 a$, 
then we are going to prove that there is $w \in [0, 2 a]$ such that 
\begin{equation}
\label{eq: borne t0 0}
\big| \gp_{w} - \gp_{t_0} \big|
\leq 4 \int_a^1 
\Psi^{-1} \left( \frac{4B_a}{u^2} \right) dp(u).
\end{equation}
If $t_0< 1- 2 a$, we can show the reverse inequality 
\begin{equation*}
\big| \gp_{1 - w'} - \gp_{t_0} \big|
\leq 4 \int_a^1 
\Psi^{-1} \left( \frac{4B_a}{u^2} \right) dp(u).
\end{equation*}
Combining both inequalities, will be enough to complete \eqref{eq: modulus alpha}.

\bigskip

Let us assume that $t_0> 2 a$, we are going to build a sequence $\{ t_n, u_n \}_n$  
$$
t_0> u_1> t_1 > u_2> \dots 
$$
such that $t_{n-1}> 2 a$ and $u_n$ is defined by 
\begin{equation}
\label{eq: recurrence}
p(u_n) = \frac{1}{2} p(t_{n-1}), \quad \text{i.e.} \quad
u_n = \frac{1}{2^{4/\gamma}} t_{n-1}.
\end{equation}
The sequence will be stopped as soon as $t_n<2 a$.

Initially $t_0> 2 a$ and $u_1$ is defined by \eqref{eq: recurrence}.
Suppose that the sequence has been built up to $t_{n-1}$.
By construction 
$$
t_{n-1} - u_n = \left( 1- \frac{1}{2^{4/\gamma}} \right) t_{n-1} > a
\quad \text{since} \quad t_{n-1} > 2 a.
$$
Thus 
$$
\int_0^{u_n} ds \;
\Psi \left( \frac{| \gp_{t_{n-1}} - \gp_s|}{p(|t_{n-1}-s|)}\right) 
= 
\int_0^{u_n} ds \;
\Psi \left( \frac{| \gp_{t_{n-1}} - \gp_s|}{p(|t_{n-1}-s|)}\right) {\bf 1}_{|t_{n-1} - s| > a}
\leq  B_a (t_{n-1}) .
$$
Furthermore
$$
  \int_0^{u_n} dt B_a (t) \leq B_a ,
$$
%In both integrals, there must be an interval of size larger than $\frac{u_n}{2}$ such that 
thus there is $t_n \in [0,u_n]$ such that 
$$
B_a (t_n) \leq \frac{2 B_a}{u_n} 
\quad \text{and} \quad 
\Psi \left( \frac{| \gp_{t_{n-1}} - \gp_{t_n}|}{p(|t_{n-1} - t_n|)}\right) 
\leq  \frac{2 B_a (t_{n-1}) }{u_n} \leq  \frac{4 B_a}{u_{n-1} \, u_n}
\leq  \frac{4 B_a}{ u_n^2}.
$$
We deduce that 
$$
| \gp_{t_{n-1}} - \gp_{t_n}| \leq
\Psi^{-1} \left(  \frac{4 B_a}{ u_n^2} \right)  p(|t_{n-1} - t_n|)
\leq \Psi^{-1} \left(  \frac{4 B_a}{ u_n^2} \right)  p(t_{n-1} ).
$$
Suppose that $t_n > 2 a$ then using that 
$$u_n > t_n \Rightarrow p(u_n) > p(t_n) = 2 p(u_{n+1}),$$ 
we get
$$
p(t_{n-1}) = 2 p(u_n) =
4 \big( p(u_n) - p(u_n)/2   \big) \leq 4 \big( p(u_n) - p(u_{n+1})   \big)
$$
and also
\begin{equation}
\label{eq: integrale 1 step}
| \gp_{t_{n-1}} - \gp_{t_n}| \leq
4 \Psi^{-1} \left(  \frac{4 B_a}{ u_n^2} \right)  \big( p(u_n) - p(u_{n+1}) \big)
\leq
4 \int_{u_{n+1}}^{u_n}   \Psi^{-1} \left(  \frac{4 B_a}{ u^2} \right)  d p (u)  .
\end{equation}
We then iterate the procedure to define $t_{n+1}$.

If $t_n <  2 a$, we set $w = t_n$
and we stop the procedure at step $n$ with the inequality
\begin{equation}
\label{eq: integrale 2 step}
| \gp_{t_{n-1}} - \gp_{w}| = | \gp_{t_{n-1}} - \gp_{t_n}| \leq
4 \int_{0}^{u_n}   \Psi^{-1} \left(  \frac{4 B_a}{ u^2} \right)  d p (u)  ,
\end{equation}
where we used that $$
p(t_{n-1}) = 2 p(u_n)  \leq 4 \big( p(u_n) - p(0)   \big).
$$
Summing the previous inequalities of the form \eqref{eq: integrale 1 step}, we deduce \eqref{eq: borne t0 0}
from 
\begin{equation}
\label{eq: integrale 3 step}
| \gp_{t_0} - \gp_{w}| \leq
\sum_{i=1}^n | \gp_{t_{i-1}} - \gp_{t_i}| \leq
4 \int_{0}^{u_1}   \Psi^{-1} \left(  \frac{4 B_a}{ u^2} \right)  d p (u)  .
\end{equation}
This completes the proof of~(\ref{eq: modulus alpha}).

\bigskip

\noindent
{\bf Step 2: proof of \eqref{eq: modulus GRR cutoff}.}\\
We are going to proceed by a change of variables. Given $x<y$ such that $y-x > 4 a$, we set
$p_{y-x}(u) = p( (y-x) u)$ and 
$\psi_t = \gp(x+ (y-x) t)$
\begin{align*}
%\label{eq: bound GRR alpha xy} 
B_{ \frac{a}{y-x}}^{(\psi)} 
& = 
\int_0^1 \int_0^1 ds dt \;
\Psi \left( \frac{| \gp_t - \gp_s|}{p_{y-x}(|t-s|)}\right) {\bf 1}_{\{ |t - s| >  \frac{a}{|y-x|} \}}\\
& =  \frac{1}{|y-x|^2}
\int_x^y \int_x^y ds' dt' \;
\Psi \left( \frac{| \psi_{t'} - \psi_{s'}|}{p( |t'-s'|)}\right) {\bf 1}_{\{ |t' - s'| > a \}}
\leq \frac{B_a}{|y-x|^2} \, .
\end{align*}
Applying \eqref{eq: modulus alpha} to the function $\psi$, 
there exists $w,w' \in [0,2 a]$  
such that
\begin{align*}
\big| \psi_{1- \frac{w'}{y-x}} - \psi_{\frac{w}{y-x}} \big|
\leq 8 \int^1_{0}  
\Psi^{-1} \left( \frac{4B_{ \frac{a}{y-x}}^{(\psi)} }{u^2} \right) dp_{y-x}(u)
\leq 8 \int^1_{0} 
\Psi^{-1} \left( \frac{4 B_a }{|y-x|^2 u^2} \right) dp_{y-x}(u).
\end{align*}
Changing again variables, we get for some constant $C$ depending only on $\gamma,b$ 
\begin{align*}
\big| \gp_{y - w'} - \gp_{x+ w } \big|
\leq 8 \, (y-x)^{\frac{\gamma}{b} - \frac{2}{b}} \, \int^1_0 
\Psi^{-1} \left( \frac{4 B_a }{u^2} \right) dp(u)
\leq C B_a^{1/b} \; (y-x)^{\frac{\gamma -2 }{b}}.
\end{align*}
By bounding $\big| \gp_y  -  \gp_{y - w'}\big|$ and $\big| \gp_{x+ w } - \gp_y \big|$ by the supremum of  the local fluctuations in a time interval less than $2 a$, we conclude
to~\eqref{eq: modulus GRR cutoff}. The proposition is proved. \qed

\section{Spohn's formula for the covariance}
\label{appendix-spohn}

\setcounter{equation}{0}

%and written in the form
%\begin{align*}
%&\bbE \Big( \gz_t  (\psi) \gz_s  (\gp) \Big)
%=  \int dz \, \psi (z)\,  \left(\mathcal{U} (t, s) \,  \gp \, f(s) \right)(z)\\
%& \qquad\qquad  
% + \int_0^{s}  d \tau   \int dx dv dw 
% \, \cR \big( f(\tau) , f(\tau)\big)(x,v,w) \,   \left(\mathcal{U}^*(\theta_1,\tau) \psi \right)(x,v)  \, \left(\mathcal{U}^*(\theta_2, \tau) \gp \right)(x,w) \,.
%\end{align*} 
%with the recollision operator 
%\begin{equation}
%\label{eq: cRt}
%\cR (g,g)(x,v_1,v_2) :=\int_{{\mathbb S}^{d-1}} \left(  g( x,v'_1 )\, g( x,v_2')  - g( x,v_1 ) \, g(x,v_2 )\right)    \left((v_2- v_1)\cdot \omega \right)_+ d\omega\, ,
%\end{equation}

For the sake of completeness, we are going to show that the covariance $\hat \cC$ of the Ornstein-Uhlenbeck process computed in  \eqref{eq: covariance hors equilibre} coincides with 
the formula obtained by Spohn in \cite{S81} and recalled below in  \eqref{eq: Spohn Formula}.
Formula \eqref{eq: Spohn Formula} is striking as the recollision operator   $R^{1,2}$ emphasizes the contribution to  the covariance of the recollisions in the microscopic dynamics.
\begin{Prop}
\label{Prop: covariance OU}
Recall that  $\mathcal{U}(t,s)$ stands for the semi-group associated with the time dependent operator $\cL_\tau$  for~$\tau$ between times $s<t$.
Given two times $t \geq s$, there holds
\begin{align}
\label{eq: Spohn Formula}
& \cC(s,t,\gp,\psi) 
%= 
%\bbE \Big( \gz_{\theta_1}  (h_1) \gz_{\theta_2}  (h_2) \Big)
=  \int dz \,\mathcal{U}^* (t, s) \psi (z) \;     \gp (z) \, f(s,z) \\
& \qquad\qquad  
 + \int_0^{t}  d \tau   \int dx dv dw 
 \, R^{1,2} \left( f(\tau) , f(\tau)\right)(x,v,w) \,   \left(\mathcal{U}^*(t,\tau) \psi \right)(x,v)  \, 
 \left(\mathcal{U}^*(s, \tau) \gp \right)(x,w) 
\, , \nonumber
 \end{align} 
where the recollision operator $R^{1,2}$ is defined by
\begin{equation}
\label{eq: cRt}
R^{1,2}(g,g) (z_1, z_2) := \int \Big(  g (z_1') g (z_2') - g (z_1 ) g (z_2)\Big) d\mu_{z_1,z_2} ( \omega)  \, .
\end{equation}
\end{Prop}

\begin{proof}
The covariance at time $t = s=0$ is indeed given by 
\begin{align*}
\bbE \left( \gz_0 ( \gp) \gz_0 (\psi) \right) = \int dz \gp (z) f^0 \psi (z) 
= \int dz \gp (z) \psi (z) f (0,z)  \, .
\end{align*}
We will simply derive \eqref{eq: Spohn Formula} when $s=t$ and the case $s<t$ can be easily deduced.
The covariance~${\bf  Cov}_t$ introduced in  \eqref{eq: noise covariance}
can be rewritten in terms of the operator $\gS_t$  
\begin{equation}
\label{eq: operateur SIGMA}
\gS_t   \psi (z_1) 
 :=  -      \int d\mu_{z_1} ( z_2, \omega)\;
\Big[ f (t,z_1 ) f (t,z_2) +  f (t,z_1' ) f (t,z_2')\Big] \Delta \psi\, ,
\end{equation}
with the notation $d\mu_{z_1}$ as in \eqref{defdmuzi} and $\Delta \psi$ as in \eqref{eq: Delta psi}.
Indeed, one can check that for any functions $\gp, \psi$, the covariance 
can be recovered as follows
\begin{align*}
\int \gp \gS_t   \psi (z_1)  dz_1 & = -\frac12  \int d\mu(z_1,z_2, \omega)
\Big[ f (t,z_1 ) f (t,z_2) +  f (t,z_1' ) f (t,z_2')\Big] \Delta \psi (\gp(z_1) +\gp(z_2))
\\
%&= - \frac14  \int d\mu(z,z_2, \omega)
%\Big[ f (t,z_1 ) f (t,z_2) +  f (t,z_1' ) f (t,z_2')\Big] \Delta \psi (\gp(z_1) +\gp(z_2)-\gp(z_1')- \gp(z'_2))
%\\
&= \frac12 \int d\mu(z_1,z_2, \omega)f (t,z_1 ) f (t,z_2) ( \Delta \psi ) ( \Delta \gp ) 
= {\bf  Cov}_t ( \gp, \psi) \, .
\end{align*}
The covariance~$\hat \cC$ of the Ornstein-Uhlenbeck process computed in  \eqref{eq: covariance hors equilibre} reads 
\begin{equation}
\begin{aligned}
\label{eq: covariance at t}
\cC (t,t ,\gp, \psi) 
%&= 
% \bbE \left(  \gz_0  \big( \mathcal{U}^*(t,0) \psi  \big) \;  \gz_0  \big( \mathcal{U}^*(s,0) \gp  \big) \right)
%+   \int_0^s  du \;
%{\bf  Cov}_u \left(\big( \mathcal{U}^*(t,u) \psi \big) , \big( \mathcal{U}^*(s,u) \gp \big)  \right)\\
=  \int dz_1 \,  \mathcal{U}^*(t,0) \psi (z_1) \, f^0\,  \mathcal{U}^*(t,0) \gp (z_1) 
+ \int_0^t du  \int dz_1 \,    \gp  (z_1) \, \big[  \mathcal{U}(t,u)\;  \gS_u \; \mathcal{U}^*(t,u) \psi \big](z_1)  \, .
\end{aligned}
\end{equation}
The following identity is the key to identify \eqref{eq: covariance at t} and \eqref{eq: Spohn Formula}
\begin{equation}
\label{eq: linearized and covariance}
\begin{aligned}
\gS_t \gp (z_1) = - \Big( f_t \cL_t^*    +  \cL_t  f_t \Big) \gp(z_1)
+ \partial_t  f (t,z_1) \,  \gp(z_1) + \int    d z_2 \, R^{1,2} \big( f(t),f(t) \big) (z_1,z_2)   \gp(z_2) \, .
\end{aligned}
\end{equation}
Let us postpone for a while the proof of this identity and 
complete first the proof of \eqref{eq: Spohn Formula}.

Replacing the expression (\ref{eq: linearized and covariance}) of $\gS_u$ in the second line of \eqref{eq: covariance at t} and recalling that~$\mathcal{U}(t,t)\varphi =\varphi $, we get that 
\begin{align*}
\int_0^t du  \int dz_1 \,    \gp  (z_1) \, & \big[  \mathcal{U}(t,u)\;  \gS_u \; \mathcal{U}^*(t,u) \psi \big](z_1)  \\
= & \int_0^t du \, \int dz_1 \;    \gp  (z_1) \; \big[  \mathcal{U}(t,u)\; \Big( - \big( \cL_u  f_u +  f_u \cL_u^*\big)
+ \partial_u  f (u) \Big) \; \mathcal{U}^*(t,u) \psi \big](z_1)   \\
& + \int_0^t du \, \int d z_1 d z_2  \,     \mathcal{U}^* (t,u)\gp  (z_1) \; R^{1,2} \big( f(u),f(u) \big)  (z_1,z_2)  \; \mathcal{U}^*(t,u) \psi (z_2)\, .
\end{align*}
Noticing that the time derivative is given by
\begin{align*}
\partial_u \Big[ \mathcal{U}(t,u)\;   f_u  \; \mathcal{U}^*(t,u)  \Big]
=
\mathcal{U}(t,u)\; \Big( - \big( \cL_u  f_u +  f_u \cL_u^*\big)
+ \partial_u  f(u)  \Big) \; \mathcal{U}^*(t,u)   \,  ,
\end{align*}
we conclude that 
\begin{align*}
& \int_0^t du  \int dz_1 \,    \gp  (z_1) \, \big[  \mathcal{U}(t,u)\;  \gS_u \; \mathcal{U}^*(t,u) \psi \big](z_1)  
 =   \int d z_1 \;  \Big(  \gp  (z_1) \; f_t  \psi (z_1) -
   \gp  (z_1) \;   \mathcal{U}(t,0) f^0 \; \mathcal{U}^*(t,0) \psi  (z_1) \Big)  \\
& \qquad  \qquad  \qquad  \qquad + \int_0^t du \int dz_1 dz_2  \;  
\mathcal{U}^* (t,u)\gp  (z_1) \; R^{1,2} \big( f(u),f(u) \big)  (z_1,z_2)  \; \mathcal{U}^*(t,u) \psi (z_2)\,  .
\end{align*}
Finally the covariance \eqref{eq: covariance at t} reads
$$
\begin{aligned}
\cC (t,t, \gp, \psi) 
=  \int dz  \gp  (z) \, f_t  \psi (z)
%\\ & \qquad  
 + \int_0^t du   \int d z_1 dz_2  \,     \mathcal{U}^* (t,u)\gp  (z_1) \, R^{1,2} \big( f(u),f(u) \big)  (z_1,z_2)  \; \mathcal{U}^*(t,u) \psi (z_2) \,   .
\end{aligned}
$$
This completes the proof of Proposition \ref{Prop: covariance OU}.
It remains then to establish the identity \eqref{eq: linearized and covariance}.
Let us  write the decomposition  $\gS_t = \gS_t^+ + \gS_t^-$ with 
\begin{align*}
%\label{eq: noise covariance}
\gS_t^+   \psi (z_1) 
 :=  -      \int  d\mu_{z_1} ( z_2, \omega) 
f (t,z_1') f (t,z_2') \Delta \psi\, ,
\qquad
\gS_t^-   \psi (z_1) 
 :=  -      \int d\mu_{z_1} ( z_2, \omega) 
f (t,z_1 ) f (t,z_2) \Delta \psi\, 
.
\end{align*}
Recall that $\cL_T^*$ was computed in \eqref{eq: L*}. We get
\begin{align*}
%\label{eq: linearized Boltzmann}
f(t) \cL_t^*   \gp(z_1)  & =  f (t) \; v_1 \cdot \nabla  \gp (z_1) 
+ \int   d\mu_{z_1} ( z_2, \omega) 
f (t,z_1 ) f (t,z_2)  \Delta \gp
 =  f (t) \; v_1 \cdot \nabla  \gp (z_1) - \gS_t^- \gp (z_1)\,.
\end{align*}
and
\begin{align*}
%\label{eq: linearized Boltzmann}
\cL_t  f(t) \gp(z_1)  & = - v_1 \cdot \nabla [ f (t) \gp] (z_1) 
 + \int   d\mu_{z_1} ( z_2, \omega)  
\Big( f (t,z_1') f (t,z_2') \big( \gp(z_1') +  \gp(z_2')  \big) \\
& \qquad \qquad    \qquad \qquad     \qquad \qquad  
  - f (t,z_1) f (t,z_2)  \big(   \gp(z_2) +   \gp(z_1) \big) \Big) \\
& = - v_1 \cdot \nabla [ f (t) \gp] (z_1)   
+ \int  d\mu_{z_1} ( z_2, \omega)  \Big(f (t,z_1') f (t,z_2')  \Delta \gp\\
& \qquad \qquad    + \big[ f (t,z_1') f (t,z_2') - f (t,z_1) f (t,z_2)  \big]  \big(   \gp(z_1) +  \gp(z_2) \big) \Big)\\
%& = -  v \cdot \nabla [ f (t) \gp] (x,v) - \gS_t^+   \gp (x,v)  \\
%&  +\int  d \nu d w \,  \big( (v - w) \cdot \nu \big)_+ 
%  \big[ f (t,x,w') f (t,x,v')   -  f (t,x,v) f (t,x,w)  \big] \big( \gp(v) +  \gp(w)  \big) \\
& = - v_1 \cdot \nabla [ f (t) \gp] (z_1) - \gS_t^+ \gp (z_1) 
 + \int    d z_2 \, R^{1,2} \big( f(t),f(t) \big)  (z_1,z_2) \big( \gp(z_1) +  \gp(z_2)  \big)\,,
\end{align*} 
where we used the notation \eqref{eq: cRt}.
As a consequence, we get that 
$$
\begin{aligned}
%\label{eq: linearized Boltzmann}
f(t) \cL_t^*   \gp(z_1)  +  \cL_t  f(t) \gp(z_1)
& = - \gp  \; v_1 \cdot \nabla  f (t,z_1) -   \gS_t \gp (z_1)
%\\& \qquad
+ \int    d z_2 R^{1,2} \big( f(t),f(t) \big)  (z_1,z_2) \big( \gp(z_1) +  \gp(z_2)  \big)\,.
\end{aligned}
$$As $ f$ solves the Boltzmann equation, we have
$$
\partial_t  f (t,z_1) = -  v_1 \cdot \nabla  f (t,z_1) +  \int    d z_2  R^{1,2} \big( f(t),f(t) \big) (z_1,z_2)\,.
$$
This leads to further simplifications  as
\begin{align*}
f(t) \cL_t^*   \gp(z_1)  +  \cL_t  f(t) \gp(z_2)
& = \gp  \; \partial_t  f (t,z_1) -   \gS_t \gp (z_1)
+ \int    d z_2 \,  R^{1,2} \big( f(t),f(t) \big)  (z_1,z_2)   \gp(z_2)\,  ,
\end{align*}
thus \eqref{eq: linearized and covariance} holds. Proposition~\ref{Prop: covariance OU}
is proved.
\end{proof}

\chapter{Large deviations}
\label{LDP-chap} 
\setcounter{equation}{0}

This  chapter is devoted to the study of large deviations, and to the proof of Theorem \ref{thmLD}.
We are going to evaluate the probability of an atypical event, namely that the empirical measure 
remains close  to a probability density $\gp$  (which is different from the solution to the Boltzmann equation $f$) during a short time interval.

\bigskip

   It is well known, see e.g.~\cite{dembozeitouni, DV}, that the large deviation functional can be deduced from the 
exponential moments  by Legendre transform.  
We recall the definition~(\ref{I-def dynamics J}) of the limiting cumulant generating function \begin{align}
\label{I-def dynamics Jbis}
\cI (t,g) 
%& :=\frac{1}{\mu_\eps} \log 
%\bbE_\eps \left( \exp \Big(  \sum_{i =1}^\cN H \big( {\bf z}_i ([0,t] \big)  \Big)   \right)\\
 :=\gL_{[0,t]} ( e^{ g(t) -\int_0^t D_s g} )  =  \lim_{\mu_\eps \to \infty}  \gL^\eps_{[0,t]} (e^{ g(t) -\int_0^t D_s g}) \, ,\end{align} which is well defined (see Theorem \ref{cumulant-thm2}) in the set 
\begin{equation}
\label{eq: space Bbis}
\begin{aligned}
{\bB }_{\alpha} :=    \Big\{g \in C^1([0,T_\alpha] \times \D;{\mathbb C}) \
&\,   : \qquad 
 | g(t,z)| \leq  ( 1-{t\over 2T_\alpha} )( \alpha +\frac{\beta_0}8 |v|^2) \, ,\\
 &  \sup_{s \in [0,T_\alpha]}   | D_s g (s,z) |  \leq {1\over 2T_\alpha} (\alpha  +\frac{\beta_0}8 |v|^2)
\Big\}\, ,
\end{aligned}
\end{equation}
as long as~$t \leq T_{\alpha }$. 
The Legendre transform of~$\cI$ 
defines implicitly the large deviation functional (see~(\ref{eq: legendre transform GD}) below), and one of the goals of this chapter is to identify it with the following functional, as    previously conjectured by Rezakhanlou \cite{RezLNM} and Bouchet \cite{bouchet}: \begin{align}
\label{eq: fonc GD sup p 2}
\widehat \cF(t,\gp) 
:=  \widehat \cF (0,\gp_0) + \sup_p \left\{ \lA p,  D \gp\rA - \int_0^t \cH \big( \gp(s) ,p(s) \big)ds 
\right\},
\end{align}
where the supremum is taken over bounded measurable functions $p$ on~$[0,t]\times {\mathbb D}$,
%(cf. Equation (1.14) in \cite{Rez2}), 
and the Hamiltonian is given by
\begin{equation}
\label{eq: Hamiltonian 2}
\cH(\gp,p) := \frac{1}{2} \int  d\mu (z_1, z_2, \omega) 
\gp (z_1) \gp (z_2) \big( \exp \big( \Delta p \big) -1 \big) \, .
\end{equation}
We have denoted as in~(\ref{spacetimeduality}) the duality  on~$[0,t] \times   \bbD$    by 
%$\la \cdot,\cdot \ra$ (resp. $\lA \cdot,\cdot \rA$)
$$
  \lA \gp, \psi \rA := \int_0^{t} ds \int_\bbD dz  \; \gp(s,z) \; \psi (s,z)\, .
$$

\bigskip
We will be able to prove that $\widehat \cF$ describes indeed  the large deviations only for  a restricted class of functions, constructed as follows.
Consider the biased Boltzmann equation already introduced in \eqref{biased-Boltz} :
\begin{align}
\label{hamiltonian-traj1 detail}
& D  \varphi = \int 
\big( \varphi  (z') \varphi  (z_2') \exp( - \Delta   p) - \varphi  (z)\varphi (z_2) \exp( \Delta   p) \big) d\mu_z ( z_2, \omega)
\quad \text{with} \quad  
\varphi (0)   =f^0 e^{p (0)} \, ,  \end{align}
where~$p$ is a Lipschitz function in space and time, and define for any~$r,T>0$ the set
\begin{equation}
\label{eq: set R 2}
\begin{aligned}
\cR_{r,T} := \Big\{ \varphi :  [0,T] \times \bbD \mapsto \bbR^+ \, : \,  \varphi \hbox{ is a strong solution of   \eqref{hamiltonian-traj1 detail} on $[0,T]$ for some  }  p \\   
\hbox{ such that } \, \,  \|p\|_{W^{1,\infty}([0,T] \times {\mathbb D})} < r  \Big\} \,.
\end{aligned}
 \end{equation}
We shall  prove
the following theorem in  Section~\ref{I-section}~:
\begin{Thm}
\label{thm: F = ha F} 
For any~$  r>0$, there is~$ \alpha>0$ (depending on~$r, C_0$ and~$\beta_0$),
%, see Remark~{\rm\ref{rmk: choice of alpha}}, 
and a time~$T \in (0,T_\alpha ]$  (recalling that $T_\alpha$ is defined in Theorem~{\rm\ref{cumulant-thm2})} such that \begin{equation}
\label{eq: F = hat F}
\forall \gp \in \cR_{r,T} \, , \quad \forall t \leq T \, , \qquad 
\widehat \cF (t,\gp) =  \cF(t, \gp)  \,,
\end{equation}
where~$\cF$ is the Legendre transform of~$\cI$
\begin{equation}
\label{eq: legendre transform GD}
\cF (t, \gp) :=  \sup_{ g \in {\bB }_{  \alpha}} \Big\{   - \lA \varphi, D g \rA + \la \varphi (t),  g (t) \ra  
-  \cI (t, g) \Big\}\,  .\end{equation}
\end{Thm}
  Building on Theorem \ref{cumulant-thm2} and 
standard methods of the  large deviation theory \cite{dembozeitouni},  we shall then prove
the following large deviation principle in Section~\ref{LD-sec}.
\begin{Thm} 
\label{thmLDbis}
Consider a system of hard spheres initially distributed according to  the grand canonical measure {\rm(\ref{eq: initial measure})} where $f^0$ satisfies~{\rm(\ref{lipschitz})}. 
Let~$r>0$ be fixed,  and the associate parameters~$ \alpha>0$ and~$T>0$ of Theorem~{\rm\ref{thm: F = ha F}}.  In the Boltzmann-Grad limit $\mu_\eps \to \infty$, the empirical measure satisfies the following
large deviation estimates~:
% large deviation principle
\begin{itemize}
\item For any closed set ${\bf F} \subset D([0,T], \cM (\bbD))$,  
\begin{align}
\label{eq: large deviations upper bound}
\limsup_{\mu_\eps \to \infty} \frac{1}{\mu_\eps} 
\log \bbP_\eps \left( \pi^\gep\in {\bf F} \right) \leq - \inf_{\gp \in \bf F}  \cF(T,\gp)\, .
\end{align}
\item For any open set ${\bf O} \subset D([0,T], \cM (\bbD))$, 
\begin{align}
\label{eq: large deviations lower bound}
\liminf_{\mu_\eps \to \infty} \frac{1}{\mu_\eps} 
\log \bbP_\eps \left( \pi^\gep \in {\bf O} \right) \geq - \inf_{\gp \in {\bf O} \cap \cR_{r,T}}\cF(T,\gp)\,.
\end{align}
\end{itemize}
\end{Thm}

\section{Identification of the large deviation functional} \label{I-section}

In this section, we prove Theorem \ref{thm: F = ha F}. From now on, we fix a real number~$r>0$.
The main step  of the proof will be to provide a more explicit formula for $\cI (t, g)$   by using that the Hamilton-Jacobi equation~\eqref{HJ-eq} has a unique solution.

\subsection{Mild solutions of the Hamilton-Jacobi equation}

For any~$\alpha>0$,  fix a function $g$ in~$\bB_{\alpha}$.
At the formal level, the Hamilton-Jacobi equation  \eqref{HJ-eq} states
that for any~$t \in [0,T_{\alpha } ]$
\begin{equation}
\label{HJ-eq 2} 
\d_t \cI (t,g) = \cH \Big( {\d \cI (t, g) \over \d g(t)}, g(t)\Big)
\quad \text{with} \quad 
\cH \Big( \varphi, p\Big) = \frac12  \int  \varphi (z_1) \varphi (z_2)\Big(e ^{\Delta p} - 1 \Big)   d\mu (z_1, z_2, \omega)\, ,
\end{equation}
with initial condition 
\begin{equation}
\label{eq: initial condition}
\cI (0, g) =   \big \la f^0, (e^{g(0)}-1) \big \ra  \,.
\end{equation}
As noticed before, all the limiting cumulants at time 0,  except the first one, equal zero so that 
$\cI (0, g)$ coincides with the exponential moment of independent variables distributed according to $f^0$ and tilted by the function $g(0)$.

\medskip

We would like to use a method of characteristics to obtain a mild solution~$\hat \cI(t,g)$ of~(\ref{HJ-eq 2})-(\ref{eq: initial condition}).
Given~$t $ in~$[0,T_{\alpha }]$, define the  Hamiltonian system on the time interval $[0,t]$ 
\begin{align}
\label{hamiltonian-traj1}
& D_s  \varphi_t = {\d \cH \over \d p } ( \varphi_t, p_t)\, , 
\quad  \text{with} \quad    
\varphi_t(0)   =f^0 e^{p_t (0)}\,  ,\\
\label{hamiltonian-traj2}
& D_s  (p_t -g) =-  {\d \cH \over \d \varphi } ( \varphi_t, p_t)\, , 
\quad  \text{with} \quad   
p_t(t)   = g(t)\,. 
\end{align}
The subscript $t$ stresses the fact that the functions $\varphi_t (s), p_t(s)$ depend on~$t$.
As customary, the boundary conditions are prescribed in terms of the initial time (for \eqref{hamiltonian-traj1}) and the final time $t$ (for \eqref{hamiltonian-traj2}).  The condition \eqref{hamiltonian-traj1} is identical to the biased Boltzmann equation 
\eqref{hamiltonian-traj1 detail} used to define $\cR_{r,T}$. 
Note that~\eqref{hamiltonian-traj2} reads
\begin{align}
\label{hamiltonian-traj2 detail}
& D_s  (p_t -g) =- \int 
\varphi_t (z_2) \big( \exp( \Delta   p_t) -1  \big) d\mu_z ( z_2, \omega)
\quad \text{with} \quad  
p_t(t)   = g(t)\,.
\end{align}
The local well-posedness of the Hamiltonian equations~\eqref{hamiltonian-traj1}-\eqref{hamiltonian-traj2}  will be obtained by a Cauchy-Kovalevskaya argument after recasting the system in more symmetric variables (see  Section~\ref{symplectic} and Appendix~\ref{EL-Cauchy}).
%in Proposition \ref{Lem: definition hat I}. 
%\begin{Prop}
%\label{Lem: well-posedness phi pI}
%For any~$\alpha>0$, there exists  a time~$ T_{\alpha }^{\rm\tiny mH}   \in (0,T_{\alpha }]$, depending only on~$\alpha$, $\beta_0$ and~$C_0$, such that the following holds. Given~$t $ in~$ [0,  T_{\alpha }^{\rm\tiny mH}]$ and~$g$ in~$  \bB_{\alpha}$,  there is a unique solution 
% $(\varphi_t , p_t)$ to the system of equations~{\rm\eqref{hamiltonian-traj1}}-{\rm\eqref{hamiltonian-traj2}} on~$[0,t]$ such that 
%\begin{equation}\label{estimatephi_tp_t}
%0<  \sup_{s \in [0,t] }  \varphi_t (s,x,v)  \leq \bar C \exp \Big(- \frac {3\beta_0}8 |v|^2\Big) \, , 
%\quad
%\sup_{s \in [0,t] }  | p_t (s,x,v) |  \leq \log \bar C +  \frac {5\beta_0}{16} |v|^2
%\end{equation}
%for    some constant~$\bar C >1$.
%\end{Prop}

\medskip

Let us now explain how  the functions $\varphi_t , p_t $ can be used to build a more explicit 
representation of the functional $\cI$.
For $g \in \bbB_{\alpha}$ and $(\varphi_t, p_t)$ solution to~\eqref{hamiltonian-traj1}-\eqref{hamiltonian-traj2}, define the action associated with the  Hamiltonian system  \eqref{hamiltonian-traj1}-\eqref{hamiltonian-traj2} by
\begin{align}
\label{hatI-def}
\widehat  \cI (t,g)   := \la f^0, (e^{p_t(0)}-1) \ra +   \lA D_s  (p_t-g) , \varphi _t \rA + \int_0^t \cH (\varphi_t(s) , p_t(s)) ds \, .
\end{align}

\begin{Prop}
\label{action-lem}
Let~$\alpha>0$ and $g \in \bbB_{\alpha}$. Assume that the   Hamiltonian system \eqref{hamiltonian-traj1}-\eqref{hamiltonian-traj2} admits a unique continuous solution  on~$[0,T]$
for any forcing~$\tilde g$ in a neighborhood of~$g$ in~$\bbB_{\alpha}$.
Denote by~$(\varphi_t , p_t)$  the solution  on~$[0,T]$ associated with~$g$. Then   the functional~$\widehat \cI$ defined by~{\rm(\ref{hatI-def})} satisfies the Hamilton-Jacobi equation \eqref{HJ-eq} on~$[0, T]$ and the following identities:
\begin{equation}
\label{eq: derivees hat I}
\frac{\partial \widehat \cI}{\partial g(t)} (t,g)= \varphi_t (t)\, , \qquad \frac{\partial \widehat \cI}{\partial D g} (t,g) = - \varphi_t  \, .
\end{equation}
\end{Prop} 

\begin{proof}
Let us first compute the time derivative of $\widehat \cI (t,g)$ for a fixed function $g$
\begin{align}
\label{eq: derivee temps hat cI} 
\d_t \widehat  \cI (t,g) = & \la f^0, e^{p_t(0)} \delta p_t (0)\ra 
+ \la D_t (p_t-g)(t) ,  \varphi _t(t) \ra + \cH (\varphi_t(t) , p_t(t)) \\
& +   \lA  D_s \delta  p_t  ,  \varphi _t\rA +  \lA  D_s (p_t-g) , \delta \varphi _t \rA 
\nonumber \\
& +   \lA  \delta \varphi_t,  {\d\cH\over \d\varphi}  (\varphi_t , p_t )  \rA 
 +  \lA \delta p_t , {\d\cH\over \d p}  (\varphi_t , p_t ) \rA\, , \nonumber
\end{align}
where  $\delta$ stands for the derivative with respect to the variations of the final time; for example
\begin{equation*}
\forall s \leq t, \quad \delta p_t (s) = \lim_{u \to 0} \frac{p_{t+u} (s) - p_t (s)}{u}\,\cdotp
\end{equation*}
In particular, we will prove that
\begin{equation}
\label{eq: variation p}
\delta p_t (t) = - \partial_t ( p_t (t) - g(t) ) = - D_t ( p_t (t) - g(t) ) \, ,
\end{equation}
where the time derivative is only with respect to the argument $s \mapsto p_t (s) - g(s)$.
The first part of~\eqref{eq: variation p} follows by  
\begin{align*}
%\label{eq: variation p}
\frac{p_{t+u} (t) - p_t (t)}{u} 
& = \frac{p_{t+u} (t) -p_{t+u} (t+u) + p_{t+u} (t+u) - p_t (t)}{u}\\
& = \frac{p_{t+u} (t) -p_{t+u} (t+u) + g (t+u) - g (t)}{u}
\xrightarrow[u \to 0]{} - \partial_t ( p_t (t) - g(t) )\, ,
\end{align*}
thanks to the boundary condition $(p_s -g)(s)=0$. Using once again the boundary condition, we deduce that 
$v \cdot \nabla_x(p_t -g)(t)=0$ so that the second equality in \eqref{eq: variation p} is proved.

Integrating by parts the first term in the second line of \eqref{eq: derivee temps hat cI}, we get  
\begin{align*}
\lA  D_s \delta  p_t  ,  \varphi _t \rA 
& = - \lA  \delta  p_t   , D_s \varphi _t \rA + \la \delta  p_t (t)  , \varphi _t(t) \ra  
- \la \delta p_t (0)  , \varphi _t(0) \ra \\
& = - \lA  \delta p_t  , D_s \varphi _t \rA   -  \la  D_t ( p_t (t) - g(t) )  , \varphi _t(t) \ra
- \la \delta p_t (0) ,f^0 e^{p_t (0)} \ra\, ,
\end{align*}
where we used the boundary conditions $\varphi_t(0)   = f^0 e^{p_t (0)}$ 
and the identity \eqref{eq: variation p}.
 From the equations~(\ref{hamiltonian-traj1})-(\ref{hamiltonian-traj2}), we deduce  that the integral contributions of $\delta p_t$ and $\delta \varphi_t$ vanish. 
Therefore $\widehat \cI$ satisfies the Hamilton-Jacobi equation
\begin{equation}
\label{dt-hatI 2}
\d_t \widehat  \cI (t,g)= \cH (\varphi_t(t) , p_t(t))\,.
\end{equation}
The mild form \eqref{HJ-eq} is then a consequence of identities~(\ref{eq: derivees hat I}) by time integration.

\medskip

Let us now fix $t$ and differentiate (\ref{hatI-def}) with respect to $g(t)$ and $D_s g$.
The corresponding variations~$\delta g(t)$  and $\delta D_s g$ are independent.
We get
\begin{align*}
\d  \widehat  \cI   (t,g) 
= & \la f^0, e^{p_t(0)} \delta p_t (0)\ra 
 +   \lA  D_s \delta p_t  ,  \varphi _t\rA 
 -    \lA  \delta D_s  g ,  \varphi _t\rA 
 +  \lA  D_s (p_t-g) , \delta \varphi _t \rA \\
&  +   \lA  \delta \varphi_t,  {\d\cH\over \d\varphi}  (\varphi_t , p_t )  \rA 
 +  \lA \delta p_t , {\d\cH\over \d p} (\varphi_t , p_t ) \rA\, .
\end{align*}
%Recall that $g$ is characterized by $D_s g$ and its final condition $g(t)$,
%thus $D_s \delta g  = 0$ since  $D_s g$ is independent from $g(t)$.
By integration by parts and using the boundary conditions $\varphi_t(0)   = f^0 e^{p_t (0)}$ and $p_t (t)=g(t)$, we obtain  
\begin{align*}
\lA  D_s \delta p_t  ,  \varphi _t \rA 
& = - \lA  \delta  p_t   , D_s \varphi _t \rA + \la \delta p_t (t)  , \varphi _t(t) \ra  
- \la \delta p_t (0)  , \varphi _t(0) \ra \\
& = - \lA  \delta  p_t   , D_s \varphi _t \rA + \la \delta g (t)  , \varphi _t(t) \ra  
- \la \delta p_t (0) ,f^0 e^{p_t (0)} \ra \,. 
\end{align*}
Thus
\begin{align*}
\d  \widehat  \cI   (t,g) 
= & \la f^0, e^{p_t(0)} \delta p_t (0)\ra  -  \lA \delta D_s  g ,  \varphi _t\rA 
 - \lA  \delta  p_t   , D_s \varphi _t \rA + \la \delta g (t)  , \varphi _t(t) \ra  
- \la \delta p_t (0) ,f^0 e^{p_t (0)} \ra \\
& +  \lA  D_s (p_t-g) , \delta \varphi _t \rA 
  +   \lA  \delta \varphi_t,  {\d\cH\over \d\varphi}  (\varphi_t , p_t )  \rA 
 +  \lA \delta p_t , {\d\cH\over \d p} (\varphi_t , p_t ) \rA \,.
\end{align*}
Combining this identity and equations (\ref{hamiltonian-traj1})-(\ref{hamiltonian-traj2}) to simplify the Hamiltonian contribution, this completes the statement  \eqref{eq: derivees hat I}
$$ 
\partial \widehat  \cI  (t,g) =  \la \delta g (t)  , \varphi _t(t) \ra  -  \lA \delta D_s  g ,  \varphi _t\rA     \,.
$$
Proposition~\ref{action-lem}
 is proved. \end{proof}

 \medskip
  As a consequence of Theorem \ref{HJ-prop} page~\pageref{HJ-prop} and Proposition \ref{action-lem},
 the functionals $\cI, \widehat \cI$ are both solutions of the Hamilton-Jacobi equation 
\eqref{HJ-eq 2} and we are going to deduce that they coincide on some short time interval. The proof of the following result   is postponed to Section \ref{sec: I = hat I} as this requires to reparametrize the Hamiltonian variables in order to show the uniqueness of the Hamilton-Jacobi equation.

\begin{Prop}
\label{prop: I = hat I}
Let~$\alpha>0$ be given. There exists a time~$T^\star_\alpha>0$ such that 
the functional~$\widehat \cI$ is well defined on~$[0,T^\star_\alpha] \times \bB_{\alpha}$ and
the functionals $\cI, \widehat \cI$ coincide on 
$[0,T^\star_\alpha] \times \bB_{\alpha}$:
$$
\cI (t,g) = \widehat \cI (t,g) \qquad \text{ for any} \quad t \leq T^\star_\alpha\, , \  g\in \bB_{\alpha}\,. 
$$
\end{Prop} 
%Thanks to the explicit form \eqref{hatI-def} of $\widehat \cI$,  
%Proposition \ref{prop: I = hat I} provides a characterisation of $\cI$ on $[0,T^\star]$. 
%This will be key to derive Theorem \ref{thm: F = ha F}, i.e. 
%that $\cF = \widehat \cF$ on the set $\cR_{r,T}$.
%

%\begin{proof}[Proof of Proposition \ref{prop: I = hat I}]
%Proposition \ref{Lem: definition hat I} will establish that the system  \eqref{hamiltonian-traj1}-\eqref{hamiltonian-traj2} is well-posed, from this we deduce that the functional $\widehat \cI$ is well-defined.
%As $\cI, \widehat \cI$ are both solutions of the Hamilton-Jacobi equation 
%\eqref{HJ-eq 2}, the identification follows from :
%\begin{itemize}
%\item the uniqueness of the Hamilton-Jacobi equation \eqref{HJ-eq 2} derived in Proposition \ref{statementuniqueness},
%\item  the regularity estimates on $\cI$ and $\widehat \cI$ obtained in Propositions \ref{prop: analyticity J}  and 
%\ref{Lem: definition hat I}.
%\end{itemize}
%The proof of these propositions   are postponed to Section \ref{sec: I = hat I} as this requires to introduce a new set of variables.
%\end{proof}

\subsection{Identification of the Legendre transform $\cF$}
\label{sec: F = ha F}

In this section, we prove  Theorem \ref{thm: F = ha F}.
 Fix  a  function   $\bar \gp$ satisfying the biased Boltzmann equation 
\eqref{hamiltonian-traj1 detail}  for some~$\bar p$ 
 such that 
\begin{equation}
\label{eq: estimate bar p}
\|\bar p\|_{W^{1,\infty} ([0,T_0] \times \bbD) } < r \, . 
\end{equation}
Noticing that
$$
{\d \cH \over \d p } ( \bar \gp , \bar p) = \int 
\big(\bar \varphi  (z') \bar\varphi  (z_2') \exp( - \Delta   \bar p) -\bar  \varphi  (z)\bar \varphi (z_2) \exp( \Delta   \bar p) \big) d\mu_z ( z_2, \omega)\, , 
$$this biased Boltzmann equation can be rewritten in the more compact form~(\ref{hamiltonian-traj1}) which we recall
\begin{align}
\label{eq: sol R}
D_t  \bar \gp = {\d \cH \over \d p } ( \bar \gp , \bar p)\, , 
\quad  \text{with} \quad    
\bar \gp (0)   =  f^0  e^{\bar p(0)} .
\end{align}
By Appendix~\ref{boltz-Cauchy} (see \eqref{Mmax-appendix}), Equation~(\ref{eq: sol R}) has a  unique solution on~$[0,  T_0 e^{-5r}] $ such that 
\begin{equation}
\label{eq: norm strong solution}
\sup_{t \in [0,T_0 e^{-5r}]} \Big  \|  \bar \gp  (t)  \, \exp \Big( \frac{\beta_0}{4} \, |v|^2  \Big)  \Big \|_\infty \leq 4 C_0 e^r \,.
\end{equation}
%A  slight modification of the argument in Appendix~\ref{boltz-Cauchy} shows that for any~$\beta_0/10 \leq \beta < \beta' \leq \beta_0/2$
%\begin{align*}
% \left  \|  \left| {\d \cH \over \d p } ( \bar \gp , \bar p) \right|  \, \exp \Big( { \beta }  \, |v|^2  \Big)  \right \|_\infty 
%  & \leq C \Big  \|  \bar \gp  \, \exp \Big( { \beta'}  \, |v|^2  \Big)  \Big \|_\infty ^2  \exp( 4 r 
%)  \frac{{\beta}_0 }{\beta'-\beta}{\beta}_0 ^{-\frac {(d+1)}2} \,.
%\end{align*}
%Since on the other hand
%$$
% \left  \|  f^0  e^{\bar p(0)} \exp \Big( \frac{ \beta_0 }2 \, |v|^2  \Big)  \right \|_\infty 
%\leq C_0e^r\, , 
%$$
%Theorem~\ref{nishidatheorem}  shows that Equation~(\ref{eq: sol R}) has a  solution on~$[0,  T_0 e^{-5r}] $ such that 
%\begin{equation}
%\label{eq: norm strong solution}
%\sup_{t \in [0,T_0 e^{-5r}]} \Big  \|  \bar \gp  (t)  \, \exp \Big( \frac{\beta_0}{10} \, |v|^2  \Big)  \Big \|_\infty \leq 2C_0 e^r \,.
%\end{equation}
We then set
$$
T := \min (T_0 e^{-5r},T^\star_\alpha)\, ,
$$
with $T^\star_\alpha$ as in Proposition \ref{prop: I = hat I}.
Note   that $\bar \varphi$ is smooth, non-negative and that the conservation of mass, momentum and energy are satisfied~:
\begin{equation}
\label{eq: conservation bis}
\la D_s \bar \varphi, 1\ra = \la D_s \bar \varphi, v_i \ra = \la D_s \bar \varphi, |v|^2 \ra = 0\,.
\end{equation}
\begin{Rmk}
\label{rmk: regularite}
It has been shown in \cite{Heydecker21, BBBC22} that  the functional $\widehat \cF$ is not relevant to describe the large deviations of some functions $\varphi$ which are weak solutions of the homogeneous Boltzmann equation but do not conserve energy. Such functions are much more irregular than those in $\cR_{r,T}$ (see e.g.~{\rm\eqref{eq: conservation bis}}), thus the counterexample in  \cite{Heydecker21}  does not contradict Theorem~{\rm\ref{thm: F = ha F}}.
\end{Rmk}

Equation \eqref{eq: sol R} implies that $\bar p$ is a critical point of the variational problem~(\ref{eq: fonc GD sup p 2}) on~$[0,T]$, which we recall:
\begin{align*}
\widehat \cF(t, \bar \gp) 
:=  \widehat \cF \big( 0, \bar \gp (0) \big) 
+ \sup_p \left\{ \lA p,  D_s \bar \gp \rA - \int_0^t \cH \big( \bar \gp(s) , p(s) \big)ds 
\right\},
\end{align*}
where the supremum is taken over bounded    $p$ on~$[0,t]\times {\mathbb D}$.
Indeed since $\bar \varphi \geq 0$, the function $p \mapsto \cH( \bar \varphi, p) $ is convex and one can check that for any bounded $p$ and for all~$t \in [0,T]$, 
\begin{align*}
 \lA p,  D_s \bar \gp \rA - \int_0^t \cH \big( \bar \gp(s) ,p(s) \big)ds  
&\leq \lA \bar p,  D_s\bar \gp\rA - \int_0^t \cH \big( \bar \gp(s) , \bar p(s) \big)ds
+ \lA p- \bar p ,  D_s \bar \gp - {\d \cH\over \d p}  \big( \bar \gp ,\bar p \big)\rA \\
&\leq  \lA \bar p,  D_s\bar \gp\rA - \int_0^t \cH \big( \bar \gp(s) , \bar p(s) \big)ds\, ,
\end{align*}
where the last term in the first inequality is equal to 0 thanks to  \eqref{eq: sol R}
and the fact that $p,\bar p$  are bounded.
The previous inequality implies that the supremum  $\widehat \cF$ is reached at $\bar p$:
%the critical points $p$ defined by $ \Delta p = \Delta \bar p$.
\begin{align}
\label{eq:  hat F bar p}
\forall t \in [0,T]\, , \quad \widehat \cF(t, \bar \gp) 
=  \widehat  \cF \big( 0, \bar \gp (0) \big) 
+  \lA \bar p,  D_s \bar \gp \rA - \int_0^t \cH \big( \bar \gp(s) , \bar p(s) \big)ds \, .
\end{align}

\medskip

We turn now to the analysis of $\cF (t, \bar \gp)$. 
By the identification of $\cI$ and $\widehat \cI$ in Proposition \ref{prop: I = hat I},
the variational problem \eqref{eq: legendre transform GD} can be rewritten, for all~$t \leq T$,
\begin{equation}
\label{eq: legendre transform GD bis}
\cF (t, \bar \gp) :=  
\sup_{ g \in \bbB_{\alpha} }\Big\{   - \lA \bar \varphi, D_s g \rA + \la \bar \varphi(t),  g(t) \ra  - \widehat \cI(t,g) \Big\} \, .
\end{equation}
 Let us first build a critical point $\bar g$ for this variational problem. 
Given~$\bar p$ satistfying~(\ref{eq: estimate bar p}) and~$ \bar \gp$ solving~(\ref{eq: sol R}), we define $\bar g$ as the solution of 
\begin{equation}
\label{eq: bar g}
D_s \bar g = D_s  \bar p  +  {\d \cH \over \d \varphi } ( \bar \gp , \bar p)
\quad  \text{with} \quad   
\bar g(t) = \bar p(t) \,. 
\end{equation}
 By assumption \eqref{eq: estimate bar p} on $\bar p$, we get 
$$
\big| D_s  \bar p \big| \leq (1+|v|) \|\bar p\|_{W^{1,\infty}}  \leq  (1+|v|)r
$$
and there holds
\begin{align*}
\left| {\d \cH \over \d \varphi } ( \bar \gp , \bar p) \right|
& = \left| \int \bar \gp (z_2) \big( \exp( \Delta  \bar p ) -1  \big) d\mu_z ( z_2, \omega) \right|\\
& \leq \left| \int  \bar \gp (z_2)  \big|  \Delta \bar p \big|  \;  \exp \big( \big| \Delta  \bar p \big| \big)  d\mu_z ( z_2, \omega) \right|\\
& \leq C C_0  r  \exp( 5 r 
)  {\beta}_0 ^{-\frac d2} \Big(|v| + \beta_0^{-\frac12} \Big) \,,
\end{align*}
where we used the weighted estimate \eqref{eq: norm strong solution}
on $\bar \varphi$ to control the divergence of the cross section. 
The constant~$C$ is universal and  depends only on the dimension.
Thus we deduce from \eqref{eq: bar g} that 
\begin{equation}
\label{eq: tuning r}
\big| D_s \bar g (s,x,v) \big| 
 \leq C C_0 r  \exp( 5 r )  {\beta}_0 ^{-\frac d2} \Big(|v|  +\beta_0^{-\frac12} \Big) + (1 + |v|)r\quad  \text{and}\quad\big| \bar g(t,x,v) \big|  \leq  r \,. 
\end{equation}
Given $r >0$ which quantifies the size of the observables in the large deviation principle, 
the parameter~$\alpha$ is then  chosen large enough by using the estimates \eqref{eq: tuning r} 
so that $\bar g$ belongs to $\bbB_{\alpha}$. Note that the larger~$ \alpha$ is chosen, the smaller~$T_{\alpha}= c\, e^{-\alpha} \beta_0^{(d+1)/2}/C_0$   will be, and hence also the time of validity of Theorem~{\rm\ref{thm: F = ha F}}.
%\begin{Rmk}
%\label{rmk: choice of alpha}
%According to~{\rm(\ref{eq: tuning r})} and the definition~{\rm(\ref{eq: fonc GD sup p 2})} of~$\bbB_{\alpha}$, and recalling that~$T_\alpha = c\, e^{-\alpha} \beta_0^{(d+1)/2}/C_0$, the parameter~$\alpha$ must satisfy, for some universal constant~$C$,
%$$
%C C_0^2 r  \exp( 5 r )  \Big(1+\beta_0  |v|^2  \Big) + C{\beta}_0 ^{\frac {d-1}2} \Big(\beta_0 +\beta_0  |v|^2  \Big) r \leq  e^{\alpha}  (\alpha  +\frac{\beta_0}8 |v|^2)\, .
%$$
%\end{Rmk}

\medskip
 By construction~$\bar \gp$ belongs to~${\mathcal R}_{r,T}$ and~$(\bar \gp, \bar p, \bar g)$ satisfy the Hamiltonian system
\eqref{hamiltonian-traj1}-\eqref{hamiltonian-traj2}   on~$[0,T]$, so  from 
Proposition \ref{action-lem}, the following holds
$$
\frac{\partial \widehat \cI}{\partial g(t)} (t,\bar g)= \bar \gp (t)  \,, 
\qquad \frac{\partial \widehat \cI}{\partial D  g} (t, \bar g) = - \bar \gp   \,.
$$ 
This implies that $\bar g$ is a critical point of
\begin{equation}
\label{functional}
 (g(t), D_s g)  \mapsto  - \lA \bar \varphi, D_s g \rA + \la \bar \varphi(t),  g(t) \ra  
- \widehat \cI(t,g)  \, .
\end{equation}
Since $\widehat \cI(t,g) = \cI(t,g) = \Lambda_{[0,t]} \big( e^{g(t) - \int_0^t D_sg}\big)$ is strictly convex with respect to $(g(t), D  g)$, the supremum in  \eqref{eq: legendre transform GD bis} is reached at $\bar g$.
Thus 
\begin{align}
\label{eq: minimum atteint F}
\cF (t, \bar \gp) & =   \la \bar \varphi(t), \bar g(t) \ra  - \lA \bar \varphi, D_s \bar g \rA 
  - \widehat \cI(t, \bar g) \\
& =   
 \la \bar \gp (t), \bar g(t) \ra  - \la f^0, (e^{\bar p(0)}-1) \ra -   \lA D_s  \bar p , \bar \gp\rA - \int_0^t \cH (\bar \gp (s) , \bar p(s)) ds  \,  , \nonumber
\end{align}
where $\widehat  \cI (t,\bar g)$ is replaced by its explicit representation \eqref{hatI-def}  in the second line.
As $\bar g(t) = \bar p(t)$ and~$\bar \gp (0)  = f^0 e^{\bar p(0)}$, an integration by parts leads to 
\begin{align*}
\cF (t, \bar \gp) =   
 \la \bar \gp (0), \bar p(0) \ra  + \la f^0  - \bar \gp (0) \ra +   \lA   \bar p , D_s \bar \gp\rA - \int_0^t \cH (\bar \gp (s) , \bar p(s)) ds   \, .
\end{align*}
As the initial large deviation functional is given by  
\begin{equation*}
\widehat \cF (0,\gp(0)) =  \Big \la \gp_0 \log \left( \frac{\gp_0}{f^0} \right)
-  \gp_0 +f^0 \Big\ra
\end{equation*}
and $\widehat \cF (t, \bar \gp)$ by \eqref{eq:  hat F bar p}, this shows that $\cF (t, \bar \gp) =  \widehat \cF (t, \bar \gp)$ on~$[0,T]$. The proof of Theorem \ref{thm: F = ha F} is complete, provided that we can construct solutions of the Hamiltonian equations to define $\hat \cI$, and prove the uniqueness of solutions to the Hamilton-Jacobi equation. 
\qed

\section{Symmetrization of the Hamiltonian system: proof of $\cI  = \widehat \cI$} \label{symplectic}
 
This section is devoted to the proof of Proposition {\rm\ref{prop: I = hat I}.

In order to prove the two missing statements, i.e. the local well-posedness of the Hamiltonian equations~\eqref{hamiltonian-traj1}-\eqref{hamiltonian-traj2}, and the uniqueness for the Hamilton-Jacobi equation~\eqref{HJ-eq}, the idea is to apply Theorem~\ref{nishidatheorem}, which requires to define  suitable  functional settings in which we have  loss continuity estimates of the type (\ref{L*-est original}). 

To do so, it will  be convenient to reparametrize the Hamiltonian variables  
 and instead of $p,\varphi$ to consider 
\begin{equation}
\label{new-variables} (\psi, \eta) := (\varphi e^{-p}, e^p)\,.
\end{equation}
In these new variables, the Hamiltonian \eqref{eq: Hamiltonian 2} is rewritten in a more symmetric form
\begin{align}
\label{eq: Hamiltonien bis}
\cH'(\psi, \eta) & := \frac12 \int \psi(z_1)  \psi(z_2) \big( \eta(z'_1) \eta(z'_2) - \eta(z_1) \eta(z_2) \big) 
\, d\mu(z_1, z_2, \omega)\\
& = - \frac14 \int \big( \psi(z_1')  \psi(z_2')  - \psi(z_1)  \psi(z_2)  \big) \big( \eta(z'_1) \eta(z'_2) - \eta(z_1) \eta(z_2) \big) \, d\mu(z_1, z_2, \omega). \nonumber
\end{align}

\subsection{Uniqueness for the Hamilton-Jacobi equation}
\label{sec: I = hat I}

Consistently we  characterize $g$ using the variables 
 $\gamma(s) := e^{g(s)}$ and $\phi (s) := D_s g(s)$ which are related by the continuity equation
\begin{equation}
\label{eq: continuiy equation}
\forall s \leq t\, , \qquad 
D_s \gamma(s) - \phi(s) \gamma(s) = 0\,.  
\end{equation}
The functional $\cI (t,g)$ becomes then
\begin{align}
\label{I-def dynamics new J}
 \cJ(t, \phi, \gamma  ) :=  \gL_{[0,t]} \left( \gamma  e^{ -\int_0^t \phi} \right)
\end{align}
and the Hamilton-Jacobi equation (\ref{HJ-eq}) can be rewritten in terms of the new Hamiltonian $\cH'$
%\begin{equation}
%\label{HJ-new}
%\d_t \cJ(t, \phi, \gamma(t) ) = \cH' \left({\d \cJ \over \d \gamma}, \gamma(t) \right)
%\end{equation}
\begin{equation}
\label{HJ-Appendix form}
 \cJ (t)   = \cJ (0) +   \int_0^t  F(\cJ(s)) \, ds \,  ,
\end{equation}
when~$\phi$ and~$\gamma(t)$ are related by~(\ref{eq: continuiy equation}) and where
$$
\begin{aligned}
F\big(\cJ(s,\phi, \gamma (s))\big ) & :=   \cH' \left({\d \cJ \over \d \gamma} (\phi,\gamma(s)) , \gamma(s) \right)\\
& =  \frac12  \int {\d \cJ\over \d \gamma} (\phi,\gamma(s))(z_1)  {\d   \cJ  \over \d \gamma }(\phi,\gamma(s) )(z_2)\Big(\gamma  (s,z'_1)  \gamma  (s,z'_2)\  - \gamma  (s,z_1) \gamma  (s,z_2) \Big)   d\mu (z_1, z_2, \omega)\, ,
\end{aligned}
$$
with initial condition \eqref{eq: initial condition} 
\begin{equation}
\label{eq: initial condition 2}
\cJ(0, 0, \gamma(0) ) =  \la f^0, (\gamma(0) -1 ) \ra  \,.
\end{equation}

%With these notations,
%\begin{equation}
%\label{J-def}
% \Lambda_{[0,t]} \left( \gamma e^{-\int_0^t \phi (s) ds}\right)=  \cJ(t, \Phi, \gamma)\,.
% \end{equation}

% We then define the functional space, analogous to the set $\bbB_{\alpha}$ introduced in \eqref{eq: space Bbis}:
% for all~$t \leq T_{\alpha }$,
%\begin{equation}\begin{aligned}
%\label{eq: ensemble B}
%\cB_{\alpha,t}  & := 
%\Big\{  ( \phi, \gamma) \in C^0([0,t ] \times \D;{\mathbb C}) \times C^0(\D;{\mathbb C}) \
%  \Big|  \,    |\gamma(x,v) | \leq  \exp \Big( \big(1 - \frac t{2T_{\alpha }}\big)( \alpha  +\frac{\beta_0}8 |v|^2)\Big) ,\\
%   & \, \qquad\qquad\qquad\qquad\qquad\qquad \qquad\qquad\qquad  \sup_{s \in [0,t]}   | \phi (s,x,v) |  \leq {1\over 2T_{\alpha }} (\alpha_r  +\frac{\beta_0}8|v|^2)
%\Big\}\,,
%\\    \mbox{and}\quad  \cB_{\alpha } & := \cB_{\alpha ,T_{\alpha }}   \, . \end{aligned}
%\end{equation}
Inspired by Appendix~\ref{CauchyKol}, we define the scale of function spaces
$$\begin{aligned}\label{defbalpha}
\cB_{\alpha,\beta,t} := 
\Big\{ ( \phi, \gamma) \in C^0([0,t] \times \D;{\mathbb C}) \times C^0(\D;{\mathbb C}) \
  :   \quad  
& |\gamma(x,v) | \leq  \exp \Big( \Big(1 - \frac t{2T_{\alpha }}\Big)( \alpha +\frac\beta8 |v|^2)\Big) ,\\
 &  \sup_{s \in [0,t]}   | \phi (s,x,v) |  \leq {1\over 2T_{\alpha }} (\alpha  +\frac\beta8 |v|^2)
\Big\}\, .
\end{aligned}
$$ 
Finally let us set
\begin{equation}
\label{eq: norm J}
\| \cJ (t)  \|_{\alpha,\beta} : = 
\sup_{(\phi,\gamma )\in {\mathcal B}_{\alpha, \beta,t} }  
\big| \cJ (t,\phi,\gamma)  \big| \, .
\end{equation}
%and
%$${\cN(\cJ) :=  \sup_{   \rho <1 \atop t < T(1-\rho)}
%  \| \cJ (t) \|_{ \alpha_r \rho  ,   \beta_0  \rho}  \Big(1- {t \over  T(1 - \rho )}\Big) }$$
%for some well chosen~$T \leq T_{\hat \alpha}$ and the norm 

\begin{Prop}
\label{statementuniqueness}
Let~$\alpha_0>0$ be given. There exists $T^{\rm\tiny HJ}_{\alpha_0 }\in (0, T_{\alpha_0 }]$ such that the Hamilton-Jacobi equation \eqref{HJ-Appendix form} has locally a unique  solution~$  \cJ $ in $[0,T^{\rm\tiny HJ}_{\alpha_0 }]$, in the class of functionals which satisfy:
\begin{itemize}
\item 
for  any $0\leq\alpha <\alpha' \leq \alpha_0$, $0\leq  \beta<\beta'\leq \beta_0$, $t \in [0,T_{\alpha_0 }]$
and $(\phi, \gamma) \in \cB_{\alpha,\beta,t}$
\begin{equation}
\label{loss-continuity-est}
 \Big\|  {\d \cJ (t,\phi,\gamma) \over \d \gamma }  \Big\|_{\cM\left((1+ |v|)    \exp \big(  \big(1-\frac t{2T_{\hat \alpha}}\big)  (  \alpha   +\frac\beta 8 |v|^2)\big) dx dv   \right)}   
\leq  
C \left( {1\over \alpha' - \alpha} +{1\over \beta' - \beta} \right)\| \cJ (t)  \|_{\alpha',\beta'} ;
\end{equation}
\item 
the derivative $\displaystyle{\d \cJ (t,\phi,\gamma) \over \d \gamma}$
is a continuous function   on~$\D$, and  there is a constant~$C$ such that for any $(\phi, \gamma) \in {\cB}_{r,T_{\alpha_0} }$,
\begin{equation}
\label{improved-est}
\forall t \leq T_{\alpha_0}\, , \qquad 
 \Big \| {\d \cJ (t,\phi,\gamma) \over \d \gamma} (1+ |v|) \exp ( \frac{\beta_0}  8 |v|^2) \Big  \|_{C^0(\D)}   
 \leq  C \,.
\end{equation}
\end{itemize}
\end{Prop}

\begin{proof}
According to Theorem~\ref{nishidatheorem}, there is a unique solution to~(\ref{HJ-Appendix form}) provided  that for all~$0\leq\alpha<\alpha'  \leq \alpha_0$, $0\leq  \beta<\beta'\leq \beta_0$
\begin{equation}
\label{loss-continuity-estF}
 \|F(\cJ(t) )-F(\cJ' (t))\|_{\alpha,\beta} \leq C \left( {1\over \alpha' - \alpha} +{1\over \beta' - \beta} \right)\| (\cJ - \cJ')(t)  \|_{\alpha',\beta'} \,   .
\end{equation}

It suffices to prove that~(\ref{loss-continuity-estF}) holds if~$\cJ$ satisfies~{\rm(\ref{loss-continuity-est})-(\ref{improved-est})}. Let us write
$$
\begin{aligned}
F(\cJ )-F(\cJ' ) =  \frac12  \int &{\d  (  \cJ -  \cJ '  ) \over \d \gamma} (s,\phi,\gamma(s))(z_1)  {\d  (\cJ+ \cJ ')  \over \d \gamma }(s,\phi,\gamma(s) )(z_2) \\
&  \times\Big(\gamma  (s,z'_1)  \gamma  (s,z'_2)\  - \gamma  (s,z_1) \gamma  (s,z_2) \Big)   d\mu (z_1, z_2, \omega)\, .
\end{aligned}
$$
If~$(\phi, \gamma) $ belongs to~$ \cB_{\alpha,\beta,t}$ then
 $$
 \forall s \leq t \, , \quad \big |\gamma (s,x,v)\big| \leq \exp \Big(
  \big(1-\frac s{2T_{  \alpha}}\big) \big(\alpha + \frac{\beta }8 |v|^2)
 \Big) \,,
 $$
so  we deduce that for any~$\alpha',\beta'$ with  $0\leq\beta<\beta'\leq \beta_0$, $0\leq\alpha<\alpha' \leq \alpha_0 $   
 $$ 
 \begin{aligned}
 \Big|  F(\cJ(s) )-F(\cJ' (s))  \Big| 
 &\leq C     \Big \|  {\d \big(  \cJ (s,\phi,\gamma)-  \cJ '(s,\phi,\gamma)\big)    \over \d \gamma} \Big \|_{\cM\left((1+ |v|)    \exp \big( \big(1-\frac s{2T_{\alpha_0}}\big)  (  \alpha   +\frac\beta 8 |v|^2)\big) dx dv   \right)}   \\
 & \qquad \qquad 
 \times \Big  \| {\d\big(  \cJ (s,\phi,\gamma)+  \cJ '(s,\phi,\gamma)\big)    \over \d \gamma} (1+ |v|) \exp ( \frac{\beta_0}  8 |v|^2)  \Big \|_{C^0(\D)} \\
 & \leq  C  \left( {1\over \alpha' - \alpha} +{1\over \beta' - \beta} \right) \|  \cJ(s)- \cJ'(s)\|_{\alpha',\beta'} 
  \end{aligned}
 $$
 where $C$ is a generic constant depending only on $\alpha_0, \beta_0$.
 Taking the supremum on all couples $(\phi,\gamma  )$ in~$ \cB_{  \alpha,  \beta, t }$, we  obtain that  
  $$\Big\|F(\cJ(s) )-F(\cJ'(s) ) \Big\|_{ \alpha, \beta}   \leq  C   \left( {1\over \alpha' - \alpha} +{1\over \beta' - \beta} \right) \Big\|  \cJ(s )- \cJ'(s )\Big\|_{\alpha',\beta'}\,.$$

% 
% % 
%We finally introduce the weight in time: we set~$\beta = \rho \beta_0$ and~$\alpha =  \rho  \alpha_r$, and we let~$\beta'$ and~$\alpha'$ depend on~$s$ in the following way: $\beta' 
%= \widetilde \rho(s)\beta_0$   and~$\alpha' =   \widetilde \rho(s)  \alpha_r$ with
%$$ \widetilde \rho(s):= \frac12 \Big( 1 +\rho - {s \over T^\star} \Big) $$
% with~$T$ to be chosen small enough. Then
% $$
% \widetilde \rho- \rho = \frac12 \Big( 1 -\rho - {s \over  T^\star} \Big) \quad \mbox{and} \quad 1- \widetilde \rho= \frac12 \Big( 1 -\rho + {s \over T^\star } \Big)\, , 
%$$and we get, for a constant~$C$  depending only on~$\alpha_r$ and~$\beta_0$,
%$$\begin{aligned}
%\Big\|\cJ (t)-\cJ '  (t)\Big\|_{\alpha, \beta}
% \Big(
% 1- {
% t \over  T^\star(1-\rho )} 
% \Big) & \leq  C \cN( \cJ - \cJ')   \Big(1- {t \over  T^\star(1-\rho )}\Big) \int_0^t   \left( {1\over \widetilde \rho - \rho  } \right) \left(1- {s \over T^\star (1- \widetilde \rho )} \right)^{-1} ds\\
%& \leq 2 C \cN( \cJ - \cJ')   \Big(1- {t \over  T^\star(1-\rho )}\Big) \int_0^t   \left( {1\over 1 - \rho - {s \over  T^\star  } } \right) \left( {1 -\rho + {s \over T^\star } \over 1 -\rho - {s \over  T^\star } } \right) ds\\
%& \leq  4C \cN( \cJ - \cJ')   \Big(1-\rho - {t \over  T^\star }\Big) \int_0^t  {ds\over \Big( 1 - \rho - {s \over  T ^\star }\Big)^2  } \\
%& \leq 4CT^\star   \cN(\cJ - \cJ') \, .
%\end{aligned} $$
%For $T^\star $ sufficiently small, we obtain that the constant~$4C T^\star$ is strictly less than 1, which implies finally that~$\cN(\cJ - \cJ') =0$.
Proposition~\ref{statementuniqueness} is proved.
\end{proof}

Having in mind to  use the  uniqueness criterion of Proposition \ref{statementuniqueness} to establish Proposition \ref{prop: I = hat I}, we now need to rewrite   $\cI$ and $\widehat \cI$ in the  new variables and to prove some regularity estimates.

\subsection{ Regularity of the limiting cumulant generating function $\cJ$}

\begin{Prop}
\label{prop: analyticity J} 
Let~$\alpha_0>0$ be fixed. For $t \leq T_{\alpha_0}$, the functional 
$ \cJ ( t, \phi, \gamma ) $ defined by {\rm(\ref{I-def dynamics new J})}  is an analytic function 
of  $\gamma$, on $ {\mathcal B}_{ \alpha_0,t} $. For any $\alpha' \in ]\alpha, \alpha_0] $, $\beta'\in]\beta, \beta_0] $ and all~$(\phi,\gamma) \in  {\mathcal B}_{  \alpha, \beta,t} $, 
the derivative~$\displaystyle {\d  \cJ ( t, \phi, \gamma) \over \d \gamma }$ 
satisfies the  loss continuity estimate {\rm(\ref{loss-continuity-est})}.
Moreover,  the derivative $\displaystyle{\d  \cJ ( t, \phi, \gamma )  \over \d \gamma}$
is a continuous function    on~$\D$
 satisfying the  estimate {\rm(\ref{improved-est})}.
 \end{Prop}

\begin{proof}
Thanks to~(\ref{J-derivative 2}) we find  that $\displaystyle{\d  \cJ ( t, \phi, \gamma )  \over \d \gamma }$ is a function on $\bbD$, 
for which we are going to establish properties \eqref{loss-continuity-est} and \eqref{improved-est}.

\noindent
{\it Step \rm{1}. Proof of  \eqref{loss-continuity-est}.}
Let~$(\phi,\gamma) $ be in~$ {\mathcal B}_{  \alpha, \beta,t} $  and let~$\Upsilon$ be   a continuous function on~$\mathbb D $ satisfying
$$
|\Upsilon(x,v)| \leq  (1+|v|)   \exp \Big( \big(1-\frac t{2T_{\alpha_0 }}\big) (  \alpha   +\frac\beta 8 |v|^2)\Big)  \, .
$$
It is easy to check that for a suitable choice of $\lambda>0$, the couple~$(\phi,\gamma +\lambda e^{i\theta} \Upsilon)$ belongs to~${\mathcal B}_{\alpha', \beta' ,t}$. Indeed it suffices to notice that
$$
\begin{aligned}
\Big|   \gamma + \lambda  e^{i\theta}\Upsilon \Big| 
&<    \big(1+ \lambda   (1+|v|)   \big)   \exp \Big( \big(1-\frac t{2T_{\alpha_0}}\big)  (  \alpha   +\frac\beta 8 |v|^2)\Big)\\
& \leq    \exp \Big(  \big(1-\frac t{2T_{\alpha_0  }}\big)  (  \alpha   +\frac\beta 8 |v|^2) +2\lambda +\frac\lambda2 |v|^2 \Big) \\
& \leq   \exp \Big(  \big(1-\frac t{2T_{\alpha_0  }}\big)  (  \alpha '  +\frac{\beta'}  8 |v|^2)\Big)\, , 
\end{aligned}
$$provided that $\lambda \leq \min \Big(\displaystyle \frac {\alpha'- \alpha } {  4}  ,  \frac{\beta' - \beta} {  4}  \Big)$.
Then by analyticity, choosing~$\lambda = \min   \Big(\displaystyle\frac{\alpha'- \alpha }  {  4} ,  \frac{\beta' - \beta} {  4}   \Big)$,  the derivative can be estimated  by  a contour integral
$$
\int_\bbD dz     {\d  \cJ ( t, \phi, \gamma )   \over \d \gamma } (z) \;  \Upsilon (z) 
% \Big \la {\d J \over \d H  } (H), \Upsilon \Big \ra  
= {1\over2\pi \lambda } \int _0^{2\pi}  \cJ \Big(t,  \phi, (\gamma +\lambda e^{i\theta} \Upsilon) \Big)  e^{-i\theta} d\theta \,,
$$
and we conclude   that for all~$(\phi,\gamma) $   in~$ {\mathcal B}_{  \alpha, \beta,t } $, 
$$
  \Big\|   {\d  \cJ ( t, \phi, \gamma )   \over \d \gamma }  \Big \|_{\cM\left((1+ |v|)     \exp \left( \big(1-\frac t{2T_{\alpha_0}}\big) (  \alpha   +\frac\beta 8 |v|^2)\right) \right) }       
 \leq  C\left( {1\over \alpha' - \alpha } +{1\over \beta' - \beta} \right) \| \cJ (t)  \|_{\alpha',\beta'}  \,.
$$
This completes \eqref{loss-continuity-est}.

\noindent
{\it Step \rm{2}. Proof of  \eqref{improved-est}.}
For the second estimate, we use the series expansion \eqref{J-derivative 2}.
The   measure~$\mu_{{\rm sing},\tilde T}$ is invariant under global translations, and since $\Upsilon$ depends only on one   variable in~$\D$,~\eqref{J-derivative 2} still makes sense
if $\exp ( - \frac{ \beta_0}8 |v|^2 )  \Upsilon $ is only a measure.
Up to changing the parameter of the weights, we get the result.

 Proposition~\ref{prop: analyticity J}  is proved.
\end{proof}

\subsection{ Definition and regularity of $\widehat \cJ$}

The same  change of variables  is used to define~$\widehat \cJ(t, \phi, \gamma(t) )$ which is the counterpart of  $\widehat  \cI (t,g)$ introduced in \eqref{hatI-def}~:
\begin{align}
\label{eq: fonc hat J sup psi eta}
\widehat \cJ(t, \phi,\gamma)
:=   &
 \la f^0 , ( \eta_t (0) -1)\ra + \lA D \eta_t , \psi_t \rA - \lA  \phi_t , \psi_t \, \eta_t \rA   
 +   \int_0^t \;  \cH'  \Big( \psi_t (s) ,  \eta_t (s)  \Big)  \, ds ,
\end{align}
 where $(\psi, \eta) = (\varphi e^{-p}, e^p)$.

In these new variables, the Hamiltonian equations (\ref{hamiltonian-traj1})-(\ref{hamiltonian-traj2}) 
on the time interval $[0,t]$ can be rewritten
\begin{equation}
\label{eq: Euler f}
\begin{aligned}
& D_s \psi_t  + \psi_t \,  \phi_t   = {\d \cH' \over \d \eta  } ( \psi_t, \eta_t), \quad \psi_t (0)   =f^0  ,\\
& D_s  \eta_t   - \eta_t \, \phi_t   =-  {\d \cH' \over \d \psi } ( \psi_t, \eta_t), \quad \eta_t (t)   = \gamma(t)\,.
\end{aligned}
\end{equation}
Note that the structure of this Hamiltonian system is more symmetric
than (\ref{hamiltonian-traj1})-(\ref{hamiltonian-traj2})
and it can be interpreted as  a system of modified Boltzmann equations.
Indeed \eqref{eq: Euler f} can be written
\begin{equation}
\label{eq: Euler f 2}
\begin{aligned}
& D_s \psi_t  = - \psi_t \, \phi_t
+ \int d\mu_{z_1}(z_2,\omega)\, 
  \eta_t (z_2) 
\Big( \psi_t(z_1') \psi_t(z_2')  -  \psi_t(z_1) \psi_t(z_2) \Big)
\quad  \text{with} \quad  
\psi_t (0) = f^0,\\
& D_s   \eta_t  =   \eta_t \, \phi_t  - \int d\mu_{z_1}(z_2,\omega)\, 
  \psi_t (z_2)  \Big( \eta_t (z_1') \eta_t (z_2') - \eta_t (z_1) \eta_t (z_2) \Big) 
\quad \text{with} \quad   
\eta_t (t) =  \gamma\, .
\end{aligned}
\end{equation}
In particular contrary to \eqref{hamiltonian-traj1}, the boundary conditions  in \eqref{eq: Euler f 2} are time independent.

We are now going to check that the  modified Hamiltonian equations \eqref{eq: Euler f 2} admit unique solutions. From this, we will deduce that $\widehat \cJ$ is well defined and satisfies the regularity assumptions of Proposition~\ref{statementuniqueness}.

\begin{Prop}
\label{Lem: definition hat I}
Let~$\alpha>0$ be fixed. There exists a time~$ T_{\alpha }^{\rm\tiny H'} \in (0,T_{\alpha }]$ such that for any~$(\phi,\gamma) $ in~$ \cB_{\alpha , \beta_0,T_\alpha}$ and $t $ in~$ [0,  T_{\alpha }^{\rm\tiny H'}]$,  there is a unique solution 
 $(\psi_t , \eta_t)$ to the system of   modified Hamiltonian equations \eqref{eq: Euler f 2} on~$[0,t]$ such that 
 for the norm introduced in \eqref{eq: espaces L infini beta}
\begin{equation}
\label{eq: bornes psi eta}
 \sup_{s \in [0,t] } \| \psi_t (s)\|_{L^\infty_{{\color{black}-3\beta_0/4}}} \leq   C,
\quad
\sup_{s \in [0,t] } \| \eta_t (s)\|_{L^\infty_{{ \color{black}\beta_0/2}}} \leq   C\,  .  
\end{equation}
  If $(\phi, \gamma)$ take real values and $\gamma>0$ 
then $(\psi_t, \eta_t)$ are  both positive functions.
For any~$t \in [0, T_{\alpha }^{\rm\tiny H'}]$, the functional $ \widehat \cJ ( t, \phi, \gamma)$ is well defined and depends analytically on $\gamma$. Furthermore, it satisfies estimates~{\rm(\ref{loss-continuity-est})} and
{\rm(\ref{improved-est})}.
\end{Prop}

\begin{proof} \ \\
{\it Step 1. Well-posedness of the system of  modified Hamiltonian equations \eqref{eq: Euler f 2}.}

This is once again  a consequence of the Cauchy-Kovalevskaya argument of Appendix~\ref{CauchyKol}. 
The proof is therefore postponed to the appendix \ref{EL-Cauchy}. Let us just point out here that to implement the strategy, it is more convenient to rewrite~\eqref{eq: Euler f 2} in a mild form, denoting~$S_s$ the transport operator in~$\D$:
\begin{equation}
\label{EL-mild}
\forall s \leq t, \qquad
\begin{aligned}
& \psi_t (s)  = S_s   f^0+\int_0^s S_{s-\sigma}  F_1 \big( \phi_t (\sigma), \eta_t (\sigma), \psi_t (\sigma) \big) d\sigma \,, \\
&  \eta_t (s)  =   S_{s- t } \gamma_t - \int_s^t S_{s-\sigma }  F_2 \big( \phi_t (\sigma), \eta_t (\sigma),\psi_t (\sigma) \big) d\sigma 
\,,
\end{aligned}
\end{equation}
with 
$$ 
\begin{aligned}
& F_1 ( \phi, \eta, \psi)  =  - \psi \, \phi+ \int   d\mu_{z_1}(z_2,\omega)\, 
\eta (z_2) 
\Big( \psi(z_1') \psi(z_2')  -  \psi(z_1) \psi(z_2) \Big)\,,\\
& F_2( \phi, \eta, \psi) =\eta  \, \phi -  \int  d\mu_{z_1}(z_2,\omega)\, 
 \psi (z_2)  \Big( \eta(z_1') \eta(z_2') -\eta(z_1) \eta(z_2) \Big)\,.
\end{aligned}
$$
The positivity of $(\psi_t, \eta_t)$ is proved by rewriting \eqref{eq: Euler f} in the form
\begin{align*}
& D_s \psi_t  +  \psi_t  \Big( \phi_t + K_1(\psi_t ,\eta_t) \Big) 
= \int d\mu_{z_1}(z_2,\omega)\,  \eta_t (z_2) \psi_t (z_1') \psi_t (z_2')  
\quad  \text{with} \quad  \psi_t(0) = f^0,\\
& D_s   \eta_t  + \eta_t \Big( -  \phi_t + K_2 (\psi_t, \eta_t) \Big)
= -  \int d\mu_{z_1}(z_2,\omega)\,  \psi_t (z_2)  \eta_t (z_1') \eta_t (z_2')  
\quad  \text{with} \quad    \eta_t (t) =  \gamma .  
\end{align*}
The first equation is a transport equation with a (nonlinear) damping term $\phi_t + K_1(\psi_t, \eta_t)$ and a source term which is nonnegative  (as long as $\psi_t, \eta_t$ are positive). It therefore preserves the positivity.
The second equation is a backward transport equation with a damping term $-  \phi_t + K_2 (\psi_t, \eta_t)$ and a source term which is nonpositive (as long as $\psi_t, \eta_t$ are positive). It also preserves the positivity.
The solution~$(\psi_t, \eta_t)$ obtained by iteration (using the fixed point argument) is therefore positive.

 \medskip

\noindent
{\it Step 2. Regularity estimates on $\widehat \cJ(t, \phi,\gamma)$.}\\
Since the solution $(\psi_t, \eta_t)$ to the Hamiltonian equations is obtained as a fixed point of a contracting (polynomial) map depending linearly on $\gamma$ (see (\ref{EL-mild})), it is straightforward to check that $(\psi_t, \eta_t)$ depends analytically on $\gamma$ (for instance using the iterated Duhamel series expansion).  Proceeding as in Proposition \ref{action-lem}, we can show 
$$
{\d \widehat \cJ(t, \phi,\gamma) \over \d \gamma} 
 =  \psi_t (t) \,.
$$
The estimates \eqref{eq: bornes psi eta} on $\psi_t$ lead directly to \eqref{improved-est}.
The inequality \eqref{loss-continuity-est} can be obtained by a contour estimate as in 
the derivation of Proposition \ref{prop: analyticity J}. Proposition \ref{Lem: definition hat I} is proved.
\end{proof}

 \bigskip

\subsection{Conclusion of the proof of Proposition {\rm\ref{prop: I = hat I}}}
By Proposition \ref{Lem: definition hat I}, the functional $\widehat \cJ$ is well defined on some time interval $[0, T_{\alpha }^{\rm\tiny H'}]$, so   $\widehat \cI$ is also well defined and the formal computations in Proposition~\ref{action-lem} are justified.
By implementing a proof similar to the one of Proposition \ref{action-lem},
$\widehat \cJ$ is a solution of the Hamilton-Jacobi equation \eqref{HJ-Appendix form} in $[0, T_{\alpha }^{\rm\tiny H'}]$
\begin{equation*}
\forall t \leq T_{\alpha }^{\rm\tiny H'}\, , \qquad 
\d_t \widehat \cJ\left(t, \phi, \gamma(t)\right ) = \cH' \left({\d \widehat \cJ \over \d \gamma}, \gamma (t)\right)\, .
\end{equation*}
The regularity assumptions of Proposition \ref{statementuniqueness} hold for $\cJ$ (see  
Proposition \ref{prop: analyticity J})  and for $\widehat \cJ$ (see Proposition \ref{Lem: definition hat I}), thus 
$\cJ$ and $\widehat \cJ$ coincide on $[0,T^\star_\alpha] \times  {\mathcal B}_{\alpha}$, up to requiring~$T^\star_\alpha \leq \min( T_{\alpha }^{\rm\tiny H'}, T_{\alpha }^{\rm\tiny HJ}) $.

  Given $g$, the functions $(\psi,\eta)$ are positive by Proposition \ref{Lem: definition hat I}, so that $\varphi = \psi \eta$ and $p = \log\eta$ are well defined.     Going back to the original variables, we conclude that $\cI$ and $\widehat \cI$ coincide on $[0,T^\star_\alpha ] \times \bbB_{\hat \alpha}$.

%%%%%%%%%%%%%%%%%%%%%%%%%%%%%%%%%%%%%%%%%%%%%%%%%%%%%%%%%%%%%%%%%%

\section[Large deviations]{The large deviation estimates}
\label{LD-sec}
\setcounter{equation}{0}

In this section, we fix $\alpha$ according to  (\ref{eq: tuning r}),  and $T$ as in Theorem \ref{thm: F = ha F}.  Recall that $\cM(\bbD)$ stands for the set of positive measures with finite mass on $\bbD$.
We are now going to prove the large deviation estimates of Theorem \ref{thmLDbis}
in terms of the functional $\cF$ given by the Legendre transform for $\gp \in D([0,T], \cM(\bbD))$
\begin{equation*}
%\label{eq: legendre transform GD}
\cF (T, \gp) :=  \sup_{ g \in \bbB_{\alpha}} \Big\{   - \lA \varphi, D g \rA + \la \varphi(T),  g(T) \ra  
-  \cI (T, g) \Big\}\,  .
\end{equation*}
The method of the proof is standard (see e.g. the textbook \cite{dembozeitouni} or \cite{DV}) as the difficult work has been achieved already in Theorems \ref{cumulant-thm1} and \ref{cumulant-thm2} to derive the convergence of the cumulant generating function  of the particle system to the limiting functional $\cI (t, g)$.
For the sake of completeness, we sketch the main steps of the proof.

We   first start by proving upper and  lower large deviation bounds in a topology weaker than the Skorokhod topology.
%To define this weak topology, we consider 
% $(h_i)_{i \geq 0} $  a countable family of continuous functions which is dense in $\bbB$. 
This weak topology on $D([0,T], \cM(\bbD))$ is
generated by open sets of the form below, for any $\nu \in D([0,T], \cM(\bbD))$ and for test functions $g$ in $\bbB_{\alpha}$ and $\gd >0$:
\begin{equation}
\label{eq: petit voisinage}
{\bf O}_{\gd,g} ( \nu) := 
\Big\{ \nu' \in D([0,T_{\alpha}], \cM(\bbD)) \, : \quad   \big| 
\big( \lA \nu' , D g \rA - \la \nu_{T} ',  g_{T} \ra \big)
- 
\big( \lA  \nu , D g \rA - \la \nu_{T},  g_T \ra  \big) \big|  < \gd/2
\Big \}    .
\end{equation} 
Then, in Section \ref{sec: Tightness},  the topology will be enhanced to the Skorokhod topology by a tightness argument.

\subsection{Upper bound}
\label{sec:  Upper bound LD}
We are going to prove the large deviation upper bound~(\ref{eq: large deviations upper bound})
for any compact set ${\bf F}$ of $D([0,T], \cM(\bbD))$ in the weak topology 
\begin{align}
\label{eq: large deviations weak upper bound}
\limsup_{\mu_\eps \to \infty} \frac{1}{\mu_\eps} 
\log \bbP_\eps \left( \pi^\gep\in {\bf F} \right) \leq - \inf_{\gp \in \bf F}  \cF  (T,\gp)\,.
\end{align}
General closed sets will be considered in Section \ref{sec: Tightness}.

We are  first going to show that for any density $\gp$ in $\bf F$ and $\gd >0$,
there exists $g \in \bbB_{\hat \alpha}$ and an open set~${\bf O}_{\gd,g} ( \gp)$ of $\gp$ such that 
\begin{align}
\label{eq: large deviations upper bound local}
\limsup_{\mu_\eps \to \infty} \frac{1}{\mu_\eps} 
\log \bbP_\eps \left( \pi^\gep \in {\bf O}_{\gd,g} ( \gp) \right) \leq -  \cF (T,\gp) + \gd\,.
\end{align}
%with 
%$$
%\hat \cF_\gd(\gp) = \hat \cF(\gp) - \gd.
%% \inf \left\{ \hat \cF(\gp) - \gd, \frac{1}{\gd} \right\}.
%$$

Then by compactness, for any $\gd>0$, a finite covering of ${\bf F} \subset \cup_{i \leq K} {\bf O}_{\gd,g_i} ( \gp_i)$ can be extracted so that 
\begin{align*}
\limsup_{\mu_\eps \to \infty} \frac{1}{\mu_\eps} 
\log \bbP_\eps \left( \pi^\gep \in {\bf F} \right) \leq - \inf_{i \leq K} \cF(T,\gp_i ) + \gd
\leq - \inf_{\gp \in  {\bf F}}  \cF  (T,\gp) + \gd\,.
\end{align*}
Letting $\gd \to 0$, we recover the upper bound \eqref{eq: large deviations weak upper bound}.

We turn now to the derivation of \eqref{eq: large deviations upper bound local}.
For any density $\gp$ in $\bf F$, we know from \eqref{eq: legendre transform GD} that there exists $g \in \bbB_{\alpha}$ 
such that 
\begin{align*}
 \cF(T,\gp) 
 \leq
 - \lA \varphi, D g \rA + \la \varphi(T),  g(T) \ra -  \cI (T, g) + \gd/2\,.
\end{align*}
This leads to the upper bound
\begin{align*}
\bbP_\eps \left( \pi^\gep \in {\bf O}_{\gd,g} ( \gp) \right) 
& \leq 
\exp \Big(   \mu_\eps \frac{\gd}{2} + \mu_\eps \lA  \gp , D g \rA - \mu_\eps \la \varphi(T),  g(T) \ra   \Big) \\
& \qquad \qquad \times  \bbE_\eps \left( \exp \Big(  - \mu_\eps \lA \pi^\gep , D g \rA + \mu_\eps \la \pi^\gep_{T} ,  g(T) \ra  \Big)  \right)\\
& \leq 
\exp \Big(   \mu_\eps \frac{\gd}{2} + \mu_\eps  \lA  \gp , D g \rA - \mu_\eps \la \varphi(T),  g(T) \ra   + \mu_\eps \; \cI^\eps (T, g) \Big) \,,
\end{align*}
with
$$
 \cI^\eps (t, g) :=  \gL^\eps_{[0,t]} \big( e^{g-\int_0^t Dg}  \big) \, .
$$
Passing to the limit thanks to Theorem \ref{cumulant-thm2}, 
this completes \eqref{eq: large deviations upper bound local}
\begin{align*}
\limsup_{\mu_\eps \to \infty} \frac{1}{\mu_\eps} 
\log \bbP_\eps \Big( \pi^\gep \in{\bf O}_{\gd,g} ( \gp) \Big) 
\leq  
\cI (T, g) +  \lA  \gp , D g \rA - \la \varphi(T),  g(T) \ra  + \gd /2
\leq -   \cF  (T ,\gp) + \gd\,.
\end{align*}

\begin{Rmk}
 Note that the proof of the upper bound holds actually up to time $T_\alpha =  c e^{-\alpha } \beta^{\frac{d+1}2}_0/C_0$, if the supremum in~{\rm(\ref{eq: legendre transform GD})} is taken over the functions~$g$ satisfying the assumptions 
$$    \sup_{t \in [0,T_{\alpha} ]}   | g(t,z)| \leq  \frac12( \alpha  +\frac{\beta_0}8 |v|^2) ,\quad
 \sup_{t \in [0,T_{\alpha} ]}   | D_t g (t,z) |  \leq {1\over 2T_{\alpha}} (\alpha   +\frac{\beta_0}8 |v|^2)
 \, .  $$
 The restriction to~$T$   will appear in the proof of the lower bound when using the fact that the supremum in \eqref{eq: legendre transform GD} is reached for some $g \in \bbB_{ \alpha}$.
 \end{Rmk}

\subsection{Lower bound}
\label{sec:  Lower bound LD}

We are going to prove the large deviation lower bound \eqref{eq: large deviations lower bound}
for any open set ${\bf O}$ in the weak topology
\begin{align}
%\label{eq: large deviations lower bound}
\liminf_{\mu_\eps \to \infty} \frac{1}{\mu_\eps} 
\log \bbP_\eps \left( \pi^\gep \in {\bf O} \right) \geq - \inf_{\gp \in{ \bf O} \cap \cR_{r,T}} \cF  (T,\gp)\,,
\end{align}
where the restricted set $\cR_{r,T}$ of trajectories was defined in \eqref{eq: set R 2} (see also Theorem \ref{thmLD}).

Contrary to the proof of the upper bound which was a direct consequence  of the convergence to $\cI$
of the cumulant generating function (Theorem \ref{cumulant-thm2}), the
derivation of the lower bound follows from the G\"artner-Ellis method \cite{dembozeitouni} and it requires an additional regularity assumption on $\cF$.
For this,  we consider observables $\gp$ such that the supremum in \eqref{eq: legendre transform GD} is reached for some $g \in \bbB_{\hat \alpha}$   
\begin{align}
\label{eq: minimum principe variationnel}
\cF (T,\gp) = \la \varphi(T) ,  g(T) \ra  -  \lA  \gp , D g \rA -  \cI (T, g)\, .
\end{align}
It was shown in \eqref{eq: minimum atteint F} that 
identity \eqref{eq: minimum principe variationnel}  is valid  for any $\gp$ in $\cR_{r,T}$.
Even though \eqref{eq: minimum principe variationnel} should be valid for a larger class of functions, 
we restrict to functions $\gp$  in ${\bf O} \cap \cR_{r,T}$ for simplicity.

\medskip

Let us fix $\gp \in {\bf O} \cap \cR_{r,T}$ and denote by $g$ the associated test function as in \eqref{eq: minimum principe variationnel}.
There exists a collection of test functions 
$g^{(1)}, \dots, g^{(\ell)}$ in $\bbB_{\hat \alpha}$ such that the following  open neighborhood  of  $\gp$ 
\begin{equation}
\label{eq: petit voisinage 2}
\begin{aligned}
{\bf O}_{\gd,\{ g^{(i)}\} } ( \gp) &:= 
\Big\{ \nu \in D([0,T], \cM(\bbD))\, :  \ 
\forall i \leq \ell\, , \\
&\qquad\qquad \Big| 
 \lA \nu , D  g^{(i)} \rA - \la \nu(T) , g^{(i)}(T) \ra  
- 
\big( \lA  \gp , D  g^{(i)} \rA - \la \varphi(T),   g^{(i)}(T) \ra  \big) \Big|  < \gd
\Big \}
\end{aligned}
\end{equation}
is included in $\bf O$ for any $\gd>0$ small enough. We impose also that $g$ is one of the test functions~$g^{(1)}, \dots, g^{(\ell)}$.
To complete the lower bound 
$$ \liminf_{\mu_\eps \to \infty} \frac{1}{\mu_\eps} 
\log \bbP_\eps \left( \pi^\gep \in {\bf O} \right)  \geq -  \cF(T,\gp)\,,$$
it is enough to show that \begin{align}
\label{eq: large deviations lower bound local}
\liminf_{\gd \to 0} \liminf_{\mu_\eps \to \infty} \frac{1}{\mu_\eps} 
\log \bbP_\eps \left( \pi^\gep \in {\bf O}_{\gd,\{ g^{(i)}\} } ( \gp) \right) \geq -  \cF(T,\gp)\,.
\end{align}

We start by tilting the measure
\begin{align*}
\bbP_\eps \left(  {\bf O}_{\gd,\{ g^{(i)}\} }( \gp) \right) & \geq 
\exp \Big( - \gd \mu_\eps + \mu_\eps  \lA  \gp , D g \rA - \mu_\eps \la \varphi(T),  g(T) \ra  \Big) \\
& \qquad \qquad \times \bbE_\eps \left( 
\exp \Big(  - \mu_\eps \lA \pi^\gep , D g \rA + \mu_\eps \la \pi^\gep_T ,  g(T) \ra  \Big)  \; \indc_{ {\bf O}_{\gd,\{ g^{(i)}\} } ( \gp) } 
\right)\\
& \geq \exp \big( - \gd \mu_\eps + \mu_\eps \cI^\eps (T,g) 
+\mu_\eps  \lA  \gp , D g \rA - \mu_\eps \la \varphi(T),  g(T) \ra  \big) \; 
\bbE_{\eps,g} \left(  \indc_{ {\bf O}_{\gd,\{ g^{(i)}\} } ( \gp) } \right),
\end{align*}
where we defined the tilted measure for any function $\Psi$ on the particle trajectories as
\begin{align*}
\bbE_{\eps,g} \left(  \Psi( \pi^\gep ) \right)
:= \exp \left( - \mu_\eps \cI^\eps (T,g) \right)
\bbE_\eps \left( \exp \Big(  - \mu_\eps \lA \pi^\gep , D g \rA + \mu_\eps \la \pi^\gep_T ,  g(T) \ra  \Big)  \; \Psi( \pi^\gep ) \right).
\end{align*}
If we can  show that the trajectory $\gp$ is typical under the tilted measure
\begin{align}
\label{eq: large deviations lower bound concentration}
\forall \gd >0\, , \qquad 
\lim_{\mu_\eps \to \infty} 
 \bbP_{\eps,g} \left( \pi^\gep \in {\bf O}_{\gd,\{ g^{(i)}\} } ( \gp) \right) =1 \, ,
\end{align}
this will complete the proof of \eqref{eq: large deviations lower bound local}.

\bigskip

Let $\tilde g$ be one of the functions $g^{(1)}, \dots, g^{(\ell)}$ used to define the weak neighborhood 
${\bf O}_{\gd,\{ g^{(i)}\} } ( \gp)$.
Choose~$u \in \bbC$ in a neigborhood of 0 so that the function below is analytic 
\begin{align*}
u \in \bbC \mapsto \cI  (T, u \tilde g + g) 
= 
\lim_{\mu_\eps \to \infty}   \cI^\eps  (T, u \tilde g + g)\, . 
%= 
%\lim_{\mu_\eps \to \infty} \frac{1}{\mu_\eps} \log 
%\bbE_\eps \left(  \exp \big( \mu_\eps \lA \pi^\gep ,(u \psi + h)  \rA  \big) \right).
\end{align*} 
As a consequence the derivative and the  limit  as $\mu_\eps \to \infty$ commute, so that 
taking the derivative at~$u=0$, we get
\begin{align*}
- \LA \frac{\d \cI}{\d D g} (T, g), D \tilde g \RA
+ \La \frac{\d \cI}{\d g(T)} (T, g), \tilde g(T)   \Ra   = 
\lim_{\mu_\eps \to \infty} 
\bbE_{\eps,g} \Big(   -   \lA \pi^\gep , D \tilde  g \rA +   \la \pi^\gep_T ,  \tilde g(T) \ra    \Big)\, .
\end{align*}

Note that 
in the above equation, the functional derivative is taken over both coordinates $Dg, g(T)$
of the functional $\cI (T, g)$.
%\begin{equation*}
%\frac{\d \cI_{[0,T_{H}]}}{\d g} (g)  :=
%{\d \cI_{[0,T_{H}]} (g,G) \over \d g} \Big|_{G =g_t} + {\d \cI_{[0,T_{H}]} (g,G) \over \d G} \Big|_{G =g_t}
%\end{equation*}
%and differs from the one used in (\ref{eq: derivative I}).
As the supremum in \eqref{eq: legendre transform GD} is reached at $g$, we deduce from 
\eqref{eq: minimum principe variationnel} that  
\begin{align}
\label{eq: differentielle Dg, g}
- \LA \frac{\d \cI }{\d D g} (T, g), D \tilde g \RA
+ \La \frac{\d \cI}{\d g(T)} (T, g), \tilde g(T)   \Ra 
%\lA \frac{\d \cI (T_{H})}{\d g} ( g), \tilde g \rA 
= 
 \la \varphi(T) ,  \tilde g(T) \ra -   \lA \varphi , D \tilde  g \rA \,    .
 %=  \lA D \varphi ,  \tilde  g \rA  -    \la \varphi_0 ,  \tilde g_0 \ra .
\end{align} 
This allows us to characterize the mean under the tilted measure
\begin{align}
\label{eq: control mean}
\lim_{\mu_\eps \to \infty} 
\bbE_{\eps,g} \left(    \la \pi^\gep_T  ,  \tilde g(T) \ra  -   \lA \pi^\gep , D \tilde  g \rA \right) 
=   \la \varphi(T) ,  \tilde g(T) \ra -   \lA \varphi , D \tilde  g \rA \,    .
\end{align} 
Taking twice the derivative, we obtain
\begin{align*}
\lim_{\mu_\eps \to \infty} \mu_\eps  
\bbE_{\eps,g} \left(  \left[ 
\Big(   \la \pi^\gep_T ,  \tilde g(T) \ra -   \lA \pi^\gep , D \tilde  g \rA   \Big)
- \bbE_{\eps,g} \Big(  \la \pi^\gep(T) ,  \tilde g(T) \ra   -   \lA \pi^\gep , D \tilde  g \rA  \Big) \right]^2 \right) < \infty\, .
\end{align*} 
Combined with \eqref{eq: control mean}, this implies that the empirical measure concentrates to 
$\gp$ in a weak sense
\begin{align*}
\lim_{\mu_\eps \to \infty}  \bbE_{\eps,g} \left(  \left[  
\Big(  \la \pi^\gep_T ,  \tilde g(T)  \ra -   \lA \pi^\gep , D \tilde  g \rA   \Big)
- \Big(  \la \varphi(T) ,  \tilde g(T) \ra  -   \lA \varphi , D \tilde  g \rA   \Big)\right]^2 \right) = 0\, .
\end{align*} 
In particular, this holds for any test functions 
$g^{(1)}, \dots, g^{(\ell)}$ defining the neighborhood  ${\bf O}_{\gd,\{ g^{(i)}\} }( \gp)$ in~\eqref{eq: petit voisinage 2}.
This completes \eqref{eq: large deviations lower bound concentration}.

\subsection{Tightness}
\label{sec: Tightness}

In this section, we are going to prove a tightness property in the Skorokhod topology which will enhance the large deviations 
proven  so far in a coarser topology (see Corollary~4.2.6 of~\cite{dembozeitouni}).

Let $(h_j)_{j \geq 0 } $ denote the basis of Fourier-Hermite functions (as in \eqref{eq: sobolev norm hypotheses}).
%, which is dense in  $W^{1, \infty} (\bbD ,[-1,1])$,  and
% such that  $\| v \cdot \nabla_x h_j \|_\infty \leq  j_1(1+ j_2)^\frac12$ for all $j = (j_1, j_2)$.
%\begin{equation*}
%B =  \{ g \in C^0(\R^{2d}) \,/ \  \| g\|_\infty \leq  r\}\,.
%\end{equation*}
We define a distance on the set of measures $\cM(\bbD)$ by
\begin{equation}
\label{eq: distance}
d(\mu,\nu) := \sum_j 2^{-j}  \left| \int dz  \;  h_j (z) \big( d \mu(z) -  d \nu(z) \big) \right| \;.
\end{equation}
\begin{Prop}
\label{prop: Tightness}
The norm of the empirical measure is concentrated in compact sets  
\begin{align}
\label{eq: tighness norm GD}
\lim_{A \to \infty} \lim_{\mu_\eps \to \infty} 
 \frac{1}{\mu_\eps} \log 
\bbP_\varepsilon \Big( \sup_{ t \in [0, T_0 ]} d(\pi_t^\eps ,0) \geq A \Big) = - \infty 
\end{align}
and the modulus of continuity is controlled by 
\begin{align}
\label{eq: tension GD}
\forall \gd'>0, \qquad \lim_{\gd \to 0} \lim_{\mu_\eps \to \infty} \frac{1}{\mu_\eps} \log
\bbP_\eps \left( \sup_{\substack{|t -s| \leq \gd \\ t,s \in  [0, T_0 ]}} d(\pi_t^\eps ,\pi_s^\eps) > \gd' \right) = - \infty\,  . 
\end{align}
Thus the sequence of measures  $(\pi^\eps_t)$ is exponentially tight.
\end{Prop}
Before proving Proposition \ref{prop: Tightness}, let us first show that it implies 
 large deviation estimates in the Skorokhod space of trajectories  $D([0,T ], \cM(\bbD))$
(for a definition see Section 12 in \cite{Billingsley}).
First of all notice that the upper bound \eqref{eq: large deviations upper bound} holds for closed sets $\bf F$ and not only compact sets as the sequence of measures $(\bbP_\eps)$ is tight and  the closed sets for the Skorokhod topology are also closed for the weak topology.

We consider now an open set ${\bf O}$ for the strong topology and $\gp$  a trajectory in ${\bf O} \cap \cR_{r,T}$, recalling that~$\cR_{r,T}$ is defined in~\eqref{eq: set R 2}.
We would like to apply the same proof as 
in Section \ref{sec:  Lower bound LD} and to  reduce  the estimates to sample paths in a weak open set of the form \eqref{eq: petit voisinage 2}.
We proceed in several steps. First note that there exists $\gd>0$ such that 
$$
\left \{ \nu\,  : \quad   \sup_{t \leq T } d(\nu_t,\gp_t) <  2 \gd \right \} \subset {\bf O}\, .
$$ 
Since $\gp$ belongs to $\cR_{r,T}$, the density $\gp$ is continuous  in time. 
Choosing a time step $\gga>0$ small enough, we can restrict to computing the distance at discrete times 
$$
\left\{ \nu \,  : \quad   \sup_{i \in \bbN \atop 
i \gamma \leq T  } d(\nu_{i \gga},\gp_{i \gga}) <   \gd  \right \} 
\bigcap 
\left \{\nu \,  : \quad  \sup_{|t -s| \leq \gga} d(\nu_t ,\nu_s) <\gd  \right \}
\subset {\bf O}\, .
$$ 
Since $\gp$ is continuous in time and we consider only $T /\gga$ times, the first set above can be approximated by a set of 
the form ${\bf O}_\gd( \gp)$ as in \eqref{eq: petit voisinage 2}.
As a consequence we have shown that there is an open set ${\bf O}_\gd( \gp)$ such that 
\begin{align*}
%\label{eq: large deviations lower bound local}
\bbP_\eps \left( \pi^\gep \in {\bf O} \right) 
&\geq 
\bbP_\eps \left( \pi^\gep \in {\bf O}_\gd( \gp)  \bigcap 
\left \{ \sup_{|t -s| \leq \gga} d(\pi_t^\eps ,\pi_s^\eps) < \gd  \right \} \right) \\
&\geq 
\bbP_\eps \left( \pi^\gep \in {\bf O}_\gd( \gp)   \right) 
-
\bbP_\eps \left( \left \{ \sup_{|t -s| \leq \gga} d(\pi_t^\eps ,\pi_s^\eps) > \gd  \right \} \right) \, .
\end{align*}
By Proposition \ref{prop: Tightness} the last term can be made arbitrarily small 
for $\gga$ small. Thus the proof of the lower bound reduces now to the one of
weak open sets as in Section \ref{sec:  Lower bound LD}.

\begin{proof}[Proof of Proposition {\rm\ref{prop: Tightness}}]
To prove \eqref{eq: tighness norm GD}, let us first note that the test functions used for defining the distance in \eqref{eq: distance} are uniformly bounded, thus the distance is bounded in terms of the  total number $\cN$ of particles 
$$
d(\pi_t^\eps ,0) \leq C \; \frac{\cN}{\mu_\eps}\, \cdotp$$
As the number of particles is fixed only by the initial distribution, it is simple to obtain 
the exponential decay claimed in \eqref{eq: tighness norm GD}
\begin{align}
\label{eq: tighness norm GD bis}
\bbP_\varepsilon \Big( \sup_{ t \in [0,  T_0]} d(\pi_t^\eps ,0) \geq A \Big) 
\leq \bbP_\varepsilon \Big( \cN \geq A \frac{\mu_\eps}{C} \Big) 
\leq  c_1 \, \exp \Big( - c_2 \mu_\eps A \Big) \, .
\end{align}

%\begin{align}
%
%\lim_{A \to \infty} \lim_{\mu_\eps \to \infty} 
% \frac{1}{\mu_\eps} \log 
%\bbP_\varepsilon \Big( \sup_{ t \in [0,T_{H}]} d(\pi_t^\eps ,0) \geq A \Big) = - \infty 
%\end{align}

By the inequality \eqref{eq: tighness norm GD bis} and the boundedness of the  test functions used  in \eqref{eq: distance}, it is enough  to consider a finite number of test functions. 
Indeed, for any $\gd'$ there is $K= K(\gd')$ such that 
\begin{equation*}
d(\mu,\nu) > \gd' \quad \Rightarrow \quad  
\sum_{|j| \leq K} 2^{-j}  \left| \int dz  \;  h_j (z) \big( d \mu(z) -  d \nu(z) \big) \right| > \frac{\gd'}{2} \, \cdotp
\end{equation*}
By the union bound, we can then reduce \eqref{eq: tension GD} to controlling a single test function $h$
\begin{align}
\label{eq: tension GD h}
\forall \gd'>0\, , \qquad \lim_{\gd \to 0} \lim_{\eps \to 0} \frac{1}{\mu_\eps} \log
\bbP_\eps \left( \sup_{|t -s| \leq \gd} \big| \la \pi_t^\eps ,h \ra - \la \pi_s^\eps ,h \ra \big| > \gd' \right) = - \infty\, ,
\end{align}
where $t,s$ are restricted to $[0,  T]$.
Next, we localize 
the constraint on the time interval $[0,T]$ to smaller time intervals
\begin{align}
\label{eq: petits intervalles}
\bbP_\eps \left( \sup_{|t -s| \leq \gd} \big| \la \pi_t^\eps ,h \ra - \la \pi_s^\eps ,h \ra \big| > \gd' \right) 
& \leq \sum_{i = 2}^{T / \gd} 
\bbP_\eps \left( \sup_{t,s \in [ (i-2) \gd, i \gd]} \big| \la \pi_t^\eps ,h \ra - \la \pi_s^\eps ,h \ra \big| > \gd' \right)\, .
\end{align}

By assumption \eqref{lipschitz}, the initial density $f^0$ is bounded, up to a multiplicative constant $C_0(2\pi /\beta_0) ^{d/2}$  by the Maxwellian $M_{\beta_0}$ (uniformly distributed in $x$).
By modifying the weights $W^{\eps 0}_{N}$ in \eqref{eq: initial measure}, 
we deduce that the probability of any event ${\mathcal A}$ under $\bbP_\eps$ can be bounded from above in terms of the probability $\tilde \bbP_\eps$ with initial density $M_{\beta_0}$
(its expectation is denoted by $\tilde \bbE_\eps$)
$$
\bbP_\eps ( {\mathcal A} ) 
\leq \frac{\tilde \cZ^ \eps}{\cZ^ \eps} \tilde \bbE_\eps ( C^\cN \; 1_{\mathcal A} )
\leq \frac{\tilde \cZ^ \eps}{\cZ^ \eps} \tilde \bbE_\eps ( C^{2 \cN} )^\frac{1}{2}
\;  \tilde \bbE_\eps ( 1_{\mathcal A} )^\frac{1}{2}
\leq \exp ( C \mu_\eps) \; \tilde \bbP_\eps ( {\mathcal A} )^\frac{1}{2}
\,,
$$
for some constant $C$ and $\tilde \cZ^ \eps$ stands for the partition function of this new density.
Using the fact  that the probability $\tilde \bbP_\eps$ is time invariant, we 
can reduce the estimate of the events in \eqref{eq: petits intervalles} to a single time interval.
Thus \eqref{eq: tension GD h} will follow if one can show that 
\begin{align}
\label{eq: tension GD tilde h}
\forall \gd'>0\, , \qquad \lim_{\gd \to 0} \lim_{\eps \to 0} \frac{1}{\mu_\eps} \log
\tilde \bbP_\eps \left( \sup_{t,s \in [ 0, 2 \gd]} \big| \la \pi_t^\eps ,h \ra - \la \pi_s^\eps ,h \ra \big| > \gd' \right) = - \infty\,.
\end{align}
 
By the Markov inequality and using the notation $L_\gd = \log | \log \gd|$, we get  
\begin{align}
\label{eq: tension GD delta'}
\tilde \bbP_\eps \left( \sup_{t,s \in [ 0, 2 \gd]} \big| \la \pi_t^\eps ,h \ra - \la \pi_s^\eps ,h \ra \big| > \gd' \right)
%\bbP_\eps \Big( \sup_{t,s \in [0,2\gd]} \big| \la \pi_t^\eps ,h \ra - \la \pi_s^\eps ,h \ra \big| > \gd' \Big) 
\leq e^{  -  \gd' \, L_\gd \, \mu_\eps}
\tilde  \bbE_\eps \Big(  \exp \Big( \sup_{t,s \in [ 0, 2 \gd]} \; L_\gd \,
\Big| \sum_{i =1}^\cN  h \big( {\bf z}^\e_i(t) \big) - h \big( {\bf z}^\e_i(s) \big)  \Big|  \Big) \Big)\\
\leq e^{  - \gd' \, L_\gd \, \mu_\eps}
\tilde  \bbE_\eps \Big(  
\exp \Big(  \sum_{i =1}^\cN \sup_{t,s \in [ 0, 2 \gd]} \; L_\gd \,
\big| h \big( {\bf z}^\e_i(t) \big) - h \big( {\bf z}^\e_i(s) \big)  \big|  
\Big) \Big)\, .
\nonumber
\end{align}
The last inequality is very crude, but  it is enough for the large deviation asymptotics and it allows us to reduce to a sum of functions
depending only on the trajectory of each particle via
$$ \tilde h \big( z( [0,2 \delta] ) \big)  := \sup_{t,s \in [ 0, 2 \gd]} \; L_\gd \, \big| h \big( z(t) \big) - h \big( z(s) \big)  \big|\, .$$
Thanks to Proposition \ref{prop: exponential cumulants}, the last expectation 
in \eqref{eq: tension GD delta'} can be rewritten in terms of the 
cumulants
\begin{equation}
\begin{aligned}
\label{eq: exponential cumulants trajectory bis}
\frac{1}{\mu_\eps}  \log \tilde  \bbE_\eps \left(  
\exp \Big(  \sum_{i =1}^\cN  \tilde h \big( {\bf z}^\e_i( [0,2 \delta] ) \big)   
\Big) \right)
 =  \sum_{ n=1}^\infty   \frac{1}{n !}
\Big |  \tilde f_{n,[0, 2 \delta]}^\eps \Big( \big(  \exp ( \tilde h  ) - 1 \big)^{\otimes n} \Big)\Big | \, ,
\end{aligned}
\end{equation}
where $\tilde f_n^\eps$ stands for the dynamical cumulant under the new distribution.

For $n \geq 2$, the statement 1  of Theorem \ref{cumulant-thm1*} page~\pageref{cumulant-thm1*} can be applied 
% (see page \pageref{cumulant-thm1*} below):
\begin{equation*}\Big |  \tilde f_{n,[0, 2 \delta]}^\eps \Big( \big(  \exp ( \tilde h  ) - 1 \big)^{\otimes n} \Big)\Big | \leq n!\big(C(2 \delta+\eps) \big ) ^{n-1}  \,
| \log \delta | ^{ 2n \| h\|_\infty} \,,
\end{equation*}
with $L_\gd = \log | \log \gd|$.
The term $n = 1$  is controlled thanks to the statement 3 of Theorem \ref{cumulant-thm1*} 
% (see page \pageref{cumulant-thm1*} below):
\begin{equation*}
\Big |  \tilde f_{1,[0, 2 \delta]}^\eps  \big(  \exp ( \tilde h  ) - 1 \big) \Big | 
\leq \delta 
  \left(\| \nabla h\|_\infty L_\gd + 1 \right) e^{ L_\gd \| h\|_\infty} 
\leq \delta 
  \left(\| v\cdot \nabla_x  h\|_\infty L_\gd + 1 \right)  |\log \delta|^{  \| h\|_\infty}\,.
\end{equation*}
Thus \eqref{eq: exponential cumulants trajectory bis} converges to 0 as $\eps \to 0$, then $\gd$ tends to 0.
Furthermore $L_\gd$ diverges to $\infty$ as $\gd$ vanishes,  one deduces from \eqref{eq: tension GD delta'} that 
\eqref{eq: tension GD tilde h} holds  for any $\gd' >0$. 
This completes the proof of \eqref{eq: tension GD h} and therefore of 
Proposition~\ref{prop: Tightness}.
\end{proof}

\section{Proof of the large deviation theorem}
\label{sec: proof of thm 3}

 Theorem \ref{thmLD} is derived by combining Theorems \ref{thmLDbis}  
and \ref{thm: F = ha F}. Indeed given $\gp \in \cR_{r,T}$, the upper bound  is obtained by  considering in~(\ref{eq: large deviations upper bound}) the closed sets $\{ d_{[0,T]} (\pi^\eps, \varphi) \leq \delta \}$, where~$d_{[0,T]}$ stands for the distance metrizing the Skorokhod topology. Since  $\cF$ is lower semi-continuous (by property of the Legendre transform) there holds
\begin{align*}
\lim_{\delta  \to 0} \inf_{ \psi, \atop d_{[0,T]} (\psi, \varphi) \leq \delta} \cF(T,\psi)
\geq   \cF(T,\gp)\, ,
\end{align*}
which gives the result since  $\cF(T,\gp)=  \widehat \cF(T,\gp)$ thanks to Theorem~\ref{thm: F = ha F}. The lower bound is obtained directly thanks to~(\ref{eq: large deviations lower bound})
and Theorem~\ref{thm: F = ha F}. \qed

\part{Uniform a priori bounds and  convergence of the cumulants }\label{partbounds-cv}

\chapter{Clustering constraints and cumulant estimates}
\label{estimate-chap} 
\setcounter{equation}{0}

In this chapter we  consider the  cumulants $\cum$, whose definition (Eq.\,\eqref{eq: decomposition cumulant}) we recall: 
\begin{equation}\label{recalldefcum}
\begin{aligned}
 \cum =   \int  dZ_n^* \mu_\eps^{n-1} \sum_{\ell =1}^n \sum_{\gl \in \cP_n^\ell}
\sum_{r =1}^\ell   \sum_{\gr \in \cP_\ell^r}  
\int  
 \Big( \prod_{i=1}^\ell d\mu\big( \Psi^\eps_{\gl_i} \big)   \cH \big( \Psi^\eps_{\gl_i} \big)  \mb_{\gl_i} \Big)   \gp_{ \gr} \;  f^{\eps 0}_{\{1,\dots,r\}} \;.
   \end{aligned}
\end{equation}
We  prove the upper bound stated in Theorem \ref{cumulant-thm1} page~\pageref{cumulant-thm1}
which is a consequence of the  following  more general statement~:

\begin{Thm}
\label{cumulant-thm1*}
Consider the system of hard spheres under the  initial measure~{\rm(\ref{eq: initial measure})},  with~$f^0 $ satisfying~{\rm(\ref{lipschitz})}.
Let  $H_n : D([0,\infty[) \mapsto \bbR$  be a continuous factorized function:
$$
H_n \big( Z_n([0,\infty[ \big) = \prod_{i=1}^n H^{(i)} \big( z_i([0,\infty[) \big)
$$ 
and define the scaled cumulant  $f_{n,[0,t]}^\eps (H_n)$ by polarization %as in \eqref{eq:formpolfn0t}.
of the $n$ linear form \eqref{eq: decomposition cumulant}.
 Then  there exists a  positive constant $C$ and a time $T_0$  such that  the following  uniform a priori bounds hold:
\begin{enumerate}
\item If $H_n$ is bounded,
then   on  $[0,T_0]$ 
  $$ 
  | f_{n,[0,t]}^\eps (H_n) |  \leq  n! \left( {CC_0  \over \beta_0^{(d+1)/2}}\right)^{n }  (t+\eps)^{n-1}  \prod_{i=1} ^n \| H^{(i)}\|_\infty  \, .
  $$
\item If $H_n$ has a controlled  growth %there exist $\bar \alpha $ and $\bar \beta \in (0,  {\beta_0} /2) $ such that 
\begin{equation}
\label{controlHnexp bis}
 \big| H_n( Z_n([0,t])) \big| \leq \exp \Big(\alpha \; n +\frac{\beta_0} 4 \sup_{s\in [0,t]} |V_n(s)|^2 \Big) \,,
 \end{equation}
then   on  $[0,T_0]$ 
$$ 
| f_{n,[0,t]}^\eps (H_n) |  \leq  \left( {CC_0 e^\alpha  \over \beta_0^{(d+1)/2}}\right)^{n }  (t+\eps)  ^{n-1}   n!\, .
$$
  
\item Fix $\gd>0$. If $H_n$  measures in addition of \eqref{controlHnexp bis}, the time regularity in the time interval $[t-\gd,t]$,  i.e.\;if for some $i\in \{1,\dots, n\}$
\begin{equation}
\label{eq: time-cont + growth exp}
 \big| H_n( Z_n([0,t])) \big| \leq C_{Lip}  \min \Big(  \sup_{t' \atop |t-t'| \leq \delta }  |z_i (t) - z_i (t')| , 1\Big)   \exp   \Big(\alpha  n+\frac{\beta_0} 4 \sup_{s\in [0,t]} |V_n(s)|^2 \Big) 
\,,
\end{equation}
  then   on  $[0,T_0]$ 
\begin{equation}  
\label{timecontinuity}
| f_{n,[0,t]}^\eps (H_n) |  \leq   C_{Lip} \delta  
\left( {C C_0 \;  e^\alpha \over \beta_0^{(d+1)/2}} \right)^n   (t+\eps)  ^{n-1}   n! \,.
  \end{equation}
\end{enumerate}
   \end{Thm}

The key idea behind this result is that the clustering structure of $\cum$ imposes   strong geometric constraints on the integration parameters $(Z_n^* ,T_m, V_m, \Omega_m)$ (where we recall that~$m$ is the size of the collision tree), which imply that the integral defining~$\cum$ involves actually only a  set of parameters with small measure
of size  $O(1/\mu_\eps^{n-1} )$.
More precisely, what we   prove  is that:
\begin{itemize}
\item there are~$n-1$ ``independent"  geometric constraints (clustering conditions)  and each of them  provides a small factor $O(1/\mu_\eps)$;
\item the integration measure (which is unbounded because of possibly large velocities in the collision cross-sections) does not induce any divergence.
\end{itemize}

\medskip

Section \ref{subsec:estDC} is devoted to characterizing the small measure set. Actually we   only provide necessary conditions for the parameters $(Z_n^* ,T_m, V_m, \Omega_m)$  to belong to such a set (which is enough to get an upper bound). This characterization can be expressed as a  succession of geometric conditions on the 
relative positions $x_1^*,\dots, x_n^*$ of the $n$  particles at time $t$. 

Section~\ref{subsec:proofThm} then explains how to control  the integral defining $ \cum$. Recall that, by \eqref{eq:boundHn} and by conservation of the energy,
\begin{equation*}\label{controlHnexp}
| \cH(\Psi^\eps _n)| =  |H_n\big( Z_n^*([0,t])\big)|\leq e^{\alpha n  + \frac{\beta_0}4|V^*_{n}(0)|^2 +\frac{\beta_0 }4|V_{m}(0)|^2}\;.
\end{equation*}
 Since the initial data satisfy a Gaussian bound
$$
(f^0)^{\otimes {n+m}} (\Psi^{\eps 0} _n) \leq C_0^{n+m} e^{- \frac{\beta_0}2|V^*_{n}(0)|^2  - \frac{\beta_0}2|V_m(0)|^2} \, , 
$$
the growth of~$| \cH(\Psi^\eps _n)|$ is easily controlled, so
the main difficulty is   to control the cross-sections
\begin{equation}\label{eq: cross section}
\cC \big( \Psi^{\eps}_{n} \big) := \prod_{k=1}^{m}    s_k\Big( \big( v_k -v_{a_k} (t_k)\big) \cdot \omega_k  \Big)_+
\end{equation}
  in the measure~$ d\mu\big( \Psi^\eps_{n} \big)  $.
  In order for this term  not to create any divergence for large $m$, we need a symmetry argument as in the classical proof of Lanford, but intertwined here with the estimates on the size of the small measure set.
A similar procedure is  used in Section \ref{subsec:estDC} to cure high energy singularities arising from the geometric constraints themselves.

\section{Dynamical constraints}
\label{subsec:estDC}
\setcounter{equation}{0}

Let $\gl \hookrightarrow \gr$ be a nested partition of $\{1^*, \dots , n^*\}$.
We fix the velocities $V_n^*$ at time $t$, as well as the collision parameters~$(m,a,T_m, V_m, \Omega_m)$ of the pseudo-trajectories. We recall that $V_m = (v_1,\dots,v_m)$ where $v_i$ is the velocity of particle $i$ at the moment of its creation.

We   denote by $$\bbV^2 :=(V_n^*)^2 + V_{m}^2 = \sum_{i=1}^{n} \left(v_i^*\right)^2 + \sum_{i=1}^{m} v_i^2$$ (twice) the total energy of the whole pseudo-trajectory $\Psi^{\e}_{n}$ appearing in \eqref{recalldefcum}, and by $K = n+m$ its total number of particles. We also indicate by $\bbV_i^2$ (resp.~$\bbV_{\lambda}^2$ for any $\lambda \subset \{1^*, \dots , n^*\}$) and $K_i$ (resp.~$K_\lambda$) the corresponding energy and number of particles of the collision tree with root at~$z_i^*$~(resp.~$Z^*_{\lambda}$), that is:
 \begin{equation}
\begin{aligned}
\label{eq:energia1}
&& \bbV_i^2 = \left(v_i^*\right)^2 + \sum_{\tiny\mbox{$j$ created in $\Psi^\e_{\{i\}}$}} v_j^2\;,\\
&&  K_i = 1 + \#\left(\mbox{particles\ created\ in\ }\Psi^\e_{\{i\}}\right)\;
\end{aligned}
 \end{equation}
and
 \begin{equation}
\begin{aligned}
 \label{eq:energia2}
&& \bbV_{\lambda}^2 = \sum_{i \tiny \mbox{ tree\ in\ }  \gl} \bbV_i^2\;, \\
&& K_\gl =\sum_{i\tiny \mbox{ tree\ in\ }   \gl} K_i\;.
\end{aligned}
 \end{equation}
Note that $\bbV^2 = \sum_{i =1}^n\bbV_i^2$ and~$K = \sum_{i =1}^n K_i= n+m$.

In what follows, it will be important to remember the notations and definitions introduced in Chapter~\ref{cumulant-chap}, as well as the rules of construction of pseudo-trajectories explained in Section \ref{sec:grct}.
 In particular we recall that, because of these rules, $\bbV^2/2$ is the energy at time zero of the configuration $\Psi^{\e 0}_{n}$, while $\bbV_i^2/2$ is not, in general, the energy of $\Psi^{\e 0}_{\{i\}}$ (because of   external recollisions which can perturb the velocities of the particles inside the tree), unless $\Psi^{\e}_{\{i\}}$ does not recollide with the other $\Psi^{\e}_{\{j\}}$, $j \neq i$.

\medskip

\noindent
-- {\it Clustering recollisions.}
We first study the constraints associated with   clustering  recollisions in the pseudo-trajectory of the generic forest $\Psi^\e_{\gl_1}$. 
Up to renaming  the integration variables, we can assume that $$\lambda_1 = \{1,\dots , \ell_1\}\;.$$
We call $x^*_{\lambda_1} := x^*_{\ell_1}$ the {\it root} of the forest.

\begin{Prop}
The set of configurations $Z_{\ell_1} ^*$ at time $t$ compatible with the forest $\lambda_1 = \{1,\dots , \ell_1\}$  on $[0,t]$ satisfies the following estimate~:
\begin{equation}\begin{aligned}
\label{eq: external recollision cost trees}
\int \!  \!  dX^*_{\ell_{1}-1}   \mb_{\gl_1}\, \indc_{\cG^{\e}} \big (  \Psi^\e_{\gl_1} \big)
  \leq \left(\frac{Ct} {\beta_0 ^{1/2} \mu_\eps}\right)^{\ell_1-1} 
\sum_{T \in \cT_{\lambda_1}}\,
\prod_{ j \in \lambda_1}\,\left(\beta_0\bbV^2_{j} + K_{j}\right)^{d_j(T)}\,,
\end{aligned}
\end{equation}
where $d_j(T)$ is the degree of the vertex $j$ in the graph $T$.
\end{Prop}

By definition of $\mb_{\lambda_1}$ and by Definition \ref{def:ClustRec} of clustering  recollisions,  there exist $\ell_1-1$ clustering recollisions occurring at times
$\tau_{\rm{rec},1}\geq \tau_{\rm{rec},2} \geq\dots \geq \tau_{\rm{rec},\ell_1-1}$. Moreover, 
the corresponding chain of recolliding trees~$\{ j_1,j'_1 \},\dots, \{ j_{\ell_1-1},j'_{\ell_1-1} \}$ is a minimally connected
graph $T \in \cT_{\lambda_1}$, equipped with an ordering of the edges. 
We shall denote by $T^{\prec}$\label{Tprec}  a minimally connected graph equipped with an ordering of edges, 
and by $\cT_{\lambda_1}^{\prec}$ the set of all such graphs on $\lambda_1$. Hence
we have  
\begin{equation*}
\mb_{\gl_1} = \sum_{T^\prec \in \cT^\prec_{\gl_1}} \mb_{\gl_1, T^\prec} 
\end{equation*}
almost surely, where $\mb_{\gl_1, T^\prec} $ is the indicator function that the clustering recollisions
for the forest $\lambda_1$ %(see Definition \ref{def:ClustRec}) 
are given by $T^\prec$. We also recall that, by definition, $\mb_{\lambda_1}$ is equal to zero whenever two particles find themselves at mutual distance strictly smaller than $\e$.

It will be convenient to represent the set of graphs $\cT_{\lambda_1}^{\prec}$ in terms of sequences of merged subforests. The subforests are obtained following the dynamics of the pseudo-trajectory $\Psi^\e_{\gl_1}$ backward in time, and putting together the groups of trees that recollide. An example is provided by Figure \ref{fig:packs}.
\begin{figure}[h] %  figure placement: here, top, bottom, or page
\centering
\includegraphics[width=5in]{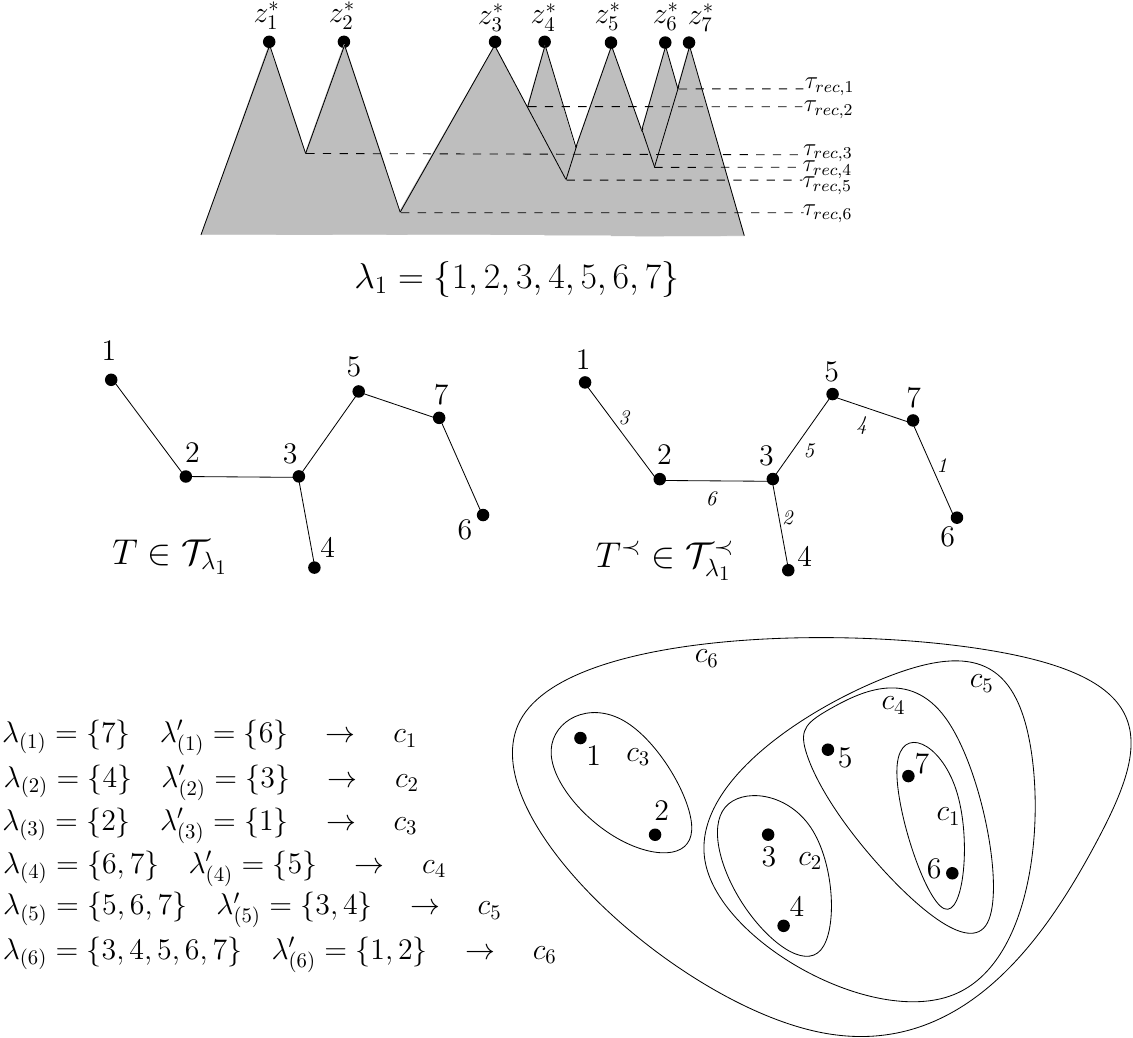} 
\caption{An example of pseudo-trajectory $\Psi^\e_{\gl_1}$ ($\ell_1 = 7$) satisfying the constraint $\mb_{\gl_1, T^\prec} $, together with its minimally connected graph $T$, ordered graph $T^\prec$, and sequence of {\it merged subforests} $\left(\gl_{(k)}, \gl'_{(k)}\right)_{k}$. The roots of the trees $z_i^* = (x_i^*,v_i^*)$ and the clustering recollision times appear in the picture on the top.}
\label{fig:packs}
\end{figure}

More precisely, we define the map which associates to any ordered tree the sequence of merging clusters
$$
\cT^\prec_{\gl_1} \ni T^\prec \mapsto \left(\gl_{(k)}, \gl'_{(k)}\right)_k
$$
 by the following iteration :
\begin{itemize}
\item start from $\gl_1 = \{1,\dots,\ell_1\}$;
%= \cup_{i=1}^{\ell_1}\gl^i_{(0)}$, where $\gl^i_{(0)} := \{i\}$, $i=1,\cdots,\ell_1$;
\item take the first edge $\{j_1,j_1'\}$ of $T^\prec$,  and set $\left(\gl_{(1)}, \gl'_{(1)}\right) =\left(\{j_1\},\{j_1'\} \right)$; these two elements are merged into a single cluster $c_1$; set $L_1 := c_1 \cup \left(\gl_1 \setminus \{j_1,j_1'\}\right)$;
\item at step $k > 1$, take $\left(\gl_{(k)}, \gl'_{(k)}\right)$ of $L_{k-1}$ in such a way that $j_k \in \gl_{(k)}, j'_k \in \gl'_{(k)}$ where $\{j_k, j'_k\}$ is the $k$-th edge of $T^\prec$, and merge them into a single cluster $c_k$; set $L_k :=  c_k \cup \left(L_{k-1} \setminus \{\gl_{(k)}, \gl'_{(k)}\}\right)$.
We can assume without loss of generality that $\max \gl'_{(k)} <\max \gl_{(k)}$.
\end{itemize}
The last step is given by $\left(\gl_{(\ell_1-1)}, \gl'_{(\ell_1-1)}\right)$, which merges the two remaining clusters. 

However this map
 is not a bijection, because the merged subforests do not specify which vertices of~$j_k\in\gl_{(k)}$ and $j'_k \in \gl'_{(k)}$ are connected by the edge. A bijection is therefore given by
\begin{equation} \label{eq:errbij}
\cT^\prec_{\gl_1} \ni T^\prec \to \left(\gl_{(k)}, \gl'_{(k)},j_{k}\in\lambda_{(k)}, j'_{k}\in\lambda'_{(k)} \right)_k\;.
\end{equation}

We define the {\it root} of the subforest $\lambda_{(k)}$ by
$$x^*_{\lambda_{(k)}} := x_{\max \gl_{(k)}}^*\,,$$
and same definition for the root of $\lambda'_{(k)}$. We can then define 
$$\hat x_k := x^*_{\gl'_{(k)}} - x^*_{\gl_{(k)}}\,,\qquad k=1,\dots,\ell_1-1$$
as the relative position between the two recolliding subforests at time $t$.
%This provides $\ell_1-1$ linearly independent variables. In fact, by construction,
%starting from $k=\ell_1-1$ and moving up to $k=1$, at each step there is necessarily
%a new available root (meaning that either $z^*_{\lambda_{(k)}}$, or $z^*_{\lambda'_{(k)}}$, 
%appear for the first time). 
%Consider the map sending the smallest available root position at step $k$, to the relative position $\hat x_k$.
%Applying this map for all $k$ (from $\ell_1-1$ to $1$), we engage all the roots, except one.
%The remaining root available is  called  {\it root} of the forest and denoted by $z^*_{\lambda_1}$ (this root will be used later on to treat the clustering overlaps). The other root positions are mapped to the relative positions:
It is easy to see that, for any given root position $x^*_{\lambda_1} = x^*_{\ell_1} \in \T^d$, the map of translations
\begin{equation} \label{eq:bijection}
X^*_{\ell_1-1} = \left(x^*_1,\dots, x^*_{\ell_1-1}\right)  \mapsto \hat X_{\ell_1-1} := 
\left(\hat x_1,\dots,\hat x_{\ell_1-1}\right) 
\end{equation}
is one-to-one on $\T^{d(\ell_1-1)}$ and such that
$$dX^*_{\ell_1-1}  = d\hat X_{\ell_1-1}\;.$$
Thus \eqref{eq:bijection} is a legitimate change of variables
in \eqref{recalldefcum}.

Our purpose is to prove iteratively that, for~$k = \ell_1-1, \dots , 1$, the variable $\hat x_k$ associated with the $k$-th clustering recollision has to be in a small set, the measure of which is uniformly small of size $O(1/\mu_\eps)$.

We define $\Psi^\e_{\lambda_{(k)}}$ (respectively $\Psi^\e_{\lambda'_{(k)}}$) the pseudo-trajectory with starting particles $\lambda_{(k)}$
$(\lambda'_{(k)})$.
%and independent from the other particles of the original $\Psi^\e_{\lambda_1}$. 
Since~$\tau_{\rm{rec},k}\geq \left(\tau_{\rm{rec},s}\right)_{s > k}$, the collision trees in $\lambda_1 \setminus \left(
\lambda_{(k)} \cup\lambda'_{(k)}\right)$ do not affect the subforests $\lambda_{(k)}, \lambda'_{(k)}$
in the time interval $(\tau_{\rm{rec},k}, t)$. 
%At step $k$, we can therefore treat $\Psi^\e_{\lambda'_{(k)}}$ and $\Psi^\e_{\lambda_{(k)}}$ as completely independent trajectories. 
The clustering structure prescribed by $T^\prec$ implies that  $\Psi^\e_{\lambda'_{(k)}}$ and $\Psi^\e_{\lambda_{(k)}}$, regarded as independent trajectories, reach mutual distance $\e$ at some time  $\tau_{\rm{rec},k} \in (0,\tau_{\rm{rec},k-1})$.

Given $\left(\hat x_{s}\right)_{s < k}$ fixed by the previous recollisions, we are going to vary $\hat x_k$ so that an external recollision between the subforests occurs.  This corresponds to moving rigidly $\Psi^\e_{\lambda'_{(k)}}$ and $\Psi^\e_{\lambda_{(k)}}$ by acting on their relative distance $\hat x_k$. In fact, the recollision condition depends only on this distance.

Given a sequence of merged subforests $\left(\gl_{(k)}, \gl'_{(k)}\right)_{k}$
and a set of variables $\left(\hat x_s\right)_{s < k}$
(with $|\hat x_s| > \e$), the~$k-$th clustering recollision condition is defined by
$$\hat x_k \in \cB_{k}:= \bigcup_{\tiny\substack{ \mbox{$q$ in the subforest $\gl_{(k)}$} \\ \mbox{$q'$ in the subforest $\gl'_{(k)}$}}} B_{qq'}\,,$$ 
with
\begin{equation}
\label{Bq-def}
B_{qq'} := \Big\{\hat x_k \in \T^d\;\;:\;\; | x_{q'}  (\tau_{{\rm{rec}},k}) - x_q(\tau_{{\rm{rec}},k})  | = \eps\ \ \ \mbox{for some $\tau_{\rm{rec},k} \in (0,\tau_{{\rm{rec}},k-1})$} \Big\}\;.
\end{equation}
Here $x_q (\tau), x_{q'} (\tau)$ are the particle trajectories in the flows $\Psi^\e_{\lambda_{(k)}},\Psi^\e_{\lambda'_{(k)}}$ (and $\tau$ is of course restricted to their existence times).  In other words there exists a time $\tau_{\rm{rec},k} \in (0,\tau_{\rm{rec},k-1})$ and a   vector~$\omega_{\rm{rec},k}  \in {\mathbb S}^{d-1}$    such that
\begin{equation}\label{recoll-changevar}
    x_{q'}  (\tau_{{\rm{rec}},k}) -   x_{q} (\tau_{{\rm{rec}},k})  = \eps\, \omega_{{\rm{rec}},k} \, .
  \end{equation}
 The particle trajectories $   x_q (\tau),   x_{q'} (\tau)$ are piecewise affine 
  (because there are almost surely a finite number of collisions and recollisions within the trees $\Psi^\e_{\lambda_{(k)}}, \Psi^\e_{\lambda'_{(k)}}$). We will denote by $v_q^{(\gd\tau_j)} , v_{ q'}^{(\gd\tau_j)}$ the velocities of $q$ and $q'$ on the interval $\delta \tau_j$.
Moreover, $(   x_q(\tau)-  x_{q'}(\tau)) - (x^*_{\gl_{(k)}}-x^*_{\gl'_{(k)}})$ does not depend on~$\hat x_k := x^*_{\gl'_{(k)}} - x^*_{\gl_{(k)}}$,
  because all positions in the collision tree are translated rigidly. This means that~$\hat x_k  $ has to be in a tube of radius~$\eps$ around the parametric curve $(x^*_{\gl_{(k)}}-x^*_{\gl'_{(k)}})- (   x_q(\tau)-  x_{q'}(\tau))$.
This tube is a union of cylinders, with two spherical caps at both ends (see Figure \ref{fig:tube}). 
Note however that we have to remove from this tube the ball corresponding to the exclusion at the  creation time (or at time $t$ if $q$ and $q'$ exist up to time $t$).
  
\begin{figure}[h] %  figure placement: here, top, bottom, or page
\includegraphics[width=3.5in,  angle =180]{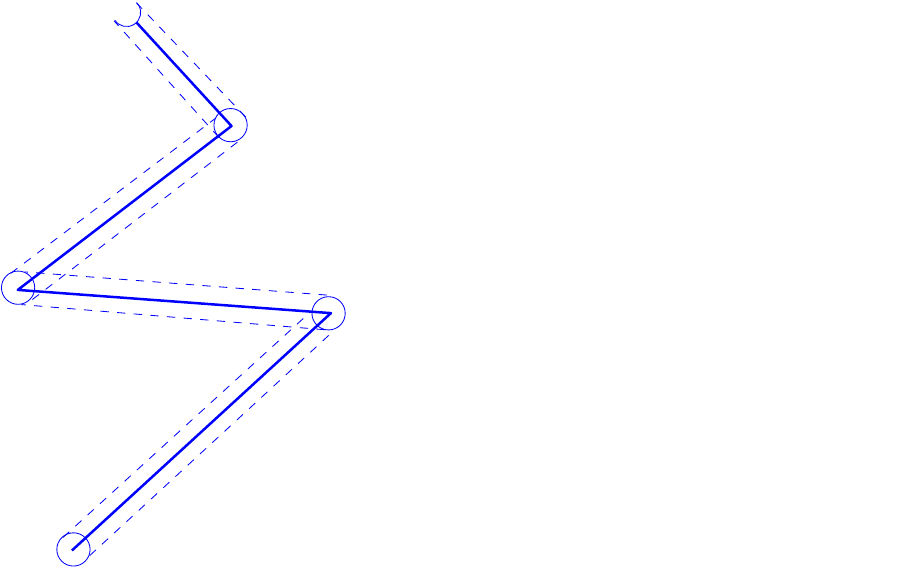} 
\caption{The tube $B_{qq'}$ leading to a recollision between particles $q$ and $q'$. The tube has section $\mu_\e^{-1}$. 
%The coloured ball represents a forbidden region, due to the exclusion condition at time $t$ ($|\hat x_k| > \e$) or at creation times (recall that in a forest pseudo-particles cannot overlap).
}
\label{fig:tube}
\end{figure}

Therefore
 $$
 B_{qq'} = \bigcup_{ j}B_{qq'}(\gd\tau_j)
 $$
 for a suitable finite decomposition of $(0,\tau_{\rm{rec},k-1})$ (depending on all the history).
 We therefore end up with the estimate (see Figure \ref{fig:tube})
  $$
  |B_{qq'}| \leq \frac C{\mu_\eps}\sum_{ j}  |v_q^{(\gd\tau_j)} - v_{ q'}^{(\gd\tau_j)}| \,|\gd\tau_j|$$
  for some pure constant $C>0$ depending only on the dimension $d$.
 
 We sum now over all $q,q'$ to obtain an estimate of the set $\cB_{k}$. To exploit the conservation of energy, we exchange the sums over $\gd\tau_j$ and over $q,q'$. 
%  Consider the set
%  $$ \cB_{k}(\gd\tau) := \bigcup _{p\in \Psi_{j_k}, q\in \Psi_{j'_k}} B_{qq'}(\gd\tau)\;.$$
  We get
  $$
  | \cB_{k}| \leq  \frac{C}{\mu_\eps} \sum_{ j}|\gd\tau_j| \sum_{q,q'} |v_q^{(\gd\tau_j)} - v_{ q'}^{(\gd\tau_j)}|\,. % + C \e^d K_{j_k}K_{j'_k}
  $$
 % where we recall that $K_{j}$ indicates the total number of particles in the collision tree $j$.
  Applying the Cauchy-Schwarz inequality, the sum over $q,q'$  is bounded by 
  $$
\sqrt{\sum_q \left(v_q^{(\gd\tau_j)}  \right)^2}\sqrt {K_{\gl_{(k)}}}\, K_{\gl'_{(k)}}
+ \sqrt{\sum_{q'} \left(v_{q'}^{(\gd\tau_j)}  \right)^2}\sqrt {K_{\gl'_{(k)}}}\, K_{\gl_{(k)}}
\leq \bbV_{\gl_{(k)}}\, \sqrt{ K_{\gl_{(k)}}}\, K_{\gl'_{(k)}}+\bbV_{\gl'_{(k)}}\, \sqrt {K_{\gl'_{(k)}}}\, K_{\gl_{(k)}}
  $$
  where  we use the notations for energy and mass of subforests introduced at the beginning of this section.
%  $\bbV^2_{j}/2$ and $K_{j}$ are respectively the total energy and the total number of particles of the collision tree $j$. %and that $K_{j}$ indicates the total number of particles in the collision tree $j$
In the above inequality, we have used the independence of $\Psi^\e_{\gl_{(k)}}$ and $\Psi^\e_{\gl'_{(k)}}$ on $[\tau_{{\rm{rec}},k}, t]$, and bounded their energies in $\gd\tau_j$ with $\bbV_{\gl_{(k)}}$ and $\bbV_{\gl'_{(k)}}$ respectively (see Eq.s \eqref{eq:energia1}-\eqref{eq:energia2}).
Therefore we infer that
 \begin{equation}
\begin{aligned}
\label{recoll B1jbj'}
  | \cB_{k}| & \leq \frac{C}{\beta_0^{1/2}\mu_\eps} \int d\tau_{{\rm{rec}},k} \indc_{ \tau_{{\rm{rec}},k} \leq \tau_{{\rm{rec}},k-1}}\,
 \left(\beta_0\bbV^2_{\gl_{(k)}} + K_{\gl_{(k)}}\right)  \left(\beta_0 \bbV^2_{\gl'_{(k)}} + K_{\gl'_{(k)}}\right)\\ & 
= \frac{C}{\beta_0^{1/2} \mu_\eps} \int d\tau_{{\rm{rec}},k} \indc_{ \tau_{{\rm{rec}},k} \leq \tau_{{\rm{rec}},k-1}}\,
\sum_{\substack{j_k \in \gl_{(k)}\\ j'_k \in \gl'_{(k)}}}
 \Big(\beta_0\bbV^2_{j_k} + K_{j_k}\Big)  \left(\beta_0\bbV^2_{j'_k} + K_{j'_k}\right)\;.
\end{aligned}
\end{equation}
 In this way we have obtained an estimate which depends only on the energy and the number of particles 
 enclosed in the trees $\Psi^\e_{\gl_{(k)}},\Psi^\e_{\gl'_{(k)}}$.
 
Coming back to Equation \eqref{recalldefcum} we observe that,
if $\mb_{\gl_1} = 1$, then there exist merged subforests such that~$\hat x_k \in \cB_{k}$ for $k=\ell_1-1,\dots,1$.
Hence, iterating the procedure leading to \eqref{recoll B1jbj'} for $k=\ell_1-1,\dots,1$, leads to
an upper bound on the cost of the clustering recollisions in  $\gl_1$:
%To  this end, we are going to  all the merged clusters one by one, according to the following iteration:
%\begin{itemize}
%\item starting from $c_{\ell_1-1}$ and ``dissolving'' $\left(\gl_{(\ell_1-1)},\gl'_{(\ell_1-1)}\right)$;
%\item choosing $c_k$ at step $k$ (when we still have $\ell_1-k$ clusters) and dissolving $\left(\gl_{(k)},\gl'_{(k)}\right)$.
%\end{itemize}
%This translates into the following inequalities:
\begin{equation} 
\begin{aligned} 
\label{eq: external recollision cost trees 0}
&\int \!  \!  dX^*_{\ell_1-1} \,  \mb_{\gl_1}\, \indc_{\cG^{\e}} \big (  \Psi^\e_{\gl_1} \big)
%&= \sum_{T^\prec \in \cT^\prec_{\gl_1}} \int \!  \! dX^*_{\ell_1}\, \mb_{\gl_1, T^\prec} \, \indc_{\cG^{\e}} \big (  \Psi_{\gl_1} \big)\\
\leq \sum_{\left(\gl_{(k)}, \gl'_{(k)}\right)} \int \!  \!  d\hat x_{1} \indc_{\cB_{1}} \! %\sum_{\gl_{(2)}, \gl'_{(2)}} 
\int   \! \! d\hat x_{2} \dots \! 
% \sum_{\gl_{(\ell_1-1)}, \gl'_{(\ell_1-1)}}
\int \!  \!  d\hat x_{\ell_1-1} 
 \indc_{\cB_{\ell_1-1}} \\
 &\qquad \leq \left(\frac{C} {\beta_0^{1/2}\mu_\eps}\right)^{\ell_1-1}
 \int_0^t d\tau_{\rm{rec},1} \cdots \int_{0}^{ \tau_{\rm{rec},\ell_1-2}} d\tau_{\rm{rec},\ell_1-1}
 \sum_{\left(\gl_{(k)}, \gl'_{(k)}\right)}
\,\sum_{\substack{j_k \in \gl_{(k)}\\ j'_k \in \gl'_{(k)}}}\,
\prod_{k = 1}^{\ell_1-1}\,  \left(\beta_0\bbV^2_{j_k} + K_{j_k}\right)  \left(\beta_0\bbV^2_{j'_k} + K_{j'_k}\right) \\
& \qquad  =   \left(\frac{Ct} {\beta_0^{1/2}\mu_\eps}\right)^{\ell_1-1}\frac{1}{(\ell_1-1)!} \sum_{\left(\gl_{(k)}, \gl'_{(k)}\right)}
\,\sum_{\substack{j_k \in \gl_{(k)}\\ j'_k \in \gl'_{(k)}}}\,
\prod_{k = 1}^{\ell_1-1}\,  \left(\beta_0\bbV^2_{j_k} + K_{j_k}\right)  \left(\beta_0\bbV^2_{j'_k} + K_{j'_k}\right)\;.
\end{aligned}
\end{equation}
% Assume that $\tau_{\rm{rec},k} \in  (\e, t)$ for all $k$. Then,
%and using \eqref{recoll B1jbj'} we deduce
%$$
%\int \!  \!  dX^*_{\ell_{1}} \,  \mb_{\gl_1}\, \indc_{\cG^{\e}} \big (  \Psi_{\gl_1} \big)
%\leq   \left(\frac{Ct} {\mu_\eps}\right)^{\ell_1-1}\frac{1}{(\ell_1-1)!} \sum_{\left(\gl_{(k)}, \gl'_{(k)}\right)_{k}}
%\,\sum_{\substack{j_k \in \gl_{(k)}\\ j'_k \in \gl'_{(k)}}}\,
%\prod_{k = 1}^{\ell_1-1}\,  \left(\bbV^2_{j_k} + K_{j_k}\right)  \left(\bbV^2_{j'_k} + K_{j'_k}\right)\,,
%$$
%where the $1/(\ell_1-1)!$ comes from the simplex in times~$\tau_{\rm{rec},1}\geq \tau_{\rm{rec},2} \dots \geq \tau_{\rm{rec},\ell_1-1}$.
Using the bijection \eqref{eq:errbij} and compensating the $1/(\ell_1-1)!$ with the ordering of the edges in $T^\prec$, we rewrite this result as 
$$
\int \!  \!  dX^*_{\ell_{1}-1}   \mb_{\gl_1}\, \indc_{\cG^{\e}} \big (  \Psi^\e_{\gl_1} \big)
  \leq \left(\frac{Ct} {\beta_0^{1/2}\mu_\eps}\right)^{\ell_1-1} 
\sum_{T \in \cT_{\lambda_1}}\,
 \prod_{\{j,j'\}  \in E(T)}\, \left(\beta_0\bbV^2_{j} + K_{j}\right)  \left(\beta_0\bbV^2_{j'} + K_{j'}\right)\;,
$$
where $E(T)$ is the set of edges of $T$. Equivalently,  we obtain (\ref{eq: external recollision cost trees}).

\bigskip

\noindent
-- {\it Clustering overlaps.} We are now going to estimate  the constraints associated with   clustering overlaps in the pseudo-trajectory of the generic jungle ${\gr_1}$. Up to a renaming of the summation variables, we can assume that $$\gr_1 = \{\lambda_1,\dots , \lambda_{r_1}\}\;.$$
The number of particles in the jungle at time $t$ is $|\rho_1|$, and at time 0 is $K_{\rho_1} = |\rho_1| + m_{\rho_1}$. We recall that each forest $\lambda_i$ has a root $x^*_{\lambda_i}$, which did not play any role in the previous estimate of clustering recollisions.
We call $x^*_{\gr_1}:=x^*_{\lambda_{r_1}}$ the {\it root} of the jungle.

\begin{Prop}
Consider some forests $\lambda_1,\dots , \lambda_{r_1}$ whose internal dynamics is fixed (prescribed by the velocities and relative positions at time $t$, as well as the creation parameters).
The set of configurations $Z_{|\rho_1| } ^*$ at time $t$ compatible with the jungle $\rho_1 = \{\lambda_1,\dots , \lambda_{r_1} \}$  on $[0,t]$ satisfies the following estimate~:
\begin{equation}\begin{aligned}  \label{eq:est jungles final}
\int \!  \!  dx^*_{\lambda_1}\cdots dx^*_{\lambda_{r_{1}-1}}   |\gp_{\gr_1}|
  \leq \left(\frac{C}{\beta_0^{1/2} \mu_\eps}\right)^{r_1-1} \,\left(t + \e\right)^{r_1-1}
\sum_{T \in \cT_{\gr_1}}\,\prod_{\lambda_j \in \gr_1}\,\left(\beta_0\bbV^2_{\lambda_j} + K_{\lambda_j}\right)^{d_{\lambda_j}(T)}\;.
\end{aligned}
\end{equation}
\end{Prop}

The argument is similar, but not identical, to the one just seen for clustering recollisions. Below we shall indicate the differences, without repeating the identical parts.

By definition of $\gp_{\gr_1}$, and by Definition \ref{def:ClustOv}, the clustering overlaps are extracted from the graph of all overlaps between the forests $ \{\lambda_1,\dots , \lambda_{r_1} \}$ via the Penrose algorithm~:
we denote by  $(\lambda_{j_1},\lambda_{j'_1}),\cdots,(\lambda_{j_{r_1-1}}, \lambda_{j'_{r_1-1}})$ the (ordered) edges of the resulting minimally connected graph $T \in \cT_{\gr_1}$. Then,  thanks to the tree inequality stated in Proposition~\ref{prop: tree inequality},
\begin{equation} \label{eq:ov cond dec trees}
|\gp_{\gr_1}|  \leq 
\sum_{T \in \cT_{\gr_1}}\, \prod_{ \{\lambda_{j},\lambda_{j'}\}  \in E(T )}\,   \indc_{ \lambda_j \sim_o \lambda_{j'} } \,.
\end{equation}

%Given a $T$, we equip the graph with an arbitrary ordering of the edges, i.e.\, we fix $T^\prec\in \cT^{\prec}_{\gr_1}$. For any
%such $T^\prec$ we can proceed exactly as we did for clustering recollisions (cf.\,\eqref{eq:rec cond dec trees}).

Note  that, as mentioned in Section \ref{sec:dyncum}, we have  more flexibility when dealing with overlaps than with recollisions,  as~$\left(\Psi^\e_{\lambda_j}\right)_{1 \leq j \leq r_1}$ are completely independent trajectories, whatever the ordering of the overlap times.  We therefore have more freedom in choosing the integration variables.

%To define the change of variables, we assign an ordering of the edges $E(T)$ in the following way. Consider $T \in \cT_{\gr_1}$ as a rooted graph, with root $\lambda_{\rho_1}$. We start from the vertices of $T$ which have the maximal depth, say $\bar k$ (the depth is defined as the number of edges connecting the vertex to the root). These vertices have degree 1, hence each one of the vertices identifies exactly one edge. We label these edges in such a way that they keep the same mutual order of the vertices, starting from the biggest one. We rename the ordered edges as $e_{r_1-1}, e_{r_1-2}, \dots$. Next we prune the edges, obtaining a smaller minimally connected tree graph, on which we can repeat the labelling operation. We iterate this procedure $\bar k$ times, producing a complete ordering of edges $e_{r_1-1},  \dots, e_1$. 
%
%Let us write $e_k = \{\gl_{[k]},\gl'_{[k]}\}$ where $\gl'_{[k]}$ has
%depth larger than $\gl_{[k]}$. 
%

We can then define 
$$\hat x_k := x^*_{\gl'_{[k]}} - x^*_{\gl_{[k]}}\,,\qquad k=1,\dots,r_1-1$$
as the relative position between the two overlapping forests at time $t$.
As in the case of clustering recollisions, for any given root position
$x^*_{\gr_1}:=x^*_{\lambda_{r_1}} \in \T^d$, the map of translations
\begin{equation} \label{eq:bijectionOV}
\left(x^*_{\gl_1},\dots, x^*_{\gl_{r_1-1}}\right)  \longmapsto  \hat X_{r_1-1} := 
\left(\hat x_1,\dots,\hat x_{r_1-1}\right) 
\end{equation}
is one-to-one on $\T^{d(r_1-1)}$ and it has unit Jacobian  determinant. 
Thus \eqref{eq:bijectionOV} is a legitimate change of variables
in \eqref{recalldefcum}.

Given a graph $T \in \cT_{\gr_1}$ and the corresponding sequence $\left(\gl_{[k]}, \gl'_{[k]}\right)_{k}$,
the $k-$th clustering overlap condition is defined by
$$\hat x_k \in \tilde\cB_{k} := \bigcup_{\tiny\substack{ \mbox{$q$ in the forest $\gl_{[k]}$} \\ \mbox{$q'$ in the forest $\gl'_{[k]}$}}} \tilde B_{qq'}\,,$$ 
with
$$\tilde B_{qq'} = \Big\{\hat x_k \in \T^d \ : \ \exists \tau \in [0,t] \quad  \text{such that}  \quad 
|x_q (\tau) - x_{q'} (\tau)  | \leq \eps \Big\}
$$
where we used \eqref{eq:deftauOV}, and $x_q (\tau), x_{q'} (\tau)$ are the particle trajectories in the flows $\Psi^\e_{\gl_{[k]}},\Psi^\e_{\gl'_{[k]}}$.
This set has small measure
\begin{equation}
\label{tildeBk}
  | \tilde\cB_{k}| \leq \frac{C}{\beta_0^{1/2} \mu_\eps} \left(t+ \e\right)
%  \sum_{
%  \substack{
%  \lambda_{j_k} \in \gr_{(k)} \\ \lambda_{j'_k} \in \gr'_{(k)}
%  }
%  }
 \left(\beta_0\bbV^2_{\lambda_{[k]}} + K_{\lambda_{[k]}}\right)  \left(\beta_0\bbV^2_{\lambda'_{[k]}} + K_{\lambda'_{[k]}}\right)
   \end{equation}
  for some  constant $C>0$. Notice that the correction of $O(\e)$ comes from the extremal spherical caps of the tubes in Figure \ref{fig:tube} (since $\indc_{ \lambda_{[k]} \sim_o \lambda'_{[k]} } =1$ inside those regions).

  \begin{Rmk}\label{overlap-rmk}
  Note that overlaps can be classified in two types
  \begin{itemize}
  \item those arising at time $t$ or involving a particle $q$ at its creation time $t_q$~: in this case, the distance between the overlapping particles at $\tau_{\rm{ov}}$ satisfies only  the inequality
  $$ |x_q(\tau_{\rm{ov}}) - x_{q'} (\tau_{\rm{ov}})| \leq \eps\, .$$
  This corresponds to one spherical end of the tube in Figure {\rm\ref{fig:tube}};
  \item and the {\rm regular} ones, for which the two overlapping particles are exactly at distance $\eps$ at $\tau_{\rm{ov}}$. We then have the same parametrization as for recollisions
  \begin{equation}
  \label{eq:omov}
   x_q(\tau_{\rm{ov}}) - x_{q'} (\tau_{\rm{ov}})= \eps \omega_{\rm{ov}}\, .
   \end{equation}
  This corresponds to the tube in Figure {\rm\ref{fig:tube}} minus the spherical end.
  \end{itemize}
  \end{Rmk}

We finally obtain (\ref{eq:est jungles final}).

%Notice that the argument leading to \eqref{eq: external recollision cost trees 0} and \eqref{eq:est jungles final} is equivalent to the change of measure
%\begin{equation}
%d\hat X_n = \mu_\e^{n-1}
%\end{equation}

\noindent
\noindent
-- {\it Initial clustering.} 
Finally, we are going to estimate the non-overlap constraints in the initial data, which are encoded in \eqref{eq : cumulants time 0}.

Recall that $f^{\e 0}_{\{1,\dots,r\}}(\Psi^{\e 0}_{\gr_1},\dots,\Psi^{\e 0}_{\gr_r})$ is a measure of the correlations between all the different clusters of particles $\Psi^{\e 0}_{\gr_1},\dots,\Psi^{\e 0}_{\gr_r}$ at time zero, and its definition has been adapted to reconstruct the dynamical cumulants.  An estimate of this correlation is obtained  by integrating over the root coordinates of the jungles $x^*_{\gr_1}, \dots, x^*_{\gr_{r-1}}$, as stated in the following proposition.

We recall that~$K_{\gr_i}: = m_{\gr_i} + |\gr_i|$  denotes the number of particles in  the configuration~$\Psi^{\eps,0}_{\gr_i}$ at time~0, and that~$\displaystyle  K:=\sum_{i=1}^r K_{\gr_i} = m+n$.  
\begin{Prop}
\label{prop:ID} 
Under Assumption~{\rm(\ref{lipschitz})}, there exists $C>0$ (depending only on the dimension $d$) such that, for $\e$ small enough, 
$$\int_{\T^{d (r-1)}} |f^{\e 0}_{\{1,\dots,r\}}(\Psi^{\e 0}_{\gr_1},\dots,\Psi^{\e 0}_{\gr_r})| \, dx^*_{\gr_{1}}\dots dx^*_{\gr_{r-1}}
\leq (r-2)!\, (CC_0)^K\,\exp\Big(-\frac{\b_0}{2} {\mathbb{V}}^2\Big)\, \e^{d({r - 1)}}
$$
for all $\Psi^{\e 0}_{\gr_i} \in {\mathcal D}^{\eps}_{K_{\gr_i}}$ at time~0. We have used   the convention $0! = (-1)! = 1$.
\end{Prop}
Recall that $f^{\e 0}_{\{1,\dots,r\}}$ is extended to 
$\D^{K} \setminus {\mathcal D}^{\eps}_{K}$ by setting $F^{\e 0}_{\omega_i} =0$ in \eqref{eq : cumulants time 0} wherever it is not defined.

The following proof is an application of known cluster expansion techniques, see e.g. \cite{PU09} and references therein.

\begin{proof} 
Set $Z_K := (\Psi^{\e 0}_{\gr_1},\dots,\Psi^{\e 0}_{\gr_r})$ with $\Psi^{\e 0}_{\gr_i}\in {\mathcal D}^{\eps}_{K_{\gr_i}}$ at time~0. To make notation lighter we  shall omit the superscripts $^{\eps 0}$ and also omit to specify the exclusion constraints inside each $\Psi^\e_{\gr_i}$ in the sequel. We define~$\Phi_{r+p}$   the indicator function of the mutual exclusion between the elements of the set~$\{\Psi^\e_{\gr_1},\dots,\Psi^\e_{\gr_r},\bar z_1,\dots,\bar z_p\}$ (where~$\Psi^\e_{\gr_1},\dots,\Psi^\e_{\gr_r}$ form~$r$ clusters and~$\bar z_1,\dots,\bar z_p$ are the configurations of~$p$ single particles):
\begin{equation*}
\label{eq:defPhi0}
\Phi_{r+p} = \prod_{h\neq h'}\indc_{\eta_h \not \sim \eta_{h'} }\,,
\end{equation*}
with $( \eta_1, \dots, \eta_{r+p} ) = (\Psi^\e_{\gr_1},\dots,\Psi^\e_{\gr_r},\bar z_1,\dots,\bar z_p)$ and $``\eta_h \not \sim \eta_{h'} "$ meaning that the minimum distance between elements of $\eta_h$ and $\eta_{h'}$ is   larger than $\eps$.  So we    start from 
\begin{equation}
\label{eq:F0mdef}
F^{\e 0}_K(Z_K) = \frac{(f^0)^{\otimes K}(Z_K)}{\cZ^\eps} \sum_{p \geq 0} \frac{\mu_\e^{p}}{p!} \int_{\D^p} (f^{0})^{\otimes p}(\bar Z_p) \,\Phi_{r+p}(\Psi^\e_{\gr_1},\dots,\Psi^\e_{\gr_r},\bar Z_p)\, d\bar Z_p\,.
\end{equation}
We want to expand~$\Phi_{r+p}$ in order to compensate the factor~$\cZ^\eps$ whose definition we recall 
\begin{equation}\label{defZeps}
\cZ^\eps :=  \sum_{p\geq 0}\frac{\mu_\eps^p}{p!}  
\int_{\D^p} (f^0)^{\otimes p}(\bar Z_p) \,\Phi_{p}(\bar Z_p)\, d\bar Z_p\,,
\end{equation}
and to identify the elements in the decomposition
$$
F^{\e 0}_K(\Psi^\e_{\gr_1},\dots,\Psi^\e_{\gr_r}) = \sum_{s = 1}^{r} \sum_{\gs \in \cP^s_r}  \prod_{i=1}^s f^{\eps 0}_{|\gs_i|}(\Psi^\e_{\gs_i})\,.
$$
This will enable us to compute, and estimate,~$f^{\e 0}_{\{1,\dots,r\}}(\Psi^\e_{\gr_1},\dots,\Psi^\e_{\gr_r})$.
To do so, we naturally develop~$\gP_{r+p}$ into $s$  clusters (each of them corresponding to one connected graph
containing at least one element of $\{\Psi^\e_{\gr_1},\dots,\Psi^\e_{\gr_r}\}$), plus a background $\bar\gs_0$ of mutually excluding particles (for which we do not expand the exclusion condition).
Such a partition can be reconstructed isolating first the background component, and then splitting $\{\Psi^\e_{\gr_1},\dots,\Psi^\e_{\gr_r}\}$ in $s$ parts, to which we adjoin the remaining  single particles  (see Figure~\ref{fig: initialclusters}).

\begin{figure}[h] %  figure placement: here, top, bottom, or page
\centering
\includegraphics[width=5in]{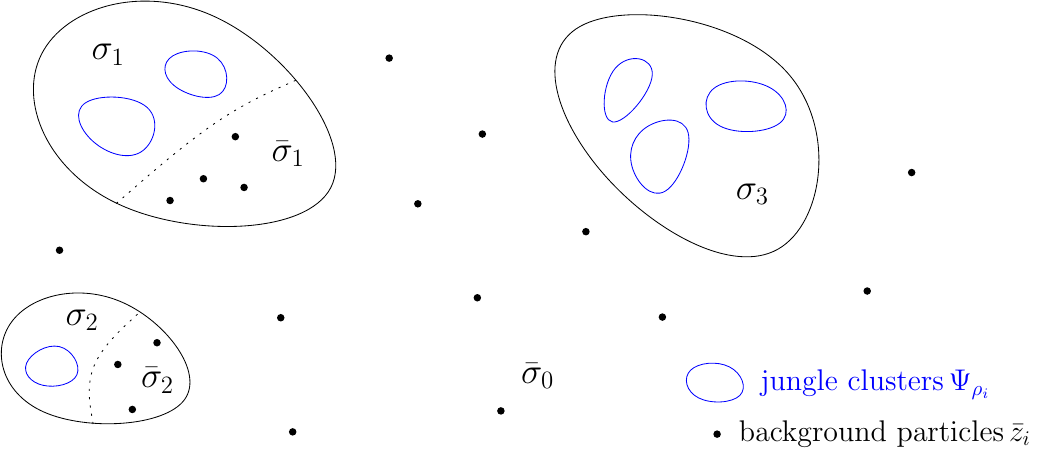} 
\caption{\small 
Initial configurations are decomposed in $s$ clusters containing at least one jungle~$\Psi^\e_{\gr_1},\dots,\Psi^\e_{\gr_r}$, plus a background of mutually excluding particles (for which we do not expand the exclusion condition).
}
\label{fig: initialclusters}
\end{figure}

This amounts to introducing truncated functions $\gp$ via the following formula:
\begin{equation}
\label{eq:decID}
\Phi_{r+p}  (\Psi^\e_{\gr_1},\dots,\Psi^\e_{\gr_r},\bar Z_p) \!=\!\!
\sum_{  \bar\gs_0\subset\{1,\dots,p\} }\gP_{|\bar\gs_0|} (\bar Z_{\bar\gs_0})
\sum_{s =1}^{r} \sum_{\gs \in \cP^s_r} \,
\sum_{
\substack{
 \bar\gs_1,\dots,\bar\gs_s \subset\{1,\cdots,p\} \\
\cup_{i=0}^s \bar\gs_i = \{1,\dots,p\} \\
\bar\gs_k \cap \bar\gs_h = \emptyset, k \neq h
}
%\bar\gs \in \cP^s_{ \{1,\dots,n\} \setminus \bar\gs_0}
}\,
%\sum_{j_1,\dots,j_s }\,
\prod_{i=1}^s\gp(\Psi^\e_{\gs_i}, \bar Z_{\bar\gs_{i}})\;.
\end{equation}
Note that the~$\bar \sigma_i$ may be empty (in particular all $\bar \sigma_i$ are empty if~$| \bar\gs_0| = p$).
%By writing $\indc_{\eta_h \not \sim \eta_{h'} }= 1-\indc_{\eta_h \sim \eta_{h'} }$ in \eqref{eq:defPhi0} and expanding the product, 
By \eqref{graph-representation}, we see that 
%$\gp$ has the explicit expression
\begin{equation*}
\label{eq:decIDgraphs}
\gp(\Psi^\e_{\gr_1},\dots,\Psi^\e_{\gr_r}, \bar Z_{p}) =
\sum_{G \in \cC_{r+p}} \prod_{(h,h')\in E(G)}(- \indc_{\eta_h\sim \eta_{h'} })\,,
\end{equation*}
where the sum runs over the set of connected graphs with $r+p$ vertices; more generally,
\begin{equation*}
\gp(\Psi^\e_{\gs_i},\bar Z_{\bar\gs_{i}}) =
\sum_{G \in \cC_{|\gs_i|+ |\bar\gs_{i}|}} \prod_{(h,h')\in E(G)}(- \indc_{\eta_h\sim \eta_{h'} })\,.
\end{equation*}
 Using the symmetry in the exchange of particle labels, we get, denoting~$\bar s_i :=|\bar\gs_i|$,
$$ 
 \binom{p}{ \bar s_1}
 \binom{p-\bar s_1}{ \bar s_2}   \dots  \binom{p- \bar s_1- \dots-\bar s_{s-1}}{ \bar s_s}    =  {p! \over \bar s_0 ! \; \bar s_1 ! \dots \bar s_s !}
$$
choices for the repartition of the background particles, so that
$$
\sum_{p \geq 0} \frac{1}{p!}  \int_{\D^p} \,\Phi_{r+p}(\Psi^\e_{\gr_1},\dots,\Psi^\e_{\gr_r},\bar Z_p)\, d\bar Z_p
 = \sum_{s=1}^r \sum_{\gs \in \cP^s_r}
\sum_{p \geq 0} 
\sum_{\substack{\bar s_0,\dots,\bar s_s \geq 0\\ \sum \bar s_i = p}}
\int_{\D^p} \frac{\Phi_{\bar s_0} (\bar Z_{ \bar s_0})}{\bar s_0!}
\prod_{i=1}^s \frac{ \gp(\Psi^\e_{\gs_i}, \bar Z_{\bar s_i})}{\bar s_i !}
d \bar Z_{p}\;.
$$
Therefore, plugging \eqref{eq:decID} into \eqref{eq:F0mdef} first and then  using~(\ref{defZeps}), we obtain
$$
\begin{aligned}
F^{\e0}_K (Z_K)&= \frac{(f^0)^{\otimes K} (Z_K)}{\cZ^\eps} 
\sum_{s=1}^r \sum_{\gs \in \cP^s_r}
\sum_{p \geq 0} 
\sum_{\substack{\bar s_0,\dots,\bar s_s \geq 0\\ \sum \bar s_i = p}}
\left( \frac{\mu_\e^{\bar s_0}}{\bar s_0 !} \int (f^0)^{\otimes \bar s_0} (\bar Z_{ \bar s_0})\Phi_{\bar s_0} (\bar Z_{ \bar s_0})d\bar Z_{\bar s_0}\right)\\
& \nonumber
\ \ \ \ \ \ \ \ \ \ \ \ \ \ \ \ \ \ \ \ \ \ \ \ \ \ \ \times\prod_{i=1}^s \frac{\mu_\e^{\bar s_i}}{\bar s_i !}\int (f^0)^{\otimes \bar s_i}   (\bar Z_{ \bar s_i}) \gp(\Psi^\e_{\gs_i}, \bar Z_{\bar s_i})d \bar Z_{\bar s_i}\\
& =(f^0)^{\otimes K}(Z_K)
\sum_{s=1}^r \sum_{\gs \in \cP^s_r}
\prod_{i=1}^s
\sum_{\bar s_i \geq 0} \frac{\mu_\e^{\bar s_i}}{\bar s_i !}\int (f^0)^{\otimes \bar s_i}   (\bar Z_{ \bar s_i}) \gp(\Psi^\e_{\gs_i}, \bar Z_{\bar s_i})d \bar Z_{\bar s_i}\,, \nonumber
\end{aligned}
$$
hence finally
\begin{equation}\label{f01r}
f^{\e 0}_{\{1,\dots,r\}}(\Psi^\e_{\gr_1},\dots,\Psi^\e_{\gr_r}) = (f^0)^{\otimes K}(Z_K)
\sum_{p \geq 0} \frac{\mu_\e^{p}}{p !}\int (f^0)^{\otimes p} ( \bar Z_{p})  \gp(\Psi^\e_{\gr_1},\dots,\Psi^\e_{\gr_r}, \bar Z_{p})d \bar Z_{p}\,.
\end{equation}
%Notice that \eqref{eq:decID} is different from the cumulant decomposition described in Section \ref{sec:Cumulants}.
%On the other hand, to control the growth of $\gp$ and the sum over $q$ in the above formulas, we shall use a more refined procedure than Proposition \ref{prop: cumulants support}. 

%The number of terms in \eqref{eq:decIDgraphs} grows badly with $q$. However, there are many cancellations between these terms. 
Applying again Proposition~\ref{prop: tree inequality} implies   that~$\gp$ is bounded by 
\begin{equation}
\label{eq:treeineq}
|\gp(\Psi^\e_{\gr_1},\dots,\Psi^\e_{\gr_r}, \bar Z_{p})| \leq \sum_{T \in \cT_{r+p}} \prod_{(h,h')\in E(T)}\indc_{\eta_h\sim \eta_{h'} }
\end{equation}
where $\cT_{r+p}$ is the set of minimally connected graphs with $r+p$ vertices labelled by $\Psi^\e_{\gr_1},\dots,\Psi^\e_{\gr_r}, \bar z_{1}, \dots, \bar z_p$.
%Such a type of estimate, going back to Penrose~\cite{Pe67}, exploits the nontrivial cancellations between graphs in \eqref{eq:decIDgraphs} to eliminate all the cycles (see~\cite{PY17} for a recent improvement).
% O. Penrose (1967): Convergence of fugacity expansions for classical systems. In Statistical mechanics: foundations and applications, A. Bak (ed.), Benjamin, New York.
% Aldo Procacci, Sergio A. Yuhjtman. Letters in Mathematical Physics January 2017, Volume 107, Issue 1, pp 31ÃÂ46 | CiteConvergence of Mayer and Virial expansions and the Penrose tree-graph identity

By Lemma \ref{lem: Cayley+}, the number of minimally connected graphs with 
specified vertex degrees~$d_1,\dots,d_{r+p}$ is given by 
$$ {(r+p-2)!} / {\displaystyle \prod_{i=1}^{r+p} (d_i-1)!} \,. $$
On the other hand, the product of indicator functions in \eqref{eq:treeineq} is a sequence of~$r+p-1$ constraints, confining the space coordinates to  balls of size $\eps$ centered at the positions of 
the clusters $\Psi^\e_{\gr_1},\dots,\Psi^\e_{\gr_r},\bar z_1,\dots,\bar z_p$. 
Such clusters have cardinality $K_{\gr_1},\dots, K_{\gr_r} \geq 1$ with the constraint~$$\sum_i K_{\gr_i} = K\;.$$
We deduce that for some $C>0$ depending only on the dimension $d$
$$
\begin{aligned}
&\int_{\T^{d (r-1)}} |f^{\e0}_{\{1,\dots,r\}}(\Psi^\e_{\gr_1},\dots,\Psi^\e_{\gr_r})| dx^*_{\gr_{1}}\dots dx^*_{\gr_{r-1}}\\
&\nonumber
\qquad \leq (CC_0)^K  \e^{d(r-1)}e^{-\frac{\b_0}{2} {\mathbb{V}}^2} \sum_{p \geq 0}\frac{(r+p-2)!}{p!}(CC_0 \eps^d \mu_\eps )^p \sum_{d_1,\dots,d_{r+p} \geq 1} \frac{\prod_{i=1}^r K_{\gr_i}^{d_i}}{\prod_{i=1}^{r+p}(d_i-1)!} \\
&\nonumber
\qquad \leq (CC_0)^K  \e^{d(r-1)}e^{-\frac{\b_0}{2} {\mathbb{V}}^2} \sum_{p \geq 0}\frac{(r+p-2)!}{p!}(C_0  \eps^d \mu_\eps )^p 
\,e^{2K + p}\\
&\nonumber
\qquad \leq (CC_0) ^K  \e^{d(r-1)}e^{-\frac{\b_0}{2} {\mathbb{V}}^2} 2^{r-2}(r-2)!\sum_{p \geq 0}(C C_0 \eps^d \mu_\eps )^p 
\,e^{2K + p}\;.
\end{aligned}
$$
In the second inequality we used that 
$$\prod_{i=1}^r\sum_{d_i \geq 1} \frac{ K_{\gr_i}^{d_i}}{(d_i-1)!} \leq \prod_{i=1}^r K_{\gr_i} e^{K_{\gr_i}}
\leq \prod_{i=1}^re^{2 K_{\gr_i}} = e^{2K}\;.
$$
Since $C \eps^d \mu_\eps $ is arbitrarily small with $\eps$, 
this proves Proposition~\ref{prop:ID}.
\end{proof}

\section{Decay estimate for the  cumulants}$ $
\label{subsec:proofThm}
\setcounter{equation}{0}

We shall now prove the bound provided in Theorem~\ref{cumulant-thm1*}.
In the previous section, we   considered  a nested partition $\gl \hookrightarrow \gr\hookrightarrow \gs$ (with $|\gs|=1$) of the set $\{1^*,\dots, n^*\}$.
We   fixed the velocities $V_n^*$  as well as the collision parameters of the pseudo-trajectories~$(m,a,T_m, V_m, \Omega_m)$. We   then exhibited $n-1$ ``independent" conditions 
on the  positions~$X_n^*$  for the pseudo-trajectories to be compatible with the partitions~$\lambda, \rho$.
Now we shall conclude the proof of Theorem~\ref{cumulant-thm1*}, by integrating successively on all the available parameters. The order of integration is pictured in Figure~\ref{fig: orderofintegration}.

 \begin{figure}[h] % figure placement: here, top, bottom, or page
\centering
\includegraphics[width=6in]{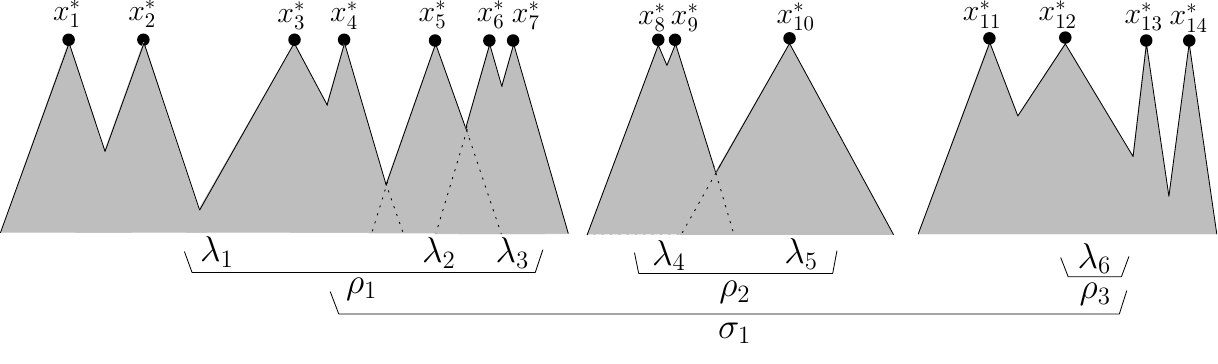}
\caption{ In this contribution to the cumulant of order $n=14$, we integrate over the positions of the roots in the following order: 
(i)  first we integrate over the initial
clustering~$\hat x _{\rho_2} = x^*_{10} - x^*_{14}$ and $\hat x_{\rho_1} = x^*_{7}- x^*_{14}$;
(ii)  secondly over the clustering overlaps $\hat x _{\lambda_4} =  x^*_{9} - x^*_{10}$ and $\hat x_{\lambda_{1}}= x^*_4- x^*_5\,, \hat x_{\lambda_{2}} = x^*_{5}-x^*_7$; 
(iii) finally over the clustering recollisions~:
$\hat x^{(\lambda_1)}_3 = x_2^* - x_3^*, \quad  \hat x^{(\lambda_1)}_2 = x_1^* - x_2^*, \quad \hat x^{(\lambda_1)}_1 = x_3^* - x_4^*,\quad \hat x^{(\lambda_3)}_1 = x_6^* - x_7^*, \quad \hat x^{(\lambda_4)}_1 = x_8^* - x_9^*, \quad \hat x^{(\lambda_6)}_3 = x^*_{13} - x^*_{14} ,\quad \hat x^{(\lambda_6)}_2 =  x_{12}^*- x_{13}^*, \quad \hat x^{(\lambda_6)}_1 =  x_{11}^*- x_{12}^*$.
Notice that the variable~$ x^*_{14}$ remains free. }
\label{fig: orderofintegration}
\end{figure}

For the proof of the first two statements in Theorem \ref{cumulant-thm1*}, we start by controlling the weight, simply using the bounds
\begin{equation}\label{controlHnexpbis}
| \cH(\Psi^\eps_n)| \leq \prod_{i=1}^n \| H^{(i)}\|_\infty\quad 
\hbox{ or }\quad | \cH(\Psi^\eps_n)| \leq  e^{ {\alpha  n} + \frac{\beta_0}4 \bbV^2} \, .
\end{equation}

Then we use that
nothing depends on the root coordinates of the jungles $x^*_{\gr_1},\dots, x^*_{\gr_{r-1}}$ inside the integrand in \eqref{recalldefcum},
except the initial datum $f^{\e 0}_{\{1,\dots,r\}}$. Therefore by Fubini and according to  Proposition~\ref{prop:ID}, 
\begin{equation}
\label{eq:ind'}
\int_{\T^{d (r-1)}} |f^{\e 0}_{\{1,\dots,r\}}(\Psi^{\e 0}_{\gr_1},\dots,\Psi^{\e 0}_{\gr_r})| dx^*_{\gr_{1}}\dots dx^*_{\gr_{r-1}}
\leq (r-2)!\, (CC_0)^K\,\exp\Big(-\frac{\b_0}{2} {\mathbb{V}}^2\Big)\, \e^{d({r - 1)}}
\end{equation}
for some~$C>0$, uniformly with respect to all other parameters.

Next, the clustering condition on the jungles gives an extra smallness when integrating over the roots of the forests
(see \eqref{eq:est jungles final})
\begin{equation}
\label{eq:ov'}
 \prod_{i =1}^r\int |\gp_{ \gr_i}| \;  \prod_{j= 1}^{r_i-1} dx^*_{\lambda_{j}}
 \leq   \left(\frac {C} {\beta_0^{1/2} \mu_\eps}\right)^{\ell - r} \,\left(t + \e\right)^{\ell-r}\,
\prod_{i=1}^r\,\sum_{T \in \cT_{\gr_i}}\, \,\prod_{\lambda_j \in \gr_i}\,\left(\beta_0 \bbV^2_{\lambda_j} + K_{\lambda_j}\right)^{d_{\lambda_j}(T)}\;,
 \end{equation}
uniformly with respect to all other parameters, for some possibly larger constant $C$.

The clustering condition on the forests gives finally an extra smallness when integrating over the remaining variables $\hat x_k$, according to \eqref{eq: external recollision cost trees}. 
Notice however that the latter inequality cannot be directly applied to \eqref{eq: decomposition cumulant}, due to the presence of the 
cross section factors \eqref{eq: cross section} in the measure \eqref{eq: measure nu}. 

It is then useful to combine the estimate with the sum over trees $  a_{|\lambda_i}$.  The argument is depicted in Figure~\ref{fig: timeslices}.
We will present the arguments for~$\lambda_1$,  assuming without loss of generality that $\lambda_1 = \{1,\dots , \ell_1\}$. %and that up to a permutation, the order of clustering recollisions is given  by $\lambda_{1,j} = j$.
We will denote by $\tilde a$ the restriction of the tree $a$ to $\lambda_1$ with fixed total numbers of particles $K_1,\cdots,K_{\ell_1}$, and by $\tilde a_{k}$, $\cC_k$  the tree variables and the cross section factors associated with the~$s_k $ creations occurring in the time interval $(\tau_{\rm{rec},k},\tau_{\rm{rec},k-1})$ for $ 1\leq k \leq \ell_1$. 

\begin{figure}[h] %  figure placement: here, top, bottom, or page
\centering
\includegraphics[width=5in]{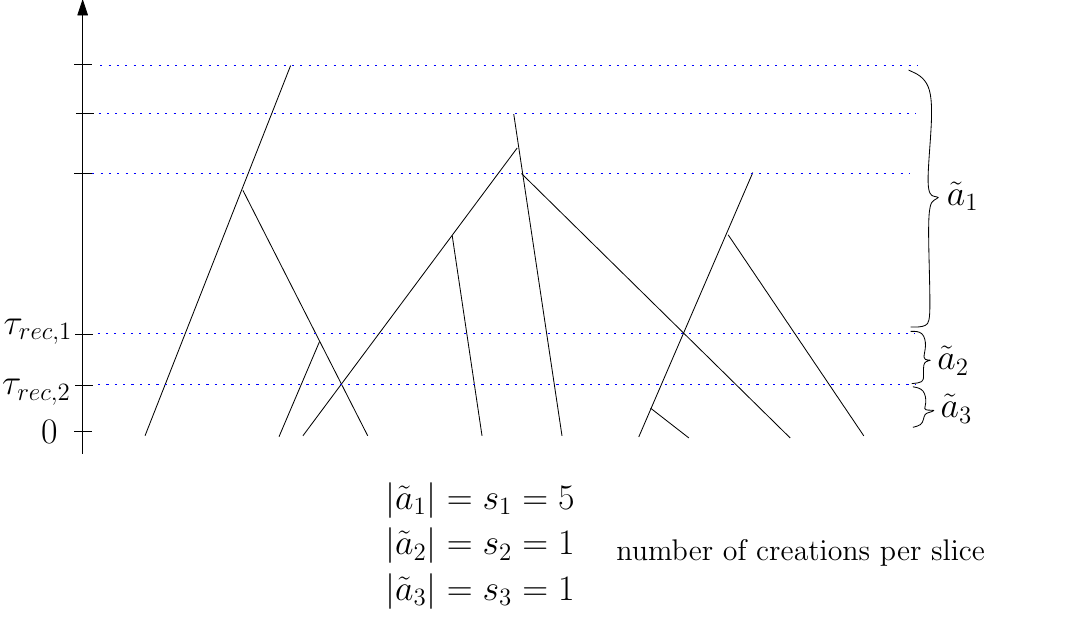} 
\caption{\small 
Integration over time slices. 
}
\label{fig: timeslices}
\end{figure}

As in the first line of \eqref{eq: external recollision cost trees 0}, we have that
\begin{equation}
\begin{aligned}
\label{eq:iterCR}
& \sum_{\tilde a }\int \!  \!  dX^*_{\ell_{1}-1} \, \mb_{\gl_1}\, \indc_{\cG^{\e}} \big (  \Psi^\eps _{\gl_1} \big) |\cC \big( \Psi^\eps_{\gl_1} \big)| \\
& \qquad
\leq \sum_{\left(\gl_{(k)}, \gl'_{(k)}\right)}\sum_{\tilde a_{1}} |\cC^\eps_{1}\, \big( \Psi_{\gl_1} \big)|\int \!  \!  d\hat x_{1} \indc_{\cB_{1}} \!  \sum_{\tilde a_{2}} |\cC_{2} \big( \Psi^\eps_{\gl_1} \big)| \int   \! \! d\hat x_{2} \dots \!  \int \!  \!  d\hat x_{\ell_1-1} 
 \indc_{\cB_{\ell_1-1}}\, \sum_{\tilde a_{\ell_1}} |\cC_{\ell_1} \big( \Psi^\eps_{\gl_1} \big)|\,.
\end{aligned}
\end{equation}

 We can therefore apply iteratively the inequality \eqref{recoll B1jbj'} and the classical Cauchy-Schwarz argument used in Lanford's proof. Denote by
 $$
 S_k := \sum_{i=1}^k s_i
 $$
the number of particles added before time $\tau_{\rm{rec},k}$, so that $$S_{\ell_1} = m_{\lambda_1}$$
(denoting abusively $\tau_{\rm{rec},\ell_1} = 0$). We get \begin{equation}
\label{eq:STE}
\begin{aligned}
\sum_{ \tilde a_{k}} \big| \cC_k \big( \Psi_{\lambda_1}\big)\big| &\leq   \prod_{s=S_{k-1}+1}^{S_k} \left( \sum_{u=1}^{s -1} | v_{ s } - v_u(t_{ s })
| +  \sum_{u=1}^{\ell_1 } | v_{ s } - v_u^*(t_{ s })
|
\right)\\
&\leq   \prod_{s=S_{k-1}+1}^{S_k} \left((\ell_1+s -1) |v_{ s }|  + \sum_{u=1}^{  s -1} |  v_u(t_{ s })| + \sum_{u=1}^{  \ell_1} |  v_u^*(t_{ s })| \right)\\
&\leq  {1\over \beta_0^{s_k/2}}  \prod_{s=S_{k-1}+1}^{S_k} \left((\ell_1+m_{\lambda_1} )(1+ \beta_0^{1/2} |v_{ s }| )  +  \beta_0 |\bbV_{\lambda_1}|^2 \right)
\end{aligned}
\end{equation}
and
\begin{equation}
\label{eq:iterCR'}
\begin{aligned}
\sum_{\tilde a }\int \!  \!  dX^*_{\ell_{1}-1} \, \mb_{\gl_1}\, \indc_{\cG^{\e}} & \big (  \Psi_{\gl_1} \big) |\cC \big( \Psi_{\gl_1} \big)|
 \leq \left(  \frac {C } {\beta_0^{1/2} \mu_\eps}   \right)^{\ell_1-1}
{\color{black} \left( \frac{1}{\beta_0} \right)^{m_{\lambda_1} /2} }
\,\left(t + \e\right)^{\ell_1-1}\,\\ 
&
\times \sum_{T \in \cT_{\lambda_1}} \,\prod_{ j \in \lambda_1}\,\left(\beta_0\bbV^2_{j} + K_{j}\right)^{d_j(T)}
 \prod_{s=1}^{m_{\lambda_1}} \left((\ell_1+m_{\lambda_1} )(1+ \beta_0^{1/2} |v_{ s }| )  + \beta_0  |\bbV_{\lambda_1}|^2 \right), 
\end{aligned}
\end{equation}
for some positive $C$.

Recall that 
$$
\exp \left(-\frac {\beta_0} {16m} |V | ^2 \right)\beta_0  | V |^2  \leq  Cm.
$$
Combining  (\ref{eq:iterCR'}) with the bound~(\ref{controlHnexpbis}) on $\cH$, \eqref{eq:ind'} and \eqref{eq:ov'} 
leads therefore to  \begin{equation}
\label{est1}
\begin{aligned}
&\int\, \Big|\,\sum_a \,
  \prod_{i=1}^{\ell}\,\mb_{\gl_i} \; \cC \big( \Psi^\eps_{\gl_i} \big) \; \indc_{\cG^{\e}} \big (  \Psi^\eps_{\gl_i} \big)\cH \big( \Psi^\eps_{\gl_i} \big)    
 \gp_{ \gr} \;    f^{\eps 0}_{\{1,\dots,r\}}(\Psi^{\eps 0}_{\gr_1},\dots,\Psi^{\eps 0}_{\gr_r})
 \Big| \,dX^*_n\, \\
 &\qquad \leq \, 
 (r-2)! \,(CC_0) ^K\,\exp \big(\alpha  n -\frac{ \beta_0} 8 \bbV ^2\big)\, \eps^{d(r-1) }\, \left(\frac{C}{\beta_0^{1/2} \mu_\e}\right)^{n-r}
 \, \left( t + \eps \right)^{n-r}  \\
 & \ \ \ \ \qquad\qquad\times \left(
 \prod_{i=1}^r\sum_{T \in \cT_{\gr_i}}\, \,\prod_{\lambda_j \in \gr_i}\,\left(\beta_0 \bbV^2_{\lambda_j} + K_{\lambda_j}\right)^{d_{\lambda_j}(T)}\right)\left(\prod_{i=1}^\ell\sum_{T \in \cT_{\lambda_i}}\, \,\prod_{ j \in \lambda_i}\,\left(\beta_0 \bbV^2_{j} + K_{j}\right)^{d_j(T)}\right)
\\ 
& \ \ \ \ \qquad\qquad\times (m+n)^m \;
{\color{black} \left( \frac{1}{\beta_0} \right)^{m/2} } \; \prod_{s = 1} ^m (1+\beta_0^{1/2}|v_s|)  \,,
\end{aligned}
\end{equation}
valid uniformly with respect to all other parameters. Here and below, we indicate by $C$ a large enough constant, depending only on the dimension $d$ and changing from line to line.

The following step then consists in integrating (\ref{est1})  with respect to the remaining parameters $(T_m  ,\Omega_m , V_m)$ and $V^*_n$ (with~$m$ fixed for the time being).
Recalling the condition that~$ t_{ 1} \geq t_{ 2}\geq  \dots \geq t_{ m}$, we get
% recalling also that~$K=N+M$
$$
\begin{aligned}
& \int\, \Big|\,\sum_a \,
  \prod_{i=1}^{\ell}\,\mb_{\gl_i} \; \cC \big( \Psi^\eps_{\gl_i} \big) \; \indc_{\cG^{\e}} \big (  \Psi^\eps_{\gl_i} \big)\cH \big( \Psi^\eps_{\gl_i} \big)    
 \gp_{ \gr} \;    f^{\eps 0}_{\{1,\dots,r\}}(\Psi^{\eps 0}_{\gr_1},\dots,\Psi^{\eps 0}_{\gr_r})
 \,dT_{ m}  d\Omega_{ m}  dV_{ m} \Big| dZ^*_n  \\
&\leq  \,
 (r-2)! \,(CC_0) ^K \eps^{d(r-1) }\left(\frac{C}{\beta_0^{1/2} \mu_\e}\right)^{n-r} \, \left( t + \eps \right)^{n-r}\,
 \frac{(CC_0  t)^m}{m!} (m+n)^m {\color{black} \left( \frac{1}{\beta_0} \right)^{m/2} }
  \\
&  \times\,\sum_{T_1 \in \cT_{\gr_1}}\!  \! \dots\! \! \sum_{T_r \in \cT_{\gr_r}} \sum_{\tilde T_1 \in \cT_{\lambda_1}}
\! \! \dots\! \! \sum_{\tilde T_\ell \in \cT_{\lambda_\ell}}
  \int \exp \left(\alpha n -\frac {\beta_0 }{16} \bbV ^2 \right)  \prod_{s = 1} ^m (1+\beta_0^{1/2} |v_s|)  dV^*_n dV_{m}\\
&\qquad \times \sup \left(  \exp  \big(-\frac {\beta_0 }{16}  \bbV ^2\big) \,\left(\prod_{i=1}^r
\,\prod_{\lambda_j \in \gr_i}\,\left(\beta_0 \bbV^2_{\lambda_j} + K_{\lambda_j}\right)^{d_{\lambda_j}(T_i)} \right) \left(\prod_{i=1}^\ell \,\prod_{ j \in \lambda_i}\,\left(\beta_0 \bbV^2_{j} + K_{j}\right)^{d_j(\tilde T_i)}\right)\right)
\,.\\
%&\leq   | \cH_p |_\infty
% r! C^{k} k^{p-N} \mu_\e^{-(N-1)} \left(C t\right)^{k-r}
% \prod_{i=1}^r\sum_{T_i \in \cT_{\gr_i}}\, \prod_{j=1}^\ell\sum_{T_j \in \cT_{\ell_j}}\,
% \prod_{\{s_i,s'_i\} } 
% \prod_{\{s_j,s'_j\} } k_{\lambda_{s_i}}k_{\lambda_{s'_i}} k_{{s_j}}k_{{s'_j}} 
% \\
\end{aligned}
$$
%Denoting $d_v(T_i)$
%the degree of the vertex $v$ in the minimally connected graph $T_i$, we have obtained
%$$
%\begin{aligned}
%& \Big| \sum_a \int  
%\;  \left( \prod_{i=1}^\ell \mb_{\gl_i} \; \cC \big( \Psi_{\gl_i} \big) \; \indc_{\cG^{\e}} \big (  \Psi_{\gl_i} \big)\cH_\a \big( \Psi_{\gl_i} \big) \right) \; 
%  \gp_{ \gr } \;  f^{\e0}_{\{1,\dots,r\}} dT_{ M}  d\Omega_{ M}  dV_{ M} dZ^*_\a    \Big|\\
%&\leq   | \cH_p |_\infty
% r! C^{k} k^{p-N} \eps^{d(r-1) } \left(\frac{C t}{\mu_\e}\right)^{N-r} \frac{(C t)^M}{M!}
% \prod_{i=1}^r\sum_{T_i \in \cT_{\gr_i}}\, \prod_{j=1}^\ell\sum_{T_j \in \cT_{\lambda_j}}
% \\
%&\times\,\int \exp \big(-\frac \beta 2 \bbV ^2\big)    \prod_{v \in \gr_i} \gga_v^{d_v(T_i)}  \prod_{u \in \lambda_j} \gga_u^{d_u(T_j)}
%%\\
%%&\qquad \times 
%\prod _{h=1}^M \Big((N+h-1) ( |v_{h}| +1) + | V_{ h-1} |^2 + |V^*_{\a}|^2 \Big) dV^*_\a dV_{M}\;.\\
%%&\leq   | \cH_p |_\infty
%% r! C^{k} k^{p-N} \mu_\e^{-(N-1)} \left(C t\right)^{k-r}
%% \prod_{i=1}^r\sum_{T_i \in \cT_{\gr_i}}\, \prod_{j=1}^\ell\sum_{T_j \in \cT_{\ell_j}}\,
%% \prod_{\{s_i,s'_i\} } 
%% \prod_{\{s_j,s'_j\} } k_{\lambda_{s_i}}k_{\lambda_{s'_i}} k_{{s_j}}k_{{s'_j}} 
%% \\
%\end{aligned}
%$$
Using the facts that 
\begin{align*}
\int \exp \left(-\frac  {\beta_0}  {16} | w |^2 \right) \beta_0^{1/2} | w | dw&\leq  C \beta_0^{-d/2} \,,\\
\exp \left(-\frac  {\beta_0}  {16} |V|^2 \right) \left(\beta_0  |V|^2 + K\right)^D & \leq  C^{K} \left( 16 D \right)^D,
\end{align*}
for positive $K,D$, we arrive at
\begin{equation}
\label{r=1isgood}
\begin{aligned}
&  \int\, \Big|\,\sum_a \,
  \prod_{i=1}^{\ell}\,\mb_{\gl_i} \; \cC \big( \Psi^\eps_{\gl_i} \big) \; \indc_{\cG^{\e}} \big (  \Psi^\eps_{\gl_i} \big)\cH \big( \Psi^\eps_{\gl_i} \big)    
 \gp_{ \gr} \;    f^{\eps 0}_{\{1,\dots,r\}}(\Psi^{\eps 0}_{\gr_1},\dots,\Psi^{\eps 0}_{\gr_r})
 \,dT_{ m}  d\Omega_{ m}  dV_{ m} \Big| dZ^*_n\\
&\leq    
 (r-2)!  \left(\frac{C \beta_0^{-1/2}\left(t + \e\right)  }{\mu_\e}\right)^{n-r}  
 \eps^{d(r-1)} (C C_0 \,  {\color{black} \beta_0^{-\frac{d+1}{2}}}  t)^m 
 (C_0 e^\alpha \beta_0^{-d/2} )^n   \\
 & \qquad \times 
 \left(\prod_{i=1}^r \,\sum_{T \in \cT_{\gr_i}}\,\prod_{\lambda_j \in \gr_i} %C^{K_{\lambda_j}/\ell}
  \left( d_{\lambda_j}(T) \right)^{d_{\lambda_j}(T)} \right)
  \left(\prod_{i=1}^\ell \,\sum_{\tilde T \in \cT_{\lambda_i}}\,\prod_{j \in \lambda_i} %C^{K_{j}/N}
  \left( d_{j}(\tilde T)\right)^{d_{j}(\tilde T)} \right)
 \,.
 \\
\end{aligned}
\end{equation}
For each forest (jungle) we ended up with a factor 
$\sum_{T \in \cT_{k}} \prod_{i =1}^{k} \left(d_{i}(T)\right)^{d_i(T)}$
where $k$ is the cardinality of the forest (jungle).
Applying again Lemma \ref{lem: Cayley+}, 
{\color{black} and using that for any integer $i$
$$
\frac{i^i}{(i-1)!} \leq i \exp (i-1) \leq \exp ( 2 i),
$$}
this number is bounded above by
\begin{equation*}
\begin{aligned}
(k-2)!  \sum_{ \substack{d_1,\cdots,d_k \\ 1 \leq d_i \leq k-1 \\ \sum_{i}d_i = 2(k-1)} }\,\prod_{i=1}^{k} \frac{d_{i}^{d_i}}{(d_i-1)!} 
%& \leq (k-2)! \,e^{k-2}\,\sum_{ \substack{d_1,\cdots,d_k\\ 1 \leq d_i \leq k-1 \\ \sum_{i}d_i = 2(k-1)} }
%\,\prod_{i=1}^{k} \,d_i
%\, \\
 & \leq (k-2)!   {\color{black}   \,e^{4(k-1)} }
 \sum_{ \substack{d_1,\cdots,d_k\\ 1 \leq d_i \leq k-1 \\ \sum_{i}d_i = 2(k-1)} }
\,  1\;.
\end{aligned}
\end{equation*}
The last sum is also bounded by $C^k$.
% 2^{5k} - 24.6.2020
Taking the sum over the number of created particles $m$,
we arrive at
\begin{equation}
\label{est2}
\begin{aligned}
& \int \left|\int  \prod_{i=1}^{\ell} \Big[ \mu (d \Psi^\eps_{\gl_i})  \mb_{\gl_i} \; \cC \big( \Psi^\eps_{\gl_i} \big) \; \indc_{\cG^{\e}} \big (  \Psi^\eps_{\gl_i} \big)\cH \big( \Psi^\eps_{\gl_i} \big) \Big]  
{} \times   \gp_{ \gr} \;  f^{\eps 0}_{\{1,\dots,r\}} (\Psi^{\eps 0}_{\rho_1}, \dots , \Psi^{\eps 0}_{\rho_r}) \right| dZ^*_{n} \\
& \leq  {\color{black} \frac{(r-2)!}{\mu_\e^{n-1}} \Big( CC_0e^\alpha \,  \beta_0^{-\frac{d+1}{2}} (t+\eps) \Big)^n}
\left( { \eps ^{r-1}   \beta_0^{r /2} \over (t+\eps)^r  }\right)
 \prod_{i=1}^r (r_i-2)!\, \prod_{j=1}^\ell (\ell_j-2)!\, 
\sum_{m} (CC_0 \,  {\color{black} \beta_0^{-\frac{d+1}{2}}} t)^m  \\
\end{aligned}
\end{equation}
valid uniformly with respect to all partitions $\gl \hookrightarrow \gr$, and for $t$ small enough.
Finally, summing \eqref{est2} over the partitions $\gl \hookrightarrow \gr$ we find (recalling the convention $0! = (-1)! = 1$)
$$
\begin{aligned}
\sum_{\ell =1}^n \sum_{\gl \in \cP_n^\ell}
\sum_{r =1}^\ell  & \sum_{\gr \in \cP_\ell^r}  \,\left(r-2\right)!\, \prod_{i=1}^r (r_i-2)!\, \prod_{j=1}^\ell (\ell_j-2)!\, \\
& = \sum_{\ell =1}^n\, \sum_{\substack{\ell_1,\cdots,\ell_{\ell} \geq 1\\ \sum_i \ell_i = n}}\,\sum_{r =1}^\ell
\, \sum_{\substack{r_1,\cdots,r_{r} \geq 1\\ \sum_i r_i = \ell}}\,
\frac{n!}{\ell! \ell_1! \dots \ell_\ell !}\, \frac{\ell!}{r! r_1!\dots r_r!}\,\left(r-2\right)!\, \prod_{i=1}^r (r_i-2)!\, \prod_{j=1}^\ell (\ell_j-2)!\\
& \leq n!  \, \left( 1+\sum_{r \geq 2}\frac{1}{r(r-1)}\right)^{2n}\;.
\end{aligned}
$$
This concludes the proof of the first two estimates in Theorem~\ref{cumulant-thm1*}.

\bigskip

The third statement (\ref{timecontinuity}) is obtained in a very similar way.
If  the pseudo-particle $i$  has no collision nor recollision during 
$[t-\delta, t]$ then
%\begin{equation*}
%\label{eq: h lisse}
%|H^{(i)} (z_i([0,t]))| \leq C_{Lip} \sup_{|t-s| \leq \delta }  |z_i (t) - z_i (s)|  \leq C_{Lip} \delta |v_i(t)| \leq C_{Lip} \delta |V_n(t)|\;.
%\end{equation*}
\begin{equation*}
\label{eq: h lisse}
 \sup_{|t- t'| \leq \delta }  |z_i (t) - z_i (t')|  \leq   \delta |v_i(t)| \leq  \delta |V_n(t)|\;.
\end{equation*}
This is enough to gain a factor $\delta$ from the assumption on $H_n$.

If a collision occurs  during $[t-\delta,t]$, then by localizing the time integral of this collision in  Duhamel formula, one gets the additional factor $\delta$ (with a factor $m$ corresponding to the symmetry breaking
 in the time integration $dT_m$).

Finally, it may happen that a  recollision  occurs during $[t-\delta,t]$.
This imposes an additional geometric constraint and the recollision time has to be integrated 
now in $[t-\delta,t]$.
Thus an additional  factor $\delta$ is also obtained (together with a factor $n$ corresponding to the symmetry breaking in the time integration $d\Theta^{\rm clust}_{n-1} $). This completes the proof of (\ref{timecontinuity}).
\qed

\begin{Rmk}
\label{m-sum}
Note that the sum over $m$ in \eqref{est2} is converging uniformly in $\eps$, which means that the contribution of pseudo-trajectories involving a large number $m$ of created particles can be made as small as needed.
In particular, to study the convergence as $\eps \to 0$, it will be enough to look at pseudo-trajectories with a controlled number $m\leq m_0$ of added particles.
\end{Rmk}

%%%%%%%%%%%%%%%%%%%%%%%%%%%%%%%%%%%%%%%%%%%%%%%%%%%%%%%
\chapter{Minimal trees and convergence of the cumulants}
\label{convergence-chap} 
\setcounter{equation}{0}
 The goal of this chapter is to prove  Theorem~\ref{cumulant-thm2} p. \pageref{cumulant-thm2}, which can be restated as follows.
\begin{Thm}
\label{cumulant-thm2*}
Let  $H_n : (D([0,+\infty[ ))^n  \mapsto \bbR$  be a  continuous factorized function $ H_n (Z_n([0,t])) = \prod_{i=1}^n H^{(i)} (z_i([0,t])) $ such that
\begin{equation}
\label{controlHnexp*-chap10}
 \big| H_n( Z_n([0,t])) \big| \leq \exp \Big(\alpha n +\frac{\beta_0} 4 \sup_{s\in [0,t]} |V_n(s)|^2 \Big) \,,
 \end{equation}
with $\beta_0$ defined in \eqref{lipschitz}.

Then  the scaled cumulant  $f_{n,[0,t]}^\eps (H_n)$ converges for any $t \leq T_0$ to 
the limiting cumulant introduced in~\eqref{eq: decomposition lcumulant}
$$
  f_{n,[0,t]} (H_n) = \sum_{T \in \cT^\pm_n} \sum_{m} \sum_{a \in \cA^\pm_{n,m} } \int d \mu_{{\rm sing }, T,a} (\Psi_{n,m})   \cH(\Psi_{n,m}) f^{0\otimes(n+m)} (\Psi_{n,m}^0)\,.
$$ 
%where  the singular measure  is defined as in \eqref{defdmusing} by
%\begin{equation}
%\label{eq: defdmusing contracte}
%d \mu_{\rm sing} (\Psi_n) = \sum_{T \in \cT^\pm_n} \sum_{m} \sum_{a \in \cA^\pm_{n,m} }  d \mu_{{\rm sing }, T,a} (\Psi_{n,m}) .
%\end{equation}
\end{Thm}

%   We recall that the limiting cumulants~$\lcum$ correspond to   minimally connected graphs, and are defined by pseudo-trajectories  which are described in detail in Chapter~\ref{Limiting cumulants}.
  
 After some preparation in Section~\ref{truncatedcumulants}, we present in Section~\ref{sec-leading order}
 the leading order asymptotics of~$ \cum $  by eliminating all pseudo-trajectories involving non clustering recollisions and overlaps. Section~\ref{sec:convergence cum} is devoted to the conclusion of the proof, by estimating the discrepancy between the remaining pseudo-trajectories~$\Psi^\eps_n$ and their limits~$ \Psi_n$.

\section{Truncation of  cumulants}\label{truncatedcumulants}
\setcounter{equation}{0}

An inspection of the arguments in the previous chapter shows   that initial clusterings are  negligible compared to dynamical clusterings. Indeed  Estimate~(\ref{est2})  shows  that the leading order term in the cumulant decomposition {\rm(\ref{eq: decomposition cumulant})} corresponds
to choosing~$r=1$: this  term is indeed  of order
$$C^n n! (t+\eps)^{n-1}   $$
while the error is smaller by one order of $\eps$.
We are therefore reduced to studying
$$   \mu_\eps^{n-1} \sum_{\ell =1}^n \sum_{\gl \in \cP_n^\ell}
\int   \Big( \prod_{i=1}^\ell d\mu\big( \Psi^\eps_{\gl_i} \big)   \cH \big( \Psi^\eps_{\gl_i} \big)  \mb_{\gl_i} \Big)  \,     \gp_{\{1,\dots,\ell\}} \;  f^{\eps 0}_{\{1\}} (\Psi^{\e 0} _{\rho_1})  \,    .
$$

We shall furthermore consider only  trees of controlled size: we define, for any integer~$m_0$,
\begin{equation}
\label{truncatedcumulantm0}   
\cumt :=   \mu_\eps^{n-1} \sum_{\ell =1}^n \sum_{\gl \in \cP_n^\ell}
\int  dZ_n^*  \int  \prod_{i=1}^{\ell} 
\Big[ d\mu_{m_0}(\Psi^\eps _{\gl_i})  \mb_{\gl_i} \cH \big( \Psi^\eps _{\gl_i} \big) \Big]     \gp_{\{1,\dots,\ell\}} \;  f^{\eps 0}_{\{1\}} (\Psi^{\e 0} _{\rho_1})  \,   ,
\end{equation}
where the measure on the pseudo-trajectories is defined as in \eqref{eq: measure nu} by 
$$
d \mu_{m_0}(\Psi^\eps _{\gl_i}) := \sum_{m_i\leq m_0} \sum_{a \in \cA^\pm_{\gl_i,m_i } }dT_{m_i}  d\Omega_{m_i}  dV_{m_i} \; \indc_{\cG^{\e} } (\Psi^\eps _{\gl_i})
\; \prod_{k=1}^{m_i}   \Big (s_k\left(\big( v_k -v_{a_k} (t_k)\big) \cdot \omega_k\right)_+\Big ) .
$$
Then by Remark \ref{m-sum}, we   have  
\begin{equation}
\label{truncation1}
\lim_{m_0\to \infty}\big |\cum-\cumt \big|  =0 \hbox{ uniformly in } \eps \,.
\end{equation}

\medskip

Next let us define
$$ \tcum := \mu_\eps^{n-1} \sum_{\ell =1}^n \sum_{\gl \in \cP_n^\ell} 
\int  dZ_n^*  \int  \prod_{i=1}^{\ell} \Big[ d\mu ( \Psi^\eps _{\gl_i}) \tilde  \mb_{\gl_i} 
%\; \cC \big( \Psi^\eps _{\gl_i} \big) 
%\; \indc_{\cG^{\e}} \big (  \Psi^\eps _{\gl_i} \big)
 \cH \big( \Psi^\eps _{\gl_i} \big) \Big]      \tilde  \gp_{ \{1, \dots, \ell\}} \;  f^{\eps 0}_{\{1\}} (\Psi^{\eps 0} _{\rho_1}) 
$$
where $\tilde \mb_{\gl_i}$ is the characteristic function supported on the forests $\lambda_i$ having exactly $|\gl_i| - 1$ recollisions, and $\tilde  \gp_{ \{1, \dots, \ell\}}$ is supported on jungles having exactly $\ell - 1$ regular overlaps, so that 
\begin{itemize}
\item all recollisions and overlaps are clustering;
\item all overlaps are regular in the sense of Remark \ref{overlap-rmk}.
\end{itemize}
Since~$\tcum $ is defined simply as the restriction of~$\cum$ to some pseudo-trajectories (with a special choice of initial data), the same estimates as in the previous chapter show that
$$
| \tcum| \leq C^n n!  ( t+\eps)^{n-1} \, .
$$
Furthermore,  defining its truncated counterpart
$$ 
\tcumt  := \mu_\eps^{n-1} \sum_{\ell =1}^n \sum_{\gl \in \cP_n^\ell} \int  dZ_n^*
\int  \prod_{i=1}^{\ell} \Big[ d \mu _{m_0} ( \Psi^\eps _{\gl_i}) \tilde  \mb_{\gl_i} 
%\; \cC \big( \Psi^\eps _{\gl_i} \big) 
%\; \indc_{\cG^{\e}} \big (  \Psi^\eps _{\gl_i} \big)
\cH \big( \Psi^\eps _{\gl_i} \big)  \Big]     \tilde  \gp_{  \{1, \dots, \ell\}} \;  f^{\eps 0}_{\{1\}} (\Psi^{\eps 0} _{\rho_1}) $$
there holds
\begin{equation}
\label{truncation2}
 \lim_{m_0\to \infty} \big|\tcum-\tcumt\big|  =0 \hbox{ uniformly in } \eps \,. 
 \end{equation}

The limits~(\ref{truncation1}) and~(\ref{truncation2}) imply that it is enough to prove that  the truncated decompositions~$\cumt$  and~$\tcumt$ are close:  we shall indeed see in the next section that   non clustering recollisions or overlaps  as well as non regular overlaps induce  some  extra smallness. 

 Note finally that the estimates provided in Theorem~\ref{cumulant-thm1*} show  that the series~$\cum / n!$ converges uniformly in~$\eps$ for $t\leq T_\alpha$, so a termwise (in $n$) convergence as $\eps \to 0$  is sufficient for our purposes.   We therefore shall make no attempt at optimality in the dependence of the constants in~$n, \alpha, C_0, \beta_0$ in this chapter.

\section{Removing non clustering recollisions/overlaps and non regular overlaps}\label{sec-leading order}\label{leadingorder}
\setcounter{equation}{0}

Let us now estimate $|\cumt- \tcumt|$. 
We   first show how to express   non clustering recollisions/overlaps  as   additional constraints on the set of integration parameters $(Z_n^* ,T_m, V_m, \Omega_m)$.
% The constraints may be either ``independent" of the constraints described in the previous chapter, or can reinforce one of them   in an explicit way: in particular, we shall see that the size of the  sets ${\mathcal B}_k $ of parameters, introduced in Section~\ref{subsec:estDC}, becomes smaller by a factor~$O(\eps^{\frac18}  )$ in the presence of a non clustering recollision (and similarly for   clustering overlaps).
% 
This argument  is actually very similar to the  argument used to control (internal) recollisions in Lanford's proof (which focuses primarily on  the expansion of the first cumulant).

\begin{Prop}
\label{prop: minimally connected}
Denote by~$\cB^\e$ the set of integration parameters leading to  pathological cumulant pseudo-trajectories~:
\begin{equation}\label{defBepsilon}
\begin{aligned}
 \cB^\eps &:=\Big\{ (Z_n^*, m, T_m, \Omega_m, V_m) \,:\, m\leq m_0  \, \\
&\qquad\qquad  \hbox{ and } \, \, \Psi^\eps \hbox{ has a non  clustering recollision/overlap or  a non regular overlap}\Big\}\, . 
 \end{aligned}
 \end{equation}
Then,  there exists a constant $C$ (depending on $\alpha, C_0, \beta_0$) such that 
$$\big | \cumt - \tcumt\big| \leq  C ^n  (t+1) ^{n+d-1} n!   \eps^{1/8} \,.$$
\end{Prop}

In the coming section we discuss one elementary step, which is the estimate of a given non clustering   event, by treating separately different geometrical cases -- we shall actually only deal with non clustering recollisions, the case of overlaps being simpler.  Then in Section~\ref{subsec:estloop} we   apply the argument to provide a global estimate.% on  dynamical cycles. 

 \subsection{Additional constraint due to  non clustering recollisions and overlaps}
\label{subsec:esterror}
\setcounter{equation}{0}

We consider a   partition $\gl $ of $\{1^*,\dots,n^*\}$ in $\ell$ forests $\lambda_1,\dots, \lambda_\ell$.
We fix the velocities~$V_n^*$,  as well as the collision parameters~$(T_{m}, V_{m}, \Omega_{m})$, with~$m\leq m_0\ell $.
 As in Section~\ref{subsec:estDC}
we   denote by~$\bbV^2 :=(V_n^*)^2 + V_{m}^2$ (twice) the total energy and by $K = n+m$
the total number of particles, and by~$\bbV_i^2$ and $K_i$  the   energy and number of particles of the collision tree~$\Psi^\eps_{\{i\}}$ with root at~$z_i^*$.

%The following notion  of {\it pseudo-particle}   will be useful. 
% \begin{Def}[Pseudo-particle]
% \label{pseudoparticle} 
% Let~$q\in\{1^*,\dots,n^*\}\cup\{1,\dots,m\}$ be the label of a particle in a tree.
% We define recursively, moving up the tree towards the root, the {\rm pseudo-particle} $\bar q$  associated with the particle $q$  to be
%\begin{itemize}
% \item $q$ as long as it exists,
% \item $ \bar q= a_q$ when $q$ disappears due to a collision, and as long $a_q$ exists,
% \item $\bar q = a_{a_q} $ when $a_q$ disappears, and as long as this latter   exists,
%  \item...
%\end{itemize}
%When there is no  possible confusion, we shall denote abusively $q$ the pseudo-particle.  
%\end{Def} 

Let us consider a pseudo-trajectory (compatible with~$\lambda$) involving a non clustering recollision.
%(recall that we shall not treat the case of overlaps as it is identical). 
We denote by~$t_{\rm rec}$ the time of occurrence of the first non clustering recollision (going backwards in time)  and we denote by~$q,q'\in\{1^*,\dots,n^*\}\cup\{1,\dots,m\}$ the labels of the two particles involved in that recollision. By definition, they  belong to the same forest, say $\lambda_1$, and we denote by~$\Psi^\eps_{\{i\}} $ and~$\Psi^\eps_{\{i'\}}$ their respective trees  (note that it may happen that~$i=i'$ in the case of an internal recollision). 

 The recollision between $q $ and $q'$  imposes  strong constraints  on the history of these particles, especially on the  first deflection of the couple~$q,q'$, moving up the forest  (thus forward in time) towards the root. These constraints can be expressed by different equations depending on the recollision scenario.

 \noindent
\underbar{Self-recollision.}\label{selfrecollpage} $ $ 
Let us assume that moving up the tree starting at  the recollision time, the first deflection of~$q $ and~$q'$ is between~$q $ and~$q'$ themselves at time $\bar t$: this means that the recollision occurs due to   periodicity in space.  
\begin{figure} [h]%  figure placement: here, top, bottom, or page
   \centering
   \includegraphics[width=7cm]{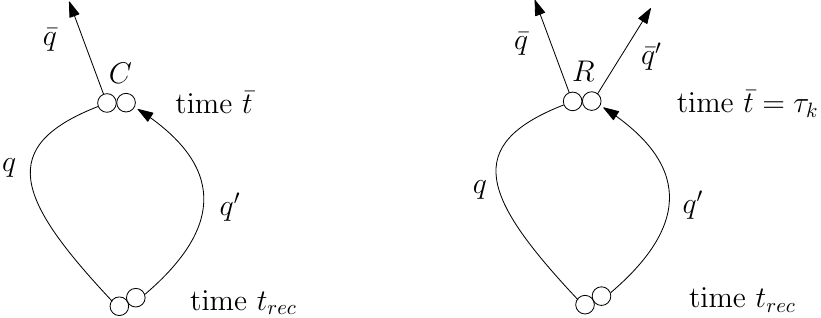} 
   \caption{The first deflection of $q$ and $q'$ can be either the creation of one of them (say $q$), or a clustering recollision.}\label{selfrecollfigure} 
\end{figure}

%Three different situations can occur (up to exchanging~$q$ and~$q'$):
%\begin{itemize}
%\item[(1a)] $q$ is created  at  time~$\bar t$, without scattering;
%%(Figure~\ref{periodicrecollfigure}(a));
%
%\item[(1b)] $q$ is created  at  time~$\bar t$,  with scattering;
%% (Figure~\ref{periodicrecollfigure}(b));
%
%\item[(1c)]   the deflection at  time~$\bar t$ corresponds to the clustering recollision between two trees, say  $\Psi_{i}$ and~$\Psi_{b(i)}$. %(Figure~\ref{periodicrecollfigure}(c)).
%\end{itemize}
This has a very small cost, 
as described in the following proposition (with the notation of Section~\ref{subsec:estDC}).   
\begin{Prop}\label{periodicrecollisionprop} Let~$q$ and~$q'$ be the labels of the two particles recolliding due to space periodicity, and denote by~$\bar t$ the first time of  deflection of~$q$ and~$q'$, moving up their respective trees from the recollision time. The following holds:
\begin{itemize}
\item If~$q$ is created next to~$q'$ at  time~$\bar t$ with collision parameters~$\bar \omega $ and~$ \bar v$, and if~$\bar v_{q}$ is the velocity of~$q$ at time~$\bar t^{+}$, then denoting by~$\Psi^\eps_{\{i\}}$ their collision tree there holds
$$
\int  \indc _{\mbox{\tiny {\rm Self-recollision with   creation of~$q$ at time~$\bar t$}}}  \, \,  \big | \big(\bar v -\bar v_{q}  \big)\cdot \bar  \omega \big |  d  \bar t d \bar \omega d\bar v  \leq {C\over \mu_\eps }  \bbV^{2d+1}(  1+t)^{d+1}\, .
$$
\item If~$\bar t$ corresponds to the~$k$-th clustering recollision in~$\Psi^\eps_{\lambda_1}$, between the trees~$\Psi^\eps_{\{j_k\}}$ and~$\Psi^\eps_{\{j'_k\}} $, then 
$$
 \int \indc _{\mbox{\tiny {\rm Self-recollision  with a clustering recollision at time~$\bar t$}}} \, \,   d \hat x_k  \leq  \frac C {\mu_\eps^2} \left( \bbV (  1+t)\right)^{d+1}  \, .
$$
\end{itemize}
\end{Prop} 
Note that in the first case the admissible collision parameters $(\bar t, \bar \omega, \bar v)$ belong to a small set of size~$O(1/\mu_\eps)$. In the second case, the condition is expressed in terms of the root $ \hat x_k$  with the notation of Section~\ref{subsec:estDC}: it is not independent of the condition~(\ref{Bq-def}) defining $  B _{qq'}$, but it reinforces it as the estimate provides a factor~$1/\mu_\eps^2$ instead of $1/\mu_\eps$.

\bigskip
 \noindent
\underbar{Generic non clustering recollision.}
Without loss of generality,  we  may  assume  that  the first deflection moving up the tree from time~$t_{\rm rec}$ involves $q$.   We denote by~$\bar t$ the time of that first deflection and by~$c\neq q,q'$ the particle involved in the collision with~$q$  (see Figure \ref{selfrecollfigure'}). The parent $\bar q$ of $q$ is the particle $q$ or $c$ existing at time $\bar t^+$, and     we denote  by~$\bar v _{ q}$  the velocity of $\bar q$ at time $\bar t^+$ . Similarly we denote by~$\bar v _{ q'}$ the velocity of particle~$q'$ at time~$\bar t$.
\begin{figure} [h]%  figure placement: here, top, bottom, or page
   \centering
   \includegraphics[width=7cm]{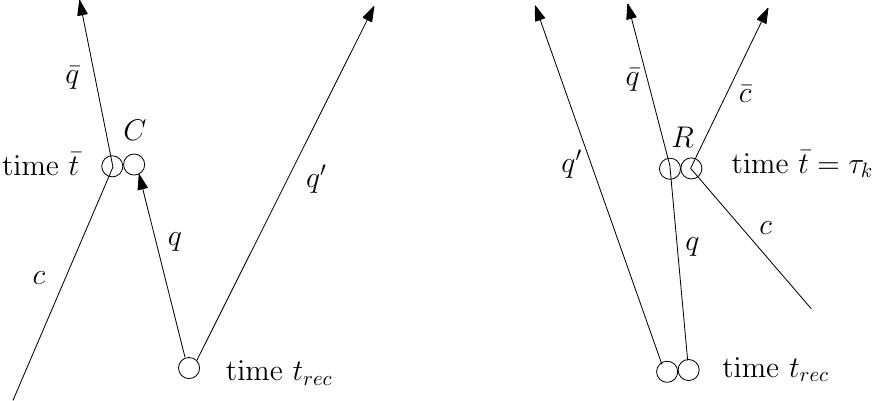} 
   \caption{The first deflection of $q$  can be either a collision, or a clustering recollision.}\label{selfrecollfigure'} 
\end{figure}
% (see Figure~\ref{recollfigure}). 

The result is the following.
\begin{Prop}\label{recollisionprop} Let~$q$ and~$q'$ be the labels of the two   particles involved in the first non clustering recollision. Assume  that  the first deflection moving up their trees from time~$t_{\rm rec}$ involves $q$ and a particle~$c \neq  q'$, at some time~$\bar t$. Then  with the above notation

\begin{itemize}

\item If  $\bar t$ is the creation time of $q$ (or~$c$), denoting by~$\bar \omega$ and~$\bar v$ the corresponding collision parameters, by~$\Psi^\eps_{\{i\}}$ their collision tree and by~$\Psi^\eps_{\{i'\}}$ the collision tree of~$q'$, there holds
$$
\int \indc _{\mbox{\tiny {\rm  Recollision with a creation at time~$\bar t$}}}\,  \big |  (\bar v-\bar v_{ q}(\bar t ) ) \cdot \bar \omega \big|   d\bar t  d\bar \omega  d\bar v \leq  C \bbV^{2d+\frac32}  (1+t)^{d+\frac12}  \min \left(1, \frac{ \eps^{1/2}  }{ |\bar v_{ q}-\bar v_{ q'}| }\right) \, .
$$

\item If  $\bar t   $ corresponds to the~$k$-th clustering recollision in~$\Psi^\eps_{\lambda_1}$, between~$\Psi^\eps_{\{j_k\}}$ and~$\Psi^\eps_{\{j'_k\}} $,  and if~$\Psi^\eps_{\{i'\}}$ is the collision tree of~$q'$, then
$$
 \int \indc _{\mbox{\tiny {\rm Recollision with a clustering recollision at time~$\bar t$}}} \,  \, d \hat x_k  \leq   {C\over \mu_\eps}  \bbV ^{\frac32} (1+t)^{\frac12} \min \left(1,{\eps^{1/2} \over  | \bar v_q - \bar v_{ q'}| } \right) \, .
$$
\end{itemize}

\end{Prop} 

%\begin{figure} [h]%  figure placement: here, top, bottom, or page
%   \centering
%   \includegraphics[width=5cm]{recollision}  
%   \caption{We distinguish two scenarios of recollision between $q$ and $q'$ depending on the first deflection moving up their trees. Case (i)~: addition of a new particle. Case (ii)~: clustering recollision.}\label{recollfigure} 
%\end{figure}
%

%
%As in the case of a periodic recollision, three different situations  may occur:
%\begin{itemize}
%
%\item[(2a)] $\bar t$ corresponds to  the  creation of    $q$  (so $\bar t = t_p$)  without scattering
%%  as in Figure~\ref{recollfigure}(a) ;
%
%\item[(2b)]  $\bar t$ corresponds to  the creation of    $q$ or $\sigma$ (so $\bar t = t_p$ or~$\bar t = t_\sigma$)  with scattering
%%  as in Figure~\ref{recollfigure}(b);
%
%
%\item[(2c)]  $\bar t =\tau_i $ corresponds to the clustering recollision between~ÃÂ$\Psi_{\{i\}}$ and~$\Psi_{b(i)} $
%%   as in Figure~\ref{recollfigure}(c).
%
%\end{itemize}

%\begin{figure} [h]%  figure placement: here, top, bottom, or page
%   \centering
%   \includegraphics[width=5cm]{recoll0}  \qquad\qquad  \includegraphics[width=5cm]{recoll1}  \qquad\qquad  \includegraphics[width=5cm]{recoll3} 
%   \caption{CHANGER LES NOTATIONS}\label{recollfigure} 
%\end{figure}

Note that as in the periodic situation, the recollision condition in the first case provides some smallness  on the set of admissible parameters $(\bar t, \bar \omega, \bar v)$, while the recollision condition in the second case is expressed in terms of the root~$ \hat x_k$, and reinforces the condition (\ref{Bq-def}) defining $  B _{qq'}$ by a factor~$\eps^{1/2}$. However in both cases the estimate involves a   singularity in velocities that  has to be eliminated.

The geometric analysis of these scenarios and the proof of Propositions~\ref{periodicrecollisionprop} and~\ref{recollisionprop} are  postponed to  Section~\ref{recollision-appendix}.
The  estimates in the first case  were actually   already proved in \cite{BGSR2}, while the second one (the case of a clustering recollision) requires a slight adaptation.

%
%
%\begin{figure} [h]%  figure placement: here, top, bottom, or page
%   \centering
%   \includegraphics[width=8cm]{recoll'}  \qquad\qquad \qquad\qquad  \includegraphics[width=8cm]{recoll''}  \qquad\qquad  %\includegraphics[width=5cm]{recoll''} 
%   \caption{CHANGER LES NOTATIONS}
%   \label{recollfigure} 
%\end{figure}
%
%

 \noindent
\underbar{Elimination of the singularity.} 
It finally remains to  eliminate the singularity~$1/ | \bar v_{ q} - \bar v_{q'}|  $, using the next deflection moving up the tree.   Note that this singularity arises only  if  the first non clustering recollision is not a self-recollision, which ensures that the recolliding particles have at least two deflections before the non clustering recollision.
The result is the following.
 
\begin{Prop}\label{vsingularityprop}
 Let~$q$ and~$q'$ be the labels of two   particles  with velocities~$v_{q}$ and~$v_{q'}$, and denote by~$\bar t$ the time of the first deflection of~$q$ or~$q'$ moving up their trees.   
 \begin{itemize}

\item If the deflection at $\bar t$  corresponds to a collision in  a tree~$\Psi^\eps_{\{i\}}$ with parameters~$ \bar \omega, \bar v$, then 
$$
\int  \indc _{\mbox{\tiny {\rm  Recollision with a creation at time~$\bar t$}}}\,    \min \left(1, \frac{ \eps^{1/2}  }{ |v_{ q}-v_{ q'}| }\right) 
   \,   \big |  (\bar v-   \bar  v_{  q } ) \cdot  \bar \omega \big|  d\bar t  d\bar v  d\bar\omega \leq C  t
 \bbV^{d+1}  \e^{\frac18} \,.
$$
\item if $\bar t  $ corresponds to the~$k$-th clustering recollision in the tree~$\Psi^\eps_{\lambda_1}$, between~$\Psi^\eps_{\{j_k\}}$ and~$\Psi^\eps_{\{j'_k\}} $, then
$$
 \int  \min \left(1, \frac{ \eps^{1/2}  }{ |v_{ q}-v_{ q'}|}\right)  \,  d\hat x_k \leq   {C  \e^{\frac18} \bbV  t\over \mu_\eps}   \, \cdotp
$$

\end{itemize}

\end{Prop}
The proposition is also proved in Section~\ref{recollision-appendix} of this chapter.

\subsection{Removing pathological  cumulant pseudo-trajectories}
\label{subsec:estloop}
\setcounter{equation}{0}

\begin{proof}[Proof of Proposition \ref{prop: minimally connected}]
We   first consider the case of pathological pseudo-trajectories involving a non regular clustering overlap. By definition  (see Remark \ref{overlap-rmk}), this means that the corresponding~$\tau_{\rm{ov}}$ has to be equal either to $t$ or to  the creation time of one of the overlapping particles. In other words, instead of being a union of tubes of volume $O( (t+\eps)/\mu_\eps)$, the set $\tilde \cB_k$ describing the $k$-th clustering overlap (see (\ref{tildeBk})) reduces to a union of balls of volume $O(\eps^d)$, so that
$$
| \tilde\cB_{k}| \leq C \eps^d   K_{\lambda_{[k]}} K_{\lambda'_{[k]}}.
$$
The non clustering condition is therefore reinforced and we gain   additional smallness.

\medskip
Let us now consider the case of pathological  pseudo-trajectories involving some non clustering recollision/overlap.
We can  assume without loss of generality that the first non clustering recollision (recall that we leave the case of regular overlaps to the reader) occurs in the forest~$\lambda_1 = \{1,\dots , \ell_1\}$.
The compatibility condition on the jungles gives  smallness when integrating over the  roots of the jungles (see \eqref{eq:ov'}). The compatibility condition on the forests $\lambda_2, \dots, \lambda_\ell$ is obtained by integrating~(\ref{eq:iterCR}) as in Section~\ref{subsec:proofThm}.
We now have  to combine the recollision condition with the compatibility conditions on $\lambda_1 $  to obtain the desired estimate.
As in the previous chapter, we denote by $\tilde a$ the restriction of the tree $a$ to $\lambda_1$, and by $\tilde a_{k}$, $\cC_k$  the tree variables and the cross section factors associated with the~$s_k $ creations occurring in the time interval $(\tau_{\rm{rec},k},\tau_{\rm{rec},k-1})$. 

We start from (\ref{eq:iterCR}), adding the recollision condition: we get
$$\begin{aligned}
& \sum_{\tilde a }\int dx^*_{\lambda_1,1}\dots dx^*_{{\lambda_1,\ell_1-1}} \, \mb_{\gl_1}\, \indc_{\mathcal G} \big (  \Psi^\eps _{\gl_1} \big) |\cC \big( \Psi^\eps _{\gl_1} \big)| \, \indc_{\Psi^\eps  _{\gl_1}  \hbox{\tiny has a non clustering recollision}} \\
& \qquad
\leq \sum_{\tilde a_{1} } |\cC_{1} \big( \Psi^\eps _{\gl_1} \big)| 
\int d\hat x_{1}  \indc_{\cB_1 } 
\sum_{\tilde a_{2}} |\cC_{2} \big( \Psi^\eps _{\gl_1} \big)|\int d\hat x_2 \dots  \\
&   \qquad \qquad \times \int d\hat x_{\ell_1- 1 } 
 \indc_{\cB_{\ell_1-1}} \sum_{\tilde a_{\ell_1}} |\cC_{{\ell_1}} \big( \Psi^\eps _{\gl_1} \big)| \, \indc_{\Psi^\eps  _{\gl_1}  \hbox{\tiny has a non clustering recollision}} \,.
\end{aligned}
$$
As shown in the previous section, the set of parameters  leading to the additional recollision can be described 
in terms of a first deflection at a time~$\bar t$. 
We then have to improve    the iteration scheme of Section \ref{subsec:proofThm}, on the time interval~$[\tau_{{\rm rec},k} , \tau_{{\rm rec},k+1}]$ containing the  time~$\bar t$.
There are   two different situations depending on whether the time~$\bar t$   corresponds to a creation, or to a clustering recollision.

\bigskip

If $\bar t$ corresponds to a creation of a particle, say~$c$, the condition on the recollision can be expressed in terms of the collision parameters $(\bar t,\bar v, \bar \omega)= (t_c, v_c, \omega_c) $. We therefore have to
\begin{itemize} 
\item use (\ref{eq:STE}) to control the collision cross sections $\big| \cC_j \big( \Psi^\eps _{\lambda_1}\big)\big|$ for   integration variables indexed by~$s\in \{c+1, \dots, S_j\}$;
\item use the integral with respect to $\bar t,   \bar \omega, \bar v $ to gain  a factor $$C (1+\bbV)^{2d+ 3/2} (1+t) ^{d+1/2}  \min \left(1, \frac{ \eps^{1/2}  }{ |\bar v_{ q}-v_{ q'}| }\right) $$
by Proposition~\ref{recollisionprop}. Note that  the geometric condition for the recollision between~$q$ and $q'$ does not depend on the  parameters which have been  integrated  already at this stage, and to simplify from now on all velocities are   bounded by~$\bbV$;
\item use (\ref{eq:STE}) to control the collision cross sections $\big| \cC_j \big( \Psi^\eps _{\lambda_1}\big)\big|$ for $s\in \{S_{j-1}+1, \dots,c- 1\} $;
\item use the integral with respect to $ \hat x_j $ to gain   smallness due to the clustering recollision.
\end{itemize}
Note that, since $\bar t$ is dealt with separately, we shall lose a power of $t$  as well as a factor $m\leq \ell m_0$ in the time integral. We shall also lose another factor $K^2$ corresponding to all possible choices of recollision pairs~$(q,q')$: at this stage we shall not be too precise in the  control of the constants in terms of~$n$, and~$m_0$, contrary to the previous chapter.

If $\bar t = \tau_{{\rm rec},k} $ corresponds to a clustering recollision, we use the same iteration as in Section~\ref{subsec:proofThm}:
\begin{itemize} 
\item use (\ref{eq:STE}) to control the collision cross sections $\big| \cC_k\big( \Psi^\eps _{\lambda_1}\big)\big|$; 
\item use the integral with respect to $ \hat x_k$ to gain some smallness due to the clustering recollision, multiplied by the additional smallness due to the non clustering recollision.  \end{itemize} 
As in the first case, we shall lose a  factor $K^2$ corresponding to all possible choices of recollision pairs.
 
After this first stage, we still need to integrate the singularity with respect to velocity variables, which requires   introducing the next deflection (moving up the root). 

We therefore perform the same steps as above, but  integrate the singularity  $$  \min \left(1, \frac{ \eps^{1/2}  }{ |v_{ q}-v_{ q'}| }\right) $$ by using Proposition~\ref{vsingularityprop}.

%With the notation of Proposition~\ref{vsingularityprop} , if~$\bar t \in [\tau_j, \tau_{j+1} ] $ corresponds to a creation of a particle~$\sigma$, the condition on the recollision can be expressed in terms of  $\bar t,\bar v, \bar \omega$. We therefore again have to
%\begin{itemize} 
%\item use (\ref{eq:STE}) to control the collision cross sections $\big| \cC_j \big( \Psi_{\lambda_1}\big)\big|$ for $s\in \{\sigma+1 ,\dots, S_j\}$;  
%\item use the integral with respect to $\bar t,\bar v, \bar \omega$ to integrate the singularity  $$  \min \left(1, \frac{ \eps^{1/2}  }{ |v_{ p}-v_{ q'}|^{1/2} }\right) $$
%by Proposition~\ref{vsingularityprop};
%\item use (\ref{eq:STE}) to control the collision cross sections $\big| \cC_j \big( \Psi_{\lambda_1}\big)\big|$ for $s\in \{S_{j-1}+1, \dots , \sigma- 1\}$;
%\item use the integral with respect to $ \hat x_{j} $ to gain some smallness due to the clustering recollision.
%\end{itemize}
%Note that again, since $\bar t$ is dealt with separately, we will loose a power of $t$  as well as a factor $m$ in the time integral. 
%
%
%If $\bar t= \tau_{j}$ corresponds to a clustering recollision, we use the same iteration as in Section~\ref{subsec:proofThm}
%\begin{itemize} 
%\item use (\ref{eq:STE}) to control the collision cross sections $\big| \cC_j \big( \Psi_{\lambda_1}\big)\big|$; 
%\item use the integral with respect to $ \hat x_{j} $ to gain some smallness due to the clustering recollision, plus some additional smallness due to the non clustering recollision. 
%\end{itemize}

\begin{Rmk}
Note that it may happen that the    two deflection times used in the process are in the same time interval $[\tau_{{\rm rec},k} , \tau_{{\rm rec},k+1}]$, which does not bring any additional difficulty. We just set apart the two corresponding integrals  in the collision parameters if both correspond to the creation of new particles.
\end{Rmk}
 
Integrating    with respect to~the remaining variables in  $(T_{ m}  ,\Omega_{ m} , V_ {m})$
and following the strategy described above leads to the  bound  
\begin{equation}
\label{NC-recollision-bound}
\begin{aligned}
& \left| \int  
\;  \left( \prod_{i=1}^\ell \mb_{\gl_i} \; \cC \big( \Psi^\eps _{\gl_i} \big) \; \indc_{\mathcal G} \big (  \Psi^\eps _{\gl_i} \big)\cH \big( \Psi_{\gl_i} \big) \right)   \,
 \gp_{\{1,\dots,\ell\}}   f^{\eps 0}_{\{1 \}} \indc_{\cB^\eps }\ dT_{ m}  d\Omega_{ m}  dV_{ m} dZ^*_n  \right|\\
&\leq    \ell !     \eps^{\frac18}  (\ell m_0)^4C^n \left(  {  (t+\eps ) \over \mu_\eps }\right) ^{n-1} (C t)^{m}  (1+t)^{d}  \, .
\end{aligned}
\end{equation}

%We therefore end up with
%$$\begin{aligned}
%& \left| \int 
%\;  \left( \prod_{i=1}^\ell   \mu (d \Psi_{\gl_i})  \mb_{\gl_i} \; \cC \big( \Psi_{\gl_i} \big) \; \indc_{\mathcal G} \big (  \Psi_{\gl_i} \big)\cH \big( \Psi_{\gl_i} \big) \right) 
% \gp_{\{1,\dots,\ell\}}   f^{0}_{\{1 \}} dZ^*_\alpha \right|\\
%&\qquad \leq  | \cH |_\infty (p-N)!N!  \eps^{(d-1)(N-1)  } \eps^{\frac14}  t^{2d+n} \sum_M  (N+M)^{2 (p-N)}  (C  t)^p \\
%&\qquad\leq   | \cH |_\infty (p-N)!N!  N^{2(N-1)}  \eps^{(d-1)(N-1)  +\frac14}  t^{2d+N}  C ^p
%\end{aligned}
%$$
Finally summing over $m\leq \ell m_0$ and over all possible partitions, we find
$$\forall n \geq 1, \qquad\big | \cumt - \tcumt\big| \leq  C ^n  (t+1) ^{n+d-1} n!   \eps^{1/8} 
\,,$$
where~$C$ depends on~$C_0,\alpha, \beta_0$ and~$m_0$. This concludes the proof of  Proposition  \ref{prop: minimally connected}.
\end{proof}

\section{Convergence of the cumulants} 
\label{sec:convergence cum}

 In order to conclude the proof of  Theorem~\ref{cumulant-thm2*}, we now have to compare $\tcumt$ and $\lcum$ defined in  \eqref{eq: decomposition lcumulant} as
\begin{equation*}
 \lcum =   \sum_{T \in \cT^\pm_n} \sum_{m} \sum_{a \in \cA^\pm_{n,m} }\int  d\mu_{{\rm sing}, T, a} ( \Psi_{n,m})   \cH(\Psi_{n,m}) \left(f^0\right)^{\otimes (n+m) } ( \Psi^0_{n,m} )  \,  .
\end{equation*}
The comparison will be achieved  by coupling the pseudo-trajectories and this requires   discarding the pathological trajectories leading to non clustering recollisions/overlaps and non regular overlaps. Thus we  define the modified limiting cumulants by 
restricting the integration parameters 
 to the set $\mathcal G^\e$, which avoids  internal overlaps in collision trees of the same forest at the creation times, and  
by removing the set $\cB^\e$ introduced in~(\ref{defBepsilon})
\begin{equation*}
\lcumt :=  \sum_{T \in \cT^\pm_n} \sum_{m} \sum_{a \in \cA^\pm_{n,m} }\int  d\mu^{m_0} _{{\rm sing}, T, a} ( \Psi_{n,m}) 
\cH(\Psi_{n,m} )\indc_{\cG^\e \setminus \cB^\e}  \left(f^0\right)^{\otimes (n+m)} ( \Psi^0_{n,m} ) \, ,
\end{equation*}
where $d\mu^{m_0} _{{\rm sing}, T, a}$ stands for the   measure with at most $m_0$ collisions in each forest. 
We stress the fact that~$\lcumt $ depends on $\e$ only through the sets $\cB^\e$ and $\cG^\e$.
We are going to check that 
\begin{equation}
\label{truncation4}
\lim_{m_0 \to \infty} \lim_{\eps \to 0}  | \lcum - \lcumt | = 0 \, .
\end{equation}
The  analysis of the two previous sections may be performed for the limiting cumulants so that 
restricting  the number of collisions to be less than~$m_0$ in each forest
and the integration parameters outside the set~$\cB^\e$ leads to a small error.
The control of internal overlaps, associated with $\cG^\eps$, relies   on the  same   geometric arguments as discussed in Section~\ref{subsec:esterror}: indeed, in order for an overlap to arise when adding particle~$k$ at time $t_k$, one should already have a particle which is at distance less than~$2\eps$ from particle $a_k$, which is a generalized recollision situation (replacing~$\eps$ by $2\eps$). This completes \eqref{truncation4}.

In order to compare  $  \lcumt $  and~$ \tcumt $, we first compare the initial measures, namely~$ f^{\eps 0}_{\{1 \}}   $ with~$ ( f^{0})^{\otimes (n+m)} $. This is actually an easy matter as returning to~(\ref{f01r}) we see that the leading order term in the decomposition of~$ f^{\eps 0}_{\{1 \}} $ is~$F^0_{n+{m }}$, which is well known to tensorize  asymptotically as~$\mu_\eps$ goes to infinity (for fixed~$n+m $), as stated by the following proposition. 
 \begin{Prop}[\cite{GSRT}]\label{exclusion-prop2}
 If $f^0$ satisfies~{\rm(\ref{lipschitz})}, there exists~$C>0$
such that  
$$
\forall m\, , \qquad 
\Big| \left(F^0_m  - \left(f^{0}\right)^{\otimes m}  \right) 
\indc_{{\mathcal D}_{\eps}^m}(Z_m)
\Big|  \leq C ^m
\eps \,    e^{-\frac{3\beta_0}8 |V_m|^2}   \,.
$$
 \end{Prop}

At this stage, we are left with a final discrepancy between $  \lcumt $  and~$ \tcumt $ which is due to the initial data and  $  \cH$ being evaluated at different configurations (namely~$\Psi _n$ and~$ \Psi^\eps_n$). We then need  to introduce a suitable coupling.

In Chapter~\ref{HJ-chapter}, we  used  the change of variables~(\ref{infinitesimalij})
 to reparametrize the limiting pseudo-trajectories in terms of~$x_n^*,V_n^*$ and~$n-1$ recollision parameters (times and angles).
 In the same way, for fixed $\eps$, we can use  the parametrization of clustering recollisions (\ref{eq:omrec}) and of regular clustering overlaps (\ref{eq:omov})  to reparametrize the non pathological  pseudo-trajectories in terms of~$x_n^*,V_n^*$ and~$n-1$ recollision parameters (times and angles). 
 The  cumulant pseudo-trajectories~$\Psi^\eps_{n,m}$ associated with the minimally connected graph~$T \in \cT_n^\pm$ and tree $a \in \cA_{n,m} ^\pm $ are obtained by fixing~$x_n^*$ and $V_n^*$,
 \begin{itemize}
 \item %a partition in forests $\gl \in \cP_n^\ell$, 
 for each~$e  \in E(T)$, a representative~$\{q_e,q'_e\}\approx e$,
  \item a collection of~$m$  ordered   creation times $T_m$,   and parameters~$( \Omega_m, V_m)$;
\item    a collection of  clustering times~$(\tau^{\rm clust}_e)_{e \in E(T)}$\label{notationclusteringtime} and  clustering angles~$( \omega^{\rm clust} _{e})_{e \in E(T)}$.
\end{itemize}
 At each creation time~$t_k$, a new particle, labeled~$k$, is adjoined at 
 position~$x_{a_k}(t_k)+\eps \omega_k $ and with velocity~$v_k$:
 \begin{itemize}
  \item   if~$s_k = +$,  then the velocities~$v_k$ and~$v_{a_k}$ are changed to~$v_k(t_k^-)$ and~$v_{a_k}(t_k^-)$ according to the laws~\eqref{V-def}, 
\item then all particles follow the backward free flow until the next creation or clustering  time.
\end{itemize}
For $\Psi_{n,m}$ to be admissible, at each  time $\tau^{\rm clust} _{e}$ the particles $q_e$ and~$q_e'$  have to collide with the following rules  $ x_{q_e}(\tau^{\rm clust} _{e}) - x_{q'_e}(\tau^{\rm clust} _{e}) = \eps \omega^{\rm clust}_{e}$ :
   \begin{itemize}
  \item  if~$s_{e} = +$, then the velocities~$v_{q_e}$ and~$v_{q'_e}$ are changed  according to the scattering rule, with scattering vector $ \omega^{\rm clust}_{e }$.
 % if~$ (v_{k_i} -v_{k_j})\cdot\omega^{\rm clust} _{\{i,j\}} >0 $ then the velocities~$v_{k_i}$ and~$v_{k_j}$ are changed  according to \eqref{V-def},
\item then  all particles follow the backward free flow until the next creation or clustering time. 
\end{itemize}

 As in \eqref{defdmusing}, we define the   measure for each tree $a\in \cA^\pm_{n,m} $  and each minimally connected graph~$T\in \cT_n^\pm$
\begin{equation}\label{defdmusingeps}
 \begin{aligned}
d\mu^\e_{{\rm sing}, T,a}  &:=  dT_{m}  d\Omega_{m}  dV_{m} dx_n^*  dV_{n}^*  d \Theta^{\rm clust}_{n-1} d\omega^{\rm clust}_{n-1} \prod_{i=1} ^m  s_i  \big( (v_{i}-v_{a_j}(t_i)\cdot\omega_i\big)_+\\
&\qquad \times\prod_{e\in E(T)} \sum_{\{q_e,q_e'\}  \approx e}  s^{\rm {clust}}_{e  }  \big( (v_{q_e}(\tau^{\rm clust}_{e })-v_{q_e'}(\tau^{\rm clust}_{e }))\cdot\omega^{\rm clust}_{e }\big)_+ \indc_{\cG^\e \setminus \cB^\e}
\end{aligned}\end{equation}
denoting  by~$ \Theta^{\rm clust}_{n-1} $ and~$\Omega^{\rm clust}_{n-1}$ the~$n-1$ clustering times~$\tau^{\rm clust}_{e }$ and angles~$\omega^{\rm clust}_{e }$ for~$e \in E(T)$.

\medskip
We can therefore couple the pseudo-trajectories $\Psi _n$ and~$ \Psi^\eps_n$ by their (identical) collision and clustering   parameters. The error between
 the two configurations~$\Psi^\eps _n$ and~$\Psi_n$ is due to the fact that collisions, recollisions and overlaps become pointwise in the limit but generate a shift of size $O(\eps)$ for fixed $\eps$. We then have
 $$
 | \Psi^\eps _n (\tau) - \Psi_n(\tau) | \leq C (n+m)\,\eps 
 \quad \hbox{ for all } \tau \in [0,t] \,.
 $$
  Such discrepancies concern only the  positions, as  the velocities remain equal in both flows.

 It follows that
$$
\Big|\left(f^0\right)^{\otimes  (n+m)} (\Psi^{\eps0} _n  )
- 
\left(f^0\right)^{\otimes (n+m)} ( \Psi^{0}_n  )  \Big|  \leq C_{n, m_0} \eps   e^{-\frac{3\beta}8 |V_{m+n}|^2}\, , 
$$
having used the Lipschitz continuity (\ref{lipschitz}) of $f^0$.  Using the same reasoning for~$  \cH$ (assumed to be continuous), 
 we find finally that for all~$n,m_0$
$$
 \lim_{\eps \to 0} \, | \tcumt - \lcumt |= 0 \,.
$$
This result, along with Proposition~\ref{prop: minimally connected}, Estimates~(\ref{truncation1}), (\ref{truncation2}) and (\ref{truncation4})    proves  Theorem~\ref{cumulant-thm2*}. \qed

%%%%%%%%%%%%%%%%%%%%%%%%%%%%%%%%%%%%%%%%%%%%%%%%%%%%%%%%%%

\section{Analysis of the geometric conditions}
\label{recollision-appendix}
\setcounter{equation}{0}

In this section we prove Propositions~\ref{periodicrecollisionprop} to~\ref{vsingularityprop}. Without loss of generality, we will assume that the velocities $\bbV_j$ are all  larger than 1.

\noindent
{\bf Self-recollision: proof of Proposition \ref{periodicrecollisionprop}.}\label{selfrecoll}
$ $ 
  Denote by $q,q'$ the recolliding particles. By definition of a self-recollision, their first deflection (going forward in time) involves both particles $q$ and $q'$. It can be either a creation (say of $q$ without loss of generality, in the tree~$\Psi^\eps_{\{i\}}$ of~$q'$), or a clustering recollision between two trees (say $\Psi^\eps_{\{j_k\}}$ and $\Psi^\eps_{\{j_k'\}}$ in~$\Psi^\eps_{\lambda_1}$) (see Figure \ref{selfrecollfigure}).

$ \bullet $ $ $   \underbar{If the first deflection corresponds to the creation of~$q$}, we denote by $(\bar t, \bar \omega, \bar v)$ the parameters encoding this creation. We also denote by~$\bar v_{q}$ the velocity of the parent~$\bar q$ just before the creation in the backward dynamics, and by~$\Psi^\eps_{\{i\}}$ the collision tree of~$q'$ (and~$q$). Denoting by~$v_q$ and~$v_{ q'}$ the velocities of~$q$ and~$q'$ after adjunction of~$q$ (in the backward dynamics) there holds
\begin{equation}
\label{self-recoll}
\eps \bar \omega + (v_q  -  v_{ q'})  (t_{\rm rec} - \bar t) = \eps \omega_{\rm rec} +\zeta \hbox{ with } \zeta\in \Z^d \setminus \{0\}
\end{equation}
which  implies that~$v_q-v_{ q'}$ has to belong to the intersection $K_\zeta$ of a cone of opening $\eps$ with a ball of radius $2\bbV$.
 
 Note that the  number of $\zeta$'s for which the sets 
are not empty is at most~$O\big(\bbV^dt^d \big)$.
\begin{itemize}
\item 
If the creation of~$q$ is without scattering, then $  v_q - v_{ q'}  = \bar v - \bar v_{q}$ has to belong to the union of the~$K_\zeta$'s, and 
$$
\begin{aligned}&\int \indc _{\mbox{\tiny Self-recollision with creation at time $\bar t$   without scattering}} \big| \big(\bar v -\bar v_{q} )\cdot \bar \omega   \big |  d  \bar t d \bar \omega d\bar v \\
& \qquad\qquad  \leq C\bbV^dt^d
\sup_{\zeta}
\int \indc _{\bar v -\bar v_{q} \in K_\zeta }
\big| \big(\bar v -\bar v_{q}  )\cdot \bar \omega  \big |  d  \bar t d \bar \omega d\bar v  \leq C \eps^{d-1} \bbV^d (\bbV  t)^{d+1} \,.
\end{aligned}
$$

\item 
 If the creation  of~$q$  is with scattering, then $v_q - v_{ q'} = \bar v - \bar v_{q} - 2(\bar v - \bar v_{q}) \cdot \bar \omega \,  \bar \omega$ has to belong to the union of the  $K_\zeta$'s. Equivalently $\bar v - \bar v_{q}$ lies in the  union of the  $S_{\bar \omega }K_\zeta$'s (obtained from~$K_\zeta$ by symmetry with respect to~$\bar \omega$), and there holds
$$
\begin{aligned}&\int \indc _{\mbox{\tiny Self-recollision with creation at time $\bar t$   with scattering}} \big| \big(\bar v -\bar v_{q}  )\cdot \bar \omega \big |  d  \bar t d \bar \omega d\bar v \\
& \qquad\qquad  \leq C\bbV^d t^d
\sup_{\zeta}
\int \indc _{\bar v -\bar v_{ q} \in S_{\bar \omega }K_\zeta }
\big| \big(\bar v -\bar v_{q}  )\cdot \bar \omega \big |  d  \bar t d \bar \omega d\bar v  \leq C \eps^{d-1} \bbV^d (\bbV  t)^{d+1} \,.
\end{aligned}
$$
\end{itemize}

$ \bullet $ $ $ \underbar{If the first deflection corresponds to the $k$-th clustering recollision}  between $\Psi^\eps_{\{j_k\}} $ and $\Psi^\eps_{\{j'_k\} }$ in the forest~$\Psi^\eps_{\lambda_1} $ for instance, in addition to the condition 
$ \hat x_k\in B_{qq'}$ which encodes the clustering recollision (see Section~\ref{subsec:estDC}), we obtain the condition
\begin{equation}
\label{self-recollc}
\begin{aligned}
\eps \omega_{{\rm rec},k}  + (v_q  -  v_{ q'})  (t_{\rm rec} - \tau_{{\rm rec},k} ) = \eps \omega_{\rm rec} +\zeta \,  \hbox{ with } \,    \zeta\in \Z^d\, \\
\hbox{ and } v_q - v_{ q'} = \bar v_q - \bar v_{q'} - 2(\bar v_q - \bar v_{q'}) \cdot \omega_{{\rm rec},k}   \, \omega _{{\rm rec},k}
\end{aligned}
\end{equation}
denoting by $\bar v_q, \bar v_{q'}$ the velocities before the clustering recollision in the backwards dynamics, and by~$\omega_{{\rm rec},k}$ the impact parameter at the clustering recollision.
We deduce from the first relation that~$v_q- v_{ q'}$ has to be in a small cone $K_\zeta$ of opening $\eps$, 
which implies by the second relation  that  $\omega_{{\rm rec},k} $
has to be in a small cone $S_\zeta$ of opening $\eps$. 

 Using the   change of variables~(\ref{infinitesimalij}), it follows that
$$
\begin{aligned}
\int \indc _{\mbox{\tiny Self-recollision with clustering at time $\bar t$}}  \, d \hat x_k & \leq C
 \eps^{d-1} t \sum_{\zeta}
\int \indc _{\omega_{{\rm rec},k} \in  S_\zeta} \big( (
 \bar v_q - \bar v_{q'}) \cdot \omega_{{\rm rec},k}
 \big)
d\omega_{{\rm rec},k}
   \\&\leq C \eps^{2( d-1) } \left(t \bbV \right)^{d+1}  \,.
\end{aligned}
$$
This concludes the proof of Proposition~\ref{periodicrecollisionprop}. \qed  

\medskip
  
   {\bf Non clustering recollision: proof of Proposition~\ref{recollisionprop}}\label{nonclustering}

 Denote by $q,q'$ the recolliding particles. Without loss of generality,  we can assume that the first deflection (when going up the tree) involves only particle $q$, at some time~$\bar t$. It can be either a creation (with or without scattering), or a clustering recollision.
\smallskip

$\bullet$ \underbar{If the first deflection of $q$  corresponds to a creation}, we denote by $(\bar t, \bar \omega, \bar v)$ the parameters encoding this creation, and by $(\bar x_q, \bar v_q)$ the position and velocity of the  parent  $\bar q$ before the creation in the backward dynamics. %(see Figure \ref{recollfigure}).
%As explained before the statement of Proposition~\ref{recollisionprop}, we use that notation even if~$q$ is created at time~$\bar t$ and~$c$ is its parent.
 Note that locally  in time (up to the next deflection) $\bar v_q$ is constant, and $\bar x_q$ is an affine function. In the same way, denoting  by $(\bar x_{q'}, \bar v_{ q'})$ the position and velocity of the particle~$q'$, we have that $\bar v_{ q'}$ is locally constant while $\bar x_{q'}$ is affine.

There are actually  three subcases~:
\begin{itemize}
\item[(a)] particle $q$ is created without scattering~: 
$ v_q = \bar v$ ; 
\item[(b)] particle $q$ is created with scattering~:  $ v_q = \bar v + (\bar v - \bar v_q ) \cdot \bar \omega \, \bar \omega
$ ; 
\item[(c)] another particle is created next to $q$, and $q$ is scattered~: $ v_q = \bar v_q  + (\bar v - \bar v_q ) \cdot \bar \omega \, \bar \omega
$.
\end{itemize}

The equation for the recollision states
\begin{equation}
\label{recoll}
\begin{aligned}
 \bar x_q (\bar t ) + \eps \bar \omega -  \bar x_{ q'} (\bar t)  + (v_q - \bar v_{ q'})  (t_{\rm rec} - \bar t) = \eps \omega_{\rm rec} +\zeta \, \hbox{ in cases (a)-(b)}  ,\\
 \bar x_q (\bar t )  -  \bar x_{ q'}  (\bar t)    + (v_q - \bar v_{ q'})  (t_{\rm rec} - \bar t) = \eps \omega_{\rm rec}+\zeta  \hbox{ in case (c)} \, . 
\end{aligned}
\end{equation}
We fix from now on the parameter~$\zeta\in \Z^d \cap B_{ \bbV  t}$ encoding the periodicity, and the estimates will be multiplied by~$\bbV^d t^d$ at the very end.
 Define
$$
\begin{aligned}
\delta x &:=  \frac1\eps ( \bar x_{ q'}(\bar t)   -\eps \bar \omega- \bar x_q (\bar t )+\zeta ) =: \delta x_\perp +   \frac1\eps (\bar v_{ q'} - \bar v_q ) (\bar t - t_0)  \hbox{  in cases (a)-(b)  }\, , \\
\delta x &:=  \frac1\eps ( \bar x_{ q'}(\bar t) - \bar x_q (\bar t )+\zeta  ) =: \delta x_\perp +   \frac1\eps (\bar v_{ q'} - \bar v_q ) (\bar t - t_0)  \hbox{  in case (c)} \, , \\
 & \qquad\qquad  \tau_{\rm rec}:= (t_{\rm rec} - \bar t)/\eps \quad \hbox{ and } \tau := (\bar t - t_0)/\eps\,,
\end{aligned}
$$
where~$\delta x_\perp $ is orthogonal to~$\bar v_{ q'} - \bar v_q $ (this constraint defines the parameter~$t_0$).
Then (\ref{recoll}) can be rewritten
\begin{equation}
\label{cylinder}
v_q - \bar v_{ q'}  = {1\over \tau_{\rm rec}} \Big( \omega_{\rm rec}  + \delta x_\perp + \tau ( \bar v_{ q'} - \bar v_q) \Big).
\end{equation}
We know that $v_q - \bar v_{ q'}$ belongs to a ball of radius $\bbV$. In the case when~$|\tau (\bar  v_{ q'} - \bar v_q) | \geq 2$,  the triangular inequality gives
$$ {1\over 2\tau_{\rm rec}}\big | \tau ( \bar v_{ q'} - \bar v_q)\big| \leq  {1\over \tau_{\rm rec}} \Big| \omega_{\rm rec}  + \delta x_\perp + \tau ( \bar v_{ q'} - \bar v_q) \Big| = |v_q - \bar v_{ q'}| \leq\bbV_{i,i'} $$
and we deduce that 
$$
\frac1{  \tau_{\rm{rec}}} \leq \frac{2\bbV}{ |\tau | | \bar v_{ q'} - \bar v_q|}
$$
hence, by (\ref{cylinder}), $v_q -  \bar v_{q'}$  belongs to a cylinder of main axis $\delta x_\perp + \tau ( \bar v_{q'}- \bar v_{ q} )$ and  of width $2\bbV/|\tau | | \bar v_{ q} - \bar v_{q'}|$.
In any case,  (\ref{cylinder}) forces $v_q - \bar v_{q'}$ to belong to a cylinder $\cR_\zeta$ of main axis $\delta x_\perp + \tau ( \bar v_{q'}- \bar v_{ q} )$ and  of width~$C\bbV\min \left( \frac{1}{|\tau| | \bar v_q - \bar v_{q'}|}, 1 \right)$.
In any dimension $d\geq 2$, the volume of this cylinder is less than~$C\bbV^d\min \left( \frac{1}{|\tau| | \bar v_q - \bar v_{q'}|}, 1 \right)$.

\medskip
\noindent
  \underbar{Case (a)}. Since $  v_q = \bar v$, Equation~(\ref{cylinder}) forces $\bar v -  \bar v_{q'}$ to belong to the cylinder~$\cR _\zeta$. Recall that $\tau$ is a rescaled time, with
  $$| ( \bar v_q -  \bar v_{q'}) \tau| \leq \frac {t}\eps | \bar v_q -  \bar v_{q'} |+ |\delta x _\parallel | \leq \frac C\eps (\bbV t +1)\,.$$
  Then 
\begin{equation*}
\begin{aligned}
\int_{|\bar v | \leq \bbV}  \indc_{\bar v  - \bar v_{q'} \in \cR_\zeta  } \;  \big |  (\bar v - \bar v_q  ) \cdot  \bar \omega \big|  d\bar t d\bar  \omega d\bar v 
& \leq  C \bbV^{d+1}  \int _{- C(\bbV t+1)  /\eps} ^{C(\bbV t+1) /\eps }   \min \left( \frac{1}{|u|}, 1 \right) \eps \frac{du}{ | \bar v_q - \bar v_{q'}| }\\
& \leq    C\bbV^{d+1} {\eps\big( |\log(\bbV t +1) |+|\log \eps|\big)\over |\bar v_q -  \bar v_{q'}|}  \,\cdotp
\end{aligned}
\end{equation*}

\medskip
\noindent
  \underbar{Cases (b) and (c)}.
By definition,~$v_q$ belongs to the sphere of diameter $[\bar v, \bar v_q ]$. The intersection~$I$ of this sphere and of the cylinder $\bar v_{q'} +\cR$ is a union of spherical caps, and we can estimate the solid angles of these caps. 

\begin{figure} [h] %  figure placement: here, top, bottom, or page
   \centering
  \includegraphics[width=11cm]{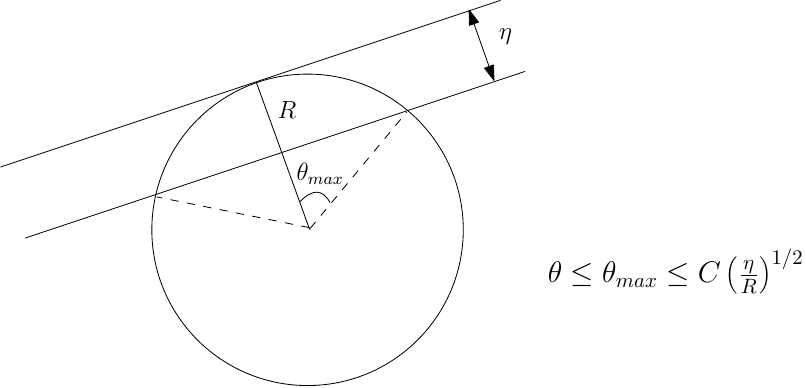} 
  \caption{Intersection of a cylinder and a sphere. The solid angle of the spherical caps is less than $C_d\min (1, (\eta/R)^{1/2})$.}
\label{fig:intersection}
\end{figure}

A basic geometrical argument shows that $\bar \omega$ has therefore  to be in a union of solid angles of measure less than $ C \min \Big( \big( \frac{ \bbV }{|\tau| | \bar v_q - \bar v_{q'}| |\bar v_q - \bar v |}\big)^{1/2} , 1 \Big) $. Integrating first with respect to $\bar \omega$ and $\bar v$, then with respect to  $\bar t$, we obtain
\begin{equation*}
\begin{aligned}
\int_{|\bar v | \leq \bbV}  \indc_{v_q \in I   } \;  \big |  (\bar v - \bar v_q  ) \cdot  \bar \omega \big|  d\bar t d\bar  \omega d\bar v 
& \leq  C \bbV^{d+1}  \int _{- C(\bbV t +1) /\eps}^{C(\bbV t+1) /\eps }  \min \Big( \frac{1  }{|u| ^{1/2}} , 1 \Big)\eps { du \over | \bar v_q - \bar v_{q'}|} \\
& \leq    C \bbV^{d+\frac32}  { \eps^{1/2} t^{\frac12 } \over |\bar v_q -\bar  v_{ q'}| }  \,\cdotp
\end{aligned}
\end{equation*}

\medskip
We obtain finally that 
$$
\int \indc _{\mbox{\tiny Recollision  of type (a)(b)(c)}}  \big |  (\bar v - \bar v_q  ) \cdot  \bar \omega \big|  d\bar t d\bar  \omega d\bar v \leq C   \bbV^{2d+\frac32} (1+t)^{d+\frac12} { \eps^{\frac12}  \over |\bar v_q - \bar v_{ q'}| }\,\cdotp$$

$ \bullet $ $ $ \underbar{If the first deflection of $q$ corresponds to a clustering recollision}.
With the notation of Section~\ref{subsec:estDC} we assume the clustering recollision is the~$k$-th recollision in~$\Psi^\eps_{\lambda_1}$  between the trees~$\Psi^\eps_{j_k}$ and~$\Psi^\eps_{j'_k}$, involving particles~$q\in \Psi^\eps_{\{j_k\}}$ and~$c\in \Psi^\eps_{\{j'_k\}}$  (with~$c \neq q'$) at time $\bar t = \tau_{{\rm rec},k}$. Then  in addition to the condition
$$ \hat x_k    \in B_{qc}$$ which encodes the clustering recollision (see Section~\ref{subsec:estDC}), we obtain the condition
\begin{equation}
\label{recolld}
\begin{aligned}
\big(\bar x_q ( \tau_{{\rm rec},k}  )  -  x_{q'} ( \tau_{{\rm rec},k} ) \big)  + (v_q - \bar v_{ q'})  (t_{\rm rec} -  \tau_{{\rm rec},k}) = \eps \omega_{\rm rec} +\zeta \, ,\\
\hbox{ and } v_q  = \bar v_q - (\bar v_q - \bar v_c) \cdot \omega_{{\rm rec},k}  \,  \omega _{{\rm rec},k} 
\end{aligned}
\end{equation}
denoting by $(\bar x_q, \bar v_q)$ and $(\bar x_c,  \bar v_c)$ the positions and velocities of $q$ and $c$ before the clustering recollision (in the backward dynamics).
Note that, as previously, $\bar v_q$ and $\bar v_c$ are locally constant.
Defining as above
$$\delta x :=  \frac1\eps ( \bar x_q ( \tau_{{\rm rec},k}  )  -  x_q ( \tau_{{\rm rec},k}) +\zeta ) =: \delta x_\perp +  (\bar v_{ q'} - \bar v_q ) ( \tau_{{\rm rec},k} - t_0)/\eps \hbox{ with } \delta x_\perp \perp (\bar v_{ q'} - \bar v_q ) \, ,
$$
and the rescaled times $$
  \tau_{\rm rec}:= (t_{\rm rec} - \tau_{{\rm rec},k})/\eps \quad \hbox{ and } \tau =: ( \tau_{{\rm rec},k} - t_0)/\eps\,,
$$
we end up with the  equation (\ref{cylinder}), which  forces $v_q - \bar v_{ q'}$ to belong to a cylinder~$\cR $ of main axis~$\delta x_\perp - \tau (\bar v_q  -  \bar v_{ q'})$ and 
 of width~$C\bbV \min \left( \frac{1}{|\tau( \bar v_q  -  \bar v_{ q'})|}, 1\right)$,
 where~$\Psi^\eps_{\{i'\}}$ is the collision tree of~$q'$.
Then $v_q$ has to be in the intersection of the sphere of diameter~$[\bar v_q, \bar v_c]$ and of the cylinder~$\bar v_{q'} + \cR$. This implies that $\omega_{{\rm rec},k}   $ has to belong to a union of spherical caps $S$, of solid angle  less than 
$ C \min \Big( \big( \frac{ \bbV}{|\tau| | \bar v_q -  \bar v_{ q'}| |\bar v_q - \bar v_c |}\big)^{1/2} , 1 \Big) $.
 Using the (local) change of variables $\hat x_k \mapsto (  \tau_{{\rm rec},k}, \eps \omega_{{\rm rec},k}  )$, it follows that
$$
\begin{aligned}\int \indc _{\mbox{\tiny Recollision  of type (d)}} d \hat x_k  & \leq {C\over \mu_\eps} \int \indc _{\omega_{{\rm rec},k}  \in  S}
 |( \bar v_q - \bar v_c)\cdot \omega_{{\rm rec},k}  | d\omega_{{\rm rec},k}  d\tau _{{\rm rec},k} \\
 &    \leq {C\over \mu_\eps}  \bbV  ^{\frac32} (1+t)^{\frac12} {\eps^{1/2} \over  | \bar v_q - \bar v_{ q'}| } \; \cdotp
\end{aligned}
$$
This concludes the proof of Proposition~\ref{recollisionprop}.
 \qed

{\bf Integration of the singularity in relative velocities: proof of Proposition \ref{vsingularityprop}}\label{integrationsingularities}

  We start with the obvious estimate
\begin{equation}
\label{obvious}
 \min \Big( 1, {\eps^{1/2} \over |v_q - v_{ q'} | } \Big)  \leq  \eps^{\frac14} + \indc_{|v_q - v_{ q'}| \leq \eps^{1/4}} \, .
 \end{equation}
 Thus we only  need  to control the set of parameters leading to small relative velocities.
 
 Without loss of generality, we shall assume that the first deflection (when going up the tree) involves  particle $q$. It can be either a creation (with or without scattering), or a clustering recollision, say  between $q \in \Psi^\eps_{\{j_k\}}$ and $c \in \Psi^\eps_{\{j_k'\}}$.

 $\bullet$ \underbar{If the first deflection of $q$  corresponds to a creation}, we denote by $(\bar t, \bar \omega, \bar v)$ the parameters encoding this creation, and by $(\bar x_q, \bar v_q)$ and $(\bar x_{q'}, \bar v_{q'})$ the positions and velocities of the pseudo-particles $q$ and $q'$ before the creation.

There are actually  four subcases~:
\begin{itemize}
\item[(a)] particle $q'$ is created  next to particle $q $ in the tree~$\Psi^\eps_{\{i\}}$: 
$ |v_q - v_{ q'} | = | \bar v - \bar v_q| $ ;
\item[(b)]  particle $q'$ is not deflected and particle $q$ is created  without scattering next to $\bar q$   in the tree~$\Psi^\eps_{\{i\}}$: 
$ |v_q - v_{ q'} | = | \bar v- \bar v_{q'}|$ ;
\item[(c)] particle $q '$ is not deflected and particle $q$ is created with scattering next to $\bar q$  in the tree~$\Psi^\eps_{\{i\}}$:  $ v_q   = \bar v - (\bar v - \bar v_q ) \cdot \bar \omega\,  \bar \omega 
$ ;
\item[(d)] particle $q'$ is not deflected,  another particle is created next to $q$ in the tree~$\Psi^\eps_{\{i\}}$, and $q$ is scattered so~$ v_q = \bar v_q+ (\bar v - \bar v_q ) \cdot \bar \omega\, \bar \omega
$ . 
\end{itemize}
 In cases (a) and (b), we obtain that $\bar v$ has to be in a small ball of radius $\eps^{1/4}$. Then,
 \begin{equation*}
\int \indc _{\mbox{\tiny Small relative velocity of type (a)(b)}}  \big |  (\bar v - \bar v_q  ) \cdot  \bar \omega \big|  d\bar t d\bar  \omega d\bar v \leq  C\bbV t\eps^{d/4} \, .
\end{equation*}
  In cases (c) and (d), we obtain that $v_q$ has to be in the intersection of a small ball of radius~$\eps^{1/4}$ and of  the sphere of diameter $[\bar v, \bar v_q]$.
  This condition imposes that $\bar \omega$ has to be in a  spherical cap of solid angle less than $ \eps^{\frac18} / | \bar v - \bar v_q|^{1/2} $ (see Figure \ref{fig:intersection}). We find that
  \begin{equation*}
\int \indc _{\mbox{\tiny Small relative velocity of type (c)(d)}}  \big |  (\bar v - \bar v_q  ) \cdot  \bar \omega \big|  d\bar t d\bar  \omega d\bar v \leq  C\bbV ^{d+\frac12} t \eps^{\frac18} \, .
\end{equation*}
  Combining these two estimates with (\ref{obvious}), we get
  \begin{equation*}
\int  \min \Big( 1, {\eps^{1/2} \over |v_q - v_{ q'} | } \Big) \big |  (\bar v - \bar v_q  ) \cdot  \bar \omega \big|  d\bar t d\bar  \omega d\bar v \leq  C\bbV ^{d+1} t \eps^{\frac18}  \, .
\end{equation*}

$ \bullet $ $ $ \underbar{If the first deflection of $q$ corresponds to the $k$-th  clustering recollision} in~$\Psi^\eps_{\lambda_1}$ between $q\in \Psi^\eps_{\{j_k\}}$ and $c \in \Psi^\eps_{\{j'_k\}}$ at time $\bar t = \tau_{{\rm rec},k} $, in addition to the condition 
$ \hat x_k \in B_{qc}$ which encodes the clustering recollision (see Section~\ref{subsec:estDC}), we obtain a condition on the velocity.

There are actually  two subcases~:
\begin{itemize}
\item[(e)] $q' = c$ and $|v_q - v_{ q'} | = |\bar v_q - \bar v_{q'} |$ ; 
\item[(f)] $q' $ is not deflected,  and $ v_q = \bar v_q - (\bar v_q - \bar v_c) \cdot \omega_{{\rm rec},k}  \, \omega_{{\rm rec},k} $ .
 \end{itemize}

In case (e),  there holds
 \begin{equation*}
\int \indc _{\mbox{\tiny Small relative velocity of type (e)}} d\hat x_k \leq {C\over \mu_\eps}  \int \indc _{|\bar v_q - \bar v_{q'} | \leq \eps^{1/4} } \big |  (\bar v_q - \bar v_{q'}  ) \cdot  \omega  \big|  d\omega  d\tau_{{\rm rec},k}  \leq  {C t  \eps^{\frac14} \over \mu_\eps} \, \cdotp
\end{equation*}
In case (f), we obtain that $v_q$ has to be in the intersection of a small ball of radius $\eps^{1/4}$ and of  the sphere of diameter $[\bar v_q, \bar v_c]$.
 This condition imposes that $\omega_{{\rm rec},k} $ has to be in a  spherical cap of solid angle less than $ \eps^{\frac18} / | \bar v_q - \bar v_c|^{1/2} $ (see Figure \ref{fig:intersection}). We find
 \begin{equation*}
\int \indc _{\mbox{\tiny Small relative velocity of type (f)}} d\hat x_k  \leq  {C\over \mu_\eps} \eps^{\frac18}  \int  \big |  \bar v_q - \bar v_c  \big|^{1/2}    d\tau_{{\rm rec},k} \leq  {Ct \bbV^{\frac12}   \eps^{\frac18}\over \mu_\eps}  \, \cdotp
\end{equation*}

 Combining these two estimates with (\ref{obvious}), we get
  \begin{equation*}
\int  \min \Big( 1, {\eps^{1/2} \over |v_q - v_{ q'} |^{1/2} } \Big) d\hat x_k \leq  {C\bbV t\eps^{\frac18} \over \mu_\eps}  \,\cdotp
\end{equation*}
This concludes the proof of Proposition~\ref{vsingularityprop}.
\qed

\appendix
\chapter{The abstract Cauchy-Kovalevskaya  theorem}
\label{CauchyKol}
In this appendix we recall the well-known Cauchy-Kovalevskaya theorem, in a generalized Banach framework as devised namely by F. Treves~\cite{Treves2}, L. Nirenberg~\cite{nirenberg}, T. Nishida~\cite{nishida}.  This result is used to prove the existence and uniqueness of a solution for short times for the Boltzmann equation (Section \ref{boltz-Cauchy}), for the linearized Boltzmann equation (proof of Proposition \ref{prop: borne U*} in Section \ref{U*-Cauchy}), for the covariance equation~(\ref{eq: systeme covariance}) (Proposition \ref{prop: well posedness covariance mild} in Section \ref{cov-Cauchy}), and for the modified Hamiltonian equations (\ref{EL-mild}) (proof of Proposition \ref{Lem: definition hat I} in Section \ref{EL-Cauchy}).  

We state the result as proved in \cite{nishida2} (Th\'{e}or\`{e}me A\footnote{The   (suboptimal) estimate on the existence time, as well as the estimates as stated in Theorem \ref{nishidatheorem}, follow from a simple adaptation of the argument in \cite{nishida2}, pages 367-368.}).

 \begin{Thmapp}[\cite{nishida2}]
 \label{nishidatheorem}  
 Let $(X_\rho)_{\rho>0}$ be a decreasing sequence of Banach spaces with increasing norms $\|\cdot\|_\rho$. Consider the equation 
\begin{equation}
\label{model-eq}
  u (t) = u_0(t) + \int_0^t F\big(t,s, u(s)\big) ds \, , \quad t \geq 0
\end{equation}
where 
\begin{itemize}
\item there are~$A_0>0, \rho_0>0$ such that~$t\mapsto u_0(t)$ is continuous for $t \in [0,A_0(\rho_0-\rho)[$ with values in~$X_{\rho}$ for all~$\rho<\rho_0$, and there is~$R_0>0$ such that  
$$\forall t \in[0,A_0(\rho_0-\rho)[\, , \quad \|u_0(t) \|_{\rho } \leq R_0  \,;
$$

\item $F(\cdot,\cdot, 0) = 0$, and  there are~$R>R_0>0, T>0$ such that~$F$ is continuous from $[0,T] \times[0,T] \times B_R(X_{\rho'}) $ to $X_\rho$ for all~$ \rho<\rho' \leq \rho_0$, with $B_R$ the open ball of radius $R$. Moreover there is a constant~$C_R  $  such that for all~$u,v\in  B_R(X_{\rho'})$, for all~$(t,s) \in [0,T] $,
\begin{equation}
\label{L*-est original}
\| F(t,s,u) - F(t,s,v) \|_\rho \leq C_R  {\rho_0 \over \rho'- \rho} \|u - v \|_{\rho'}\, , \qquad 
 \rho_0/2 \leq \rho<\rho'\leq \rho_0 \,.
\end{equation}
\end{itemize}
Then there exists a constant $c$ (not depending on any of the previous parameters) such that {\rm(\ref{model-eq})}  has a unique solution on the time interval~$[0,  T  ]$ with $T = c/C_{4R_0}$, which is continuous in time and satisfies
$$\sup_{\rho_0/2 \leq \rho < \rho_0 \atop 
0\leq  t <  4T  (1-\rho/\rho_0)   }  \|u (t)\|_\rho \Big(1- {t \over 4T (1 - \rho/\rho_0)}\Big) \leq 2R_0$$
and in particular
$$  \|u (t)\|_{\rho_0/2 } \leq 4R_0\;,\qquad t \in [0,T]\;.$$
%Moreover,
%$$ \sup_{0 \leq t \leq  \rho_0 \eta/2} \|u (t)\|_{\rho_0 - 2t/\eta} < 2R_0 \,.$$
%Furthermore there is a universal constant~$c$ such that  
%$$
%  \sup_{0 \leq t \leq  c\rho_0/C} \|u (t)\|_{\rho_0 /2} \leq  2R_0\, .
%$$
\theoremstyle{plain}
\renewcommand{\Thmapp}{\Alph{theorem}} 
\end{Thmapp}

\section{Local well-posedness for the biased Boltzmann equation}
\label{boltz-Cauchy}

The local well-posedness of the Boltzmann equation (\ref{boltzmann-eq}) can be deduced directly from the previous theorem (as pointed out first in \cite{ukai-CK}). 
In this section, 
we are going to consider the well-posedness of the biased Boltzmann equation \eqref{biased-Boltz} recalled below
\begin{equation}
\label{biased-Boltz appendix}
\begin{aligned}
 D_t\varphi = \int  \Big( \varphi  (t,z') \varphi (t,z_2') e^{-\Delta p}   
  - \varphi  (t,z) \varphi (t,z_2) e^{\Delta p}   \Big)  \, d\mu_{z}  ( z_2 ,  \omega) 
  \quad  \text{with} \quad    
 \gp (0)   =  f^0  e^{\bar p(0)} ,
   \end{aligned}
\end{equation}
with $\|p\|_{W^{1,\infty}([0,T] \times {\mathbb D})} \leq r$.

%check the assumptions of Theorem \ref{nishidatheorem}  

We first define the weighted $L^\infty$ spaces\label{defLinftybeta}
$$L^\infty_\beta := \left \{ \gp = \gp(x,v) \, : \, 
\|\gp\|_{ L^\infty_\beta} :=\sup_\D  \left( \exp \big( -\frac\beta2  |v|^2\big) |\gp (x,v) | \right)  <+\infty \right\}\,.$$
Note that, by assumption (\ref{lipschitz}), the initial data $f^0$ belongs to $L^\infty_{-\beta_0}$
so that 
$$ 
\|  \gp (0)   \|_{ L^\infty_{-\beta_0}} \leq C_0 e^r \,.
$$
Note also that these functional spaces are invariant by the free transport operator $S_t$ over $\D$.

The mild formulation of \eqref{biased-Boltz appendix} states
\begin{equation}
\label{mild-boltz}
\gp (t) = S_t \gp(0) +\int_0^t S_{t-s} Q_p (\gp(s), \gp(s)) ds
 \end{equation}
where the collision term
$$ Q_p (\gp,\gp) (z) := \int \Big( \varphi  (t,z') \varphi (t,z_2') e^{-\Delta p}   
  - \varphi  (t,z) \varphi (t,z_2) e^{\Delta p}   \Big)  \, d\mu_{z}  ( z_2 ,  \omega) 
%{\D}\int_{{\mathbb S}^{d-1}}   \! \Big(  f(x,w') f(x,v') - f(x,w) f(x,v)\Big) ((v-w)\cdot \omega)_+ d\omega dw
$$
satisfies the following loss continuity estimate for $ \beta_0/2 \leq \beta<\beta' \leq \beta_0$
\begin{equation}
\label{eq: estimation Q}
\begin{aligned}
\| Q_p (\gp,\gp)\|_{L^\infty_{-\beta}}& \leq 2 \| \gp\|_{L^\infty_{-\beta'}} ^2 \, e^{4r} \, \sup_v \left(\int  \exp \left( - {\beta'-\beta\over 2} |v|^2\right)  \exp \left( - {\beta' \over 2} |w|^2)\right) |v-w| dw d\omega \right) \\
& \leq c_d \| \gp\|_{L^\infty_{-\beta'}}^2  \, e^{4r} \, 
{\beta_0 \over \beta' - \beta} \beta_0 ^{-(d+1)/2} \,,
\end{aligned}
\end{equation}
where the constant $c_d$ depends only on the dimension $d$.

Then choosing $T_0 = c_d C_0^{-1}\beta_0^{(d+1)/2}$, we obtain by 
Theorem \ref{nishidatheorem} that the mild formulation of the Boltzmann equation (\ref{mild-boltz}) has a unique solution $\gp$  which is continuous on $[0,T_0 e^{- 5r}]$ and satisfies
$$
\sup_{\beta_0/2 < \beta < \beta_0 \atop 
0\leq  t <  4T_0 e^{- 5 r} (1-\beta/\beta_0 )   }  \| \gp (t)\|_{L^{\infty}_{-\beta}} 
\Big(1- {t \over 4T_0 e^{-5r} (1 - \beta/\beta_0)}\Big) \leq 2C_0 \,,
$$
and
\begin{equation}
\label{Mmax-appendix}  
\| \gp (t)\|_{L^{\infty}_{-\beta_0/2}} \leq 4C_0 e^r \;,\qquad t \in [0,T_0 e^{- 5r}]\;.
\end{equation}

\section{Well-posedness of the linearized Boltzmann (adjoint) equation.} 
\label{U*-Cauchy}

We prove now Proposition \ref{prop: borne U*}.
Let us recall the definition~(\ref{eq: espaces L2beta}) of the function spaces
$$L^2_\beta := \left \{ \gp = \gp(x,v) \, : \, 
\|\gp\|_{L^2_\beta}^2:=\int _\D \exp \big(- \frac\beta2  |v|^2\big) \, \gp^2 (x,v) dxdv  <+\infty \right\}\,.
$$

We need to prove that if~$\gp$ is in~$ L^2_{\beta_0/4}$,  then   $\mathcal{U}^* (t,s) \gp $ belongs to~$ L^2_{3\beta_0/8}$ for any $s \leq t \leq T $ for~$T$ small enough. We get from \eqref{eq: linearized backward}-\eqref{eq: L*} the backward Duhamel formula
\begin{equation}
\label{eq: linearise psi}
\mathcal{U}^* (t,s) \gp = S_{s- t }\varphi + \int_s^t S_{s-\sigma}   {\bf L}_\sigma^* \,\mathcal{U}^* (t,\sigma) \gp \, d\sigma \,.
\end{equation}

Using the uniform bound (\ref{Mmax-appendix}), we first establish a loss continuity estimate for the operator ${\bf L}_s^*$ defined by \eqref{eq: L*}.
By the Cauchy-Schwarz inequality, for any function~$\gp$ and 
any~$\frac{\beta_0}{4}\leq \beta'<\beta\leq  \frac{3\beta_0}8$,
\begin{equation}
\label{L*-est}
\begin{aligned}
\| {\bf L}_s^*\,\gp \|^2_{L^2_\beta} & \leq  \int dxdv \exp ( - \frac\beta2 |v|^2)  \left( \int |v-w|^2 f^2(s,x,w) \exp ( \frac{\beta'} 2 |w|^2)  dw d\omega \right) \\ 
&\ \ \ \ \ \times \left(\int (\Delta \gp)^2(s,x,w) \exp ( - \frac{\beta'} 2 |w|^2)dwd\omega \right) \\
& \leq c_d C_0^2 \| \gp\|_{L^2_{\beta'}}^2 \beta_0^{-d/2} 
\ \sup_v   \left(   \exp ( - \frac{\beta-\beta'}2 |v|^2) \int |v-w|^2 \exp (- \frac{5\beta_0}{16} |w|^2) dw \right) \\
& \leq   c_d C_0^2 \beta_0^{-(d+1)} {\beta_0 \over \beta- \beta'  } \| \gp \|_{L^2_{\beta'}}^2\,,
\end{aligned}
\end{equation}
where $c_d$ denotes a constant depending only on the dimension $d$ which may change from line to line.

Since the transport~$S_s$ preserves the spaces $L^2_\beta$,  we are   in position to apply Theorem \ref{nishidatheorem}.
The only difference is that~\eqref{eq: linearise psi} defines a backward evolution, rather than a forward one, and that the $L^2_\beta$ spaces   are increasing rather than decreasing. Up to these slight adaptations, Theorem \ref{nishidatheorem}  provides  the existence of $T\leq T_0$, also  of the form $T= c_d \beta_0^{(d+1)/2} /C_0$, such that for any $\gp\in L^2_{\beta_0/4}$,  \eqref{eq: linearise psi} has a unique solution satisfying
 $\mathcal{U}^* (t,s) \gp \in L^2_{3\beta_0/8}$ for any $s \leq t \leq T $.
  Proposition~\ref{prop: borne U*} is proved. \qed
  
Notice that, for the linear equation \eqref{eq: linearise psi}, the fixed point argument leading to the Cauchy-Kovalevskaya  theorem provides in particular a convergent series representation for the solution, of the form
\begin{equation}
\mathcal{U}^* (t,s) \gp = S_{s- t }\varphi + \sum_{n \geq 1}
\int_s^t d\sigma_1 \cdots \int_{\sigma_{n-1}}^t d\sigma_n S_{s-\sigma_1}{\bf L}_{\sigma_1}^* \cdots
{\bf L}_{\sigma_n}^* S_{\sigma_n - t} \gp\;.
\end{equation}
In particular, the following properties are easily verified.
\begin{Cor} 
\label{cor: mild formula lin properties} 
For $T \leq T_0$ as in Proposition~{\rm\ref{prop: borne U*}} and for any $s \leq t \leq T $, $\mathcal{U}^* (t,s)$ is a  semigroup  satisfying 
$$ \mathcal{U}^* (t,s) = \mathcal{U}^* (\sigma,s)\,\mathcal{U}^* (t,\sigma)\;,\qquad \sigma \in [s,t]$$
and
$$
\mathcal{U}^*(t,s) \gp =  S_{s - t}\gp  + \int_s^t d\sigma\, \mathcal{U}^*(\sigma,s) {\bf L}_\sigma^* S_{\sigma-t}\gp\, .
$$
\end{Cor}

\section{Well-posedness of the covariance equation}
\label{cov-Cauchy}

\begin{Prop} 
\label{prop: well posedness covariance mild}
%Let~$(\psi, \gp)$ belong to~$
% L^2_{\beta_0/4}$ and consider~$(\phi_s)_{s \in [0,T_0]}$   a family of bounded functions with values in~$ L^2_{\beta_0/4}$. 
There exists a time~$T>0$ of the form $T= c_d \beta_0^{(d+1)/2} /C_0$ such that the system~{\rm(\ref{eq: systeme covariance})} has a unique solution~$\mathcal C$ on $[0,T]^2$, which is defined as a bilinear form on $ L^2_{\beta_0/4}$.
\end{Prop}

\begin{proof}
System~(\ref{eq: systeme covariance}) consists in two equations. Let us start by solving the first one, namely
\begin{equation}\label{eq: systeme covariance1}
\begin{aligned} 
& \mathcal C(t,t,\psi, \gp)  = \mathcal C(0,0,S_{-t}\psi,S_{-t}\gp ) + \int_0^t ds\, {\bf  Cov}_s( S_{s-t}\psi, S_{s-t}\gp)   \\
  &  \qquad\qquad + \int_0^t ds\, \mathcal C(s,s,S_{s-t}\psi, {\mathbf L}_s^*S_{s-t}\gp)  
  + \int_0^t ds\, \mathcal C(s,s, {\mathbf L}_s^*S_{s-t}\psi, S_{s-t}\gp)\, .
  \end{aligned}
\end{equation}
  We are going to apply Theorem  \ref{nishidatheorem}, with the family of spaces~${\mathcal X}_\beta$ of bilinear forms defined by
  $$
  {\mathcal X}_\beta:= \Big\{
   \mathcal C:= \mathcal C(\psi, \gp) \, / \, 
    \| \mathcal C\|_{  {\mathcal X}_\beta}< \infty   
  \Big
  \} \, , \quad  \| \mathcal C\|_{  {\mathcal X}_\beta} := \sup_{\|\psi\|_{L^2_\beta} \leq 1 ,\|\varphi\|_{L^2_\beta} \leq 1 } \big|\mathcal C(\psi, \gp) \big|\, .
  $$
    Notice that, since the spaces $L^2_\beta$ are increasing,  the spaces~$ {\mathcal X}_\beta$ are decreasing.
 Given~$\beta \leq \beta_0$ and~$\psi, \gp$ in~$ L^2_{\beta}$ of norm smaller than~1, there holds
    $$
    \begin{aligned}
 \big | \mathcal C(0,0,S_{-t}\psi,S_{-t}\gp )\big| 
& \leq  \int f^0(z) |S_{-t}\psi(z)|  |S_{-t}\gp(z)| \, dz \\
& \leq  C_0  \int  e^{(\frac{\beta }2-\frac{\beta_0}2) |v|^2}  
e^{-\frac{\beta }{4}  |v|^2} |S_{-t}\psi(z)|e^{-\frac{\beta }{4}  |v|^2} |S_{-t}\gp(z)| \, dxdv
  \end{aligned}$$
  so by the Cauchy-Schwarz inequality we infer
  $$
  \big  \| \mathcal C(t=0,t=0)  \big\|_{ {\mathcal X}_{\beta/2}}   \leq C_0  \, .
  $$
  Similarly, as in the proof of Proposition~\ref{prop:uniq-cov} page~\pageref{prop:uniq-cov}, we find that
  $$
  \begin{aligned}
 & \big |{\bf  Cov}_s( S_{s-t}\psi, S_{s-t}\gp) \big | \leq \frac{1}{2}  \int d\mu(z_1,z_2,\omega) f(s,z_1) f(s,z_2) |\Delta S_{s-t}\psi|  |\Delta S_{s-t}\gp| \, \\
&\ \ \ \ \   \leq C \, C_0^2  \int d\mu (z_1, z_2, \omega)  e^{(\frac{\beta }2-\frac{\beta_0}4) (|v_1|^2+|v_2|^2) }   \Big(e^{-\frac{\beta }{2}  |v_1|^2} \psi^2(s,z_1) +e^{-\frac{\beta }{2}  |v_1|^2} \gp^2(s,z_1)\Big)e^{-\frac{\beta }{2}  |v_2|^2} \\
&\ \ \ \ \   \leq c_d C_0^2 \beta_0^{-(d+1)/2}
  \end{aligned}$$
  if~$\psi, \gp$ belong to~$ L^2_{\beta}$  for $\beta \leq 3\beta_0/8$, and norm bounded by 1.
  
  Finally setting
  $$
  F(t,s,\mathcal C(s,s,\cdot, \cdot)):=\mathcal C(s,s,S_{s-t}\cdot, {\mathbf L}_s^* S_{s-t}\cdot)+ \mathcal C(s,s,{\mathbf L}_s^* S_{s-t}\cdot, S_{s-t}\cdot)
  $$
  let us prove the loss estimate~(\ref{L*-est original}).  
   There holds, for $\beta_0/4 \leq \beta'<\beta \leq 3\beta_0/8$,
     $$
       \begin{aligned}
\big| F(t,s,\mathcal C(s,s,\psi,\gp)) \big|  & \leq 2 \|\mathcal C(s,s)\|_{{\mathcal X}_{\beta}} \|S_{s-t}\psi\|_{L^2_{ \beta}} \| {\mathbf L}_s^* S_{s-t}\gp\|_{ L^2_{\beta}}\\
& \leq c_d C_0 \beta_0^{-(d+1)/2} {\beta_0 \over \beta- \beta'  }  \|\mathcal C(s,s)\|_{{\mathcal X}_{\beta}} \|\psi\|_{ L^2_{\beta'}}  \|\gp\|_{L^2_ {\beta'} }
     \end{aligned}$$
  where we have used
  the fact that the spaces~$L^2_\beta$ are increasing, along with the loss estimate~(\ref{L*-est}). 
  Thanks to  Theorem  \ref{nishidatheorem}, we find that there exists a time~$T>0$, proportional to $\beta_0^{(d+1)/2} /C_0$,  such that~(\ref{eq: systeme covariance1}) has a unique solution which is continuous on~$[0,T] $, with values in~${\mathcal X}_{\beta_0/4}$.

The argument is the same for the second equation of~(\ref{eq: systeme covariance}), namely
\begin{equation}\label{eq: systeme covariance2}
\begin{aligned} 
 \int_0^t   \mathcal C(t,\sigma ,\psi,\phi_\sigma) \,d \sigma
 =   \int_0^t d\sigma  \, 
\left( 
\mathcal C(\sigma ,\sigma , S_{\sigma-t}\psi,\phi_\sigma) 
+  \int_ \sigma^t ds   \;   
\mathcal C \Big( s, \sigma,   {\mathbf L}_s^* S_{s-t}\psi , \phi_\sigma \Big)  \right) \, , 
  \end{aligned}
\end{equation}
 applying  Theorem  \ref{nishidatheorem} to
  $$
  {\mathcal K}(t,\psi,\Phi):=   \int_0^t   \mathcal C(t,\sigma ,\psi,\phi_\sigma) \,d \sigma
      $$
      which satisfies, thanks to the Fubini theorem,
      $$
        {\mathcal K}(t,\psi,\Phi) =  \int_0^t d\sigma\, \mathcal C(\sigma ,\sigma , S_{\sigma-t}\psi, \phi_\sigma) 
+   \int_0^t ds  \,  {\mathcal K}( s, {\mathbf L}_s^* S_{s-t}\psi ,\Phi)\, .
      $$
Note that $  {\mathcal K}(t)$ is now a bilinear form on $ L^2_\beta \times L^\infty ( (0,t);L^2_\beta)$.
      The same estimates as above allow to conclude.
        \end{proof}

\section{Well-posedness of the modified Hamiltonian equations}
\label{EL-Cauchy}

We are now going to check the well-posedness of the modified Hamiltonian equations \eqref{EL-mild}
which are recalled below
\begin{equation}
\label{EL-mild bis}
\forall s \leq t, \qquad
\begin{aligned}
& \psi_t (s)  = S_s   f^0+\int_0^s S_{s-\sigma}  F_1 \big( \phi_t (\sigma), \eta_t (\sigma), \psi_t (\sigma) \big) d\sigma \,, \\
&  \eta_t (s)  =   S_{s- t } \gamma_t - \int_s^t S_{s-\sigma }  F_2 \big( \phi_t (\sigma), \eta_t (\sigma),\psi_t (\sigma) \big) d\sigma 
\,,
\end{aligned}
\end{equation}
with $\psi_t(0) = f^0, \eta_t (t) =  \gamma$ and 
$$ 
\begin{aligned}
& F_1 ( \phi, \eta, \psi)  =  - \psi \, \phi+ \int   d\mu_{z_1}(z_2,\omega)\, 
\eta (z_2) 
\Big( \psi(z_1') \psi(z_2')  -  \psi(z_1) \psi(z_2) \Big)\,,\\
& F_2( \phi, \eta, \psi) =\eta  \, \phi -  \int  d\mu_{z_1}(z_2,\omega)\, 
 \psi (z_2)  \Big( \eta(z_1') \eta(z_2') -\eta(z_1) \eta(z_2) \Big)\,.
\end{aligned}
$$ 
This is a coupled system and~$\eta_t$ satisfies a backward equation, so this is not exactly the standard formulation to apply Theorem \ref{nishidatheorem}.

Let us fix~$\alpha >0$ and a time $t  \leq T_\alpha$.
Using the fact that~$(\phi,\gamma) $ belongs to~$ \cB_{\alpha, \beta_0,T_\alpha}$, we   have in particular that
$$
\sup_{s \in [0, t]} \big| \phi(s,x,v) \big| \leq  C (1+|v|^2) \quad \mbox{and} \quad  \gamma(t) \in L^\infty_{\beta_0/4} \, ,
$$
 where the constant~$C$ depends on~$\alpha, \beta_0, C_0$.
Recall moreover that~$f^0$ belongs to~$L^\infty_{-\beta_0}$, so let us define
$$
\bar C: = 4 \left( \|\gamma\|_{ L^\infty_{\beta_0 /4}} + \|f^0\|_{L^\infty_{-\beta_0} }\right) .
$$
By a computation as in \eqref{eq: estimation Q}, one can check that 
for any~$3 \beta_0/4 <\beta _1<  \beta'_1 \leq \beta_0 $ and $\beta_0/4\leq\beta'_2<\beta_2 <   \beta_0 /2$ there are constants~$C_1$  and~$C_2$   such that 
\begin{eqnarray}
\| F_1 ( \phi, \eta, \psi) \|_{L^\infty_{- \beta_1}}& \leq \displaystyle {C_1  \beta_0\over \beta'_1- \beta _1} \| \psi\|_{ L^\infty_{-\beta'_1}} \Big( 1 +  \| \psi\|  _{ L^\infty_{- \beta'_1}} \| \eta\|_{L^\infty_{\beta_0/2}} \Big)
\,,\label{estimateF1}
\\
\| F_2( \phi , \eta, \psi) \|_{L^\infty_{\beta_2}}& \leq\displaystyle  {C_2 \beta_0\over \beta_2- \beta'_2}  \| \eta\|_{ L^\infty_{\beta'_2}} \Big( 1 + \| \psi\| _{ L^\infty_{-3\beta_0/4} }\| \eta\| _{L^\infty_{\beta'_2}} \Big)
 \,.\label{estimateF2}
\end{eqnarray}
 The second equation in \eqref{EL-mild bis} evolves backward so that as in Section~\ref{U*-Cauchy},  the regularity in \eqref{estimateF2} is coded in the opposite direction of the forward flow.

By the method of Theorem~\ref{nishidatheorem},  a fixed point argument can 
be implemented (by solving at each iteration both the forward and backward equations). 
In this way, we find a time $T_\alpha^{\rm\tiny H'}>0$ such that there exists a unique solution}   $(\psi_t,\eta_t)$ to~\eqref{EL-mild bis} on $[0,t]$ for any~$t \leq T_\alpha^{\rm\tiny H'}$, satisfying
$$
\sup_{s \in [0,t] } \| \eta_t (s)\|_{L^\infty_{\beta_0/2}} \leq \bar C \, , 
\quad  
\sup_{s \in [0,t] } \| \psi_t (s)\|_{L^\infty_{-3\beta_0/4}} \leq \bar C\,.
$$
  Step 1 of the proof of Proposition~\ref{Lem: definition hat I} is now complete.

\backmatter


\begin{thebibliography}{99}


 \bibitem{Ale75} R.K. Alexander. {\it The infinite hard sphere system}.
Ph.D. Thesis, Dep. of Math., University of California at Berkeley, 1975.

 \bibitem{Ayi} N. Ayi. From Newton's law to the linear Boltzmann equation without cut-off. {\it Comm. Math. Phys.} 350(3):1219--1274, 2017.
 
 \bibitem{BBBC21}
G. Basile, D. Benedetto, L. Bertini and E. Caglioti. 
Large deviations for a binary collision model: energy evaporation. 
arXiv:2202.0731.

 \bibitem{BBBC22}
G. Basile, D. Benedetto, L. Bertini and E. Caglioti. 
Asymptotic probability of energy increasing solutions to homogeneous Boltzmann equation.
{\it Math. in Engin.} 5(1):1-12, 2022.

\bibitem{BBBO21}
G. Basile, D. Benedetto, L. Bertini and C. Orrieri. 
Large deviations for Kac-Like Walks. 
{\it J. Stat. Phys.} 184, 2021.
 
 \bibitem{BLLS80} H. van Beijeren, O.E. Lanford III, J.L. Lebowitz and H. Spohn.
Equilibrium time correlation functions in the low--density limit.
{\it J. Stat. Phys.} 22(2):237-257, 1980.

\bibitem{BdSGJ-LL15}
L. Bertini, A. De Sole, D. Gabrielli, G. Jona-Lasinio and C. Landim. Macroscopic fluctuation theory. 
{\it Rev. Mod. Phys.} 87:593-636, 2015.

\bibitem{Billingsley} P. Billingsley. {\it Probability and measure}.  John Wiley \& Sons, 1979.
 

\bibitem {BGSR2}
T. Bodineau, I. Gallagher and L. Saint--Raymond.
From hard sphere dynamics to the Stokes-Fourier equations: an $L^2$ analysis of the Boltzmann--Grad limit.
{\it Annals PDE} 3(2), 2017.

\bibitem{BGSR3} T. Bodineau, I. Gallagher and L. Saint-Raymond. Derivation of an Ornstein-Uhlenbeck process for a massive particle in a rarified gas of particles. {\it Ann. IHP} 19(6):1647-1709, 2018.



\bibitem{BGSRS18}
T. Bodineau, I. Gallagher, L. Saint-Raymond and S. Simonella. One-sided convergence in the Boltzmann-Grad limit.
{\it Ann.\;Fac.\;Sci.\;Toulouse Math. Ser. 6} 27(5):985-1022, 2018.  

\bibitem{BGSRSshort}
T. Bodineau, I. Gallagher, L. Saint-Raymond and S. Simonella. Fluctuation Theory in the Boltzmann--Grad Limit. {\it J. Stat. Phys.} 180:873-895, 2020.

\bibitem{BGSRScov}
T. Bodineau, I. Gallagher, L. Saint-Raymond and S. Simonella. 
Long-time correlations for a hard-sphere gas at equilibrium. 
{\it Comm. Pure Appl. Math.}, in press.

\bibitem{BGSRStcl}
T. Bodineau, I. Gallagher, L. Saint-Raymond and S. Simonella. 
Long-time derivation at equilibrium of the fluctuating Boltzmann equation. 
arXiv:2201.04514.

\bibitem{BGSRSsurv}
T. Bodineau, I. Gallagher, L. Saint-Raymond and S. Simonella. 
Dynamics of dilute gases: a statistical approach. 
arXiv:2201.10149.


\bibitem{bouchet} F. Bouchet. Is the Boltzmann Equation Reversible? A Large Deviation Perspective on the Irreversibility Paradox. {\it J. Stat. Phys.}, 2020.

\bibitem{BH77}
W. Braun and K. Hepp. The Vlasov dynamics and its fluctuations in the $1/N$ limit of interacting classical particles. {\it Commun.Math. Phys.} 56:101-113, 1977. 

\bibitem{Ce72}
C. Cercignani. On the Boltzmann equation for rigid spheres. {\it Transp. Theory Stat. Phys.} 2:211-225, 1972.

\bibitem{CGP97}
C. Cercignani, V. I. Gerasimenko and D. Y. Petrina. 
Many-particle dynamics and kinetic equations. 
{\it Math. and its Appl.} 420, Kluwer Academic Publishers Group, Dordrecht, 1997.

\bibitem{CIP94}
C. Cercignani, R. Illner and M. Pulvirenti. {\it The Mathematical Theory of Dilute Gases}.
Applied Math. Sci. 106, Springer--Verlag, New York, 1994.


\bibitem{daley}  D.J. Daley and D. Vere-Jones.  {\it An introduction to the theory of point processes. Vol. II. General theory and structure}. Probability and its Applications, Springer, New York, 2008.

\bibitem{dembozeitouni} 
A. Dembo and O. Zeitouni.
{\it Large Deviations Techniques and Applications}. Springer, 2010.

\bibitem{denlinger} R. Denlinger. The propagation of chaos for a rarefied gas of hard spheres in the whole space. {\it Arch. Rat. Mech. and Anal.} 229(2):885-952, 2018.


\bibitem{Dolmaire} Th. Dolmaire. About Lanford's theorem in the half-space with specular reflection. arxiv:2102.05513, {\it to appear in KRM}.   

\bibitem{DV} M. D. Donsker and S. R. S. Varadhan. 
Asymptotic evaluation of certain Markov process expectations for large time. 
{\it Comm. Pure Appl. Math.} 28:1-47, 1975.

\bibitem{EC81} M.H. Ernst and E.G.D. Cohen. Nonequilibrium Fluctuations in $\mu$ Space.
{\it J. Stat. Phys.} 25(1):153-180, 1981.

\bibitem{EGM11}
R. Esposito, Y. Guo and R. Marra. 
Validity of the Boltzmann equation with an external force. 
{\it Kin. \& Rel. Mod.} 4(2):499-515, 2011.

\bibitem{GSRT}
I. Gallagher, L. Saint Raymond and B. Texier. From Newton to Boltzmann: hard spheres and short-range potentials. {\it Zurich Lect. in Adv. Math.} 18, EMS, 2014.

\bibitem{GG18}
V.I. Gerasimenko and I.V. Gapyak.
Low-density Asymptotic Behavior of Observables of Hard Sphere Fluids.
{\it Advances in Mathematical Physics}, 2018.

\bibitem{GG21}
V. I. Gerasimenko and I. V. Gapyak. Boltzmann-Grad asymptotic behavior of collisional dynamics.
{\it Rev. Math. Phys.} 33(2):32, 2021.

\bibitem{GG22}
V. I. Gerasimenko and I. V. Gapyak. Propagation processes of correlations of hard spheres. 
arXiv:2111.13940.



\bibitem{GV79}
J. Ginibre, G. Velo. 
The classical field limit of scattering theory for non-relativistic many-boson systems. I. 
{\it Comm. Math. Phys.} 66:37-76, 1979.

\bibitem{Gr49} H. Grad. On the kinetic theory of rarefied gases. 
{\it Comm. Pure and App. Math.} 2(4):331-407, 1949. 

\bibitem{Gr58}
H. Grad. Principles of the kinetic theory of gases. In: {\it Handbuch der Physik}
12:205-294, Springer, 1958.

\bibitem{Gradshteyn} I.-S. Gradshteyn and I.-M. Ryzhik. 
{\it Table of integrals, series and products}.
Alan Jeffrey and Daniel Zwillinger ed., Elsevier, 2007.

\bibitem{HL73}
K. Hepp and H. Lieb.
Phase Transitions in Reservoir-Driven Open Systems with Applications to Lasers and Superconductors.
{\it Helv. Phys. Acta} 46, 1973.

\bibitem{Heydecker21}
D. Heydecker. 
Large deviations of Kac's conservative particle system and energy non-conserving solutions to the Boltzmann equation: a counterexample to the predicted rate function. arXiv:2103.14550.

\bibitem{holley1978generalized}
R. Holley and  D. Stroock.
Generalized Ornstein-Uhlenbeck processes and infinite particle branching Brownian motions.
{\it Publications Res. Inst. for Math. Sci.} 14(3):741-788, 1978.

 \bibitem{IP89} R. Illner and M. Pulvirenti. 
Global Validity of the Boltzmann equation for a Two-- and Three--Dimensional
Rare Gas in Vacuum: Erratum and Improved Result. {\it Comm. Math. Phys.} 121:143-146, 1989.

\bibitem{gibbspp} S. Jansen. {\it Gibbsian Point Processes}. Online available at: \\ http://www.mathematik.uni-muenchen.de/~jansen/gibbspp.pdf.

\bibitem{Kac56}
M. Kac. 
Foundations of kinetic theory.
Proceedings of the Third Berkeley Symposium onMathematical Statistics and
Probability, University of California Press, Berkeley and Los Angeles, 1956.

\bibitem{KL76}
M. Kac and  J. Logan. Fluctuations and the Boltzmann equation. {\it Phys. Rev. A} 13:458-470, 1976.

\bibitem{vK74}
N.G. van Kampen. Fluctuations in Boltzmann's equation. {\it Phys. Rev.} 50A(4), 1974.

\bibitem{MT12}
K. Matthies and F. Theil. A semigroup approach to the justification of kinetic theory. 
{\it SIAM J. Math. Anal.} 44(6):4345-4379, 2012.

\bibitem{nishida2} T. Kano and T. Nishida.  Sur les ondes de surface de l'eau avec une justification math\'ematique des \'equations des ondes en eau peu profonde. (French)
{\it J. Math. Kyoto Univ.} 19(2):335-370, 1979.
 
\bibitem{Ki75} F. King. {\it BBGKY Hierarchy for Positive Potentials}. Ph.D. Thesis, Dep. of Math., Univ. California, Berkeley, 1975.

\bibitem{La75} O.E. Lanford. Time evolution of large classical systems. In: Dynamical systems, theory and applications, {\it Lect. Notes in Phys,} 38, J. Moser ed., Springer--Verlag, Berlin, 1975.

\bibitem{LeBihan21}
C. Le Bihan. Boltzmann-Grad limit of a hard sphere system in a box with diffusive boundary conditions. 
{\it Disc. Cont. Dyn Syst.}, in press.

\bibitem{Leonard95}
C. L\'{e}onard. On large deviations for particle systems associated with spatially homogeneous Boltzmann type equations. {\it Probab. Theory Relat. Fields}, 101(1), 1995.

\bibitem{LMN16} J. Lukkarinen, M. Marcozzi and A. Nota.
Summability of Connected Correlation Functions of Coupled Lattice Fields.
{\it J. Stat. Phys.} 171:189-206, 2018.

\bibitem{meleard}
S. M\'el\'eard. Convergence of the fluctuations for interacting diffusions with jumps associated with {B}oltzmann equations. {\it Stochastics and Stoch. Rep.} 63(3-4):195-225, 1998.

\bibitem{nirenberg} L. Nirenberg, An abstract form of the nonlinear Cauchy-Kowalewski theorem. {\it Jour. Diff. Geom.} 6:561-576, 1972.
 

\bibitem{nishida} T. Nishida, A note on a theorem by Nirenberg. {\it Jour. Diff. Geom.} 12:629-633, 1977.
 
\bibitem{Pe67} O. Penrose. Convergence of fugacity expansions for classical systems. {\it Statistical mechanics: foundations and applications},  A. Bak ed., Benjamin, New York, 1967.

 \bibitem{PU09} O. Poghosyan and D. Ueltschi. 
Abstract cluster expansion with applications to statistical mechanical systems.
{\it J. Math. Phys.} 50, 2009.

\bibitem{PSS17} M. Pulvirenti, C. Saffirio and S. Simonella.
On the validity of the Boltzmann equation for short range potentials. 
{\it Rev. Math. Phys.} 26(2), 2014.

\bibitem{PS17}
M. Pulvirenti and S. Simonella. The Boltzmann-Grad limit of a hard sphere system:
analysis of the correlation error. {\it Inventiones Math.} 207(3):1135-1237, 2017.

\bibitem{PS21}
M. Pulvirenti and S. Simonella. 
On the cardinality of collisional clusters for hard spheres at low density. {\it Disc.\,\&\,Cont.\,Dyn.\,Syst.} 41(8):3903-3914, 2021.

\bibitem{PT} E. Pulvirenti and D. Tsagkarogiannis.  Finite volume corrections and decay of correlations in the Canonical Ensemble, {\it J. Stat. Phys.} 5:1017-1039, 2015. 

\bibitem{Rez}
F. Rezakhanlou.  Equilibrium fluctuations for the discrete Boltzmann equation. 
{\it Duke math. Jour.} 93(2): 257-288, 1998.

\bibitem{Rez2}
F. Rezakhanlou. Large deviations from a kinetic limit.
{\it Annals of Prob.} 26(3):1259-1340, 1998.

\bibitem{RezLect}
F. Rezakhanlou. {\it Lectures on the Large deviation Principle}. Online available at: \\ http://math.berkeley.edu/rezakhan


\bibitem{RezLNM} 
F. Rezakhanlou and C. Villani. Entropy Methods for the Boltzmann Equation. 
{\it Lect. Notes Math.} 1916, Springer, 2001.


\bibitem{Ru69} D. Ruelle. {\it Statistical Mechanics. Rigorous Results}. W.A. Benjamin Inc., NewYork, 1969.


\bibitem{Scola}
G. Scola. 
Local Moderate and Precise Large Deviations via Cluster Expansions. 
{\it J. Stat. Phys.} 183(2), 2021.

\bibitem{S} S. Simonella. 
Evolution of correlation functions in the hard sphere dynamics, 
{\it J. Stat. Phys.} 155(6):1191-1221, 2014.

\bibitem{S81} H. Spohn.
Fluctuations Around the Boltzmann Equation.
{\it J. Stat. Phys.} 26(2), 1981.

\bibitem{S83}
H. Spohn. Fluctuation theory for the Boltzmann equation. {\it Nonequilibrium Phenomena I: The Boltzmann Equation}, Lebowitz and Montroll ed., North-Holland, Amsterdam, 1983.

\bibitem{S2}
H. Spohn. {\it Large scale dynamics of interacting particles}. Texts and Monographs in
Physics, Springer, Heidelberg, 1991.

\bibitem{Tanaka82}
H. Tanaka. 
Fluctuation Theory for Kac's One-Dimensional Model of Maxwellian Molecules. 
{\it The Indian Jour. Stat.}, Series A, 44(1)1:23-46, 1982.

 

\bibitem{Treves2} F. Treves, An abstract nonlinear Cauchy-Kowalewska theorem, {\it Trans. American. Math. Soc} 150:77-92, 1970.

\bibitem{Uchiyama83}
K. Uchiyama. A fluctuation problem associated with the Boltzmann equation for a gas of molecules with a cutoff potential. {\it Japan J. Math.} 9:27-53, 1983.

\bibitem{Uchiyama88}
K. Uchiyama. Fluctuations in a Markovian system of pairwise interacting particles.
{\it Probab. Theory Relat. Fields} 79:289-302, 1988.

\bibitem{ukai-CK} S. Ukai. The Boltzmann-Grad Limit and Cauchy-Kovalevskaya Theorem, {\it  Japan J. Indust. Appl. Math.} 18:383-392, 2001.


\bibitem{Varadhan} S. Varadhan. Stochastic processes. {\it American Math. Soc.} 16, 2007.


\end{thebibliography}
\end{document}